\documentclass[a4paper,11pt,leqno]{article}
\makeatletter
\usepackage{amsfonts,delarray,amssymb,amsmath,amsthm,a4,a4wide}
\usepackage{amsmath, euscript}
\usepackage[ansinew]{inputenc}

\newtheorem{theo}{Theorem}
\newtheorem{pro}{Proposition}[section]
\newtheorem{lem}[pro]{Lemma}
\newtheorem{coro}[pro]{Corollary}
\newtheorem{remark}[pro]{Remark}
\newtheorem{defi}[pro]{Definition}
\newtheorem{conj}{Conjecture}
\newtheorem{example}{Example}

\def\Xint#1{\mathchoice
   {\XXint\displaystyle\textstyle{#1}}%
   {\XXint\textstyle\scriptstyle{#1}}%
   {\XXint\scriptstyle\scriptscriptstyle{#1}}%
   {\XXint\scriptscriptstyle\scriptscriptstyle{#1}}%
   \!\int}
\def\XXint#1#2#3{{\setbox0=\hbox{$#1{#2#3}{\int}$}
     \vcenter{\hbox{$#2#3$}}\kern-.5\wd0}}
\def\dashint{\Xint-}

\DeclareMathOperator{\supp}{Supp}
\def\lb{\langle}
\def\rb{\rangle}
\def\Lj{{\Gamma(j)}}
\def\Lji{{\Gamma_(J_i)}}
\def\tq{{\frac34}}
\def\LK{{\mathcal L_\ep}}
\def\LKk{{\mathcal L_{\ep,i}}}
\def\jnint{{j_\ep}}

\def\({\left(}
\def\){\right)}
\def\1{\mathbf{1}}
\def\admissible{{\mathcal{A}_{m_\lambda}}}

\def\ainfty{{\mathcal{A}_1}}
\def\ai{{A_{1,\ep}}}

\def\ao{{A_{0,\ep}}}
\def\a{\alpha}

\def\bg{{\bar g}}
\def\bh{{\underline h}}

\def\bix{{\mathbf x}}
\def\bj{{\bar \jmath}}

\def\bnu{{\bar\nu}}

\def\bo{\partial \Omega}\def\D{\displaystyle}
\def\boe{{\bar\omega_\ep'}}

\def\B{{\mathcal{B}}}

\def\Ce{c_{\ep, N}}

\def\cl{{m_\lambda}}
\def\com{{\lambda_\Omega}}
\def\co{m_\ep}

\def\curl{{\rm curl\,}}

\def\diam{\mathrm{diam}\ }
\def\dist{\text{dist}\ }
\def\div{\mathrm{div} \ }
\def\dl{\hal \lep}
\def\dt0{{{\frac{d}{dt}}_{|t=0}}}

\def\D{\displaystyle}

\def\ella{{L_m}}
\def\el{{\ell_\ep}}
\def\elp{{\ell_\ep'}}
\def\ep{\varepsilon}
\def\extb{{h_{\rm out}}}

\def\fep{{f_\ep}}
\def\fe{{F_\ep}}

\def\F{{\mathcal{F}}}

\def\gep{{g_\ep}}
\def\gfep{{\mathbf f_\ep}}

\def\gF{{\mathbf F}}
\def\gf{{\mathbf f}}
\def\gom{{G_\Omega}}

\def\hal{\frac{1}{2}}

\def\hciii{{H_{c_3}}}
\def\hcii{{H_{c_2}}}
\def\hci{{H_{c_1}}}
\def\HE{\widehat{E}}
\def\he{{h_{\rm ex}}}
\def\trl{\mathrm{triangular}}

\def\hint{{h_{\rm int}}}

\def\hl{{H_\Lambda}}
\def\hm{H_m}
\def\homb{{\underline h_0}}
\def\hom{{h_0}}
\def\ho{{h_{0,\ep}}}

\def\indic{\mathbf{1}}
\def\indic{\mathbf{1}}

\def\io{\int_{\Omega}}

\def\je{G_\ep}
\def\jie{{F_\ep}}

\def\lattices{{\mathcal{L}}}

\def\lep{{|\mathrm{log }\ \ep|}}
\def\lip{{\text{\rm Lip}}}

\def\llep{\log\ \lep}
\def\loc{{\text{\rm loc}}}
\def\Lp{{L^p_\loc(\mr^2,\mr^2)}}
\def\Lploc{{L^p_\loc}}

\def\L{\log \ \frac{1}{\ep \sqrt{\he}}}
\def\l|{\left|}

\def\mc{\mathbb{C}}
\def\me{m_{\ep, N}}
\def\moep{m_{0,\ep}}

\def\mn{\mathbb{N}}
\def\mn{\mathbb{N}}
\def\mr{\mathbb{R}}
\def\muo{{\mu_{0,\ep}}}

\def\mz{\mathbb{Z}}
\def\nab{\nabla}

\def\nm{N_0^-}
\def\np{\nab^{\perp}}
\def\ns{N_0^+}

\def\oma{{\omega_m}}
\def\omo{{\omega_{0,\ep}}}
\def\omQ{{\omega_Q}}
\def\om{\Omega}
\def\pae{{A_\ep'}}
\def\pa{{A'}}

\def\pep{{\varepsilon'}}

\def\pfep{{F_\ep'}}

\def\pge{{g_\ep'}}

\def\pje{{j_\ep'}}
\def\pje{{j_\ep'}}
\def\plep{{|\log\ep'|}}
\def\pmue{{\mu_\ep'}}
\def\pmue{{\mu_\ep'}}

\def\pOe{{\Omega_\ep'}}
\def\poe{{\omega_\ep'}}
\def\pue{{u_\ep'}}
\def\pu{{u'}}
\def\pxe{{x_\ep'}}

\def\pzep{{\zeta_\ep'}}

\def\p{\partial}

\def\radon{\mathcal{M}}

\def\ro{\rho}
\def\r|{\right|}
\def\sd{\bigtriangleup}
\def\sm{\setminus}
\def\som{{S_\Omega}}
\def\sql{\mathrm{square}}
\def\supp{\text{Supp}}

\def\tho{h_{\ep, N}}

\def\tej{{\tilde \jmath_\ep}}
\def\hej{{\hat \jmath_\ep}}
\def\tix{{\tilde x}}

\def\tls{{\mathbb{T}_{\Lambda^*}}}
\def\tl{{\mathbb{T}_{\Lambda}}}

\def\tmuo{\mu_{\ep, N}}
\def\toe{\omega_{\ep, N}}
\def\tom{{\tilde\omega_\ep}}
\def\oh{{\hat\omega_\ep}}

\def\T{{\mathbb{T}}}

\def\ul{{U_\Lambda}}
\def\URC{{\mathbf U_{R+C}}}
\def\URc{{\mathbf U_{R-C}}}
\def\UcR{{\mathbf U_{cR}}}
\def\URI{{\mathbf U_{R+1}}}
\def\UR{{\mathbf U_R}}

\def\veta{{\mathcal{V}_\eta}}
\def\vp{\varphi}
\def\VR{{\mathbf V_R}}
\def\VRc{{\mathbf V_{R-C}}}

\def\vu{{\vec u}}

\def\vv{{\vec v}}

\def\xom{{x_0}}

\def\zep{{\zeta_\ep}}

\def\tix{{\tilde x}}

\def\lue{{\tilde u_\ep}}
\def\lAe{{\tilde A_\ep}}
\def\lom{{\tilde\om}}
\def\lG{{\tilde \je}}
\def\leep{{\tilde \ep}}

\def\lhe{{\tilde h_{\rm ex}}}
\def\lL{\log \ \frac{1}{\leep \sqrt{\lhe}}}
\def\lho{{h_{\leep, \tilde N}}}
\def\lmuo{{\mu_{\leep, \tilde N}}}
\def\facteur{{\sigma}}
\def\oms{{\om_\facteur}}
\def\probas{{\EuScript P}}

 \numberwithin{equation}{section}
\title{From the Ginzburg-Landau Model to Vortex Lattice Problems}
\author{Etienne Sandier and Sylvia Serfaty}
\date{Second version, January 2012} 
\begin{document}
\maketitle

%\tableofcontents

\begin{abstract}

We introduce a 
 ``Coulombian renormalized energy" $W$ which  is a  logarithmic
 type of interaction between points in the plane, computed by a ``renormalization."  We prove various of its properties, such as the existence of minimizers,  and show in particular, using results from number theory,  that among lattice configurations the triangular lattice is the unique minimizer. Its  minimization in general remains open.

Our motivation is the study of  minimizers of the two-dimensional Ginzburg-Landau energy with applied magnetic field, between the first and second critical fields $\hci$ and $\hcii$. In that regime, minimizing configurations exhibit densely packed  triangular vortex lattices, called Abrikosov lattices. 
We derive, in some asymptotic regime, $W$ as a $\Gamma$-limit of the Ginzburg-Landau 
energy. More precisely we show that the vortices of minimizers   of Ginzburg-Landau,  blown-up at  a suitable scale, converge to minimizers of $W$, thus  providing a first rigorous hint at the Abrikosov lattice. This is a next order effect compared  to the mean-field type results we  previously established.

 The derivation of $W$  uses  energy methods: the framework of   $\Gamma$-convergence, and an abstract scheme  for obtaining lower bounds for ``2-scale energies" via the ergodic theorem that we introduce.
\end{abstract}

\noindent
{\bf keywords:} Ginzburg-Landau, vortices, Abrikosov lattice, triangular lattice,  renormalized energy, Gamma-convergence.

\section{Introduction}
In this paper, we are interested in deriving a ``Coulombian renormalized energy" from the Ginzburg-Landau model of superconductivity.
We will start by defining and presenting the  renormalized energy in Section \ref{sec14}, then state some results about it in Section \ref{period-intro}. In Section \ref{secergo}, we then  present an abstract method for lower bounds for two-scale energies using ergodic theory. In Sections \ref{sec1.1}--\ref{sec1.10} we turn to the Ginzburg-Landau model, and give our main results about it as well as ingredients for the proof.

\subsection{The Coulombian  renormalized energy $W$ }\label{sec14}
 The interaction energy $W$ that we wish to define just below  is a natural energy for the Coulombian interaction of charged particles in the plane screened by  a uniform background: it could  be called a ``screened Coulombian renormalized energy". It  can be seen  in our context as  the analogue for an
  infinite number of points in $\mr^2$   of the renormalized
  energy $W$ introduced  in Bethuel-Brezis-H\'elein  \cite{bbh} for a finite number of points
   in a bounded domain, or of the Kirchhof-Onsager function.  We believe that this energy is quite  ubiquitous in all problems that have an underlying Coulomb interaction:  it already  arises in  the  study of weighted Fekete sets and of the statistical mechanics of  Coulomb gases and  random matrices  \cite{ss6},
     as well as a limit in some parameter regime for the Ohta-Kawasaki model \cite{ms}.   In \cite{ss7} we introduce a one-dimensional analogue (a renormalized logarithmic interaction for points on the line) which we also connect to one-dimensional Fekete sets as well as  ``log gases" and random matrices.
     
         We will discuss more at the end of this subsection and in the next, but  let us first give the  precise definition.
       %Here $\lambda\in(1/2,+\infty]$ and $\cl = 1-1/2\lambda$. In particular if
        %$\lambda= +\infty$ then $\cl = 1$.

In all the paper,  $B_R$ denotes the ball centered at $0$ and of radius $R$, and $|\cdot|$ denotes the area of a set.

\begin{defi}\label{defA} Let $m $ be a positive number.
Let  $j$
be a vector field in $\mr^2$. We say $j$
belongs to the admissible class
$\mathcal{A}_m $ if
\begin{equation}\label{eqj}
\curl j = \nu
- m, \qquad \div j=0,\end{equation} where $\nu$ has the form
$$\nu= 2\pi \sum_{p \in\Lambda} \delta_{p}\quad \text{ for
some discrete set} \  \Lambda\subset\mr^2,$$ and
\begin{equation}\label{densbornee}
\frac{    \nu (B_R )  } {|B_R|}\quad \text{ is bounded by a constant independent of $R>1$}.
\end{equation}
 \end{defi}

For any family of sets $\{\UR\}_{R>0}$ in $\mr^2$ we  use the
notation $\chi_{\UR}$ for positive  cutoff functions satisfying, for some
constant $C$ independent of $R$,
\begin{equation}
\label{defchi} |\nab \chi_{\UR}|\le C, \quad \supp(\chi_{\UR})
\subset \UR, \quad \chi_{\UR}(x)=1 \ \text{if } d(x, \UR^c) \ge
1.\end{equation}
We will always implicitly assume that $\{\UR\}_{R>0}$ is an increasing family of bounded open sets, and we will use the following set of additional assumptions:
\begin{itemize}
\item
\begin{equation}\label{hypsets} \{\UR\} \ \text{is   a  Vitali family
 and} \quad
\lim_{R\to +\infty} \frac{|(\lambda + \UR)\sd \UR
|}{|\UR|} = 0. \end{equation}
for any $\lambda\in\mr^2$. Here, a  Vitali family (see \cite{nmriv}) means  that the intersection of the closures is $\{0\}$, that $R\mapsto|\UR|$ is left continuous, and that $|\UR - \UR|\le C|\UR|$
for some constant $C>0$ independent of $R$.
\item
There exists $\theta<2$ such that for any $R>0$
\begin{equation}\label{hypsetsbis} \UR+B(0,1)\subset \URC,\quad \URI\subset \UR+ B(0,C),\quad |\URI\sm\UR| = O(R^\theta).\end{equation}
\end{itemize}

\begin{defi} \label{def1.2}The Coulombian renormalized energy $W$ is defined,
for $j \in \mathcal{A}_m$,  by
\begin{equation} \label{Wroi} W(j)= \limsup_{R \to \infty}
\frac{W(j, \chi_{B_R})}{|B_R|} ,
\end{equation} where  for any function $\chi$ we denote
\begin{equation}\label{WR}W(j, \chi) = \lim_{\eta\to 0} \(
\hal\int_{\mr^2 \backslash \cup_{p\in\Lambda} B(p,\eta) }\chi
|j|^2 +  \pi \log \eta \sum_{p\in\Lambda} \chi (p) \).
\end{equation}
We similarly define  the renormalized energy relative to the family
$\{\UR\}_{R>0}$ by
\begin{equation}\label{WU} W_U(j)= \limsup_{R \to \infty}
\frac{W(j, \chi_{\UR})}{|\UR|} .
\end{equation}
\end{defi}

Let us make several remarks about the definition.

\begin{itemize}
\item[1.] We will see in Theorem \ref{thw}  that the value  of $W$ does not depend on $\{\chi_{B_R}\}_R$
 as long as  it satisfies \eqref{defchi}. The corresponding
 statement holds for $W_U$ under the assumptions \eqref{hypsets}--\eqref{hypsetsbis}.

\item[2.] Since in the neighborhood of $p\in\Lambda$ we have $\curl j = 2\pi\delta_p - 1$, $\div j = 0$, we have near $p$ the decomposition $j(x) = \nab^\perp \log|x-p| + f(x)$ where $f$ is smooth, and it easily follows that the limit \eqref{WR} exists.  It also follows that $j$ belongs to $L^p_\loc$ for any $p<2$.
\item[3.] From \eqref{eqj} we have  $j= -\np H$  for some $H$, and then  $$-\Delta H= 2\pi \sum_{p\in\Lambda} \delta_p -m.$$ Then the energy in (\ref{WR}) can be seen as the (renormalized) interaction energy between the ``charged particles" at $p \in \Lambda$ and between them and a constant background $-m $. We prefer to take $j = -\np H$ as the unknown, though, because it is related to the superconducting current $j_\ep$.

\item[4.] We will see in Theorem \ref{thw} that
 the  minimizers and the value of the minimum of $W_U$ are
 independent of $U$, provided \eqref{hypsets} and \eqref{hypsetsbis} hold.
  However there are examples of
admissible $j$'s (nonminimizers)  for which $W_U(j)$ depends on the family of shapes $\{ \UR\}_{R>0}$ which
is used.

\item[5.] Because the number of points is infinite,
 the interaction over large balls needs to be normalized
 by the volume and thus  $W$ does not feel compact perturbations of the configuration of points. Even though the interactions are long-range, this is not difficult to justify rigorously.

\item[6.] The  cut-off function $\chi_R$ cannot simply be replaced by the characteristic function of  $B_R$ because for every $p\in\Lambda$
$$\lim_{\substack{R\to|p|\\R<|p|}} W(j,\indic_{B_R}) = +\infty,\quad \lim_{\substack{R\to|p|\\R>|p|}} W(j,\indic_{B_R}) = -\infty.$$

\item[7.] It is easy to check that if $j$ belongs to $\mathcal{A}_m$ then
$j'=\frac{1}{\sqrt{m}} j(\cdot /\sqrt m)$ belongs to $\ainfty $ and
\begin{equation}\label{minalai1}
W(j)= m \(W(j') - \frac14 \log m\).\end{equation} 
%Consequently  if $j$
% is a minimizer of  $W$ over $\mathcal{A}_m$, then $j'$ minimizes  $W$ over $\ainfty$ and
%\begin{equation}\label{minalai}
%\min_{\mathcal{A}_m}  W = m \( \min_{\ainfty} W -\frac14\log m \),\end{equation}
 so we may reduce to  the study  of $W$ over $\ainfty$.
\end{itemize}

When  the set of points $\Lambda$ is  periodic with respect to some lattice $
\mz \vu+ \mz \vv$ then  it can be viewed as a set of $n$ points $a_1, \cdots , a_n$ over the torus
$\T_{(\vu,\vv)}= \mr^2/(\mz \vu +\mz\vv)$. There also exists a unique periodic (with same period) $j_{\{a_i\}}  $ with mean zero and satisfying \eqref{eqj} for some $m$
which from
    \eqref{eqj} and the periodicity of $j_{\{a_i\}}$ must be equal
    to  $2\pi n$ divided by  the surface of the periodicity cell.
 Moreover $j_{\{a_i\}}  $ minimizes  $W$ among $(\vu, \vv)$-periodic solutions of \eqref{eqj} (see Proposition~\ref{renper}).
The  computation of $W$ in this setting where both $\Lambda $ and $j$ are periodic is quite  simpler (the need for the limit $R\to \infty$ and the cutoff function disappear). By the scaling formula \eqref{minalai1}, we may reduce to working in $\mathcal{A}_n$ in a situation where the volume of the torus is $ 2\pi $. Then we will see in Section \ref{rem52} the following
\begin{lem}\label{lemperiointro}With the above notation, we have 
\begin{equation}\label{Wcasper}
W(j_{\{a_i\}}   )= \hal\sum_{i \neq j} G(a_i- a_j) +   n c_{(\vu, \vv)}\end{equation}
where $c_{(\vu, \vv)}$ is a constant depending only on  $(\vu, \vv)$ and $G$ is the Green function of the torus with respect to its volume form, i.e. the solution to
$$- \Delta G(x) = 2\pi \delta_0 - 1 \quad \text{in } \ \T_{(\vu, \vv)}.$$ Moreover, $j_{\{a_i\}}$ is the minimizer of $W(j)$ among all $\T_{(\vu,\vv)}$-periodic $j$'s satisfying \eqref{eqj}.
\end{lem}
\begin{remark}
The Green function of the torus admits an explicit Fourier series expansion, through this we can obtain a more explicit formula for the right-hand side of \eqref{Wcasper}:
\begin{equation}\label{formexpl}
W(j_{\{a_i\}} )= \hal \sum_{i \neq j} \sum_{p \in (\mz \vu + \mz \vv)^*\sm \{0\}} \frac{e^{2i\pi p \cdot (a_i- a_j)}  }{4\pi^2 |p|^2 } + \frac{n}{2}
 \lim_{x\to 0} \( \sum_{p \in (\mz \vu + \mz \vv)^*\backslash
\{0\} } \frac{ e^{2i\pi p \cdot x}}{ 4\pi^2 |p|^2}  + \log
|x|\) \end{equation} where $^*$ refers to the dual of a lattice.
\end{remark}

The function $\sum_{i \neq j} G(a_i - a_j)$ is the sum of pairwise Coulombian interactions between particles on a torus. It arises for example in number theory (Arakelov theory), see \cite{langara} p. 150, where a result attributed to  Elkies is  stated:  $\sum_{i\neq j} G(a_i-a_j) \ge - \frac{n}{4} \log n +O(n)$ (on any Riemann surface of genus $\ge 1$). Note that we can  retrieve this estimate  in the case of the torus by using the fact that  $\min_{\ainfty} W$ is finite and formula \eqref{minalai} with $m=n$.

So conversely, another way of looking at  our energy $W$ is that it provides a way of computing an analogue of $\sum_{i\neq j} G(a_i- a_j) $ in an infinite-size domain.

\subsection{Results and conjecture  on the renormalized energy}\label{period-intro}
The following theorem summarizes the basic results about the minimization of $W$. Note that by the scaling relation \eqref{minalai1} we may reduce to the case of $\ainfty$, and we have
\begin{equation}\label{minalai}
\min_{\mathcal{A}_m}  W = m \( \min_{\ainfty} W -\frac14\log m \).\end{equation}

\begin{theo}\label{thw} Let $W$ be as in Definition \ref{def1.2}.
\begin{enumerate}
\item  Let  $\{\UR\}_{R>0}$ be a family of sets satisfying \eqref{hypsets}--\eqref{hypsetsbis}, then for any $j\in \ainfty$,  the value of $W_U(j)$ is independent of the choice of $\chi_\UR $ in its definition as long as it satisfies \eqref{defchi}. \item $W$ is Borel measurable on $L^p_{loc} (\mr^2, \mr^2)$, $p<2$.
\item  $\min_{\mathcal{A}_1} W_U$   is achieved and finite, and it is independent of the choice of $U$, as long as $\{\UR\}$ satisfies \eqref{hypsets}--\eqref{hypsetsbis}.
  \item There exists a minimizing sequence $\{j_n\}_{n \in \mn}$ for $\min_{\ainfty} W$  consisting of vector-fields which are periodic (with respect to a square lattice of sidelength $\sqrt{2\pi n}$).
\end{enumerate}
\end{theo}
 The question of identifying the minimum and minimizers of $W$ seems very difficult. In fact   it is natural to expect  that the
triangular lattice (of appropriate volume) minimizes  $W$ over any $\mathcal{A}_m$, we will come back to this below.
 We show here a weaker but nontrivial result: the triangular
 lattice is the unique minimizer among lattice configurations.

When the set of points $\Lambda$ itself  is a lattice, i.e. of the form $\mz\vu\oplus\mz\vv$,
  denoting by $j_\Lambda $ the $j$ which is as in Lemma \ref{lemperiointro}, that lemma   shows that $W(j_\Lambda) $ is equal to $c_{(\vu, \vv)}$ and  only depends on  the lattice $\Lambda$.
   We will denote it in this case by $W(\Lambda)$.
For the sake of generality, we state the result for any volume normalization:
\begin{theo}\label{th2}
Let $\lattices = \{\Lambda\mid \text{$\Lambda$ is a  lattice  and $j_\Lambda\in\mathcal{A}_m$}\}$. Then the  minimum of  $\Lambda\mapsto W(\Lambda)$ over $\lattices$  is achieved  uniquely, modulo  rotation, by the triangular lattice
$$\Lambda_m  = \sqrt{\frac{4\pi }{m\sqrt 3}}\((1,0)\mz\oplus\(\hal,\frac{\sqrt3}2\)\mz\).$$
\end{theo}
%\footnote{j'ai change le facteur devant car il semblait faux}
The normalizing factor ensures that the periodicity cell has area
$2\pi/m$, or equivalently that $\Lambda_m\in\mathcal{A}_m $.
\begin{remark} The  value of $W$ for the triangular lattice  with $m=1$ estimated numerically from  formula \eqref{eis} or \eqref{Weta} is $\simeq -0.2011 $. For the square lattice it is  $\simeq - 0.1958$.
\end{remark}
Theorem \ref{th2} is proven by expressing $j_\Lambda$ using Fourier
 series. Then  minimizing $W$ over $\lattices$ is a
 limit case
of minimizing the Epstein  $\zeta$ function
$$\Lambda \mapsto \sum_{p \in \Lambda\backslash \{0\}}
\frac{1}{|p|^{2+x}} $$ over $\lattices$ when  $x \to 0$.
 This question was answered by  Cassels, Rankin, Ennola, Diananda, \cite{cassels,rankin,ennola,ennola2,diananda}. In a  later self-contained paper  \cite{montgomery},
Montgomery shows that  the $\zeta$-function  is actually the Mellin transform of another  classical function  from number theory: the Theta function
\begin{equation}\label{theta}
\theta_\Lambda (\a)= \sum_{p \in \Lambda} e^{-\pi \a|p|^2}.\end{equation}
He then deduces the minimality of the $\zeta$ function at the triangular lattice from the corresponding result for the $\theta$ function (we will give more details in Section \ref{latticecase}).

The main open question  is naturally to show  that the triangular lattice
is a minimizer  among all configurations:
\begin{conj}\label{conj}
 The lattice $\Lambda_m$ being defined as in Theorem \ref{th2},
  we have $W(\Lambda_m)= \min_{\mathcal{A}_m} W.$
  \end{conj}
 Note that a minimizer of $W$
cannot be unique since compact perturbations  do not affect the
value of $W$, as seen in the fifth remark in Section \ref{sec14}. However, it could be that the triangular lattice is the only minimizer which is also
a local minimizer in the following sense: if $\Lambda'$ is any set
of points differing from $\Lambda$ by a finite number of points, and $j_{\Lambda'}$ a corresponding perturbation of $j_\Lambda$,
then
$$\lim_{R \to \infty} W(j_{\Lambda'} , \chi_{B_R})- W(j_{\Lambda},
\chi_{B_R})  \ge 0.$$(Note that here we do not normalize
by $|B_R|$.)

A first motivation for this conjecture comes from  the physics:   in the experiments on superconductors, triangular lattices of vortices, called   Abrikosov lattices, are observed, as predicted by the physicist Abrikosov  from Ginzburg-Landau theory.
 But we shall prove here, cf. Theorem \ref{th1},  that  vortices of minimizers of  the Ginzburg-Landau energy functional (or rather, their associated ``currents") converge in some asymptotic limit to minimizers of $W$,  so if one believes experiments show ground states, then it can be expected that the triangular lattice corresponds to  a minimizer of $W$.
Another motivation for this conjecture is that,   returning to  the expression \eqref{Wroi}--\eqref{WR},
  $W$ can be seen as a renormalized way of computing $\|2\pi \sum_p \delta_p - 1\|_{H^{-1}}$, thus minimizing $W$ over $\ainfty $  is heuristically  like trying to minimize
  $\|2\pi \sum_p \delta_p - 1\|_{H^{-1}}$ over points in the plane,
   or trying to allocate points in the plane in the most uniform
   manner. By analogy with packing problems and other
   crystallisation problems, it seems natural,
   although far out  of reach, that this
   could be accomplished by the triangular lattice. Positive answers are found in \cite{radin} for packing problems, \cite{theil} for some very short-range pairwise interaction potential,  and   one also finds the same conjecture and some supporting arguments in \cite{ck} Section 9,  for  a certain (but different) class of interaction potentials. But we note however  again that here the interaction
   between the points is logarithmic hence long range, in contrast with  these known results. 
Finally,  in dimension 1, the situation is much easier, since we can prove in \cite{ss7} that  the minimum of the  one-dimensional analogue of $W$ is indeed achieved by the perfect lattice $\mz$ (suitably rescaled).

We have seen in Theorem \ref{thw}  that $W$ has a minimizer which is a limit of periodic configurations with large period.  This  will be used crucially for the energy upper bound  on Ginzburg-Landau in Section \ref{bornesup}. It also connects the question of minimizing $W$ in all generality to the simpler one of minimizing it in the periodic setting.  By the formula \eqref{Wcasper} the problem in the periodic setting reduces to minimizing $\sum_{i \neq j}G(a_i- a_j)$ or \eqref{formexpl} over the torus. The points that achieve such minima (on all types of surfaces) are called Fekete points (or weighted Fekete sets) and are important in potential theory, random matrices, approximation, see \cite{safftotik}. Their average distribution  is well-known (see \cite{safftotik}), in our setting it is uniform, but their precise location is more delicate to study (in \cite{ss6} we make progress in that direction and connect them to $W$).
 As already mentioned in the last remark of the previous subsection, estimates on the minimum value of $W$ in this setting are also used in number \medskip  theory.

The energy  $W$ also bears some ressemblance with a
   nonlocal interaction energy related to diblock copolymers, sometimes called ``Ohta-Kawasaki model"
     and studied in particular in a recent paper of Alberti,
     Choksi and Otto \cite{aco},  see also \cite{muratov,gms1}
      (and previous references therein):
      there, one also has  a logarithmic interaction,
      but the Dirac masses are replaced by nonsingular charges,
      and  so no renormalization is needed.   More precisely
     the interaction energy is 
      $\|u-m\|_{H^{-1}}$ where $m$ is a fixed constant in $[-1,1]$
      and $u$ takes values in $\{-1,1\}$, whose $BV$ norm is also
       penalized.
  There, triangular lattice configurations are also observed, it can even be shown \cite{ms} that
  the energy $W$ can be derived as a limit in the regime where $m \to 1$ (in this regime, the problem becomes singular again). Also,
   the analogue result to  our  Theorem \ref{th2}
   is proven for that model  in \cite{chen-oshita}
    using also modular  functions.
      In \cite{aco}  it is proven  that for all $m \in [-1,1]$
       the energy for minimizers is uniformly distributed. Our  study of $W$ in Section \ref{secW} is similar in spirit. Note that  results of equidistribution of energy analogous to \cite{aco} could also most likely  be proven with our  method.

\subsection{Lower bounds for two-scale energies via the ergodic theorem}\label{secergo}

In this subsection we present an abstract framework for proving lower bounds on energies which contain two scales (one much smaller than the other). This framework will then be crucially used in this paper, both for proving the results of $W$ in Theorem \ref{thw} and for obtaining the lower bounds for Ginzburg-Landau in Theorem \ref{th1} and \ref{th1prime}. We believe it is of independent interest as well.
  
  The question is to
deduce from a  $\Gamma$-convergence  (in the sense of De Giorgi)  result at a certain scale a statement
at a larger scale. The framework 
can thus  be seen as a type of $\Gamma$-convergence result for 2-scale energies.  The lower bound is expressed  in terms of a probability  measure, which can be seen as a Young measure on  profiles (i.e. limits of the configuration functions viewed in the  small scale). Following the suggestion of  Varadhan, this
is achieved by using Wiener's multiparameter ergodic theorem, as
stated in  Becker \cite{becker}.     Alberti and M\"uller introduced in
\cite{am}  a different framework for a somewhat similar goal, with a similar notion of Young measure, that they called ``Young measures on micropatterns".
In contrast with Young measures, these measures (just like ours)  are not a probability measure on values taken by the functions, but rather  probability measures on the whole  limiting profile.
The spirit of both frameworks is the same, however Alberti and M\"uller's method did not use the ergodic theorem. It was also a bit more 
general since it dealt with problems that are not homogeneous at the
larger scale but admit slowly varying parameters (however we generalize to such dependence in forthcoming work), and it adressed
the $\Gamma$-limsup aspect as well. Our method is more
rudimentary, but maybe also more flexible. 
  In forthcoming work, we refine it  and include    a version with dependence on the ``slow  (larger scale) variable" in \cite{ss6}, as well as an application to random homogenization in \cite{bss}.

% We shall refine this method and extend it to the study
%of stochastic functionals (for the purpose of homogenization) in
%forthcoming work \cite{bss}.
\medskip

Let $X$ denote a Polish metric space (for reference see \cite{dudley}). When we speak of
measurable functions on $X$ we will always mean Borel-measurable. We
assume that there exists  an $n$-parameter  group of transformations
$\theta_\lambda $ acting continuously on $X$. More precisely we require that
\begin{itemize}
\item[-] For all $u\in X$ and $\lambda, \mu\in \mr^n$, $ \theta_\lambda (\theta_\mu u)=\theta_{\lambda+\mu} u$, $\theta_0 u=u$.\item[-] The map $(\lambda, u) \mapsto \theta_\lambda u $ is measurable  on $\mr^n \times X$.
\item[-] The map $(\lambda,u)\mapsto \theta_\lambda u$ is continuous  with respect to each variable (hence measurable with respect to both).
\end{itemize}
Typically we think of $X$ as a space of functions defined on $\mr^n$
and $\theta$ as the action of translations, i.e. $\theta_\lambda
u(x)=u(x+\lambda)$.

We also consider  a family  $\{\omega_\ep\}_\ep$ of domains of
$\mr^n$ such that  for any $R>0$ and letting $\omega_{\ep,R}=\{ x \in \omega_\ep\mid\dist(x, \p \omega_\ep)
> R\}$, we have
\begin{equation}\label{domaine} \text{$|\omega_\ep|\sim|\omega_{\ep,R}|$ as $\ep\to 0.$}\end{equation}
In particular the diameter of $\omega_\ep$ tends to $+\infty$ as
$\ep\to 0$.

Finally we let $\{f_\ep\}_\ep$ and $f$ be measurable nonnegative
functions on $X$, and assume that for any family $\{u_\ep\}_\ep$
such that
\begin{equation}\label{2.1bis}
\forall R>0, \quad \limsup_{\ep\to 0}\int_{B_R} f_\ep(\theta_\lambda u_\ep) \, d\lambda<+\infty\end{equation}
the following holds:
\begin{enumerate}
\item(Coercivity) $\{u_\ep\}_\ep$ admits a convergent subsequence.
\item ($\Gamma$-liminf)   If %$f_\ep$ $\Gamma$-converges to some function $f$ on $X$, in the sense that
$\{u_{\ep}\}_\ep$  is a convergent subsequence (not relabeled) and
$u$ is its limit,  then
\begin{equation}\label{2.1ter}
\liminf_{\ep\to 0} f_\ep(u_\ep) \ge f(u).\end{equation}
\end{enumerate}

The abstract result is 
\begin{theo}\label{gamma} Let $\{\theta_\lambda\}_\lambda$, $\{\omega_\ep\}_\ep$ and $\{f_\ep\}_\ep$, $f$ be as above. Let $\{\UR\}_{R>0}$ be a family satisfying  \eqref{hypsets}. Let
\begin{equation}\label{FEP}
F_\ep (u) = \dashint_{\omega_\ep} f_\ep( \theta_\lambda u) \,d\lambda.\end{equation}
Assume that $\{F_\ep(u_\ep)\}_\ep$ is bounded. Let  $P_\ep$ be the
image of the normalized Lebesgue measure on $\omega_\ep$  under
the map $\lambda \mapsto \theta_\lambda u_\ep$. Then $P_\ep$
converges along a subsequence to  a Borel probability measure $P$ on $X$ invariant under
the action  $\theta$ and  such that
\begin{equation}
\label{rg1} \liminf_{\ep \to 0} F_\ep(u_\ep) \ge \int f(u)\, dP(u).
\end{equation}
Moreover, for any family $\{\UR\}_{R>0}$ satisfying \eqref{hypsets}, we have
\begin{equation}
\label{rg2} \int f(u)\, dP(u)= \int f^*(u) \, dP(u), 
\end{equation} with $f^*$ given by 
\begin{equation}\label{fstarr}
f^*(u)  := \lim_{R\to
+\infty}\dashint_{\UR} f(\theta_\lambda
u)\,d\lambda.\end{equation}
In particular, the right-hand side of \eqref{rg2} is independent of the choice of  $\{\UR\}_{R>0}$.
\end{theo}

\begin{remark}\label{remF}  The result \eqref{rg2} is simply the ergodic theorem in multiparameter  form.    Part of the result is that the limit in \eqref{rg2} exists for $P$-almost every $u$.
\end{remark}
The probability measure $P$ is the ``Young measure" on limiting profiles  we were referring to before, indeed it encodes the limit of all translates of $u_\ep$. 
 Note  that $P_\ep\to P$ implies that  $P$ almost every $u$ is of the form  $\lim_{\ep\to 0}\theta_{\lambda_{\ep}} u_\ep$.  
 The limit defining $f^*$ can be viewed as a ``cell problem", using the terminology of homogenization, providing the limiting small scale functional.
 \medskip
 
To illustrate this result,  we shall give two examples, in order of increasing generality, both models for what we will use here. 
\begin{example}\label{ex1}
Consider $X=\mathcal{M}(\mr^n)$ the set of positive bounded measures on $\mr^n$, and $\theta_\lambda$ the action of translations of $\mr^n$. Let $\chi $ be a given nonnegative smooth function with support in the unit ball of $\mr^n$,  and define  $f_\ep(\mu)=f(\mu)=\int_{\mr^n} \chi \, d\mu .$
Then one can check that $F_\ep$, as defined in \eqref{FEP}, is  $$ F_\ep(\mu)=\frac{1}{|\omega_\ep|}\int_{\mr^n} \chi * \indic_{\omega_\ep}\, d\mu.$$ The result of the theorem (choosing balls for example) is the the assertion that
$$\liminf_{\ep \to 0}   F_\ep(\mu_\ep )\ge    \int f(\mu)\, dP(\mu)= \mathbb{E}^P \( \lim_{R\to +\infty} \int \chi *\indic_{B_R} \, d\mu\).$$  The probability $P$ gives a measure over all  limiting profiles of $\mu_\ep$ (depending on the centering point), and the result says that the quantity we are computing, here
 the average of $\mu_\ep$ over the large sets $\omega_\ep$, can be bounded below by an average, this time over $P$, of a similar quantity for the limiting profiles.
This implies in particular that there is a $\mu_0$ in the support of $P$, hence of the form $\lim_{\ep \to 0} \mu_\ep(\lambda_\ep+ \cdot) $ such that
$$\liminf_{\ep \to 0} \frac{1}{|\omega_\ep|}\int_{\mr^n} \chi * \indic_{\omega_\ep}\, d\mu_\ep \ge \lim_{R\to+ \infty} \int \chi*\indic_{B_R} \, d\mu_0.$$
In other words we can find a good centering sequence $\lambda_\ep$  such that the average of $\mu_\ep$ over the large sets $\omega_\ep$  can be bounded from below by the average over large balls of the limit after centering,  $\mu_0$.
The averages over $\omega_\ep$ or $B_R$,  and the average with respect to $P$ are not of the same nature (in standard ergodic settings, the first ones are usually time averages, while the latter is a space average).\end{example}
\begin{example}\label{ex2}
 Let us assume we want to bound from below an energy which is the average over large (as $\ep \to 0$) domains $\omega_\ep$ of some nonnegative energy density $e_\ep(u)$, defined on a space of functions $X$ (functions over $\mr^n$),  $\dashint_{\omega_\ep } e_\ep(u(x))\, dx$, and we know the $\Gamma$-liminf behavior of $e_\ep(u)$ on small (i.e. here, bounded) scales, say we know how to prove that $\liminf_{\ep \to 0}\int_{B_R}e_\ep(u) \, dx\ge \int_{B_R} e(u)\, dx$. We cannot always directly apply such a knowledge to obtain a lower bound on
the average over large domains, this may be due to a  loss of boundary information, or due to the difficulty to   reverse limits $R\to \infty$ and $\ep \to 0$. (These are two obstructions that we encounter specifically both for Ginzburg-Landau and $W$, as illustrated by the 6th remark in Section \ref{sec14})
What we can do is   let $\chi $ be a smooth cutoff function as above, and  define  the functions  $f_\ep$ by
$$f_\ep(u)=\int e_\ep(u(x))\, \chi(x)\, dx,$$
that is $f_\ep$ can be seen as  the small scale local functional. Since we know the $\Gamma$-liminf behavior of  the energy density $e_\ep$ on small scales, let us assume we can prove that \eqref{2.1ter} holds for some function(al) $f$, a function on $u$ (we may expect that $f$ will also be of the form $\int e(u(x))\, \chi(x)\, dx.$)

Defining $F_\ep$ as in \eqref{FEP} and using Fubini's theorem, we see that
$$F_\ep(u)=\dashint_{\omega_\ep} \int_{\mr^n} e_\ep(u(x+y))\chi(x)\, dx\, dy
 = \frac{1}{|\omega_\ep|}\int_{z\in \omega_\ep-x}\int_x e_\ep(u(z))\chi(x)\, dx\, dz.$$
Since $\chi$ is supported in $B_1$, $\int \chi=1$,  and $\omega_\ep$ satisfies \eqref{domaine}, we check that
$$F_\ep(u)\sim_{\ep \to 0}   \dashint_{\omega_\ep} e_\ep(u(y)\, dy,$$
hence $F_\ep$ is asymptotically equal to the average we wanted to bound from below. We may thus apply the theorem and it yields a lower bound for the desired quantity:
$$\liminf_{\ep \to 0} F_\ep(u_\ep) \ge \int f^*(u)\, dP(u)
$$ with \begin{equation}
\label{cellprob}
f^*(u)=\lim_{R\to \infty}\dashint_{B_R} f(u(x+\cdot) ) \, dx\end{equation}
 The ``cell-function" $f^*$ is simply an average over large balls of  the local $\Gamma$-liminf $f$.
If typically $f$ is of the form $\int e(u(x))\, \chi(x)\, dx$ then  we can compute by Fubini that
$$f^*(u)= \lim_{R\to \infty}\frac{1}{|B_R|} \int e(u(x))(\chi*\indic_{B_R})(x)\, dx.$$
\end{example}

\subsection{The Ginzburg-Landau model}
Our original motivation  in this article is to  analyze  the  behaviour of minimizers of the  Ginzburg-Landau energy, given by
\begin{equation}
\label{je} \je(u,A)= \hal \io |\nab_{A} u|^2 +
|\curl A -\he|^2+\frac{(1-|u|^2)^2}{2\ep^2}.\end{equation} This is a
celebrated model in physics, introduced by Ginzburg and Landau in
the 1950's as a model for superconductivity. Here $\om$ is  a two
dimensional bounded  and simply connected domain,  $u$ is a
complex-valued function, ``order parameter" in physics, describing
the local state of the superconductor, $A:\om \mapsto \mr^2$
 is the vector potential of the magnetic field $h = \curl A= \nab \times A$, and $\nab_A $ denotes the operator $\nabla - i A$. Finally the parameter $\he$ denotes the intensity of the applied magnetic field and $\ep$ is a constant corresponding to a characteristic lengthscale (of the material). It is the inverse of $\kappa$, the Ginzburg-Landau parameter in physics.
 We are interested in the $\ep \to 0$ asymptotics.
 The quantity $|u|^2$ measures the local density of superconducting electron pairs ($|u|\le 1$). The material is in the superconducting phase wherever $|u|\simeq 1$ and in the normal phase where $|u|\simeq 0$.
  We focus our attention on   the zeroes of $u$ with nonzero topological degree (recall that $u$ is complex-valued), also known as the {\it vortices} of $u$. Here the typical lengthscale of the set where $|u|$ is small, hence of the vortex ``cores", is $\ep$.

The Euler-Lagrange equations associated to this energy with natural boundary conditions are
the Ginzburg-Landau equations
\begin{equation}
\label{gleq}\left\{
\begin{array}{ll}
-( \nab_A)^2 u = \frac{u}{\ep^2}(1-|u|^2) & \text{in } \ \om\\ [2mm]
- \np h = (iu, \nab_A u) & \text{in} \ \om\\
h=\he & \text{on} \ \bo
\\
\nu \cdot\nab_{A} u=0 & \text{on} \ \bo.
\end{array}\right.\end{equation}
where $\np$ denotes the operator $(-\p_2, \p_1)$, $\nu$ the outer
unit normal to $\bo$ and $(\cdot,\cdot)$ the canonical scalar
product in $\mc$ obtained by identifying $\mc$ with $\mr^2$.

This model is also famous as the simplest ``Abelian gauge theory".
Indeed it admits the $\mathbb{U}(1)$ gauge-invariance : the energy
(\ref{je}), the equation (\ref{gleq}), and all the physical
quantities are invariant under the transformation
$$\left\{\begin{array}{l}
u \to u e^{i\Phi}
\\
A \to A + \nab \Phi.\end{array} \right.$$

For a more detailed mathematical presentation of the functional, one
may refer to \cite{livre} and the references therein.

\subsection{Physical behaviour: critical fields and vortex lattices}\label{sec1.1}
 Here and in all the paper $a \ll b$ means $\lim a/b=0$.

For  $\ep$  small,   the minimizers of \eqref{je} depend on  the intensity $\he$ of the applied field as follows.

\begin{itemize}
\item[-] If  $\he$ is below a critical value called the {\it first critical field} and denoted by $\hci$, the
     superconductor is in the so-called {\it Meissner state} characterized by the expulsion of the magnetic field and the fact that $|u|\simeq 1$ everywhere. Vortices are absent in this phase.
\item[-] If  $\he$ is above  $\hci$, which is equivalent to $\lambda_\Omega\lep$ as $\ep\to 0$, then energy minimizers  have one vortex, then two,... The number of vortices increases with $\he$, so as to become equal to leading order to $\he/2\pi$ for fields much larger than $\lep$ (see \cite{livre}).  In this case, according to the picture we owe to A. Abrikosov and dating back to the late 1950's, vortices  repell each other and  organize themselves in {\it triangular lattices} named {\it Abrikosov lattices}. The theoretical predictions of Abrikosov have received ample experimental confirmation and there are numerous and striking observations of the lattices.
\item[-] If  $\he$ increases beyond  a second critical value $\hcii = 1/\ep^2$, another phase transition occurs where superconductivity disappears from the material, except for a boundary layer. Even this boundary layer completely disappears  when $\he$ is above a third critical field $\hciii$.
\end{itemize}

Since the works of  Ginzburg, Landau and Abrikosov, this model has
been largely studied in the physics literature. We refer to the classic  monographs and textbooks by De Gennes \cite{dg}, Saint-James Sarma - Thomas \cite{sst}, Tinkham \cite{t}.

The above picture describes the phenomenology of the model for small
values of $\ep$ (or high values of $\kappa$) in a casual way, but by now many rigorous
mathematical results support this picture. Except for the third
critical field however, whose existence was proven in a strong sense
by Fournais and Helffer \cite{fh,fournaishelffer-livre}, mathematical
results  really prove  the existence of intervals where the transition
between the different types of behaviour occur, rather than critical
values, with estimates on these  intervals as $\ep \to 0$.

Our goal is the study and description of the vortices
in  the whole range  $\hci < \he \ll \hcii$,  where minimizers of the
Ginzburg-Landau energy have a large number of vortices, expected to
form Abrikosov lattices.

\subsection{Connection to earlier mathematical  works}\label{sec12}
There is an abundant mathematical literature related to the study of
\eqref{je} and to the  justification of the physics picture above. A
relatively extensive bibliography is given in \cite{livre}, Chap.
14.

The techniques most relevant to us for the description of vortices
originate in the pioneering  book of Bethuel-Brezis-H\'elein \cite{bbh} and have
been further expanded by several authors including  in particular
Jerrard, Soner \cite{jerr,js}, ourselves, etc...
 For a more detailed description of these techniques we refer to \cite{livre}.
 In that book we describe how these techniques allow to derive  the asymptotic values of the critical field $\hci$ as $\ep \to 0$ as well as the mean-field description of minimizers and their vortices in the regimes $\he \ll \hcii$.

In order to describe the vortices of minimizers, one introduces the
{\it vorticity} associated to a configuration $(u,A)$,  defined by
\begin{equation}\label{muua}
 \mu(u,A) = \curl j(u,A) + \curl A, \quad \text{where $
j(u,A):=(iu,\nab_A u)$}\end{equation}
is the superconducting current. Here $(a, b)$ denotes the Euclidean scalar product in $\mc$ identified with $\mr^2$, so $j(u, A)$
 may also be written as $\frac i2\(u \overline{\nab_A u}  -
\bar u \nab_A u\)$, or as  $\rho^2(\nab\vp-A)$ if $u= \rho e^{i\vp}$, at
least where $\rho\neq 0$.

This  vorticity is the appropriate quantity to consider in this
context, rather than the quantity $\curl (iu, \nab_A u)$ which could
come to mind first.  It may be seen as a gauge-invariant version of
$\curl (iu, \nab u)$ which is also (twice) the Jacobian determinant
of $u$. Indeed if $A = 0$, then $\mu(u,A) =
2\partial_x u\times\partial_y u$. One can prove (this is the
so-called Jacobian estimate, see \cite{js, livre}) that assuming a
suitable bound on $\je(u,A)$, the vorticity $ \mu(u,A)$ is well
approximated, in some weak sense, as $\ep\to 0$ by a measure of the form $2 \pi \sum_i d_i \delta_{a_i}$. As points of concentration of
the vorticity, the points $\{a_i\}_i$ are naturally called vortices
of $u$ and  $d_i$, which  is an  integer, is called the degree of
$a_i$. The vorticity $\mu(u,A)$ may either describe individual
vortices or, after normalization by the number of vortices,  their
density.

In  previous work (summarized in  \cite{livre}, Chap. 7) we showed
that
\begin{equation} \label{hc1} \hci \sim \com \lep,\end{equation} for a constant $\com >\hal$ depending only on $\om$ (and such that $\com \to \hal $ as $\om \to \mr^2$).
More precisely, letting   $\he =
\lambda\lep$, we established by a $\Gamma$-convergence approach that
minimizers $(u_\ep,A_\ep)$ of $\je$ satisfy
\begin{equation}\label{1.3b}
\text{$\frac{\curl A_\ep}{\he}$ converges to
$h_\lambda$},\quad \text{$\frac{\mu(u_\ep,A_\ep)}\he$ converges to
$\mu_\lambda$, as $\ep\to 0$},\end{equation} where $\mu_\lambda =
-\Delta h_\lambda +h_\lambda$, and $h_\lambda$ is the solution of
the following minimization problem
\begin{equation}
\label{obs} \min_{h-1\in H^1_0(\om)}\frac{1 }{2\lambda} \io |-\Delta
h+h|+ \hal\io |\nab h|^2 +|h-1|^2.
\end{equation}
This problem is in turn equivalent to an obstacle problem, and as a
consequence there exists a subdomain $\omega_\lambda$ such that
\begin{equation}\label{mumin}
\mu_\lambda = m_\lambda \indic_{\omega_\lambda},\quad\text{where }
m_\lambda = 1-\frac 1{2\lambda},\end{equation} and where $\indic_A$
denotes the characteristic function of $A$. In other words the
optimal limiting vortex density $\mu_\lambda$ is  uniform  and equal to $m_\lambda$ over a
subregion of $\om$, completely determined by $\lambda$ i.e. by the
applied field $\he$. The constant  $\com$  introduced in (\ref{hc1})
is characterized by the fact that $\omega_\lambda= \varnothing $ if
and only if $\lambda <\com$. Since $\com >\hal$,  if $\omega_\lambda
\neq \varnothing$ we have \begin{equation}\label{126bis}
1\ge m_\lambda\ge 1- \frac{1}{2 \com}
>0.\end{equation}
  Since the  number of vortices is proportional to $\he = \lambda \lep$, it tends to $+\infty$ as $\ep \to 0$. It is  also established in \cite{livre} that as $\ep\to 0$, the  minimal energy has the following expansion,
\begin{equation}
\label{minG} \min \je = \frac{1}{2}\he \L \io \mu_\lambda +
\frac{\he^2}{2} \io |\nab h_\lambda |^2 + |h_\lambda -1 |^2 + o\(
\he \L\).\end{equation} When the applied field is much larger than
$\lep$, but much less than $1/\ep^2$, \eqref{1.3b} and  \eqref{minG}
still hold, replacing  $\lambda$ by $+\infty$. In this case
$h_\lambda = 1$ and $\omega_\lambda = \om$.

There are  other cases where  the distribution of vortices for
minimizers of the Ginzburg-Landau functional is understood. First in
a periodic setting: Minimizing the Ginzburg-Landau energy among
configurations $(u,A)$ which are periodic (modulo gauge equivalence)
with respect to a certain lattice independent of $\ep$, one obtains
(see \cite{aydi,aydisandier}) as above a limiting vortex density
$\mu_\lambda= - \Delta h_\lambda + h_\lambda$ where $h_\lambda$
minimizes the energy in  (\ref{obs}) among periodic functions. The
minimizer in this case is clearly a constant, which is easily found
to be $\max\( 1- \frac{1}{2\lambda}, 0\)$.

Second, in the regime of applied fields $\he$  where the number of
vortices tends to $+\infty$, but is negligible compared to $\lep$,
(this corresponds to applied fields such that  $\log\lep\ll
\he-\com\lep\ll\lep$ as $\ep \to 0$),
 it is shown in \cite{livre} that for simply connected domains satisfying a certain generic property (see \eqref{assumption} below) --- including convex domains --- vortices concentrate around a single point (a finite number of points for general simply connected domains). Then, blowing up at the suitable scale  and normalizing the vorticity, one obtains in the limit a probability measure $\mu$ which describes the distribution of vortices around the point, and $\mu$ is the unique minimizer among probability measures in $\mr^2$ of
$$I(\mu) = -\pi \iint \log |x-y|\, d\mu(x) \, d\mu(y) +  \pi \int Q(x)\, d\mu(x).$$
Note that here $Q$ is a positive definite  quadratic form which depends on
the domain $\om$. It is the Hessian of a certain function at the
point of concentration of the vortices. Note  that the precise
regime of applied field modifies the number of vortices, i.e. the
normalizing factor of the vorticity,  and the scaling, but does not
influence the limit distribution $\mu$, which is a characteristic of
$Q$.

In our previous work, the treatment for the ``intermediate" regime $\log\lep\ll
\he-\com\lep\ll\lep$ and for the regime $\he = \lambda \lep$, $\lambda>\lambda_\om$ were different, the former one being more delicate. Here we provide (and this is part of the technical difficulties) a unified approach for both, and treat all regimes $\he \ll \hcii$  where the number of vortices is blowing up.

\subsection{Main result on Ginzburg-Landau}\label{secmain}

The mean field description above  tells us that the vortices tend to
be distributed   uniformly in $\omega_\lambda$ but  is insensitive
to the  pattern formed by vortices. This pattern is in fact, as we
shall see,  selected by the minimization of the next term in the
asymptotic expansion of the energy as $\ep\to 0$. The proof of this
is achieved in this paper by a splitting of the energy that
separates the leading order term found in (\ref{minG}) from a
remainder term, and then by studying the remainder term after blow
up at the scale of the expected intervortex distance, which from the
above considerations is of the order of  $1/\sqrt{\he}$; this remainder term is then shown to 
$\Gamma$-converge to  $W$ (as introduced in Section \ref{sec14}), hence allows to distinguish among vortex configurations. 

As before, we use the  current $j(u,A) = (iu,\nab_A u)$ (cf. \eqref{muua})  to describe the vortex locations, we study  through the abstract framework of Section \ref{secergo} the probability measure carried by all 
possible limiting profiles  of blow ups of $j$ at the scale $1/\sqrt{\he}$ centered at all possible blow-up points in $\omega_\lambda$, and show that this probability is concentrated on minimizers of $W$.

 Before stating the simplest form of
 our main result, let us explain the main
  assumptions we need to make.
 First, we make the simplifying assumption that the domain
  $\om$ is convex. This is only used at one point in the proof (the upper bound construction) and avoids the possibility that $\omega_\lambda$ may have cusps. We believe our results still   hold without this restriction.
 Then letting $h_0$ be  the solution to $- \Delta h_0 + h_0=0$ with $h_0=1 $ on $\bo$ --- which is consistent with the definition of $h_\lambda$ in \eqref{obs} --- it is known (see Caffarelli-Friedman \cite{cf}, the result in dimensions greater than two can be found in  \cite{kl}) that in the case where $\om$ is convex,  $h_0$ is strictly convex hence achieves its minimum at a unique point $x_0\in \om$. Moreover this minimum is nondegenerate:
 \begin{equation}\label{assumption}
\min_\om  h_0\  \text{is achieved at a unique point} \  x_0 \ \text{and } \
Q:=D^2 h_0(x_0) \ \text{is positive definite}.\end{equation}
For a family $(u_\ep, A_\ep)$, we denote
\begin{equation}\label{1.8b}
\tilde \jmath_{\ep, x}(\cdot):= \frac{1}{\sqrt{\he}} j(u_\ep,A_\ep) \( x +  \frac{\cdot} { \sqrt{
\he } }\)\end{equation} their blown-up current,
where in the right-hand side $j(u_\ep, A_\ep)$ is implicitly extended by $0$ outside the domain $\om$.

We have
\begin{theo}\label{th1} Assume that $\om $ is convex, so  that \eqref{assumption} is satisfied, and that
$$ \text{ $\he = \lambda \lep$   with   $\lambda > \com$ .}$$
Let $(u_\ep,A_\ep)$ be a minimizer of $\je$, and let $\tilde \jmath_{\ep, x}$ be as in \eqref{1.8b} for $x \in \omega_\lambda$.
Then, given $1<p<2$, there exists a probability measure $P$ on  $L^p
_{loc}(\mr^2,
\mr^2)$ such that the following hold:
\begin{enumerate}
\item  Up to extraction,  for any bounded continuous function $\Phi$ on $L^p_{loc}(\mr^2, \mr^2)$, we have
    \begin{equation}\label{ym}\lim_{\ep \to 0}\frac{1}{|\omega_\lambda|}
\int_{\omega_\lambda}\Phi(\tilde \jmath_{\ep, x} ) \, dx = \int \Phi(j) \, dP(j).\end{equation}
%\item $P$-almost every $j$ satisfies
%$$\curl j  = 2\pi \sum_{p\in \vtx} \delta_p -\cl, \qquad \div j=0,$$
%for some discrete subset  $ \vtx$ of $\mr^2$.\footnote{on pourrait remplacer par appartient a $\admissible$ ou meme supprimer car c'est sous-entendu par ce qui suit.}
 \item  $P$-almost every $j$ minimizes $W$ over $\admissible$ and
 \begin{equation}\label{truc} \je(u_\ep,A_\ep)= \je^{N_0} + \he |\omega_\lambda |( \min W + cst)  + o(\he)\quad \text{as} \ \ep \to 0,\end{equation}
where $\je^{N_0}\in \mr$ is  explicited in \eqref{120b} below.
\end{enumerate}
\end{theo}
We may informally describe $\je^{N_0}$ : It is the minimum value
 for an obstacle  problem similar to  \eqref{obs}. But whereas  \eqref{obs} is derived assuming the vortex energy is $\pi\lep$, to obtain $\je^{N_0}$ we must use instead the more precise value $\pi\plep$, where
$$\ep' = \ep\sqrt\he.$$
As $\ep \to 0$ we have $\lep\approx\plep$, but using $\plep$ induces a correction which  {\em is not} $o(\he)$, hence is  important to us.

As we have seen in Section \ref{period-intro}, $W$ allows to distinguish between configurations of points since it distinguishes between lattices, and it is expected to favor the triangular lattice. The result of Theorem \ref{th1} can be informally understood as follows: if one chooses a point at random in $\omega_\lambda$ and blows up at the scale $1/\sqrt{\he}$, then in the limit $\ep \to 0$, almost surely (with respect to the blow up center), one sees a minimizer of $W$.  
This derivation of this limiting energy $W$ is, to our knowledge, the first
rigorous  justification of the Abrikosov triangular lattice in this
regime: at least the triangular lattice is the best among
lattice configurations, and it is conjectured to be a global minimizer (see Section \ref{period-intro}).

In Theorem~\ref{th1prime} below, we will give a more precise and a full $\Gamma$-convergence  version of Theorem \ref{th1},  valid  for the other regimes of applied field where vortex lattices are
expected to arise. The latter
result encompasses the regimes of Theorems 1.3, 1.4 and 1.5 in
\cite{livre} and  allows to reprove \medskip these results.

Returning to the reference to the Theta function  \eqref{theta} in Section \ref{period-intro},  it is striking to observe that the problem of minimizing the Theta  function also arises in the context of Ginzburg-Landau, but in a very different regime: when $\he \sim \hcii$ or more precisely when $\he= \frac{b}{\ep^2}$ with $b \nearrow
1$. As seen in \cite{ss4,as} this is a regime which is essentially linear (contrarily to the one we study here) and the energy minimization can be reduced to the minimization of a function on a finite dimensional space (the ``lowest Landau level"- this is essentially the result of Abrikosov's original calculation).  This function can be viewed as the linear analogue of $W$ and reduces, in the case  where the points are on a lattice, to the $\theta$ function $\theta_\Lambda$ (and so again the optimal lattice is the triangular one).   In that sense the limiting  lattice energies for Ginzburg-Landau  in the regime $\hci \ll  \he \ll \hcii$ and in the regime $\he \sim \hcii$  can be viewed as Mellin transforms of each \medskip other.

We now go into more detail on the method of the proof of Theorem
\ref{th1}.  It  follows from a result of $\Gamma$-convergence, i.e. by showing a general
lower bound for the energy and a matching upper bound via an
explicit construction. Thus the minimality of $(u_\ep, A_\ep)$ and the Euler-Lagrange equation it solves is not used per se.
 The proof of the lower bound involves three
ingredients: an energy splitting, a blow-up, and the abstract method  of Theorem \ref{gamma}.
\subsection{The energy splitting}\label{sec1.8}

The first ingredient of the proof, detailed in Section
\ref{sec-splitting}, is a new algebraic   splitting of the energy,
which allows to isolate the constant  leading order part from the
next-order.

First we define a mean field $\ho$  similar to $\he h_\lambda$,
except that  when computing it  we  take more precisely into account
the cost of a vortex which is $\pi\plep$, where we recall $\ep'
=\ep\sqrt\he$. Accordingly, assuming $\he < 1/\ep^2$, we  let $\ho$
be the minimizer of
\begin{equation}
\label{obs1} \min_{h-\he\in H^1_0(\om)}\hal\plep \io |-\Delta h+h|+
\hal\io |\nab h|^2 +|h-\he|^2.
\end{equation}
Thus $\ho/\he$ solves  (\ref{obs}), with $\lambda$ replaced by
$\lambda_\ep = \frac{\he}{\plep}.$

This is equivalent (see \cite{bre,brs}) to saying that $\ho/\he$
minimizes the $H^1$ norm subject to the constraints $\ho/\he=1$ on
$\bo$ and
\begin{equation}\label{mep} \D \frac{\ho}{\he} \ge   \moep:= 1 - \frac\plep{2\he} \quad  \text{in} \ \om .\end{equation}
In other words $\ho$ is the solution to an obstacle problem with constant obstacle. Letting then
 $\muo= -\Delta \ho+\ho$, we have
\begin{equation}\label{omep}\muo =  \moep \he   \indic_{\omo},\quad\text{where}\quad\omo = \left\{x\mid \ho(x) = \he\moep\right\}.\end{equation}
We  define
\begin{equation}\label{defn0}
N_0= \frac{1}{2\pi} \io \muo.\end{equation}
The splitting function  could be taken to be  $\ho$ in most regimes of applied field. Then one should define $\je^{N_0}$ in Theorem \ref{th1}  as
\begin{equation}\label{120b}
\je^{N_0}: = \pi N_0 \plep + \hal \|\ho- \he \|_{H^1(\om)}^2,\end{equation}
where $\|\cdot\|_{H^1}$ denotes the Sobolev space  norm $\(\|\cdot \|_{L^2}^2+ \|\nab \cdot \|_{L^2}^2 \) ^\hal$.

However, when $\he - \lambda_\om\lep\ll\lep$ --- i.e. when the number of vortices is small compared to $\he$ though divergent as $\ep\to 0$ --- then \eqref{obs1}, which is a refinement of \eqref{obs}, must itself be refined to take into account the constraint that the vorticity is quantized. For the other regimes, the error made by ignoring this constraint is negligible in our analysis.

More precisely, given $N$ such that $0 \le N\le \frac{1}{2\pi} \he |\om|$, we consider  $\tho$
 the minimizer of
\begin{equation}\label{deftho}
\min_{\substack{h - \he \in H^1_0(\om)\\   \io |\- \Delta h+h|= 2\pi N}}
\hal \io |\nab h|^2 + |h-\he|^2,\end{equation}  and we define
 \begin{equation}\label{121b}
 \je^N: =\pi N \plep + \hal \|\tho- \he\|_{H^1(\om)}^2.\end{equation}
 Since $\he $ will be a given function of $\ep$, we denote the dependence of $\tho$ as of $\ep$ instead of $\he$. $\je^N$ is minimal at $N=N_0$ and we  will see that $h_{\ep, N_0}= \ho$.

The refinement with respect to \eqref{obs1} consists in taking $N$ to be an integer, when $N_0$ is not necessarily one. More precisely, $N$ will be taken to be either $\nm$, the largest integer $\le N_0$, or $\ns$, the smallest integer $\ge N_0$.  With that choice, we will sometimes call $\tho$ the ``splitting function".

The leading order term in the energy is not exactly $\je^{N_0}$ but  rather $\min_{N \in \{\nm, \ns\}} \je^N$.    We may immediately check however  that for $N\in \{\nm, \ns\}$, as $\ep \to 0$, 
$$\je^N \sim \je^{N_0} \sim \frac{1}{2}\he \plep \io \mu_\lambda +
\frac{\he^2}{2} \io |\nab h_\lambda |^2 + |h_\lambda -1 |^2 $$ when $\he \sim \lambda \lep$
in the notation of Section \ref{sec12}, hence it recovers the leading order term of the minimal energy \eqref{minG}.
The difference between $\min_{N \in \{\nm, \ns\}}
\je^N $ and $\je^{N_0}$  is $o(\he)$  hence  it is negliglible for the  precision of $o(N)$ we want to achieve as soon as  $\lambda> \lambda_\om$, but not always when $\lambda= \lambda_\om$.

$\tho$  is the solution to an obstacle problem
and we have $\tho \ge \me \he $ and $- \Delta \tho + \tho= \he \me \indic_{\toe}$ where $\toe = \{\tho = \me\he\}$ is called the coincidence set, for some constant
 $\me$ such that $\me \to m_\lambda$
as $\ep \to 0$.
 Note that $\toe$ will depend on the choice of $N= \nm $ or $\ns$
  but sometimes we will forget it and  simply  write $\omega_\ep$.
   We let  $\tmuo= - \Delta \tho + \tho= \he \me \indic_{\toe}$. It is a perturbation of $\mu_\lambda$ and  $\toe$ a perturbation of $\omega_\lambda$, defined in Section \ref{sec12}.
 We have the relation \begin{equation}\label{toen}
 |\toe|\me \he = \io \tmuo = 2\pi N.\end{equation}
%Note that in the case $\he = \lambda\lep$ with $\lambda>\com$, we
%have  $\co\approx m_\lambda$ as $\ep\to 0$, and that $\co$ is an
%increasing function of $\he$ on $[0,1/\ep^2]$. In particular if
%$\he\ge\com\lep$ then, since $\com>1/2$,
%\begin{equation}\label{half} \liminf_{\ep\to 0}\co \ge 1 - \frac1{2\com} >0.\end{equation}

Then we  (temporarily) introduce $\ao =  \np \tho $.
Letting  $(u,A)$ be a
configuration of finite energy, we will write
$$\ai= A- \ao= A- \np \tho ,$$
where $\ao=\np\tho$ is understood as the leading order term, and $\ai$ as a remainder term.
The energy-splitting is the observation of the following  identity (valid even if  $N$ is not an integer):
\begin{multline}
\label{dec1} \je(u,A)=\je^N   +\io (\tho-\he)\mu(u,\ai)
+\hal \io |\nab_{\ai}
u|^2 + |\curl \ai - \tmuo |^2 + \frac{(1-|u|^2)^2}{2\ep^2}\\ - \Ce \io \tmuo  - \hal\io
(1-|u|^2)|\nab\tho|^2,\end{multline}
where we recall \eqref{muua}, and  $\Ce$ is a constant explicited in \eqref{Cep}.  The last term in \eqref{dec1} can be shown to be negligible  if $\je (u,A)$ is not too large.   Thus the study
 of the energy near its minimum reduces to that of   the remainder
\begin{multline}
\label{g1-intro}\jie(u,\ai):= \hal \io |\nab_{\ai} u|^2 + |\curl \ai -\tmuo|^2 +
\frac{(1-|u|^2)^2}{2\ep^2} \\
+ \io (\tho-\he)\mu(u,\ai)- \Ce \io \tmuo .\end{multline}
It turns out that when we make the right specific choice $N= \nm $ or $N= \ns$ (depending on $(u_\ep, A_\ep)$), this expression simplifies and one has
\begin{equation}\label{fua1}
\jie(u, \ai) \ge \hal \io |\nab_{\ai} u|^2 + |\curl \ai -\tmuo|^2 +
\frac{(1-|u|^2)^2}{2\ep^2} - \io \zeta_\ep \mu(u,\ai)+o(1)\end{equation}
where $\zeta_\ep$ is a positive function, equal to its maximum $ \hal \plep$ on  $\toe = \supp(\tmuo)$. It is this remainder $F_\ep$ (with this choice of $N$) whose $\Gamma$-convergence we study.
We will see through the upper bound construction that
$$
\min \jie \le C \he|\toe|$$  (this is equivalent to $\min\jie \le C N$  or $\le C N_0$ from \eqref{toen}), and we will  work in that class
$\jie (u,\ai)\le C N_0 $, thus reducing to a relatively narrow class of ``almost minimizers" i.e. configurations whose  leading order  energy is the  minimal one $G_\ep^{N_0}$.   Note that once we
know that this remainder $\jie$ is of lower order, this shows that
the leading order component  in $A=\np \tho+ \ai$ is $\np\tho$ (in other
words $\ai \ll \ao$) so also at leading order $\mu(u_\ep, A_\ep)  \sim \tmuo\sim \mu_\lambda$
and this allows to recover in essence the results of \cite{livre}
Theorems 1.3, 1.4, 1.5 for all configurations in this almost minimizing class.

The next step is to make the  change of scales  $x'=\sqrt{\he}x$ in
order to study (\ref{g1-intro}).  Under  this rescaling, the
inter-vortex distance becomes of order 1 (recall that the average
vortex density is precisely $\me \he \indic_{\omega_\ep}$ with $\me \to m_\lambda$ and \eqref{mumin}). After
this change of scales the right-hand side of \eqref{fua1} becomes, in terms of $\pu(x') =
u(x)$ and $\pa(x')= A(x)/ \sqrt{\he}$, and $\zeta_\ep'(x')= \zeta(x)$,
\begin{equation}
\label{fep-intro} \pfep(\pu,\pa)= \hal \int_{\om_\ep'}
|\nab_\pa\pu|^2 + \he |\curl \pa- \co\indic_{\poe} |^2 +
\frac{(1-|\pu|^2)^2}{2\pep^2} - \int_{\om_\ep'}\zeta_\ep'(x')
\mu(\pu,\pa)+o(1)\end{equation} where $\poe $ is the rescaled domain
$\sqrt{\he} \omega_{\ep,N}$ and we recall $\ep'= \ep\sqrt{\he} $.

Combining all these elements, the conclusion of the splitting procedure, found in Section~\ref{sec-splitting}  is
\begin{pro}\label{pro1.5}
For any $(u,A)$, there exists $N \in \{ \ns, \nm\}$ such that
$$\je(u,A)\ge \je^N + \pfep(\pu, \pa) +o(1)$$where $\pfep$ is as in \eqref{fep-intro}, $\pu(x')=u(x'/\sqrt{\he})$, $\pa(x')=A(x'/\sqrt{\he})$
and $\zeta'$ is a positive function, equal to its maximum $\hal\plep$ on $\poe$.
\end{pro}

\subsection{Full version of the main result}\label{sec1.7}
We may now give  the more complete version of Theorem~\ref{th1} and the stronger statement of $\Gamma$-convergence  of $\frac{1}{N}( \je(u, A)- \je^N).$
  In all the paper, the weak convergence of probabilities will mean convergence against  bounded continuous  test-functions, see \cite{billingsley}. We will say that a  probability  measure is concentrated on a set if   that set has probability 1.

  We consider configurations $(u_\ep, A_\ep)$ and  assume that $|u_\ep|\le 1$ everywhere together with the second Ginzburg-Landau equation, i.e. we assume 
\begin{equation}\label{gl2ee}
|u|\le 1\  \text{in} \ \om,  \qquad   - \np \curl A= (iu, \nab_A u)\ \text{in} \ \om
\end{equation} is satisfied.
This is obviously true for minimizers and critical points of \eqref{je}. Moreover  given $(u,A)$,   replacing $u$ by $u/|u|$ wherever $|u|\ge 1$  and replacing  $A$ by the minimizer of $A \mapsto \je(u, A)$, with $u$ remaining fixed,   decreases the energy without displacing the vortices, and the modified $(u,A)$ verify \eqref{gl2ee}. Since the change can only decrease the energy, our results   allow to bound from below $\je(u_\ep,A_\ep)$  for the original arbitrary $(u,A)$.  Thus, assuming \eqref{gl2ee} is no loss of generality.  Note that the only consequence of \eqref{gl2ee} that we will really use is that $\div j(u, A)=0$.

Note that when $\he - \com \lep =O(\log \lep) $  we already know
from \cite{livre}, Chap. 12, that  the number of vortices remains
bounded as $\ep \to 0$ and we already  characterized their limiting
location. For  $\he \ge \frac{\beta}{\ep^2}$ the study of vortices
involves an analysis quite different from that used in the present
paper, see \cite{ss4,as,fournaishelffer-livre}.  So we focus on the remaining regimes,
namely
\begin{equation}\label{range}\log\lep\ll \he -\com \lep \quad\text{and}\quad  \he \ll 1/\ep^2,\end{equation} and   we have

\begin{theo}\label{th1prime}Assume  $\om$ is convex, hence satisfies \eqref{assumption}. Assume
 \eqref{range} and   $\frac{\he}{\lep} \to \lambda \in [\lambda_\om, +\infty]$ as $\ep \to 0$. Let any  $1<p<2$ be chosen.

 \begin{enumerate}
 \item Let  $(u_\ep, A_\ep)$ satisfy \eqref{gl2ee}, and
        $$G_\ep(u_\ep, A_\ep) \le
   \min_{N\in \{\nm, \ns\}}\je^N+ CN_0
         $$ for some $C$ independent of $\ep$.
 Then
 there is a choice of  $N\in \{\nm, \ns\}$ such that,
 letting $P_\ep$ be the probability measure on $\Lp$ defined as  the push-forward of the normalized uniform measure on $\toe$ by the map
 $x\mapsto\tilde \jmath_{\ep,x} $ (cf. \eqref{1.8b}),  as $\ep \to 0$
the measures $\{P_\ep\}_\ep$ converge   up to extraction
 to a probability measure $P$ on $\Lp$ which is invariant
 under the action of  translations, concentrated on $\admissible$  ($m_\lambda$ as in \eqref{mumin}) and
 \begin{equation}\label{borneinffinale2}
  \je(u_\ep, A_\ep) \ge  \min_{N\in \{\nm, \ns\}}\je^N + N
   \( \frac{2\pi}{m_\lambda}  \int  W_U(j) \, dP(j) + \gamma    +o(1) \) ,\end{equation} where $\gamma$ is defined in \eqref{defgamma}, and $W_U$ is computed according to \eqref{WU} relatively to any family of sets satisfying \eqref{hypsets}--\eqref{hypsetsbis}.

In the case  $\he \le \frac{1}{\ep^\beta}$ for some
  $\beta>0$ small enough,   then
  \begin{equation}
  \label{mumuo0}
  \|\mu(u_\ep, A_\ep)  - \tmuo\|_{W^{-1,p} (\om) }\le C_p\sqrt{ N}.
  \end{equation}

 \item
  For any probability $P $ on $\Lp$
which is invariant under the action of translations and concentrated on $\admissible$ and for every $N \in \{\nm, \ns\}$,  there exists $(u_\ep, A_\ep)$  such that, letting $P_\ep$ be the push-forward of the normalized Lebesgue measure on $\toe$ by the map $x \mapsto
\frac{1}{\sqrt{\he}}  j(u_\ep, A_\ep) \( x + \frac{\cdot}{\sqrt{\he}} \) $, we have  as $\ep \to 0$, $P_\ep \to P$  and
\begin{equation}\label{resth5k} \je(u_\ep, A_\ep) \le \je^N + N \(\frac{2\pi}{m_\lambda} \int W_K(j) \, dP(j)  +\gamma  +o(1)\),\end{equation}
where $W_K$ is the renormalized energy relative to the family of squares $\{K_R\}_R$, where $K_R = [-R/2,R/2]^2$, as defined in \eqref{WU}.
\item
If we assume  that $(u_\ep,A_\ep)$ minimizes $\je$ then it satisfies all  the assumptions of item 1),  $P$-almost every $j$ minimizes $W$ over $\admissible$, and there is equality in \eqref{borneinffinale2}.

\end{enumerate}
\end{theo}

\begin{remark}
\begin{itemize}
\item[-] The constant $\gamma$ in \eqref{borneinffinale2} and \eqref{resth5k} was introduced in \cite{bbh} and  may be  defined by
\begin{equation}\label{defgamma}\gamma=\lim_{R \to \infty}\( \hal  \int_{B_R} |\nab u_0|^2
+ \frac{(1-|u_0|^2)^2}{2} -\pi \log R\) ,\end{equation} where $u_0 (r,\theta) =
f(r) e^{i \theta} $  is the unique (up to translation and rotation)
radially symmetric degree-one vortex (see \cite{bbh,miro}).
\item[-] There exists $\{(u_\ep,A_\ep)\}_\ep$ satisfying \eqref{resth5k} for any $N$ such that $1\ll N\le \frac{\he|\om|}{2\pi}$, see Theorem~\ref{th5} in Section~\ref{bornesup}.
\item[-]
As we already mentioned in Section \ref{sec1.8}, this theorem allows to retrieve the results of \cite{livre}, but it gives a stronger result: in \cite{livre} we establish the leading order behaviour of the vorticity for minimizers:  $\mu(u_\ep, A_\ep) \sim \he \mu_\lambda \sim \tmuo$, while here \eqref{mumuo0} gives (at least for small enough applied fields) for the whole class of ``almost minimizers"  the order of the fluctuations of the vorticity around the constant density $\tmuo$:
it is of order $\sqrt{N}$, hence  in particular the number of vortices is $N$ with an error of order $\sqrt{N}$.
\item[-] For energy minimizers it is possible to deduce from this result that  $\bar{N}$, now being defined as the total degree of the vortices, satisfies  $\je^{\bar{N}}=\min_{N \in \mn }\je^N +o(\bar{N})$. On the other hand $\min_{N \in \mn} \je^N$ is achieved at $\nm $ or $\ns$. From examining carefully the variations of $N \mapsto \je^N$ one should be able to deduce that $\bar{N}=\nm $ or $\ns$ with a smaller error than previously (at least for $\he \le \hci + O(\sqrt{\lep})$ we expect this error to be $0$), see Remark  \ref{remsec6} for more details. So we expect, for small enough fields, to be able  to estimate exactly the total degree of the vortices of a minimizer, and for larger fields, to estimate it with an error which is at least  better than $\sqrt{N}$.
\end{itemize}
\end{remark}

\subsection{Use of the ergodic theorem for Ginzburg-Landau} \label{sec1.10}As announced, the method consists in applying the framework of Section \ref{secergo} to the Ginzburg-Landau energy.

We sketch the method in the case $\he=\lambda\lep$, $\lambda_\om < \lambda <+\infty
$. The case of higher fields $\he \gg \lep$ will reduce to this one by scaling.

Let $(u_\ep,A_\ep)$ --- or $(\pue,\pae)$ in rescaled coordinates ---
denote a minimizer of  $\je$ and let $\pmue= \mu(\pue,\pae)$.
The splitting result of Section \ref{sec1.8} combined with the blow up procedure reduces us (cf. Proposition \ref{pro1.5})  to bounding from below  $\pfep(\pue, \pae)$.

Thus we are in the setting of Example \ref{ex2}  in Section \ref{secergo}, i.e. the case where we want to bound from below the average over large domains $\omega_\ep'$  of some energy density $e_\ep$. Here the energy density is of the form
$$e_\ep(u, A)= \hal |\nab_A u|^2 + \he|\curl A - m\indic_{\omega_\ep'}|^+\frac{(1-|u|^2)^2}{2\ep^2} - \int \zeta' \, d\mu(u,A).$$

From the
Jacobian estimate    $\pmue$ is well approximated by $\nu_\ep=2\pi \sum_i
d_i \delta_{a_i}$ (where $a_i$ denotes the vortex center and $d_i$
its degree) and we should have $a_i\in\poe$.  Using in addition that $\zeta' = \hal \plep$ there,  we may thus formally replace the energy-density above by
$$e_\ep(u, A)= \hal |\nab_A u|^2 + \he|\curl A - m\indic_{\omega_\ep'}|^2+\frac{(1-|u|^2)^2}{2\ep^2} - \pi  \sum  \lep \sum_i d_i \delta_{a_i}.$$

 To apply the framework of Section \ref{secergo} we need to check the coercivity and $\Gamma$-liminf properties of $f_\ep(u,A)=\int e_\ep (u, A)\, \chi $ for a cut-off function $\chi$ as in the example of Section \ref{secergo}.  
 But the framework also requires that  $f_\ep$ be nonnegative (or bounded below by a constant will do) but this is obviously not the case, 
 and one of
the major difficulties of our analysis consists in getting around
this problem. 
This part of the analysis was carried out in our companion paper \cite{compagnon}, where we introduced a method that consists in displacing the negative part of $e_\ep$ so as to  absorb  it into the positive part and to obtain a
density bounded below, called $g_\ep$. We show there that $e_\ep$ can be replaced by $g_\ep \ge - C$ with making only a small error (in average at the small scale).
Note that since we are dealing with cancelling
leading order terms,  we need very precise estimates of  the
free-energy $f_\ep$, in fact we need to make errors which are at most $o(1)$
per vortex (the total acceptable error is $o(N)$).  Hence in \cite{compagnon} we need refined estimates in the
``ball construction  methods"  which are devised  to obtain
general lower bounds for the energy of the vortices  even when their
number is unbounded.
We also show in \cite{compagnon}  the crucial fact that $g_\ep$ controls the number of vortices. 

Returning to the question of proving the coercivity of $f_\ep$, we may now heuristically replace $e_\ep$ by $g_\ep$, and 
the  coercivity requires proving \ref{2.1bis} which in our case becomes
 $$\forall R>0, \ \limsup_{\ep \to 0} \int g_\ep ( \chi * \indic_{B_R}) <\infty\Longrightarrow (u_\ep,A_\ep) \ \text{compact}.$$
This is  satisfied thanks to the fact  that $g_\ep$ controls the number of vortices, in other words we have roughly
$$\forall R>0, \ \limsup_{\ep \to 0} \int g_\ep  (\chi * \indic_{B_R} )<\infty\Longrightarrow  \sum_{a_i \in B_R}|d_i|\le C_R.$$
To prove the $\Gamma$-liminf relation on $f_\ep$, we may reduce to that setting, i.e. that where the number of vortices is bounded independently of $\ep$ on the compact support of $\chi$. In that setting, it is now standard to retrieve  very precise estimates for the Ginzburg-Landau energy, using an analysis of the type of \cite{bbh}.
This way, we obtain the precise $\Gamma$-liminf for $f_\ep$ and can show it can be expressed as a function of $j$ (limit of $j(u_\ep', A_\ep')$ and is equal (up to a constant) to $f(j)=W(j, \chi)$.
But then, using the definition \eqref{fstarr}, the ``cell-energy" is  
$$f^*(j)= \lim_{R\to +\infty}\dashint_{\UR} W(j(\lambda + \cdot), \chi)\, d\lambda.$$ 
One may immediately check, using Fubini, that this is equal to $\lim_{R\to +\infty }\frac{1}{|\UR|} W(j, \chi* \indic_{\UR})$, which is in turn clearly equal to $W_U(j)$ (cf. \eqref{WU} and  the first item in Theorem \ref{thw}).
Combining all these elements, the result of Theorem \ref{gamma} yields
$$\liminf_{\ep \to 0} \dashint_{\omega_\ep'} \hal |\nab_{A_\ep'} u_\ep'|^2 + \he|\curl A_\ep' - m\indic_{\omega_\ep'}|^+\frac{(1-|u_\ep'|^2)^2}{2\ep^2} - \int \zeta' \, d\mu(u_\ep',A_\ep')\ge \int W(j) \, dP(j),$$ which is essentially the desired lower bound. 

The rigorous proof of this is detailed in Section \ref{section:lb}.

\subsection{Plan of the paper} The paper contains  two parts. The first part is completely independent of the Ginzburg-Landau energy and can thus be read independently.
It starts in Section \ref{sec:abstract} with the proof of Theorem \ref{gamma}.
In Section \ref{latticecase},  we  prove Theorem \ref{th2} i.e. that $W$ is minimized among lattice configurations by the triangular lattice. In Section
\ref{secW}, we study $W$ more generally, and prove Theorem \ref{thw}.

The second part is about the application of the tools of the first part to the Ginzburg-Landau energy and the derivation of $W$ as its $\Gamma$-limit.
In Section \ref{sec-splitting}, we prove
the new energy-splitting formula,
in the same section we recall or prove some results on the splitting function that will be needed in the sequel and  we also derive simple a
priori bounds for energy-minimizers.
In Section \ref{section:lb}, we show how to apply the abstract framework  of Section \ref{secergo} in the specific case of the Ginzburg-Landau energy, and this way we obtain the main
energy lower bound, and prove Theorems~\ref{th1} and~\ref{th1prime}
assuming the upper bound, which is proven in Section~\ref{bornesup}.  This requires  the improved lower bounds for the energy of vortices borrowed from \cite{compagnon}.
 In Section
\ref{bornesup}, we prove the matching upper bound for the energy via an explicit
construction, using the periodic minimizing sequence  found in Theorem \ref{thw}.

In the Appendix, we prove some additional qualitative results on the solutions to the obstacle problem which can be of independent interest, in particular estimates when the size of the coincidence set is small, which are needed in the regime $\he - \hci \ll \lep$.

\vskip .5cm {\bf Acknowledgements:} We would like to thank warmly
S.R.S. Varadhan for his suggestion to use the ergodic theorem. Many
thanks also to our colleagues  A. Venkatesh,  C.S. Gunt\"urk, L.
 Caffarelli, A. Chambert-Loir,  N. Fournier, L. Mazet, M. Traizet
and E. Lesigne for providing information and/or help
on various steps. We also thank  the referee,    as well as  our students and post-docs H. Ben Moussa, D. Goldman and  A. Contreras, for their careful and critical reading.
 This work would also not have been possible without several invitations of E.S to New York University and S.S to Paris XII. We
thank these two institutions, and acknowledge the  financial support of the National Science Foundation, an EURYI award (SS), and the Institut Universitaire de France (ES).

\part{The renormalized energy}

\section{Proof of Theorem \ref{gamma}}\label{sec:abstract}
Before giving the proof of the theorem, we state a preliminary lemma.
\begin{lem}\label{lesigne} (E. Lesigne) Assume $P_n$ are Borel probability measures on a Polish metric  space $X$ and that for any $\delta>0$ there exists $\{K_n\}_n$ such that $P_n(K_n)\ge 1-\delta$ for every $n$ and such that if $\{x_n\}_n$ satisfies for every $n$ that $x_n\in K_n$, then any subsequence of $\{x_n\}_n$  admits a convergent subsequence (note that we do not assume $K_n$ to be compact).

Then $\{P_n\}_n$ admits a subsequence which converges tightly, i.e. converges weakly to a {\em probability} measure $P$.
\end{lem}
\begin{proof}
From Prohorov's Theorem, it suffices to show that the sequence of
measures is tight. As a finite Borel measure on a Polish space, the measure  $P_n$ is regular \cite{cohn}, thus there is a compact
subset $K_n'\subset K_n$ such that $P_n(K_n')\ge 1 - 2\delta$. Then,
letting
$$K = \overline{\bigcup_{n}K_n'}, $$
we have $P_n(K)\ge 1-2\delta$ for every $n$ and the assumption made
on $\{K_n\}_n$ implies that $K$ is compact.  Indeed, a sequence in $K$ is either included in a finite union of the compact sets $K_n$ (then is compact) or has a subsequence which can be relabelled $(x_n)$ and satisfied $x_n'\in K_n'$ along a subsequence $n'$, hence compact by assumption. 
Therefore $\{P_n\}_n$ is
tight.
\end{proof}
We now start proving the theorem.  First we choose a sequence $\{\ep_n\}_n$ tending to $0$ such that
$$\lim_{n\to +\infty} F_{\ep_n}(u_{\ep_n}) = \liminf_{\ep \to 0} F_\ep(u_\ep).$$
In this proof, $\ep$ will always be assumed to belong to this
sequence, and $\lim_{\ep\to 0}$ will mean the limit along this
sequence.

Recall that  $P_\ep$ is the image of the normalized Lebesgue measure
restricted to $\omega_\ep$  under the map $\lambda \mapsto \theta_\lambda
u_\ep$. In particular for any positive measurable  function $\Phi$ on $X$
\begin{equation}
\label{2.1.2}\int \Phi(u) \, dP_\ep(u)  = \dashint_{\omega_\ep}
\Phi(\theta_\lambda u_\ep)\, d\lambda.
\end{equation}

\begin{proof}[Step 1: $\{P_{\ep}\}_\ep$ is tight]
Letting $\omega_{\ep,R}=\{ x \in \omega_\ep\mid\dist(x, \p \omega_\ep)
> R\}$ we have
\begin{equation}\label{eqet}
\begin{split}\int_{\omega_{\ep,R}}\int_{B_R} f_\ep(\theta_{\lambda+\mu} u_\ep)\,d\lambda\,d\mu
&= \iint_{\mr^2\times\mr^2} \indic_{\omega_{\ep,R}}(\lambda)\indic_{B_R}(\mu)f_\ep(\theta_{\lambda+\mu} u_\ep)\,d\lambda\,d\mu\\
& = \int_{\mr^2} \indic_{\omega_{\ep,R}}\ast \indic_{B_R}(\lambda)f_\ep(\theta_{\lambda} u_\ep)\,d\lambda\\
&\le |B_R| \int_{\omega_\ep} f_\ep(\theta_\lambda u_\ep)\,d\lambda\le C|B_R||\omega_\ep|.
\end{split}\end{equation}
Let us denote $Y_{\ep, R}$ the image of $\omega_{\ep , R} $ by $\lambda \mapsto  \theta_\lambda u_\ep$, and
$$X_{R,K}^\ep = \left\{u\in X\mid \dashint_{B_R} f_\ep(\theta_\lambda u)\,d\lambda > K\right\}.$$
The left-hand side in \eqref{eqet} is larger than  $K |B_R| |\omega_\ep| P_\ep\( X_{R, K}^\ep \cap Y_{\ep, R}\) $.
In addition $P_\ep( X_{R, K}^\ep) \ge   P_\ep \( X_{R, K}^\ep \cap Y_{\ep, R}\)- P_\ep (Y_{\ep, R}^c) \ge    P_\ep \( X_{R, K}^\ep \cap Y_{\ep, R}\)-
 \frac{|\omega_\ep\sm \omega_{\ep,R}|}{|\omega_\ep|}$, so we deduce from \eqref{eqet}  that
$$P_\ep(X_{R,K}^\ep)\le \frac CK + \frac{|\omega_\ep\sm \omega_{\ep,R}|}{|\omega_\ep|}.$$
From  \eqref{domaine} and for any $\delta >0$, there exists a subsequence $\{\ep_n\}_n$ such that $$\frac{|\omega_{\ep_n}\sm \omega_{\ep_n,R}|}{|\omega_{\ep_n}|} < \delta 2^{-n}$$
and then
$$P_{\ep_n}\(\cup_{k=1}^n X^{\ep_n}_{k,2^k/\delta}\)\le C\delta.$$
Now we have that the hypotheses of Lemma~\ref{lesigne} are satisfied. Indeed, letting  $K_n$ be the complement of  $\cup_{k=1}^n X^{\ep_n}_{k,2^k/\delta}$, we have $P_{\ep_n}(K_n)\ge 1 - C\delta$. Moreover,  if $u_n\in K_n$ for every $n$ then $\forall R>0$, $\forall n>R$ we have $u_n\notin X^{\ep_n}_{R,2^R/\delta}$, i.e.
$$\dashint_{B_R} f_{\ep_n}(\theta_\lambda u_n)\,d\lambda \le \frac{2^R}\delta.$$
 Then, from the coercivity assumption, a subsequence of $\{u_n\}$ converges.

Applying Lemma~\ref{lesigne} we can conclude that $\{P_{\ep_n}\}$ is tight and then that a subsequence converges weakly to a probability measure $P$.
\end{proof}
\begin{proof}[Step 2: $P$ is  $\theta$-invariant]

Let $\Phi$ be bounded  continuous   on $X$. Then from the
definition of  $P_\ep$,
$$
\int \Phi(u) \, dP(u) =  \lim_{\ep \to 0} \int \Phi(u) \, dP_\ep(u)=
\lim_{\ep \to 0 } \dashint_{\omega_\ep } \Phi(\theta_\lambda u_\ep)\,
d\lambda.$$ Moreover,
\begin{equation*}\int_{\omega_\ep }
\Phi(\theta_\lambda u_\ep)\, d\lambda -\int_{\omega_\ep }
\Phi(\theta_{\lambda+\lambda_0} u_\ep)\, d\lambda= \int_{\omega_\ep}
\Phi(\theta_\lambda u_\ep)\, d\lambda - \int_{ \omega_\ep + \lambda_0}
\Phi(\theta_\lambda u_\ep)\, d\lambda.
 \end{equation*}
Thus,
 \begin{equation*}
\left|\dashint_{\omega_\ep } \Phi(\theta_\lambda u_\ep)\, d\lambda
-\dashint_{\omega_\ep } \Phi(\theta_{\lambda+\lambda_0} u_\ep)\,
d\lambda\right| \le
 \frac{|\omega_\ep\sd(\omega_\ep + \lambda_0)|}{|\omega_\ep|}
\|\Phi\|_{L^\infty} , \end{equation*} where $  \sd$ denotes the symmetric difference between sets, and follows from \eqref{domaine} that $|\omega_\ep\sd(\omega_\ep + \lambda_0) |= o(|\omega_\ep|)$ as $\ep\to 0$.
We deduce that
$$\int \Phi(u)\, dP(u)=\lim_{\ep \to 0 }
\dashint_{\omega_\ep } \Phi(\theta_{\lambda+\lambda_0} u_\ep)\, d\lambda =
\int \Phi(\theta_{\lambda_0}u)\, dP(u),$$ hence $P$ is  invariant
under the action  $\theta$.
\end{proof}
We state the proof of  \eqref{rg1} as a lemma.

\begin{lem}\label{abstrait} Assume that $X$ is a Polish metric space, that $\{P_\ep\}_{\ep>0}$, $P$  are Borel probability measures on $X$ such that $P_\ep\to P$  as $\ep\to 0$, and that $\{f_\ep\}_{\ep>0}$ and $f$ are positive measurable functions on $X$ such that $\liminf_{\ep\to 0} f_\ep(x_\ep)\ge f(x)$ whenever $x_\ep\to x$.

Then,
\begin{equation}\label{eqlem} \liminf_{\ep\to 0} \int f_\ep\,dP_\ep\ge \int f\,dP.\end{equation}
\end{lem}

This is \eqref{rg1} since $\int f_\ep\,dP_\ep  = \dashint_{\omega_\ep}  f_\ep(\theta_\lambda u) \, d\lambda.$

\begin{proof}[Proof of the Lemma] It suffices to show that for any
$\lambda,\delta>0$, we have
\begin{equation}\label{fu5}\liminf_{\ep \to 0} P_\ep(\{ f_\ep>\lambda-\delta\})\ge P(\{f > \lambda\}).\end{equation}
Indeed, using the standard
expression for the integral of a positive function $f$
$$\int f(x) \, d\mu(x)= \int_0^\infty \mu(\{ f>\lambda\}) \, d\lambda,$$
we find by applying it to  $f_\ep$ and $P_\ep$, and then to $f$ and
$P$, in view of (\ref{fu5}) that
\begin{eqnarray*}
\liminf_{\ep \to 0} \int f_\ep(u)\, dP_\ep(u)
&  = & \liminf_{\ep \to 0} \int_0^\infty P_\ep(\{f_\ep>\lambda\}) \, d\lambda \\
 & \ge &  \int_0^\infty \liminf_{\ep\to 0} P_\ep(\{f_\ep>\lambda\}) \, d\lambda \\
& \ge & \int_0^\infty P(\{f >\lambda+\delta \}) \, d\lambda \\
& \ge & \int_0^\infty P(\{ f>\lambda \}) \, d \lambda - \delta\\
& = & \int f(u) \, dP(u) - \delta,
\end{eqnarray*}
where we have used Fatou's lemma, and the fact that $P$ is a
probability measure. Since this is true for any $\delta>0$, we have \eqref{eqlem}.

We prove \eqref{fu5}. For any $\delta>0$ and $u\in X$  we claim that there exists
an open neighborhood $V_u$ of $u$ and $\eta>0$ such
that,\begin{equation} \label{gc1}\text{ $\forall \ep<\eta$, $\forall
v \in V_u$, $f_\ep(v)> F(u)-\delta$.}\end{equation}
Indeed, assume this were wrong, then there would exist $\delta>0$ and $u \in X$ together with  a sequence $\{\ep\}$ tending to $0$  and
a corresponding sequence  $\{u_\ep\}_\ep$  tending to $u$  such that $f_\ep(u_\ep) \le f(u)- \delta$, thus  contradicting
the $\Gamma$-liminf assumption.  Hence (\ref{gc1}) holds.

We denote by $\veta$ the set of $u$'s such that $f_\ep(v) > F(u)-\delta$ holds on  $V_u$ for every
$\ep <\eta$. Clearly, $\{\veta\}_\eta$ is decreasing and from the above $\cup_{\eta>0
} \veta=X$. Thus, if we  let  for some $\lambda>0$
$$E=\{ u \in X\mid f(u) > \lambda\},\quad E_\eta = E\cap \veta,$$
then $\{E_\eta\}_\eta$ is decreasing and $E = \cup_\eta E_\eta.$
Moreover, from the definition of $\veta$, we have  $f_\ep > \lambda - \delta$ on the  open set
$O_\eta = \cup_{u\in E_\eta} V_u$ for every $\ep<\eta$.

It follows that for any $\ep<\eta$,
\begin{equation}
\label{fu3}P_\ep(\{ f_\ep>\lambda-\delta\}) \ge P_\ep( O_\eta). \end{equation}
Then, since  $P_\ep\to P$ and  since $O_\eta$ is open and contains $E_\eta$, we have
$$\liminf_{\ep \to 0} P_\ep(O_{\eta}) \ge P(O_{\eta})\ge P(E_\eta).$$
It then follows from (\ref{fu3}) that
$$ \liminf_{\ep \to 0} P_\ep(\{ f_\ep>\lambda-\delta\})\ge
P(E_\eta).$$ Since $E_\eta\nearrow E$ as $\eta\searrow 0$  we deduce by monotone convergence that \eqref{fu5} holds.
\end{proof}

\begin{proof}[Step 4: Proof of \eqref{rg2}]

Since $P$ is invariant w.r.t. the action $\theta$, we may apply the
ergodic theorem as stated in \cite{becker}  to obtain that
$$ {\bf E}^P (f(u)) = {\bf E}^P (f^*(u)),$$
where $f^*$ is $\theta$-invariant and $P$-a.e. equal  to
\begin{equation}\label{fstar}
\lim_{R \to \infty} \dashint_{B_R} f(\theta_\lambda u)
\, d\lambda.\end{equation}It is also true (see \cite{becker})  that this limit
exists for $P$-a.e. $u$, and that balls may be replaced by any Vitali family satisfying   \eqref{hypsets}. This
proves \eqref{rg2}.

\end{proof}

\begin{remark}\label{rem2.1}\mbox{}
\begin{enumerate}\item
The lower bound of Theorem \ref{gamma} implies in particular that $\liminf_{\ep \to 0} F_\ep (u_\ep) \ge \min f^*$
 where $f^*$ is given by \eqref{fstar}.
If we assume in addition  that for some family there is equality i.e. that
$\limsup_{\ep \to 0} F_\ep(u_\ep) \le \min f^*$, then comparing with the lower bound obtained in Theorem \ref{gamma} we deduce that we have
$$ f^*(u)= \min f^* \quad \text{ for } P- a.e. u. $$
\item
We may apply the same reasoning as Theorem \ref{gamma}  in $B(x_\ep, R)$  instead of $\omega_\ep$, with $R$ large, to the functional $\dashint_{B(x_\ep, R)} f_\ep(\theta_\lambda u)\, d\lambda$,
and we will obtain
$$\liminf_{\ep \to 0} \dashint_{B(x_\ep, R)} f_\ep(\theta_\lambda u_\ep)\, d\lambda\ge \min f^* + o_R(1)$$ where $o_R(1)\to 0$ as $R\to \infty$ (the $o_R(1)$ is due to   near boundary errors).
If in addition we know that for some $x_\ep$,
\begin{equation}\label{bornesuplocale}  \limsup_{\ep \to 0}
\dashint_{B(x_\ep, R)} f_\ep(\theta_\lambda u)\, d\lambda \le \min f^*+o_R(1)\end{equation} we deduce
\begin{equation}\label{borneinflocale}\lim_{R\to \infty} \lim_{\ep \to 0} \dashint_{B(x_\ep, R)} f_\ep(\theta_\lambda u_\ep) \, d\lambda = \min f^*.\end{equation}But if we assume $\limsup_{\ep \to 0} F_\ep (u_\ep) \le \min f^*$, by a Fubini argument
\eqref{bornesuplocale} holds for most $x_\ep \in \omega_\ep$, so the local estimate \eqref{borneinflocale} too. This says that  when the upper and lower bounds match, the energy density is essentially uniformly distributed, at any scale $\gg 1$.
\end{enumerate}\end{remark}

\section{Minimization of $W$ in the  periodic case : optimality of the triangular lattice}
\label{latticecase}
\subsection{Calculation of $W$ in the periodic case}\label{rem52}
In this section we study the minimization of
the renormalized energy $W(j)$ among lattices   in the following sense: We assume that
$\Lambda = \mz
\vu +\mz\vv$, where the  vectors $(\vu,\vv)$
form a basis of $\mr^2$, that  $j$ is invariant w.r.t.
translations by the vectors $\vu$, $\vv$ and $$\curl j = 2\pi\sum_{p\in\Lambda}\delta_p - 1, \qquad \div j =0.$$
It is not difficult to check that there exists  such a current $j$
if and only if  $\det(\vu,\vv) = 2\pi$.
We will denote $\mathbb{T}_\Lambda$ the torus $\tl = \mr^2/\Lambda$.
 We have
\begin{pro}\label{renper}  Assume $\Lambda = \mz
\vu +\mz\vv$ and $j$ are as above.
Then, letting $\{\UR\}$ be any family satisfying \eqref{hypsets}, \eqref{hypsetsbis},  defining $W_U(j) $ as in \eqref{WU} we have
\begin{equation}
\label{pro51.1} W_U(j) =\lim_{\eta\to 0} \frac{1}{2\pi}
\(\hal\int_{\tl\sm B(0,\eta)} |j|^2  +\pi\log \eta\). \end{equation}
Moreover, letting $\hl$   be the unique solution (with mean zero)
on $\tl$ to
 \begin{equation}\label{dhl}
- \Delta \hl = 2\pi \delta _0 - 1,\end{equation}  and denoting
\begin{equation}\label{pro51.2}
j_\Lambda = - \np H_\Lambda\end{equation} we have that $j-j_\Lambda
$ is a constant and $$W_U(j)\ge W_U(j_\Lambda) $$ with equality if and
only if $j=j_\Lambda$.

\end{pro}
We will denote $W_U(j_\Lambda)$ simply $W(\Lambda)$: it is the minimum
of $W_U(j) $ among $\Lambda$-periodic  configurations, and using
(\ref{pro51.1}) it is given by
$$ W(\Lambda)=\lim_{\eta\to 0} \frac{1}{2\pi}
\(\hal\int_{\tl\sm B(0,\eta)} |\nab \hl|^2  +\pi\log \eta\).$$
\begin{proof}
Since $\curl(j-j_\Lambda)=\div (j-j_\Lambda)=0$ and $j, j_\Lambda$ are periodic, we have that
$j-j_\Lambda$ is constant. We now prove (\ref{pro51.1}). Let us
denote by  $\cal K$ the set of cells $K$ of the form $\{ t\vu +
s\vv, t\in (l-\hal, l+\hal) , s\in (m-\hal, m+\hal)\}$, where $l,m \in \mz$.
For $K \in \cal K$, let  $c_K$  be the center of the cell.  The cells
of $\cal K$ tile  $\mr^2$, hence for every $R>0$, writing $\chi_R$ for $\chi_{\UR}$,
$$W(j, \chi_R)= \lim_{\eta\to 0}
\sum_{K \in \cal K} \hal \int_{K \backslash B(c_K, \eta)} |j|^2
\chi_R + \pi \chi_R (c_K) \log \eta.$$ There are only finitely many
$K$'s on which $\chi_R$ is not identically zero, thus  we may write
$$W(j, \chi_R)= \sum_{K \in \cal K} w_K(j, \chi_R)$$
where $$w_K(j, \chi_R)= \lim_{\eta\to 0} \hal \int_{K \backslash
B(c_K, \eta)} |j|^2 \chi_R + \pi \chi_R (c_K) \log \eta.$$ Then, if
$\chi_R \equiv 1 $ on $K$, we have (by periodicity of $j$)
\begin{equation}
\label{pro51-cas1} w_K(j, \chi_R)= \lim_{\eta\to 0} \int_{\tl
\backslash B(0, \eta)} |j|^2  + \pi  \log \eta.
\end{equation}
On the other hand, there exists $C>0$ such that,
 \begin{equation}\label{pro51-cas2}
\forall R>0,  \forall K \in {\cal K},   \quad |w_K(j, \chi_R)|\le
C.\end{equation} Indeed, $j= cst-\np H_\Lambda$ with $H_\Lambda(x)=
- \log |x|+ U(x)$ where $U$ is a $C^2$ function, so for any $r>0$,
we have
$$E(r):=\hal \int_{\tl\backslash B(0,r)}|j|^2  \le C+ \pi \log \frac{1}{r}.$$
From this we deduce first, by letting $r=\eta$,  that $$\left|\hal
\int_{\tl\backslash B(0, \eta)} |j|^2 \chi_R(0)+ \pi \chi_R(0)\log
\eta\right|\le C,$$ and second, that, since $|\nab \chi_R|\le C$,
for any $0<\eta<r_0$,
\begin{multline*}\left|\int_{B(0, r_0)\backslash B(0, \eta)} |j|^2 (\chi_R- \chi_R(0) ) \right|\le C \int_{B(0, r_0)\backslash B(0, \eta)} |j|^2 |x|=- C \int_\eta^{r_0} E'(t) t \, dt\\
= - C \( E(r_0) r_0-E(\eta) \eta+ \int_{\eta}^{r_0} E(r) \, dr\) \le
C.\end{multline*} Adding the two proves  (\ref{pro51-cas2}). To
conclude, we note that by definition of $\chi_R$ and \eqref{hypsetsbis}
\begin{eqnarray*} & \# \{ K \in {\cal K}| \chi_R \equiv 1\ \text{on} \ K\} \sim_{R \to \infty}\frac{|\UR|}{2\pi}\\
& \# \{ K \in {\cal K}| \chi_R \not\equiv 0 \ \text{and} \ \chi_R
\not\equiv  1 \ \text{on} \ K\}=o(|\UR|) \ \text{as} \ R\to
+\infty.\end{eqnarray*} Together with
(\ref{pro51-cas1})--(\ref{pro51-cas2}) this yields
$$W(j, \chi_R) = \frac{|\UR|}{2\pi}
\lim_{\eta\to 0} \int_{\tl \backslash B(0, \eta)} |j|^2  + \pi  \log
\eta+o(|\UR|).$$

Combining with (\ref{WU}), we deduce (\ref{pro51.1}).
 It remains to show that $W_U(j)\ge W_U(j_\Lambda)$ with equality iff $j=j_\Lambda$. Since $j=j_\Lambda + c$, using (\ref{pro51.1}),  we have
 \begin{equation*}
 4 \pi W_U(j)= \lim_{\eta \to 0}
 \int_{\tl \backslash B(0, \eta)} |j_\Lambda|^2 +|c|^2 + 2 c \cdot j_\Lambda   + \pi  \log \eta
 = 4\pi W_U(j_\Lambda) +  |\tl|c^2\end{equation*}
 since $\int_{\tl  }c \cdot j_\Lambda= - c \cdot \int_{\tl } \np H_\Lambda=0.$ The result follows.
 \end{proof}
We next express $W(\Lambda)$ as a series using the Fourier decomposition of $\hl$.
\begin{lem}
For all $\Lambda \in \lattices$ we have\begin{equation} \label{eis}
W(\Lambda)= \hal \lim_{x\to 0} \( \sum_{p \in \Lambda^*\backslash
\{0\} } \frac{ e^{2i\pi p \cdot x}}{ 4\pi^2 |p|^2}  + \log
|x|\),\end{equation} where $\Lambda^*$ denotes the lattice dual to
$\Lambda$.
\end{lem}
\begin{proof}
Integrating by parts and using (\ref{dhl}) we find
\begin{equation}\label{ipp}
\hal\int_{\tl\sm B(0,\eta)} |\nab \hl|^2  +\pi\log \eta= \hal  \(  \int_{\p (B(0,\eta ) )}\hl
\frac{\p \hl}{\p \nu} - \int_{\tl \backslash B(0, \eta)} \hl\)
+\pi\log\eta.\end{equation} But  $\int_\tl \hl=0$ and  $H_\Lambda(x) + \log|x|$ is a
$C^1$ function in a neighbourhood of $0$, thus passing to the limit
$\eta\to 0$ above we find
\begin{equation}\label{59bis}
W(\Lambda)=\hal \lim_{x\to 0} (\hl(x)+\log|x|).\end{equation}
Using the following normalisation of the Fourier transform
$$\hat f(y) = \int_{\mr^2} f(x)e^{-2i\pi x\cdot y}\,dx,$$
since $-\Delta \hl = 2\pi\sum_{p\in\Lambda}\delta_p - 1$  in $\mr^2$ and $\hl$
has zero mean we have
$$\hat \hl(y) = \frac{1}{4\pi^2|y|^2} \sum_{p\in\Lambda^*\sm\{0\}}\delta_p(y),$$
where $\Lambda^*$ is the dual lattice of $\Lambda$, i.e. the set of
vectors $q$ such that $p\cdot q\in\mz$ for every $p\in\Lambda$. By
Fourier inversion formula, we obtain the expression of $H_\Lambda $
in Fourier series:
\begin{equation}\label{hfour}
H_\Lambda(x)= \sum_{p \in \Lambda^*\backslash \{0\} } \frac{
e^{2i\pi p \cdot x}}{ 4\pi^2 |p|^2}  ,\end{equation} and the result follows from \eqref{59bis}.
\end{proof}

We may now prove Lemma \ref{lemperiointro}.
For the more general periodic  situation of Lemma \ref{lemperiointro}
 where  $\Lambda$ is assumed to
be merely $(\vu,\vv)$ periodic instead of being itself a   lattice, and $j$ is  also  $(\vu,\vv)$ periodic with  $\curl j = 2\pi
\sum_{p\in\Lambda} \delta_p - n$, we observe that Proposition \ref{renper} can be adapted with identical proofs.
In this case, denoting by $\T_{(\vu,\vv)}$ the
quotient $\mr^2/(\mz \vu +\mz\vv)$, the configuration $\Lambda$ may
be seen as a finite family of points $(a_1,\dots,a_n)$ in
$\T_{(\vu,\vv)}$, and the statements of Proposition \ref{renper} remain
true, replacing $\tl$ by $\T_{(\vu,\vv)}$, replacing $\tl\sm
B(0,\eta)$ by $\T_{(\vu,\vv)}\sm\cup_iB(a_i,\eta)$ and replacing $-
\Delta \hl = 2\pi \delta _0 - 1$ by $- \Delta H_{\{a_i\}}  = 2\pi \sum_{i=1}^n
\delta _{a_i} - n$, in particular $W_U(j_{\{a_i\}}  ) = \frac{1}{2\pi}  W(j_{\{a_i\}  } , \indic_{\T_{(\vu, \vv)} } )$.
Moreover, writing $j_{\{a_i\}}$  as $ - \np H_{\{a_i\} }$ and using the translation invariance of the equation, we have
$$H_{\{a_i\}} (x) = \sum_{i=1}^n G(x-a_i) $$ where $G$ is the solution to $- \Delta G(x) =2 \pi \delta_0  -1$ on the torus $\T_{(\vu ,\vv)}$.
We may also define $R(x)= G(x) + \log |x|$, which is known to be a continuous function. 
Next, integrating by parts exactly as in \eqref{ipp}--\eqref{59bis},  we easily find that
$$W(j_{\{a_i\}} )= \hal\( \sum_{i \neq j} G(a_i- a_j) + \sum_{i=1}^n R(0)\),$$
i.e. we deduce the result of Lemma \ref{lemperiointro}.

 Note that we may also compute $W$ in Fourier series just as above and we find \eqref{formexpl} i.e. 
 $$W(j_{\{a_i\}} )= \hal \sum_{i \neq j} \sum_{p \in (\mz \vu + \mz \vv)^*\sm \{0\}} \frac{e^{2i\pi p \cdot (a_i- a_j)}  }{4\pi^2 |p|^2 } + \frac{n}{2}
 \lim_{x\to 0} \( \sum_{p \in (\mz \vu + \mz \vv)^*\backslash
\{0\} } \frac{ e^{2i\pi p \cdot x}}{ 4\pi^2 |p|^2}  + \log
|x|\), $$
which is formula \eqref{formexpl}.

This may also be rewritten in the form
$$W(j_{\{a_i\}} )=\hal \sum_{i\neq j} H_{\mz \vu + \mz \vv} (a_i- a_j)+ n W(\mz \vu + \mz \vv)$$
where $H_{\mz \vu + \mz \vv}$ is as in \eqref{dhl} or \eqref{hfour} and  $W(\mz \vu + \mz \vv)$ as in \eqref{59bis}. This way $W$ is expressed as a sum of the form $\sum_{i \neq j} f(a_i - a_j)$.
Note that the series defining $H_{\mz \vu + \mz \vv}$ is an Eisenstein series and can thus be expressed in terms of the Dedekind eta function via the ``second Kronecker limit formula" as in the proof of Lemma \ref{eisen} below, see also \cite{bs} for more computations.
\subsection{Proof of  Theorem~\ref{th2}}

First, from \eqref{minalai} we may consider only the case $m=1$.
Thus the lattice $\Lambda$ is in $\lattices$ iff its fundamental cell has area $2\pi$. We return to the expression \eqref{eis} for $W(\Lambda)$, and, using standard functions and formulas from number theory,
 give a closed form for it in terms of  the Dedekind eta function.
One can also view the series expression of $W(\Lambda)$ as a regularization, or
renormalization, of the divergent series
$$ \sum_{ p \in \Lambda^*\backslash \{0\} }  \frac{1}{8\pi^2 |p|^2}.$$
We also show that modulo a constant independent of $\Lambda$, this particular regularization is equal to the regularization which uses the Zeta functions
$$\zeta_{\Lambda^*}(x) =  \sum_{p \in \Lambda^* \backslash \{0\}}
\frac{1}{8\pi^2 |p|^{2+x}},$$
for which the minimizers w.r.t.  $\Lambda $ are known.  Both results are the object of the following lemmas:
\begin{lem}\label{eisen}
We have \begin{equation}\label{Weta}
W(\Lambda)  = - \hal\log (\sqrt{2\pi b } |\eta(\tau)|^2)\end{equation} with $\eta $ the Dedekind eta function,
where the dual lattice $\Lambda^*$ to $\Lambda$ has, up to rotation, a fundamental cell given by the vectors $\frac{1}{\sqrt{2\pi b}} (1,0)$ and $\frac{1}{\sqrt{2\pi b}}(a, b)$, and $\tau $ denotes $a+ib$.\end{lem}
\begin{proof}
We may parametrize  $\Lambda^* $ as $\{ \frac{1}{\sqrt{2\pi b}} ( m \tau +n) , (m, n) \in \mz^2\}$. Then for $x=(x_1, x_2)$ we have
\begin{equation}\label{relserie}
\sum_{p \in \Lambda^*\backslash \{0\} } \frac{
e^{2i\pi p \cdot x}}{ 4\pi^2 |p|^2} = \frac{1}{2\pi } \sum_{(m,n) \in \mz^2 \backslash \{0\}} e^{\frac{2i\pi}{\sqrt{2\pi b} } ( m (ax_1 +bx_2) + n x_1)      }\frac{b}{|m\tau + n|^2}.\end{equation}
One can recognize that
this is an Eisenstein series i.e. of the form
$$E_{u,v}(\tau)= \sum_{(m,n) \in \mathbb{Z}^2\backslash \{0\}}
e^{2i\pi (mu+ nv)}\frac{b}{|m\tau+ n|^2}$$ with $\tau = a+ib$,  $u = \frac{1}{\sqrt{2\pi b} } (ax_1 +bx_2) $ and $v= \frac{1}{\sqrt{2\pi b}}x_1$.
 But the ``second Kronecker limit formula"
(see \cite{lang}) states that
$$E_{u,v} (\tau) = - 2\pi \log | f(u- v\tau , \tau) q^{v^2/2}|$$ where
$$f(z,\tau):= q^{1/12} (p^{1/2}- p^{-1/2}) \prod_{n \ge 1} (1-q^n
p) (1-q^n/p)$$ with $q=e^{2i\pi\tau} $, $p = e^{2i \pi z}$. Here $z= u - v\tau= \sqrt{\frac{b}{2\pi} }(x_2-ix_1).$  As $x\to 0$, we have $p \to 1$, and then
 the expression above can be expressed in terms of  the  Dedekind eta function
$$ \eta(\tau) := q^{1/24} \prod_{n\ge 1}
(1-q^n).$$
Inserting all the above in
\eqref{relserie} we obtain  that as $|x|\to 0$,
$$\sum_{p \in \Lambda^*\backslash \{0\} } \frac{
e^{2i\pi p \cdot x}}{ 4\pi^2 |p|^2}\sim  - \log (\sqrt{2\pi b}  |\eta(\tau)|^2 |x|)$$
and combining with \eqref{eis} we find the result \eqref{Weta}.\end{proof}
\begin{lem}
\label{517}
There exists $C\in\mr$ such that  for any  $\Lambda \in \lattices$ we have
$$ W(\Lambda)=C + \lim_{x \to 0}\( \zeta_{\Lambda^*}(x) - \int_{\mr^2} \frac{\pi}{1+4\pi^2
|y|^{2+x}}\, dy\).$$
\end{lem}
Two proofs can be given for this :  one, which we give in the following subsection
below, uses standard analysis methods. The other uses the closed form \eqref{Weta} and the fact that
 the Dedekind eta function is in turn   related to the
 Epstein zeta function
 $$Z(\tau, s) = \sum_{(m,n) \in \mz^2 \backslash 0}
\frac{b^s}{|m\tau + n|^{2s}} =\frac{1}{(2\pi)^s} \sum_{p \in \frac{1}{\sqrt{2\pi b}}  (\mz + \tau
\mz)\backslash 0} \frac{1}{|p|^{2s}}$$ via the ``first Kronecker
limit formula"
$$Z(\tau, s) =
\frac{\pi}{s-1} + 2\pi (\gamma_0 - \log 2 - \log (\sqrt{b}
|\eta(\tau)|^2 ) ) + O(s-1)\quad \text{as}\  s \to 1, $$ where $\gamma_0$ is this
time the Euler constant. Putting  these pieces together correctly
yields a proof of Lemma \ref{517}.

Then, in order to minimize $W(\Lambda)$ over $\lattices$, i.e. over lattices under the
constraint  $|\tl| = 2\pi$, one is reduced to the question of
minimizing  $\zeta_{\Lambda^*}(x)$.  It is proven  in \cite{montgomery}  that
   $$\frac{2^{1+ \frac{x}{2}  } 8 \pi^2  \Gamma(1+ x/2 )}{(2\pi)^{1+ x/2}} \zeta_{\Lambda^*}(x) =   \frac{2}{x}- \frac{1}{ 1+ \frac{x}{2}} +
     \int_{1}^\infty
   (\theta_{\Lambda^*}(a)-1)(a^{-x/2}+a^{1+x/2}) \frac{da}{a}$$
   where  \begin{equation}\label{xill}
   \theta_\Lambda (a)= \sum_{p \in \Lambda}     e^{-\pi a |p|^2}
  \end{equation}
  is the Jacobi Theta function and  $\Gamma$ is the Gamma function. Moreover, from
  \cite{cassels,rankin, ennola, diananda, montgomery} and modulo rotation, the minimum of $\theta$ over $\lattices^*$ is uniquely achieved by  $(\Lambda_0)^*$, where $\Lambda_0 =  \alpha\(\mz (1,0) +\mz \(1/2,\sqrt3/2\)\)$ and the factor $\alpha$ is chosen such that $|\T_{\Lambda_0}| = 2\pi$. But, from the above formula,
  \begin{equation}
  \label{asd1} \frac{ 2^{1+ \frac{x}{2}  } 8 \pi^2\Gamma(1+ x/2 )}{(2\pi)^{1+ x/2}}
  (\zeta_{\Lambda^*}(x)- \zeta_{\Lambda_0^*} (x))=
   \int_{1}^\infty
   (\theta_{\Lambda^*}(a)-\theta_{\Lambda_0^*}(a) )(a^{-x/2}+a^{1+x/2})
   \frac{da}{a}.\end{equation}
In view of (\ref{xill}) the integrand in (\ref{asd1}) is dominated
by an integrable function independently of $x$. Hence, letting $x$
tend to $0$, Lebesgue's dominated convergence theorem yields
$$8 \pi \Gamma(1)
\lim_{x\to 0} \zeta_{\Lambda^*}(x)- \zeta_{\Lambda_0^*} (x) =
\int_1^\infty (\theta_{\Lambda^*}(a)-\theta_{\Lambda_0^*}(a) )
(1+ a) \, \frac{ da}{a}.
$$
Hence, using Lemma \ref{517}, we have
$$8 \pi \Gamma(1) (
W(\Lambda ) -  W(\Lambda_0))
 = \int_1^\infty
(\theta_{\Lambda^*}(a)-\theta_{\Lambda_0^*}(a) ) ( 1 + a ) \,  \frac{da}{a}.$$
We deduce that
 $ W(\Lambda ) \ge   W(\Lambda_0)$ for any $\Lambda$, with
equality if and only if $\theta_{\Lambda^*} (a)=
\theta_{\Lambda_0^*}(a)$ for almost every $a$, i.e. if,  modulo rotations,
 $\Lambda=\Lambda_0$. This proves Theorem~\ref{th2}.

\subsection*{Analytic proof of Lemma \ref{517}}
Denoting by $G_0$ the unique solution of $-\Delta G_0+G_0= 2\pi
\delta_0$ in $\mr^2$, we may write
$$\hl(x)+\log|x| =\ul +\(G_0(x) +\log|x|\),$$
where $\ul = \hl - G_0$ and $G_0(x) +\log|x|$ are $C^1$ near the
origin. Taking limits, we obtain from \eqref{59bis}
$$W(\Lambda)=\gamma_0 + \hal \ul(0),$$ where
$\gamma_0 = \hal \lim_{x\to 0}\( G_0(x) +\log|x|\)$.
 Denote by $\vp(x) =(2\pi)^{-1} e^{-|x|^2/2}$ the Gaussian distribution in $\mr^2$ and for any $n\in\mn$ let $\vp_n (x) = n^2\vp(nx)$, so that $\{\vp_n\}_n$ is an approximate identity.   We have   $\hat\vp_n(y) = e^{-|y|^2/2n^2}$. Then, since $\ul$ is continuous at $0$ and bounded in $\mr^2$, we have
$$\ul(0) = \lim_{n\to+\infty} w(n,\Lambda),\text{ where } w(n,\Lambda) = \int_{\mr^2}\vp_n(x) \ul(x)dx = \int_{\mr^2}\hat\vp_n(y) \hat\ul(y)\, dy.$$
Also, it  is standard that $\hat G_0 = 2\pi/(1+4\pi^2 |y|^2)$. Then, since $\hat\ul  = \hat{H_\Lambda} - \hat{G_0}$, we get
$$w(n,\Lambda) = \sum_{p\in\Lambda^*\sm\{0\}}\frac{e^{-|p|^2/2n^2}}{4\pi^2|p|^2} - 2\pi\int_{\mr^2}\frac{e^{-|y|^2/2n^2}}{1+4\pi^2|y|^2} \,dy.$$
 Note that $|\tls| = 1/|\tl|^{-1} = 1/2\pi$, which accounts for the factor $2\pi$ in front of the integral.

The second step is to show that
\begin{equation}\label{serie2}\lim_{n\to+\infty} w(n,\Lambda) = \lim_{x\to 0} v(x,\Lambda), \text{ where } v(x,\Lambda) =  \sum_{p\in\Lambda^*\sm\{0\}}\frac1{4\pi^2|p|^{2+x}} - \int_{\mr^2}\frac{2\pi}{1+4\pi^2|y|^{2+x}}\,dy.\end{equation}
First we truncate  $w(n,\Lambda)$ and $v(x,\Lambda)$, letting for each
$N\in\mn$
$$w^N(n,\Lambda) =  \sum_{\substack{p\in\Lambda^*\sm\{0\}\\|p|<N}}\frac{e^{-|p|^2/2n^2}}{4\pi^2|p|^2} - 2\pi\int_{B(0,N)}\frac{e^{-|y|^2/2n^2}}{1+4\pi^2|y|^2}\,dy, \quad \tilde w^N(n,\Lambda) = w(n,\Lambda) - w^N(n,\Lambda),$$
and defining similarly $v^N(x,\Lambda)$, $\tilde v^N(x,\Lambda)$. It
is clear  that for any $N$
\begin{equation}\label{proches} \lim_{n\to+\infty}w^N(n,\Lambda)  =  \lim_{x\to 0} v^N(x,\Lambda) = \frac1{4\pi^2}\sum_{\substack{p\in\Lambda^*\sm\{0\}\\|p|<N}}\frac1{|p|^2} - 2\pi\int_{B(0,N)}\frac1{1+4\pi^2|y|^2}\,dy.\end{equation}
Then we claim that
\begin{equation}\label{lointain} \lim_{N\to +\infty}\lim_{n\to+\infty} \tilde w^N(n,\Lambda) = \lim_{N\to+\infty}\lim_{x\to 0}\tilde v^N(x,\Lambda)  = 0,\end{equation}
 which together with \eqref{proches}  clearly implies \eqref{serie2}.  To prove \eqref{lointain}, for any $p\in \Lambda^*$ we let  $K_p$ be the Voronoi cell centered at $p$, i.e. the set of points in $\mr^2$ which are closer to $p$ than to any other point of $\Lambda^*$. Then
\begin{equation}\label{twnN}\tilde w^N(n,\Lambda) = \sum_{\substack{p\in\Lambda^*\sm\{0\}\\|p|\ge N}} \dashint_{K_p}\frac{e^{-|p|^2/2n^2}}{4\pi^2|p|^2} - \frac{e^{-|y|^2/2n^2}}{1+4\pi^2|y|^2}\,dy  +\delta(N,n,\Lambda),\end{equation}
where
$$\delta(N,n,\Lambda) = 2\pi \(\sum_{\substack{p\notin B(0,N)}}\int_{K_p\cap B(0,N)}\frac{e^{-|y|^2/2n^2}}{1+4\pi^2|y|^2}\,dy - \sum_{\substack{p\in B(0,N)}}\int_{K_p\sm B(0,N)}\frac{e^{-|y|^2/2n^2}}{1+4\pi^2|y|^2}\,dy\).$$
We can bound $|\delta|$ by the integral of  $(1+4\pi^2|y|^2)^{-1}$
over those cells which intersect both $B(0,N)$ and its complement,
the union of which is included in $B(0,N+C_\Lambda)\sm
B(0,N-C_\Lambda)$. Thus we have for every $\Lambda$
\begin{equation}\label{deltanN}\lim_{N\to +\infty}\lim_{n\to+\infty}\delta(N,n,\Lambda) = 0.\end{equation}
Then we consider the sum in \eqref{twnN}. If $|p|>n^2$ it is
straightforward to  bound the contribution of $K_p$, which we denote
$M_p$, by $C_\Lambda e^{-|p|/2} /|p|^2$, and from there to deduce
\begin{equation}\label{mp1}\lim_{N\to +\infty}\lim_{n\to+\infty}\sum_{\substack{p\in\Lambda^*\sm\{0\}\\|p|\ge n^2}} |M_p| = 0.\end{equation}
If $|p|\le n^2$, we use the fact that the gradient of $y\to
e^{-|y|^2/2n^2}$ is bounded by $C/n$ on $\mr^2$, and that
$e^{-|p|^2/2n^2}\le 1$ to write
$$|M_p|\le \left|\dashint_{K_p}\frac 1{4\pi^2|p|^2} - \frac1{1+4\pi^2|y|^2}\,dy\right| + \frac Cn \dashint_{K_p} \frac{dy}{1+4\pi^2|y|^2}.$$
The sum w.r.t $p\in\Lambda^*$ of the first term on the right-hand side is
convergent, therefore the sum w.r.t $p\in\Lambda^*$, $n^2\ge |p|\ge
N$ tends to $0$ as $N\to+\infty$, uniformly in $n$. For the second
term, the corresponding sum may be compared with the integral
$$\frac Cn 2\pi \int_{N\le |y|\le n^2} \frac{dy}{1+4\pi^2|y|^2},$$
which is $O(\log n/n)$ as $n\to +\infty$. Therefore
$$\lim_{N\to +\infty}\lim_{n\to+\infty}\sum_{\substack{p\in\Lambda^*\sm\{0\}\\n^2\ge |p|\ge N}} |M_p| = 0.$$
Together with \eqref{mp1} and \eqref{deltanN}, this proves that $\lim_{N\to +\infty}\lim_{n\to+\infty} \tilde w^N(n,\Lambda)=0$. The proof that $\lim_{N\to +\infty}\lim_{x\to 0} \tilde v^N(x,\Lambda)=0$ is similar, we omit it. Then \eqref{lointain} holds, which proves \eqref{serie2} and the lemma.

\section{The renormalized energy in the general case: proof of Theorem \ref{thw}}\label{secW}
 In this section we prove Theorem \ref{thw}. 
We recall that we can  reduce by scaling
 to studying  the case
of   $\mathcal{A}_m= \ainfty$ i.e.  where (see Definition \ref{defA})
 \begin{equation}\label{eqji}
 \curl j = \nu - 1\qquad \div j=0,\end{equation}
where $\nu = 2\pi \sum_{p \in \Lambda} \delta_p$ is such that   $\{\frac{\nu(B_R)}{|B_R|}\}_R$ is bounded.
 For convenience, once the class $\ainfty$ has been chosen,  if $p<2$ we may extend the definition of $W_U$ to all $j\in \Lp$ by letting
$$ W_U(j) = +\infty  \quad\text{if $j \notin\ainfty$}.$$  We now assume $W$ has been  likewise extended for a certain $1<p<2$.  In what follows we will show that the minimizers and the minimum of $W_U$ do not
depend on $U$. However $W_U$ itself does.
When a statement is independent of $U$, we will sometimes  write $W$ instead of $W_U$.

One of the difficulties about $W$ is that it is not lower semi-continuous. Here is a hint as to why: consider  a set of points $\Lambda_n $  which is equal to the square  lattice (of density $1$) in the ball $B_n$ and to the triangular lattice (of density $1$)  outside $B_n$.   Let $j_n$ be corresponding vector fields.  Then, since $W$ is insensitive to compact perturbations, we have for every $n$, in informal notation,  $W(j_n)= W(\trl)$. On the other hand, as $n \to \infty$, $j_n\to j_{\sql}$ by construction. So if $W$ was lower semi-continuous, we should have
$W(\sql) \le W(\trl)$, which is false by Theorem \ref{th2}.

However, we will show that $W$ is lower semi-continuous ``up to translations".

We split the results into several propositions:
\begin{pro}\label{pw1} Let $U$ refer to  any  family $\{\UR\}$ satisfying \eqref{hypsets}, \eqref{hypsetsbis}. Let $1<p<2$. We have:
\begin{itemize}
\item[-] The value of $W_U(j)$ is independent of the choice of cutoff functions $\{\chi_\UR\}_R$ satisfying \eqref{defchi}.
\item[-] $W_U : \Lp\to \mr\cup\{+\infty\}$ is a Borel function.
\item[-] The sublevel sets  $\{j, W_U(j)\le \alpha\} $ are ``compact in $\Lp$ up to translation", more precisely: for every $j_n $  such that $\limsup_{n \to \infty} W(j_n) <+\infty$, after extraction of a subsequence, there exists a sequence  $\lambda_n \in \mr^2$ and $j\in \Lp$  such that $ j_n(\lambda_n + \cdot ) \to j$  in $\Lp$ and
\begin{equation}\label{wuj}
W_U(j)\le \liminf_{n \to \infty}  W_U(j_n).\end{equation}
 In particular
$\inf_{\Lp} W_U = inf_\ainfty W_U$ is achieved and is finite.
\item[-] The minimizers and the value of the minimum of $W_U$ are independent of $U$.
\end{itemize}
\end{pro}

We will also prove a result of  uniform approximation by  construction, which implies in particular   that the minimum of $W$ may
 be approximated by suitable periodic configurations.
 This construction will be crucial in particular in constructing test functions.
In what follows $K_R$ denotes the square $[-R, R]^2$, and $W_K$ denotes the renormalized energy relative to $\{K_R\}_R$. In the following results, squares could easily be replaced by arbitrary parallelograms.
\begin{pro}\label{pw2bis} Let $G\subset \mathcal{A}_1 $  be such that there exist $C>0$ such that for any $j \in G$,  
we have \begin{equation}\label{hypunif} 
\forall R>1, \quad \frac{\nu(K_R)}{|K_R|} <C
\end{equation} for the associated $\nu$'s and such that,   uniformly with respect to $j \in G$, 
\begin{equation}
\label{cvwunif} \lim_{R\to + \infty} \frac{W(j, \chi_{K_R})}{|K_R|} = W_K(j) \le C.\end{equation}
 Then for any $\ep>0$, there exists $R_0>0$ such that if $R>R_0$ and $|K_R|\in 2\pi \mn$, for any $j \in G$ there exists $j_R$ such that 
\begin{equation}\label{eqjr}
\left\{ \begin{array}{ll}
\curl j_R =
 2\pi\sum_{p\in\Lambda_R}\delta_p - 1\  & \text{ in} \  K_R,\\
  j_R \cdot \tau = 0 & \text{on} \ \p K_R,
 \end{array}\right.\end{equation}
   where
 $\Lambda_R$ is a discrete subset of the interior of $K_R$, and
 \begin{itemize}
 \item[-]
 if $ x\in K_{R - 2R^{3/4} } $ then $j_R(x)=j(x)$, 
\item[-] $\D \frac{W(j_R,\indic_{K_R})}{|K_R|}
\le  W_K(j)+\ep.$
\end{itemize}
\end{pro}
\begin{remark} An inspection of the proof allows to see that the construction can alternatively  be made in any rectangle which is a small perturbation of the square $K_R$, i.e whose sidelengths are in  $[2R, 2R(1+\eta)]$, where $\eta$ depends on $\ep$.\end{remark}

Note that this result is close  to establishing  that $\ainfty$ is ``uniformly $W$-approximable" in the sense of \cite{am}, Definition 4.14.

The following corollaries are proved at the end of this section.
\begin{coro}
 \label{periodisation} Given any $j$ such that $W_K(j) <\infty$,
there exists a sequence $\{j_R\}_{R^2 \in 2 \pi \mn}$ in $\ainfty$ such that
each $j_R$ is $K_R$ periodic (i.e.
 $j_R(x+ 2Rk e_1+ 2R l e_2)= j_R(x) $ for $k, l \in \mz$ where $(e_1, e_2)$ is the
 canonical basis of $\mr^2$)  and
\begin{equation}\label{bsupper}\limsup_{R\to \infty} W_K(j_R) = \limsup_{R\to \infty} \frac{W(j_R,\indic_{K_R})}{|K_R|}\le  W_K(j).\end{equation}
In particular there exists a
minimizing sequence
for $\min_{\ainfty} W$  consisting of periodic vector-fields.
\end{coro}

\begin{coro}
\label{mesyoung} Let $p<2$ and
let $P$ be a probability measure  on $L^p_{loc}(\mr^2, \mr^2)$ which is invariant under the action of  translations
 and concentrated in  $ \ainfty$. Then there exists a sequence $R \to + \infty$ and  a sequence $\{j_R\}_{R}$ of vector fields  defined over $K_R$ such that
\begin{itemize}
\item[-]   There exists a finite subset $\Lambda_R$ of the interior of $K_R$ such that
\begin{equation}\label{c4i}\left\{\begin{array}{ll}
\curl j_R =
 2\pi\sum_{p\in\Lambda_R}\delta_p - 1 & \text{ in }\  K_R\\
 j_R \cdot \tau=0 & \text{on} \ \p K_R.\end{array}\right.
\end{equation}

\item[-] Letting $P_R$ be the probability measure  on $L^p_{loc}(\mr^2, \mr^2) $ which is defined as the image of the normalized Lebesgue measure on $K_R$ by
$x\mapsto j_R(x + \cdot )$ where $j_R$ is extended periodically to the whole $\mr^2$, we have
$P_R \to P$ weakly  as $R \to \infty$.
\item[-] $\D\limsup_{R\to\infty} \frac{W(j_R,\indic_{K_R}    )}{|K_R|}
\le\int  W_K(j) \, dP(j)$.
\end{itemize}
\end{coro}

\subsection{A mass displacement result}
In this subsection, we show that even though the integrand in the definition of $W$ is not bounded below, we can somehow reduce to that case by a ``mass displacement" method, which is
an adaptation of that of \cite{compagnon}. The situation is  much simpler here due to  the fact that  all degrees are $+1$. It  follows the same steps however, consisting in ball construction combined with mass displacement.

We begin with a ball construction argument \`a la  Sandier and
Jerrard \cite{sandier,jerr}. For the sake of generality, we prove the result for an average density $m$ which may depend on $x$.
\begin{pro}\label{boulesren} Assume $\curl j = \nu -m$ in the sense of distributions in some open set $U$,  where $m\in L^\infty(U)$ and $\nu = 2\pi \sum_{p\in\Lambda} \delta_p$ for some finite subset $\Lambda$ of $U$, and that $j\in L^2_\loc(U\sm\Lambda)$. Let $n = \#\Lambda$ and $\eta_0>0$ be the minimal distance between points of $\Lambda$. Then there exists for any $r\in(0,1]$ a family of disjoint closed balls $\B_r$ of total radius $r$ covering $\Lambda$  such that:
\begin{itemize}
\item[-] If $\Lambda = \{p\}$ then $\B_r = \{B(p,r)\}$.
\item[-] If $r/n\le \eta_0$, then $\B_r = \{B(p,\frac{r}{n})\}_{p\in\Lambda}$.
\item[-]  The set $\cup_{B\in\B_r}B$ is increasing   as a function of $r$. Moreover, for any $\eta \le r/n$,  and  every $B \in \B_r$ such that $B \subset U$ we have, letting $d_B= \#(\Lambda\cap B)$, 
$$\hal \int_{B\sm\cup_{p\in\Lambda}B(p,\eta)} |j|^2\ge \pi d_B\(\log\frac r{n\eta} - \|m\|_\infty\frac{r^2}{2}\).$$
\item[-] If $B\in\B_r$  and  $\chi$ is a positive function with support in  $U$, then
$$ \int_{B\sm\cup_{p\in\Lambda}B(p,\eta)} \chi |j|^2 - 2\pi\(\log\frac r{n\eta} - \|m\|_\infty\frac{r^2}{2}\)\sum_{p\in B\cap\Lambda} \chi(p)\ge - 2 r \nu(B) \|\nab\chi\|_\infty.$$
\end{itemize}
\end{pro}
\begin{proof}[Proof of the  first three items] We let $M = \|m\|_\infty$.
The first items are obtained by a standard ball construction argument \`a la Jerrard/Sandier. Since $\curl j = \nu -m$ we have for any circle $C = \partial B$ of radius $r(B)$ not intersecting $\Lambda$, and letting $d_B =  \#(\Lambda\cap B)$,
$$\int_{C} j\cdot\tau \ge  \nu(B) - M\pi {r_B}^2 = 2\pi d_B  - M\pi {r_B}^2.$$
Using the Cauchy-Schwarz inequality  and the fact that $d_B$ is
 a nonnegative integer, we deduce  that
\begin{equation}\label{lecercle}
\int_{C}|j|^2 \ge \frac{2\pi}{r_B} (d_B)^2 - 2\pi M d_{B} r_B
\ge 2\pi d_{B}\( \frac{1}{r_B} - M r_B \).\end{equation}

Define $\F(x,r) = \int_{B(x,r)} |j|^2$. The above yields
$$\frac{\partial \F}{\partial r} \ge   2\pi d_B\(\frac1{r_B} - M r_B\).$$

%\footnote{il faut revoir la preuve qui suit car j'ai essaye de clarifier en utilisant la redaction de Serfaty-Tice mais suis pas sure d'y etre arrivee}

In order to define $\mathcal{B}_r$, we first fix a reference family of balls produced via a ball-growth.  Set $\eta_1 = \min\{\eta_0/3,1/(n+1)\}$ and let $\mathcal{B}_0  = \{ \bar{B}(p,\eta_1)\}_{p \in \Lambda}$.  According to the definition of $\eta_0$, we have that $\mathcal{B}_0$ is a finite, disjoint collection of closed balls of total radius $n \eta_1 < 1$. We then apply Theorem 4.2 of \cite{livre} to $\mathcal{B}_0$  to produce a family of collections $\{ \mathcal{B}(s) \}_{s \in [0,\log\frac{1}{n\eta_1}]}$, 
such that the total radius of the balls in $\mathcal{B}(s)$ is $r=n\eta_1e^s$. 
Then, given $n\eta_1<r<1$ we may choose $s = \log\frac r{n\eta_1}$ and write $\B_r$ for $\B(s)$. We then extend this reference family ``backward'' to radii smaller than $n \eta_1$ by letting $\B_r = \{\bar{B}(p,r/n)\}_{p\in \Lambda}$ for any $0 < r \le n \eta_1$.  

Since the balls in these collections never become tangent when $r<n\eta_1$, for any $\eta>0$ we may trivially view $\{\B(s)\}_s$, where $\B(s) = \B_r$, with $r = {n\eta}e^s$ as having been generated by a ball-growth from $\B(0) = \{\bar{B}(p,\eta)\}_{p\in \Lambda}$, i.e. satisfying all the results of Theorem 4.2 in \cite{livre}. Then each $\B_r$  has total radius $r$ and covers $\Lambda$ and from Proposition~4.1 in \cite{livre} for every $B\in\B_r = \B(s)$, with $r = {n\eta}e^s$, 
$$\int_{B\sm\cup_{p\in\Lambda}B(p,\eta)}|j|^2 \ge \int_0^s \sum_{\substack{B'\in \B(t)\\B'\subset B}}  2\pi d_{B'} (1 - M{r_{B'}}^2) \,dt.$$ 
Since the sum of the $d_{B'}$'s is $\#(B\cap\Lambda)$ and since the sum of the $r_{B'}$'s is less than $n\eta e^t$, we find
$$\int_{B\sm\cup_{p\in\Lambda}B(p,\eta)}|j|^2 \ge 2\pi d_B\int_0^s (1 - M( n\eta e^t)^2) \,dt = 2\pi d_B \(\log\frac r{n\eta} - M\frac{r^2-n^2{\eta}^2}2\).$$
Since  this is true for any $0<\eta<r/n$, the first three items are satisfied.
\end{proof}

\begin{proof}[Proof of the last item]

Let $B_\eta = B\sm\cup_{p\in\Lambda}B(p,\eta)$. Then by the ``layer-cake" theorem
\begin{equation}\label{machin}\int_{B_\eta} \chi |j|^2  = \int_0^{+\infty} \(\int_{B_\eta\cap\{\chi>t\}}|j|^2\)\,dt.\end{equation}

Now  if $a\in\Lambda\cap B$,  then for
any $s\in(0,r]$  there exists a closed  ball $B_{a,s}\in\B_s$ containing $a$. For $t>0$ we call
$$ s(a,t) = \sup \{ s \in (0, 1], B_{a,s}\subset \{\chi>t\}\}$$
if this set is nonempty, and
let $s(a,t) = 0$ otherwise, i.e. if $\chi(a)\le t$. Then we let $B_a^t = B_{a,s(a,t)}.$ Note that $a$ is not necessarily the center of $B_a^t$, note also that $s(a,t)$  bounds from above the radius of $B_a^t$, but is not necessarily equal to it.

As noted above   $s(a,t) = 0$ iff $\chi(a) \le t$ while if  $s(a,t)\in(0,r)$ then $B_a^t\not\subset \{\chi>t\}$, otherwise there would exist $s'>s(a,t)$ such that  $B_{a,s'}\subset \{\chi>t\}$, contradicting the definition of $s(a,t)$. Thus, choosing $y$ in $B_a^t\sm \{\chi>t\}$, we have
\begin{equation}\label{boundr} \chi(a) - t\le \chi(a) - \chi(y)\le 2 s(a,t)\|\nab\chi\|_\infty.\end{equation}
Also, for any $t\ge 0$ the collection $\{B_a^t\}_a$, where $a\in\Lambda$ and the $a$'s for
which $s(a,t) = 0$ have been excluded, is disjoint. Indeed if  $a,b\in \Lambda$ and $s(a,t)\ge s(b,t)$ then, since
$\B_{s(a,t)}$ is disjoint, the balls $B_{a,s(a,t)}$ and
$B_{b,s(a,t)}$ are either equal or disjoint. If they are disjoint we
note that $s(a,t)\ge s(b,t)$ implies that $B_{b, s(b,t)}\subset
B_{b,s(a,t)}$ and therefore $B_b^t = B_{b,s(b,t)}$ and $B_a^t =
B_{a,s(a,t)}$ are disjoint. If they are equal, then
$B_{b,s(a,t)}\subset E_t\cap B$ and therefore $s(b,t)\ge s(a,t)$,
which implies $s(b,t) = s(a,t)$ and then $B_b^t = B_a^t$.

Now assume that  $B'\in \{B_a^t\}_a$ and let $s$ be the common value of $s(a,t)$ for $a$'s in $B'\cap\Lambda$ and $n = \#\Lambda$.  Then the previous item of the proposition yields for any  $\eta<\min(\eta_0,r/n)$  (but the inequality is trivially true if $\eta>r/n$),
$$\int_{B'\sm\cup_{p\in\Lambda}B(p,\eta)} |j|^2\ge \nu(B')\(\log\frac s{n\eta} - M\frac{s^2}{2}\)_+.$$
We may rewrite the above as
$$\int_{B'\sm\cup_{p\in\Lambda}B(p,\eta)} |j|^2\ge 2\pi\sum_{a\in B'\cap\Lambda}\(\log\frac {s(a,t)}{n\eta} - M\frac{s(a,t)^2}{2} \)_+$$
and summing over $B'\in \{B_a^t\}_a$ we deduce, noting that the $a$'s for which $s(a,t)=0$ do not contribute to the sum,
$$\int_{B_\eta\cap\{\chi>t\}} |j|^2\ge 2\pi\sum_{a\in B\cap\Lambda}\(\log\frac {s(a,t)}{n\eta} - M\frac{s(a,t)^2}{2} \)_+.$$
Integrating the above  in view of \eqref{machin} yields, using \eqref{boundr} and the fact that $s(a,t)\le r$,
\begin{multline*}\int_{B_\eta} \chi |j|^2\ge 2\pi\sum_{a\in B\cap\Lambda}\int_0^{\chi(a)} \(\log\frac {s(a,t)}{n\eta} - M\frac{s(a,t)^2}{2} \)_+\,dt \ge \\
2\pi\sum_{a\in B\cap\Lambda}\int_0^{\chi(a)} \(\log\frac r{n\eta} +\log \(\frac{\chi(a) - t}{2r\|\nab\chi\|_\infty}\wedge 1\) - M\frac{r^2}{2} \)\,dt.
\end{multline*}
The right-hand side is greater than
$$2\pi\sum_{a\in B\cap\Lambda}\(\chi(a)\(\log\frac r{n\eta} -M\frac{r^2}{2}\) - 2r \|\nab\chi\|_\infty\),$$
which proves the result.
\end{proof}

We deduce  a control of $j$ by $W$, which we point out is substantially 
improved in \cite{st} into a control in a critical Lorentz space.
 \begin{lem}\label{wlp} Let $\chi$ be a positive function  compactly supported in  an open set $U$  and assume that $\curl j = \nu -m$ in $$
\widehat U := \{x\mid d(x,U)<1\}$$ where $\nu = 2\pi \sum_{p\in\Lambda} \delta_p$ for some finite subset $\Lambda$ of $\widehat U$.  Then, there exists $C>0$ universal and for any $p\in[1,2)$, $C_p>0$ depending only on $p$, such that 
$$ \int_{U} \chi^{p/2} |j|^p \le C (|U|+C_p)^{1-p/2}\(W(j,\chi)+ n(\log n+ \|m\|_\infty) \|\chi\|_{\infty}+ n\|\nab\chi\|_\infty\)^{p/2}.$$
where $n = \nu(\widehat U)/2\pi = \#\Lambda$.
\end{lem}
\begin{proof} We use an argument of M. Struwe \cite{struwe}.
 Let $M = \|m\|_\infty$. We construct as above balls $\B_r$ in $\widehat U$ and define for $k\ge 1$ the set $U_k = \B_{2^{-k}}\setminus \B_{2^{-(k+1)}}$ and $U_0 = \widehat U\sm \B_{1/2}$. Then since $\chi$ is supported in $U$, we have by H\"older's inequality,
\begin{equation}\label{uk}\int_{U} \chi^{p/2} |j|^p = \int_{\widehat U} \chi^{p/2} |j|^p = \sum_{k=0}^\infty \int_{U_k} \chi^{p/2} |j|^p \le \sum_{k=0}^\infty |U_k|^{1-p/2}\(\int_{U_k} \chi |j|^2\)^{p/2}.\end{equation}
For any given $k\in\mn$ and if $\eta$ is small enough we have   $U_k\subset U_\eta\sm (\B_{2^{-(k+1)}})_\eta,$ where we have written $A_\eta = A\sm\cup_{p\in\Lambda}B(p,\eta)$. For any ball $B\in\B_{2^{-(k+1)}}$, if $B$ intersects the support of $\chi$ then $B\subset \widehat U$ and thus the previous proposition yields
$$\int_{B_\eta} \chi |j|^2 \ge 2\pi\(\log\frac{2^{-(k+1)}}{n\eta} - CM\)\sum_{p\in B\cap\Lambda} \chi(p) - C \nu(B) \|\nab\chi\|_\infty.$$
Summing over balls in $\B_{2^{-(k+1)}}$ and subtracting from $\int_{U_\eta}\chi|j|^2$ we find
$$ \int_{U_k}  \chi |j|^2\le \int_{U_\eta}\chi |j|^2 -  2\pi\(\log\frac{2^{-(k+1)}}{n\eta} - CM\)\sum_{p\in U\cap\Lambda} \chi(p) + C \nu(U) \|\nab\chi\|_\infty.$$
Taking the limit $\eta\to 0$ and plugging into \eqref{uk} we find, recalling that $n =  \#\Lambda$,
$$\int_{U} \chi^{p/2} |j|^p \le  \sum_{k=0}^\infty |U_k|^{1-p/2}\(W(j,\chi) +Cn  \(M+k + \log n\)\|\chi\|_{\infty} + C n \|\nab\chi\|_\infty\)^{p/2}.$$
We have $|U_0|\le |U|$ while  for $k\ge 1$ we have $|U_k|\le C2^{-k}$. It follows easily that for some $C_p,C>0$,
$$ \int_{U} \chi^{p/2} |j|^p \le C (|U|+C_p)^{1-p/2}\(2W(j,\chi)+ n(\log n+M) \|\chi\|_{\infty}+ n\|\nab\chi\|_\infty\)^{p/2}.$$
\end{proof}

\begin{lem}\label{wsci} Assume that $m\in L^\infty(B_R)$ and  $\{(j_n,\nu_n)\}_n$ is a sequence in $\Lp\times\radon$ such that $\nu_n$ restricted to $B_R$  is of the form $2\pi\sum_i\delta_{a_i}$ for every $n$, with $\div j_n=0$ and $\curl j_n=\nu_n-m$,  and which converges in the distributional sense to $(j,\nu)$ on $B_R$.

Then   $\div j = 0$ and $\curl j = \nu-m$ on $B_R$, where $\nu$ is a locally finite sum of the form $2\pi \sum_{p\in\Lambda}d_p \delta_p$, where $\Lambda$ is a discrete subset of $B_R$ and $d_p\in\mn^*$. Moreover, if  $\chi\in C^\infty_c(B_R)$ is  positive and $\sup_n W(j_n,\chi) <+\infty$, then $d_p=1$ for every $p$ such that $\chi(p)\neq 0$. In addition, for any smooth function $\xi$ compactly supported in $\{\chi>0\}$ we have
$$\liminf_{n\to +\infty} W(j_n,\xi)= W(j,\xi).$$
\end{lem}
\begin{proof}
This is essentially a consequence of the type of  analysis of \cite{bbh}.

First, since $\nu_n$ is positive for each $n$, the distributional convergence of $\{\nu_n\}_n$ is in fact a weak convergence of the  measures, and $\{\nu_n\}_n$ is locally bounded on $B_R$, and we deduce that $\nu$ is of the form $2\pi \sum_{p\in\Lambda}d_p \delta_p$ with $d_p\in\mn^*$. That $\div j = 0$ and $\curl j = \nu-m$ on $B_R$ follows by passing to the limit in the corresponding equations satisfied by $j_n$, $\nu_n$.

Now assume $\sup_n W(j_n,\chi) <+\infty$, let $U$ be an open set compactly included in $B_R$ and containing the support of $\chi$,  and choose $\eta>0$. Let $U_\eta = U\sm\cup_{p\in\Lambda} B(p,\eta)$. If $n$ is large enough depending on $\eta$ we have $\Lambda_n\cap U_\eta = \varnothing$ hence $\div (j-j_n) = \curl(j-j_n) = 0$ in $U_\eta$. Elliptic regularity then implies  that $j_n\to j$ in $C^k(U_{2\eta})$  for any $k\in\mn$ and thus we have convergence in $C^k_\loc(B_R\sm\Lambda)$. In particular we have for any $\rho>0$
$$\lim_{n\to\infty} \hal\int_{U_\rho}\chi |j_n|^2 =  \hal\int_{U_\rho}\chi |j|^2.$$

We choose $\rho$ small enough so that the balls $B(p,\rho)$ for $p\in\Lambda\cap U$ are disjoint. Then for each $p\in\Lambda\cap U$ and $n$ large enough, there are exactly $d_p\ge 1$ points $p_n^1,\dots,p_n^{d_p}$ in $\Lambda_n\cap B(p,\rho)$, and these points converge to $p$. We apply Proposition~\ref{boulesren} to $(j_n,\nu_n)$  in $B(p,\rho)$, with $r = \rho^2$. We deduce the existence  of a family of balls $\B_n$  of total radius $\rho^2$ containing the points $p_n^k$ and such that for $\eta>0$ small enough
$$ \int_{\B_n\sm\cup_{k}B(p_n^k,\eta)} \chi |j_n|^2 - 2\pi\(\log\frac{ \rho^2}{d_p\eta} - CM\)\sum_{k=1}^{d_p} \chi(p_n^k)\ge - C d_p \rho^2 \|\nab\chi\|_\infty.$$
 Moreover, since the $p_n^k$'s converge to $p$ as $n\to +\infty$, if $n$ is chosen large enough then they are at distance less than $\rho^2$ from $p$ hence $\B_n\subset B(p,2\rho^2)$. Using \eqref{lecercle} to bound from below $\int_{\partial B(p,t)}|j_n|^2$ for $2\rho^2<t<\rho$, we find
$$\hal \int_{B(p,\rho)\sm\B_n}\chi |j_n|^2 \ge \hal\( \min_{B(p,\rho)} \chi \)\int_{B(p,\rho)\sm\B_n}|j_n|^2\ge \pi \(\min_{B(p,\rho)} \chi \) {d_p}^2 \log \frac\rho{2\rho^2} - C d_p\rho.$$
Also, since on $B(p,\rho)$ we have $\min_{B(p,\rho)}\chi \ge \chi - 2\rho\|\nab\chi\|_\infty$, we may write
$$ \(\min_{B(p,\rho)} \chi \) {d_p}^2 \ge \sum_k \chi(p_n^k) + \chi(p) \({d_p}^2 - d_p\)  - C\rho{d_p}^2\|\nab\chi\|_\infty.$$
Putting together the above lower bounds, replacing $\chi(p_n^k)$ by $\chi(p) - C \ro \|\nab \chi\|_{\infty}$,   and summing with respect to $k$ and $p$, we deduce
\begin{multline*}\hal\int_{U\sm\cup_{p\in\Lambda_n}B(p,\eta)}\chi |j_n|^2\ge
\hal \int_{U_\ro} \chi |j_n|^2 +
  \pi \log\frac{\rho^2}{\eta}\sum_{p\in\Lambda_n}\chi(p) +
\pi \log \frac\rho{2\rho^2}\sum_{p\in\Lambda_n}\chi(p) \\
+
\pi \log \frac\rho{2\rho^2}\sum_{p\in\Lambda}({d_p}^2-d_p)\chi(p)
- C\sum_{p\in\Lambda} d_p (\log d_p +1) C(\chi)
- C \rho \log \frac1{2\rho}\sum_{p\in\Lambda}{d_p}^2 C(\chi),
\end{multline*}
where $C(\chi)$ denotes a constant depending only on the function $\chi$ and possibly on $M$. Adding $ \pi\log\eta \sum_p\chi(p)$ on both sides and passing to the limit $\eta\to 0$ we find
$$ W(j_n,\chi)\ge \hal\int_{U_\rho}\chi |j_n|^2 + \pi \log\frac{\rho}{2}\sum_{p\in\Lambda_n}\chi(p) +
\pi \log \frac1{2\rho}\sum_{p\in\Lambda}({d_p}^2-d_p)\chi(p) - R, $$
where the error term $R$ is bounded independently of $\rho\in(0,1]$, $n$. This inequality is true for $n$ large enough, but the left-hand side was assumed to be bounded above independently of $n$, hence the right-hand side is bounded above independently of $\rho$, $n$, which can only be true if $({d_p}^2-d_p-1)\chi(p) \le 0$ for every $p\in\Lambda$, i.e. (since $d_p$ is a nonzero integer)  if $d_p=1$  for every $p\in\Lambda$ such that $\chi(p)\neq 0$.

We  now prove  that $\liminf_{n\to +\infty} W(j_n,\xi) =  W(j,\xi)$ if $\supp(\xi)\subset\{\chi>0\}.$ For this we note that since $j_n\to j$ in $C^k_\loc(B_R\sm\Lambda)$, then for any $\rho>0$ we have
$$I(j_n,\rho):= \hal\int_{U_\rho} \xi|j_n|^2 + \pi\sum_{p\in\Lambda_n} \xi(p)\log\rho \xrightarrow{n\to +\infty} I(j,\rho).$$
Then, the convergence of $W(j_n,\xi)$ to $W(j,\xi)$ will follow if we may reverse the limits $n\to\infty$ and $\rho\to 0$, which is the case provided for instance that $I(j_n,\rho)$ tends to $W(j_n,\xi)$ as $\rho\to 0$ {\em uniformly with respect to} $n$.

We prove this fact. Denote  by $\{p_1,\dots,p_k\}$ the set $\Lambda\cap\supp(\xi)$. We know already that the points are distinct hence there exist a neighborhood $\mathcal{V}$ of  $\supp(\xi)$ and  $k$ sequences $\{p^n_i\}_n$, $1\le i \le k$ such that $\Lambda_n \cap \mathcal{V} = \{p^n_i\}_i$ and such that $p^n_i\to p_i$ as $n\to +\infty$. There exists $N_0$ such that if $n>N_0$ then $|p^n_i - p_i|<\rho_0$ when $i\neq j$, where $\rho_0 = \frac14 min_{i\neq j} |p_i-p_j|$.

Now choose $\rho_1<\rho_2<\rho_0$ and $n>N_0$. We have  $j_n = -\np H_n$, where $-\Delta H_n = 2\pi\delta_{p_i^n}-m$  in $B(p_i^n, 2\rho_0)$ with  $\|m\|_\infty<M$. Then $H_n = \log|\cdot - p_i^n| + H_{i,n}$, where $H_{i,n}$ is bounded by a constant independent of $n$ in $C^1(B(p_i^n,\rho))$. It follows straightforwardly   by writing $|j_n|^2 = |\nab\log + \nab H_{i,n}|^2$, expanding and estimating each term, that
$$\left|\hal\int_{B(p_i^n,\rho_2)\sm B(p_i^n,\rho_1)}\xi |j_n|^2 - \pi\xi(p_i^n)\log\frac{\rho_2}{\rho_1}\right|\le C\rho_2\|\nab\xi\|_\infty+C\sqrt{\rho_2}\|\xi\|_\infty, $$
where $C$ is independent of $n\le N_0$. The left-hand side is nothing but $|I(j_n,\rho_2) - I(j_n,\rho_1)|$ thus we have proved the uniform convergence w.r.t. $n$ of $I(j_n,\rho)$ as $\rho\to 0$, and then it follows  that $W(j_n,\xi)\to W(j,\xi)$.
\end{proof}

The energy density defining $W$, $|j|^2 + \log \eta \sum_p \delta_p$ has no sign, which makes it impossible to  apply directly Theorem \ref{gamma} to it. We now show that it can be modified into a density bounded below by a constant, using again the mass displacement method introduced in \cite{compagnon} (but in a simpler setting), that is by absorbing the negative part into the positive part while making a controlled error.

The next proposition  summarizes the properties of the modified density $g$.

\begin{pro}\label{wspread} Assume $U\subset\mr^2$ is open
 and $(j,\nu)$ are such that $\nu = 2\pi \sum_{p\in\Lambda}
  \delta_p$ for some finite subset $\Lambda$ of $\widehat U$
   and  $\curl j = \nu -m$, $\div j = 0$ in $\widehat U$,
    where $m\in L^\infty(\widehat U)$. Then there exists a measure  $g$ supported  on $\widehat U$ and such that
\begin{itemize}
\item[-] $g\ge -C(\|m\|_\infty^2+1)$ on $\widehat U$, where $C$ is a universal constant.
\item[-]  For any function $\chi$ compactly supported in $U$ we have
\begin{equation}\label{wg}\left|W(j,\chi) - \int \chi\,dg\right|
\le C n (\log n + \|m\|_\infty)\|\nab\chi\|_\infty,\end{equation}
where  $n=\#\{p\in\Lambda\mid B(p,C)\cap\supp(\nabla\chi)\neq\varnothing\}$ and $C$ is universal.
\item[-] For any  $E\subset U$,
\begin{equation}\label{bgnalpha}\#(\Lambda\cap E) \le C\(1+\|m\|_\infty^2|\widehat E| + g(\widehat E)\),\end{equation} where $C$ is universal.

\end{itemize}
\end{pro}

\begin{proof} The proof follows the method of
\cite{compagnon}. Throughout $M = \|m\|_\infty$.  We cover $\mr^2$ by the balls
of radius $1/4$ whose centers are in $\frac{\mz}{4}\times\frac{ \mz}{4}$.
We call this
cover
 $\{U_\alpha\}_\alpha$ and $ \{x_\alpha\}_\alpha$ the centers.
  In each $U_\alpha\cap\widehat U$ and for any $r\in(0,1/4)$
  we construct disjoint
   balls $\B_r^\alpha$ using Proposition~\ref{boulesren}.
   Then choosing a small enough  $\rho\in(0,1/4)$ to be
    specified below,  we may extract  from $\cup_\alpha
    \B_\ro^\alpha$ a disjoint family which still covers $\Lambda$
    as follows: Denoting by $\mathcal{C}$ a connected component of
    $\cup_\alpha  \B_\rho^\alpha$, we claim that there exists
     $\alpha_0$ such that $\mathcal{C}\subset U_{\alpha_0}$. Indeed if
     $x\in \mathcal{C}$ and letting $\ell$ be a Lebesgue number of the
      covering  $\{U_\alpha\}_\alpha$ (in our case $\ell\le 1/4$),
      there exists $\alpha_0$ such that $B(x,\ell)\subset
      U_{\alpha_0}$. If $\mathcal{C}$ intersected the complement of
       $U_{\alpha_0}$, there would exist a chain of balls
       connecting $x$ to  $(U_{\alpha_0})^c$, each of which
       would intersect $U_{\alpha_0}$. Each of the balls in
       the chain would belong to some $ \B_\rho^\beta$ with
       $\beta$ such that $\dist(U_\beta, U_{\alpha_0})\le 2\rho<
       1/2$. Thus, calling $k$ the maximum number of $\beta$'s such
       that $\dist(U_\beta, U_{\alpha_0})< 1/2$, the length of the chain
        is at most $2 k \rho$ and thus
         $\ell \le  2 k \rho$. If we choose $\rho < \ell/2k$,
          this is impossible and the claim is proven. Let us
          then choose  $\rho = \ell/4k$.
           By the above, each $\mathcal{C}$ is included in some $U_\alpha$.

Then, to obtain a disjoint cover of $\Lambda$ from $\cup_\alpha
\B_\rho^\alpha$, we let $\mathcal{C}$ run over all the  connected components
of  $\cup_\alpha  \B_\rho^\alpha$ and for a given $\mathcal{C}$ such that $\mathcal{C}\subset U_{\alpha_0}$,  we remove from $\mathcal{C}$ the balls which do not
belong to $\B_\rho^{\alpha_0}$. We will still denote
$\B_\rho^\alpha$ the family with deleted balls, and let $\B_\rho =
\cup_\alpha\B^\alpha_\rho$. Then $\B_\rho$ covers $\Lambda$
and is   disjoint.

 We then proceed to the mass displacement.
 Note that by construction every ball in $\B_\rho^\alpha$ is included in $U_\alpha$.

From the last  item of Proposition~\ref{boulesren} applied to a ball $B\in\B_\rho^\alpha$, if $\eta$ is small enough then,  letting $B_\eta = B\sm\cup_{p\in\Lambda}B(p,\eta)$, for any function $\chi$ vanishing outside $B\cap\widehat U$ we have
$$\int_{B_\eta} \chi |j|^2 - 2\pi\(\log\frac \rho{n_\alpha\eta} - CM\)\sum_{p\in B\cap\Lambda} \chi(p)\ge - C \nu(B) \|\nab\chi\|_{L^\infty(B)},$$
where $n_\alpha = \nu(U_\alpha)$ and $M = \|m\|_\infty$. Then applying Lemma~3.1 of  \cite{compagnon} to
$$ f_{B,\eta} =\hal \(  |j|^2 \indic_{B_\eta}- \(\log\frac \rho{n_\alpha\eta} - CM\) 2\pi\sum_{p\in B\cap\Lambda} \delta_p\) $$
we deduce the existence of a positive  measure $g_{B,\eta}$ such that $ \|f_{B,\eta} - g_{B,\eta}\| \le C\nu(B)$, where the norm is that of the dual of the space $\lip$ of Lipschitz functions in $B$ which vanish outside $B\cap\widehat U$.  Now we let $\eta\to 0$. Since  $g_{B,\eta}\ge 0$, it subsequentially converges to  a positive measure $g_B$ and  for any $\chi\in\lip$,
\begin{equation}\label{cwg}
\left|\int \chi \,dg_B - W_B(j,\chi)\right| \le C\nu(B) \|\nab\chi\|_{L^\infty(B)},\quad\text{where}\quad W_B(j,\chi) = \lim_{\eta\to 0}\int \chi \,df_{B,\eta}.\end{equation}

Next we note that, letting   $W'(j,\chi) = W(j,\chi)-\sum_{B\in\B_\rho} W_B(j,\chi)$, we have
$$W'(j,\chi)= \int\chi  \,df',\quad\text{where}\quad f' =\hal |j|^2\indic_{U\sm\B_\rho} +  \pi \sum_{p\in\Lambda}\(\log\frac \rho{n_{\alpha_p}} - CM\)\delta_p,$$
where, denoting  $\Lambda_\alpha$ the set of $p\in\Lambda$  belonging to a ball of $\B_\rho^\alpha$, we define $\alpha_p$ as the index $\alpha$ such that $p$ belongs to $\Lambda_\alpha$.

We define a set $C_\alpha$ as follows: recall that $\rho$
 was assumed equal to $\ell/4k $, where $\ell\le 1/4$ and $k$
 bounds the number of $\beta$'s such that  $\dist(U_\beta,
  U_\alpha)< 1/2$ for any given $\alpha$. Therefore the total
   radius of the balls in $\B_\rho$ which are at distance less
    than $\hal$ from $U_\alpha$ is at most $k\rho < 1/16$. In particular, letting $T_\alpha$ denote the set of $t\in (\frac14,\hal)$ such that the circle of center $x_\alpha$ (where we recall $x_\alpha$ is the center of $U_\alpha$)  and radius $t$ does not intersect $\B_\rho^\alpha$, we have $|T_\alpha|\ge 3/16$.
    We let $C_\alpha = \{x\mid |x-x_\alpha|\in T_\alpha\}$ and $D_\alpha= U_\alpha\cup C_\alpha $.
   If $U_\alpha\cap U\neq\varnothing$ then $d(x_\alpha,U)\le 1/4$ hence $B(x_\alpha,\hal)\subset\widehat U$. In particular,  we have $C_\alpha\subset D_\alpha\subset \widehat U$.   Then  there exists universal constants $c>0$ and $C$ such that
\begin{equation}\label{lacouronne}\int_{C_\alpha}|j|^2 \ge c {n_\alpha}^2 - CM^2.\end{equation}
To see this, apply the lower bound \eqref{lecercle} on the circle $S_t = \{|x-x_\alpha|=t\}$, i.e. with $r_B = t$ and $d_B = \#(\Lambda\cap B(x_\alpha,t))$. Using $d_B\ge n_\alpha$ and $t\in(\frac18, \frac12)$ we deduce that $\int_{S_t}|j|^2\ge \pi (n_\alpha)^2 - \pi M n_\alpha$ and integrating with respect to $t\in T_\alpha$ yields \eqref{lacouronne}.

The overlap number of the sets $\{C_\alpha\}_\alpha$,
defined as the maximum number of sets to which a given $x$
belongs is bounded by the overltap number of $ \{B(x_\alpha,3/4)\}_\alpha$,
 call it $k'$. Then,  letting
$$f_\alpha =   \frac{1}{2k'}|j|^2\indic_{C_\alpha} +
 \pi \sum_{p\in\Lambda_\alpha}\(\log\frac \rho{n_{\alpha}} - CM\) \delta_p,$$
we have \begin{equation}\label{denfp}
f' - \sum_\alpha f_\alpha\ge |j|^2\indic_{U\sm\B_\rho} - \frac{1}{2k'}|j|^2\indic_{C_\alpha}\ge \hal |j|^2\indic_{U\sm\B_\rho}
\ge 0 \end{equation}
and, from \eqref{lacouronne},
\begin{equation}\label{caf}  f_\alpha(D_\alpha) = \frac{1}{2k'}
\int_{C_\alpha} |j|^2 + \pi n_\alpha\(\log\frac \rho{n_{\alpha}} - CM\)\ge c {n_\alpha}^2 - C(M^2+1),\end{equation} where the constants may have changed.
We then apply  \cite{compagnon} Lemma~3.2 over $D_\alpha $ to $f_\alpha +\frac{C(M^2+1)}{|D_\alpha|}$,  where $C$ is the constant in the right-hand
side of (\ref{caf}). We deduce {\em for any $\alpha$ such that } $U_\alpha\cap U\neq\varnothing$ the existence of  a measure  $g_\alpha$ supported in $D_\alpha$ such that $g_\alpha\ge -\frac{C(M^2+1)}{|D_\alpha|}$ and such that, for every Lipschitz function $\chi$
\begin{equation}\label{611b} \int \chi \,d(f_\alpha -
 g_\alpha) \le C \|\nab\chi\|_{L^\infty(D_\alpha)}
 (f_\alpha)_-(D_\alpha) \le C  n_\alpha (\log n_\alpha +C(M+1))\|\nab\chi\|_{L^\infty(D_\alpha)}.\end{equation}
In particular, taking $\chi = 1$,
\begin{equation}\label{galphan} g_\alpha(D_\alpha) =  f_\alpha(D_\alpha) \ge  c {n_\alpha}^2 - C(M^2+1).\end{equation}
Now we let
$$g = \sum_{B\in\B_\rho} g_B + \sum_{\substack{\alpha\ \text{s.t.}\\ U_\alpha\cap U\neq\varnothing}}g_\alpha + \(f' - \sum_\alpha f_\alpha\).$$
The term $g_B$ is positive for every $B$,    $f' - \sum_\alpha
 f_\alpha$ is bounded below by \eqref{denfp},
and
 $\sum_\alpha g_\alpha$  is
 bounded below by $-k' C(M^2+1)$ since $g_\alpha\ge -C(M^2+1)$. Thus
 $g$ is bounded below by  $ \hal  |j|^2\indic_{U\sm\B_\rho}    -C(M^2+1) \ge - C(M^2+1)$, and we have proved the first item.
In addition, if $\chi$ has support in $U$ then
$$\sum_\alpha \int\chi\,df_\alpha = \sum_{\substack{\alpha\ \text{s.t.}\\ U_\alpha\cap U\neq\varnothing}} \int \chi\,df_\alpha$$
hence
\begin{multline*}
\int \chi \,dg  =  \int \chi \,d\( \sum_B g_B + \sum_{\substack{\alpha\ \text{s.t.}\\ U_\alpha\cap U\neq\varnothing}}(g_\a-f_\a) + f'\)= \\ \sum_B \int \chi \,dg_B+ \sum_{\substack{\alpha\ \text{s.t.}\\ U_\alpha\cap U\neq\varnothing}}\int \chi \,d(g_\a-f_\a)+ W'(j,\chi)
\end{multline*}
hence in view of the definition of $W'$, \eqref{cwg} and \eqref{611b},
\begin{equation}\label{gchi}\begin{split}\int\chi \,dg - W(j,\chi) &= \sum_{B\in\B_\rho} \(\int\chi \,dg_B - W_B(j,\chi)\) +  \sum_{\substack{\alpha\ \text{s.t.}\\ U_\alpha\cap U\neq\varnothing}} \int\chi \,d(g_\alpha - f_\alpha)\\
&\le C \(\sum_{B\in\B_\rho}\nu(B)\|\nabla\chi\|_{L^\infty(B)} + \sum_\alpha n_\alpha (\log n_\alpha + C(M+1))\|\nabla\chi\|_{L^\infty(D_\alpha)}\).\end{split}\end{equation}
Then if we denote by $A$ the set of $\alpha$'s such that $\|\nabla\chi\|_{L^\infty(D_\alpha)}\neq 0$, and $k$ is the overlap number of the $U_\alpha$'s,
$$2\pi \sum_{\alpha\in A} n_\alpha \le k
 \nu\(\cup_{\alpha\in A} U_\alpha\),$$
and  $x\in U_\alpha$ with $\alpha\in A$ implies that
$\nabla\chi(y)\neq 0$ for some $y\in D_\alpha$ hence
 $B(x,5/4)\cap\supp\nabla\chi\neq\varnothing$. It follows that
 $\sum_{\alpha\in A} n_\alpha \le k n $, where $n$ is defined after \eqref{wg}. Similarly,
the sum of $\nu(B)$ for $B$'s in $\B_\rho$ such that $\|\nabla\chi\|_{L^\infty(B)}\neq 0$ is less then $2\pi n$. Thus \eqref{gchi} may be rewritten as \eqref{wg}.

Finally, summing  \eqref{galphan} for $\alpha$'s such that $U_\alpha\cap E\neq\varnothing$   and recalling that $g - \sum_\a g_\a \ge 0$, we easily deduce  \eqref{bgnalpha}.
\end{proof}

\begin{remark}
We have in  fact proved the following stronger property on $g$:
  There exists $\ro >0$ and  a family $\mathcal{B}_\ro $ of disjoint closed balls covering $\Lambda$, such that the sum of the radii of the balls in $\mathcal{B}_\ro $ intersected with any ball of radius $1$ is bounded by $ C\ro<\hal$, and such that on $\hat{U}$
    $$g\ge -C(\|m\|_\infty^2+1)+ \hal |j|^2 \indic_{U\backslash \mathcal{B}_\ro}  .$$ \end{remark}

\subsection{Application: proof of Proposition \ref{pw1}}
We start with
\begin{lem}\label{wnu} Assume $j \in\ainfty$.
 Then, for any family $\{\UR\}_R$ satisfying \eqref{hypsetsbis} and $R>C$, for any $p>1$,
$$\max\left\{\left|\nu(\UR) - |\UR|\right|,\ \nu(\widehat\UR) -
\nu(\UR)\right\} \le C \(R^\theta + R^{\theta(1 - \frac1p)}\)
\( \|j\|_{L^p(\URC)} +1\),$$
where $\theta < 2$ is the exponent in \eqref{hypsetsbis}, and $C$ only depends on the constants in \eqref{hypsetsbis}.
\end{lem}
\begin{proof} Let $d_R$ denote the signed distance to
 $\partial\UR$, i.e. $d_R(x) = - d(x,\UR^c)$ if $x\in\UR$
 and $d_R(x) = d(x,\UR)$ otherwise. Let
 $\UR^t = \{x\mid d_R(x)\le t.\}$.

Integrating \eqref{eqji} over $\UR^t$  and using H\"older's inequality we find
\begin{equation}\label{dnp} \left|\nu(\UR^t) - |\UR^t|\right| = \left|\int_{\partial \UR^t} j\cdot\tau\right|\le |\partial \UR^t|^{1-\frac1p} \|j\|_{L^p(\partial\UR^t)}.\end{equation}
Then, using the coarea formula,
$$\int_0^1 |\partial \UR^t|\,dt = |\UR^1\sm\UR|,\qquad \int_0^1 \|j\|_{L^p(\partial\UR^t)}^p\,dt = \|j\|_{L^p(\UR^1\sm \UR)}^p$$
and from \eqref{hypsetsbis} there exists $C>0$ such that $\UR^1\subset\URC$. Using a mean-value argument and \eqref{hypsetsbis} again there exists $t\in (0,1)$ such that
$$|\partial \UR^t|\le CR^\theta,\quad \|j\|_{L^p(\partial\UR^t)}^p\le 2\|j\|_{L^p(\UR^1)}^p\le 2\|j\|_{L^p(\URC)}^p,$$
and therefore
\begin{equation}\label{dnm} \left|\nu(\UR^t) - |\UR^t|\right| \le C R^{\theta(1 - \frac1p)} \|j\|_{L^p(\URC)}.\end{equation}
A similar mean value argument  yields the existence of $s\in(-1,0)$ such that
$$ \left|\nu(\UR^s) - |\UR^s|\right| \le C R^{\theta(1 - \frac1p)} \|j\|_{L^p(\UR)}.$$
Since $t\mapsto \nu(\UR^t)$ is increasing we deduce, since $\URc\subset \UR^s$,
$$ \nu(\UR^s) - |\UR^s| \le \nu(\UR) - |\UR| + |\UR\sm\URc|  \le \nu(\UR^t) - |\UR^t| + |\URC\sm\URc|,$$
which in view of \eqref{dnp}-\eqref{dnm} and \eqref{hypsetsbis}
proves the bound for $\left|\nu(\UR) - |\UR|\right|$. The bound for $\nu(\widehat\UR) - \nu(\UR)$ follows easily since $\UR\subset\widehat\UR\subset\URC$.
\end{proof}

\begin{coro}\label{wgcoro} Assume that $j\in\ainfty$,  that
 $\{\UR\}_R$, $\{\VR\}_R$ satisfy  
 \eqref{hypsetsbis} and that $c>0$ is such that $\UcR\subset \VR$
 for any $R\ge 1$ (it is easy to show  from \eqref{hypsetsbis} that such a $c$ exists).
  We assume that  either $W_U(j)<+\infty$, or $W_V(j)<+\infty$.

Then, denoting   $g_\UR$ the result of applying Proposition~\ref{wspread} to $(j,\nu)$ in $\UR$, we have
\begin{equation}\label{is}
 \lim_{R\to +\infty} \frac1{R^2}\(\int_\UcR\chi_\UcR \,dg_\VR -
  W(j,\chi_\UcR)\) = 0.\end{equation}  
  In particular if $\{\VR\}_R= \{\UR\}_R  $, we may take $c=1$
to obtain
\begin{equation}\label{iis}
 \lim_{R\to +\infty} \frac1{R^2}\(\int_\UR\chi_\UR \,dg_\UR -
 W(j,\chi_\UR)\) = 0.\end{equation}
As a consequence, $W_U$ does not depend on the particular choice of $\chi_\UR$ satisfying \eqref{defchi}.
\end{coro}
\begin{remark}\label{aprescoro}
Let $G\subset\ainfty$ satisfy the hypothesis of Proposition~\ref{pw2bis}  and let $\{\UR\}_R=\{K_R\}_R$. Then  the convergence in \eqref{iis} is uniform with respect to $j\in G$. %rate of convergence in \eqref{iis} only depends on the constants in \eqref{hypunif}--\eqref{cvwunif}.
 \end{remark}
\begin{proof} Since $W_U(j)<+\infty$ and $j \in \ainfty$, both
 $W(j,\chi_\UR)$ and  $\nu(\UR)$ are $O(R^2)$ as $R \to \infty$. Applying  Lemma~\ref{wlp} in $\URC$ we find, choosing some $p\in(1,2)$, that $\int_{\UR}|j|^p = O(R^2\log R)$. Then, using Lemma~\ref{wnu}, we have  $|\nu(\widehat\UR) - \nu(\UR)|=O(R^{\theta_p})$ for some exponent $ \theta_p\in(0,2)$, and thus the same holds for $|\nu(\UR) - \nu(\URc)|$. Inserting into \eqref{wg} yields
\begin{equation}\label{1s}
 \left|\frac{W(j,\chi_{\UcR})}{|\UcR|} - \dashint_{\UcR}\chi_{\UcR}
  \,dg_\VR\right|  \xrightarrow{R\to +\infty} 0,\end{equation}
  hence \eqref{is}, and \eqref{iis} easily follows.

In view of \eqref{iis}, proving that the definition of $W_U$ is independent of  the choice of $\{\chi_{\UR}\}_R$ satisfying \eqref{defchi} now reduces to proving the same statement for 
$$\limsup_{R\to +\infty} \dashint_\UR \chi_\UR \,dg_\UR.$$
But, since $g_\UR\ge -C$ and since $\URc\subset \{\chi_\UR = 1\}$,
$$ g_\UR(\URc) - C|\UR\sm\URc|\le \int_\UR\chi_\UR \,dg_\UR\le
g_\UR(\UR) + C|\UR\sm\URc|.$$
Dividing by $|\UR|$ and in view of \eqref{hypsetsbis} we obtain that
$$\limsup_{R\to +\infty} \dashint_\UR \chi_\UR \,dg_\UR =
 \limsup_{R\to +\infty} \frac{g_\UR(\UR)}{|\UR|},$$
which is clearly independent of $\chi_\UR$ and finite thanks to \eqref{iis}.

The proof of Remark \ref{aprescoro} follows from the fact that under the hypothesis of Proposition~\ref{pw2bis} we clearly have bounds $\nu(K_R)<CR^2$ and $W(j,\chi_{K_R})<CR^2$ which are uniform with respect to $j\in G$ and thus that the convergence in \eqref{1s} is uniform with respect to $j\in G$ as well, when $\{\UR\}_R=\{K_R\}_R$.
%from inspecting the above proof and checking that the constants involved only depend on those in \eqref{hypunif}--\eqref{cvwunif}.
\end{proof}

We may now give the proof of Proposition~\ref{pw1}, in several steps.

\begin{proof}[$W_U$ is measurable] First we show that $\ainfty$ (recall Definition \ref{defA})  is a Borel subset of  $X:=\Lp$. For $R$, $\ep>0$ we let  $A_{R,\ep}$ be the set of $j \in \Lp$ such that first  $\curl j = \nu-1$ and $\div j=0$,  where the restriction of $\nu$ to $B_{2R}$ is of the form $2\pi\sum_i \delta_{a_i}$ with $a_i\in B_R$ and $|a_i - a_j|\ge \ep$ (in particular the sum is finite), and second $\|j\|_{L^p(B_{2R})}\le 1/\ep$. We also let
$$ A_{R,\ep,C} = \{j \in A_{R,\ep}\mid \nu(B_R)\le CR^2 \}.$$
Clearly both $A_{R,\ep}$, and $A_{R,\ep,C}$ are closed. Noting that
$$\ainfty = \(\cup_{C>1}\cap_{R>1 }\cup_{\ep>0} A_{R,\ep,C}\), $$
we find that $\ainfty$ is a Borel subset of $X$.

For $j\in\ainfty$ we have $W(j) = \limsup_R W(j,\chi_\UR)/|\UR|$, hence proving  that $W$ is Borel reduces to proving  that
$$W_\chi:j\mapsto \begin{cases} W(j,\chi)  & \text{if $j\in\ainfty$}\\ +\infty & \text{otherwise}\end{cases}$$ is Borel for any smooth, positive $\chi$ with compact support.

This follows from Lemma~\ref{wsci}. Choosing $R>0$ such that $\supp(\chi)\subset B_R$ the lemma implies that $W_\chi$ restricted to the closed set $A_{R,\ep}$ is continuous, therefore $\{j \in A_{R,\ep}\mid W_\chi\le t\}$ is closed for any $t$, and
$$\{j\in\ainfty\mid W_\chi(j)\le t\} =  \ainfty \cap \(\cap_{R>0}\cup_{\ep>0} \{j\in A_{R,\ep}\mid W_\chi\le t\}\)$$
is Borel.
\end{proof}

\begin{proof}[$\inf_\ainfty W_U$ is finite.]  The results of
Section \ref{latticecase} for example show that $\inf W_U<+\infty$.
 The fact that  $\inf W_U >-\infty$  is a direct consequence of
  Corollary~\ref{wgcoro} and the fact that for any $R>0$ we
  have $g_{\UR}\ge -C$, where $g_{\UR}$ is the result of applying
  Proposition \ref{wspread}
  on $\UR$.
\end{proof}

\begin{proof}[Sub-level sets are compact up to translation, $\min W_U$ is achieved.] Consider a family $\{\VR\}_{R}$ satisfying \eqref{hypsets}--\eqref{hypsetsbis} and 
consider a sequence $\{\bj_n\}_n$ such that $W_V( \bj_n) \le \alpha  $. In particular $\bj_n \in \ainfty$.
 Let $\bnu_n$ be $\curl \bj_n +1$.
 Then since $\bnu_n(B_R)=O(R^2)$ and from the definition of $W_V$ as a $\limsup$, for any fixed arbitrary $C>0$ and any $\theta>1$ there exists a sequence   $R_n\to +\infty$   such that
$$\liminf_{n\to +\infty} \(W_V(\bj_n)
 - \frac{W(\bj_n,\chi_{V_{R_n}})}{|V_{R_n}|}\)\ge 0,\quad \bnu_n\(V_{R_n+C}\sm V_{R_n-C}\) \le {R_n}^\theta.$$
Indeed, given $n$, the second relation is satisfied by arbitrarily large $R$'s, using a mean value argument.

Now, letting $V_n = V_{R_n}$,  we apply Proposition~\ref{wspread} to $(\bj_n,\bnu_n)$ with  $U = V_n$ and deduce the existence of a measure $\bg_n\ge -C$ satisfying the properties  described there.  The choice of $\{R_n\}_n$ ensures that
\begin{equation}\label{wgn}\liminf_{n\to +\infty} W_V(\bj_n)\ge \liminf_{n\to+\infty} \frac{W(\bj_n,\chi_{V_n})}{|V_n|} = \liminf_{n\to+\infty} \dashint_{V_n} \chi_{V_n}\,d\bg_n, \end{equation}
using \eqref{wg}.

Now we apply the abstract scheme described in Theorem \ref{gamma}. Let $X = \Lp\times\radon_0$, where $\radon_0$ is the space of Radon measures on $\mr^2$ bounded below by twice the fixed constant given by item 1 of  Proposition \ref{wspread} (this means that we are considering measures $\mu$ such that $\mu + C$ is a positive Radon measure), and the topology is that of  convergence in $\Lp$ and weak convergence on $\radon_0$.  $X$ is a Polish space, and on it  we have the natural continuous action $\theta_\lambda (j,g) = (j(\lambda+\cdot),g(\lambda+\cdot))$. We may check the hypotheses of Section \ref{sec:abstract} are satisfied.
Then we
choose a smooth positive $\chi$ with compact support in $B(0,1)$  and integral $1$, we let $V_n' = V_{R_n-C}$ where $C$ is chosen (according to \eqref{hypsetsbis}) such that 
\begin{equation}\label{nmb}
\chi_{V_n}\ge \chi * \indic_{V_n'}   \ \text{everywhere}\quad \text{and} \ \ \chi_{V_n} = \chi * \indic_{V_n'}=1  \ \text{in } \ V_{R_n - 2C}\end{equation}
 and define
$$ \gf_n(j,g) = \begin{cases} \int \chi(y)\,dg(y) & \text{if there exists $\lambda\in V_n'$ such that $(j,g) = \theta_\lambda(\bj_n,\bg_n)$,}\\
+\infty & \text{otherwise.}\end{cases}$$
We also let
$$ \gF_n(j,g) = \dashint_{V_n'} \gf_n\(\theta_\lambda(j,g)\)\,d\lambda =  \begin{cases} \frac1{|V_n'|}\int \chi*\indic_{V_n'}\,d\bg_n & \text{if  $(j,g) = (\bj_n,\bg_n)$,}\\
+\infty & \text{otherwise.}\end{cases}$$
Since $\bg_n \ge -C$, the property \eqref{nmb}  implies that
\begin{equation}\label{wgf} \int \chi_{V_n} \,d\bg_n \ge |V_n'|\gF_n(\bj_n,\bg_n)  - O(R_n).\end{equation}

Then we check the coercivity and $\Gamma$-$\liminf$ properties of $\{\gf_n\}_n$ as in \eqref{2.1bis}--\eqref{2.1ter}.  The latter is the trivial observation  that if $(j_n,g_n) \to (j,g)$ then for any subsequence (not relabeled) such that $\{\gf_n(j_n,g_n)\}_n$ is bounded we have
$$\lim_{n\to\infty} \gf_n(j_n,g_n) = \lim_{n\to +\infty}\int \chi \,dg_n= \gf(j,g),\quad{where}\quad \gf(j,g) = \int\chi \,dg.$$
To prove coercivity, assume  as in \eqref{2.1bis} that for every $R$
\begin{equation}\label{wcoercif}\limsup_{n\to+\infty} \int_{B_R} \gf_n(\theta_\lambda (j_n,g_n))\,d\lambda < +\infty.\end{equation}
Then for every $R$, if $n$ is large enough the integrand is finite a.e.  hence there exists $\lambda_n\in V_n'$ such that $(j_n,g_n) = \theta_{\lambda_n} (\bj_n,\bg_n)$ and $\lambda +\lambda_n\in V_n'$ for almost every $\lambda\in B_R$, i.e. $\lambda_n+ B_R\subset V_n'$. Then \eqref{wcoercif} reads
$$\limsup_{n\to+\infty}\int_{B_R} \int \chi(x- \lambda_n-\lambda)\,d
\bg_n(x)\,d\lambda = \int \chi*\indic_{\lambda_n+B_R}\,d\bg_n < +\infty,$$
which is equivalent to saying that $\{g_n = \bg_n
(\lambda_n+\cdot)\}_n$ is bounded in $L^1(B_R)$ for every $R$. This implies that a subsequence converges in $\radon_0$.

Then, in view of \eqref{bgnalpha} this proves that $\{\bnu_n(\lambda_n+\cdot)\}_n$ is locally bounded. Inserting this information into \eqref{wg} we find that  $\{W(\bj_n(\lambda_n+\cdot),\chi_R)\}_n$ is bounded and then using Lemma~\ref{wlp} we deduce that $\{\bj_n(\lambda_n+\cdot)\}_n$ is bounded in $L^p(B_R)$ for any $R$.  Thus going to a further subsequence $\{j_n = \bj_n(\lambda_n+\cdot)\}_n$ converges to $j$  locally weakly in $L^p$.  Moreover $\div j_n=0$ and by the above $\curl j_n$ is locally bounded in the sense of measure, hence weakly compact in $W^{-1, p}_{loc}$. By elliptic regularity it follows that the convergence of $j_n$ to $j$ is strong in $\Lp$.
This proves coercivity.

We may now apply Theorem~\ref{gamma}. Letting $P_n$ be the image of the normalized Lebesgue measure on $B_n'$ by the map $\lambda\mapsto \theta_\lambda(\bj_n,\bg_n)$, there is a subsequence such that $P_n\to P$, where $P$ is a probability measure on $X$ and
$$\liminf_{n\to +\infty} \gF_n(\bj_n,\bg_n) \ge \int \gf^*_U(j,g)\,dP(j,g), \quad \text{where}\quad \gf^*_U(j,g) = \lim_{R\to
+\infty}\dashint_{\UR}\int \chi(x-\lambda) \,dg(x)\,d\lambda,$$
for any  family of open sets $\{\UR\}_{R>0}$  satisfying  (\ref{hypsets}).

We claim that if $(j,g)\in\supp (P)$ and $\gf^*_U(j,g)\le +\infty$, then
\begin{equation} \label{claim}\text{$j\in\ainfty$ and $\gf^*_U(j,g)= W_U(j)$.}\end{equation}
Assuming this, and choosing $(j,g)\in\supp (P)$ such that $\gf^*_U(j,g)\le \liminf_n \gF_n(\bj_n,\bg_n)$ we obtain, using \eqref{wgn} and \eqref{wgf},
\begin{equation}\label{wliminf} \liminf_n W_V (\bj_n)\ge \liminf_n \gF_n(\bj_n,\bg_n)\ge W_U(j).\end{equation}
Choosing $V=U$ shows that  $\inf_\ainfty W_U$ is achieved.    Taking for  $\bj_n$  a minimizing sequence for $W_V$, it follows from \eqref{wliminf} that
$\min_{\ainfty} W_V \ge \min_{\ainfty} W_U$, hence the value of  $\min_{\ainfty} W_U$ is independent of $U$.

We prove the claim \eqref{claim}. Since $\gf^*_U(j,g)\le +\infty$, we have $g(\UR)<C|\UR|\le CR^2$ and thus $\forall R>1$ there exists $N_R\in\mn$ such that $n\ge N_R$ implies $g_n(\UR)\le CR^2$. Using \eqref{bgnalpha} this in turn implies that if $n\ge N_R$ then $\nu_n(\UR)\le CR^2$ and then, passing to the limit $n\to\infty$, that $\nu(\UR)\le C R^2$, so that in particular $j\in\ainfty$.

Moreover, still if $n>N_R$,  from $g_n(\UR)$, $\nu_n(\UR)\le CR^2$ and using \eqref{wg} we deduce that $W(j_n,\chi_\UR)\le CR^2\log R$ and then, as in the proof of Corollary~\ref{wgcoro}, that
$$\left|W(j_n,\chi_\UR) - \int \chi_\UR\,dg_n\right| \le o(R^2),$$
for some $\theta<2$. Passing to the liminf  $n\to \infty$ we obtain in view of Lemma~\ref{wsci} the same relation for $j$, $g$. Dividing by $|\UR|$ --- which is bounded below by $cR^2$ for some $c>0$ --- and letting $R\to +\infty$ we find that $W_U(j) = \gf^*_U(j,g)$, which finishes the proof of \eqref{claim}.

\end{proof}

\begin{proof}[Independence w.r.t. the shape] We have just seen that if $U$ and $V$ refer to two families of sets,  the infimum of $W_V$ and $W_U$ are both achieved and  are equal.
There remains to show that minimizers are also the same.

Consider $j_V$ a minimizer of $W_V$.
By Corollary \ref{wgcoro}
 we have for any $\{\UR\}_R$ satisfying 
  \eqref{hypsetsbis},
 \begin{equation}\label{619b}
 \lim_{R \to \infty}\frac{ W(j_V,\chi_\UR)}{|\UR|}
 - \dashint_{\UR} \chi_{\UR}\,dg_{\UR}=0,\end{equation}
 where $g_{\UR}$ is the result of applying Proposition
\ref{wspread} in $\UR$.  We deduce that
\begin{multline}\label{623} W_U(j_V) =
\limsup_{R\to \infty} \dashint_{\UR}  \chi_{\UR}\,dg_{\UR}=
\limsup_{R\to \infty} \int \gf(j,g)
\, dP_{\UR} (j,g)\ge \\ \liminf_{R\to \infty}\int \gf (j,g) \, dP_{\UR}(j,g)
\ge \int \gf(j,g) \, dP_U(j,g),\end{multline}
where $\gf (j,g) = \int\chi\,dg$, $\chi_\UR = \indic_\URc * \chi$,
 and where  $P_\UR$ is the image of the normalized Lebesgue measure
  on $\URc$ by $\lambda\mapsto \theta_\lambda(j_V,g_\UR)$. Moreover we have chosen a subsequence $\{R\}$ such that $\{P_\UR\}_R$ converges to a probability measure $P_U$.

Since $P_U$ is $\theta$-invariant and  using the ergodic theorem and \eqref{claim} we get
$$\int\gf(j,g)\,dP_U(j,g) = \int \gf^*_U (j,g) \, dP_U(j,g)=
\int W_U(j) \,dP_U(j,g)\ge \min_\ainfty W, $$
where $W$ can be defined using any family of sets satisfying \eqref{hypsets}, \eqref{hypsetsbis}, {\em not necessarily the family} $U$. Together with \eqref{623} we get
\begin{equation}\label{624} W_U(j_V) = \limsup_{R\to \infty} \int \gf(j,g)\, dP_{\UR} (j,g)\ge \int \gf(j,g) \, dP_U(j,g) = \int W_U(j) \,dP_U(j,g)\end{equation}
and this is bounded below by $\min_\ainfty W$. From the minimality of $j_V$ and since $\min W_U = \min W_V$, applying \eqref{624} to $\{\UR\}_R=\{\VR\}_R$ implies that there is equality everywhere and therefore
$$ \text{(i)}\ \limsup_R\int\gf\,dP_\VR = \int\gf\,dP_V,\quad \text{(ii) $P_V$-almost every $j$ minimizes $W$.}$$
Since $\gf$ is continuous and bounded below  on the support of $P_\VR$ independently of $R$,
$$ \limsup_R\int\gf\,dP_\VR = \int\gf\,dP_V \iff \sup_R \int_{\{\gf>M\}} \gf\,dP_\VR \xrightarrow{M\to +\infty} 0.$$
Now choose $c>0$ such that $\UcR\subset \VRc$ for every
 $R\ge 1$ and let $P_\VR'$be the image of the normalized Lebesgue
  measure on $\UcR$ by $\lambda\mapsto\theta_\lambda(j_V,g_\VR)$. Then
\begin{equation}\label{dPVR} P_\VR\ge \frac{|\UcR|}{|\VRc|} P_\VR'\ge
\delta P_\VR',\end{equation}
for some $\delta>0$ independent of $R$ (this follows from \eqref{hypsetsbis}). It follows that
$$\sup_R \int_{\{\gf>M\}} \gf\,dP_\VR' \xrightarrow{M\to +\infty} 0$$
and then that, choosing a subsequence $\{R\}$ such that $P_\VR'\to P_V'$, that $\limsup_R \int\gf\,dP_\VR' = \int\gf \,dP_V'$ (cf. \cite{billingsley,dudley}).

Now we claim that the support of $P_V'$ is included in the support of $P_V$. Indeed if $\vp\ge 0$ is continuous with compact support in $(\supp\  P_V)^c$, then $\int\vp\,dP_\VR\to\int\vp\,dP_V = 0$ hence from
\eqref{dPVR}
$$ \delta \int\vp\,dP_\VR'\xrightarrow{R\to\infty} 0$$ and thus $\int\vp\,dP_V' = \lim_R\int\vp\,dP_\VR' = 0.$

It follows that $P_V'$-almost every $j$ minimizes $W_U$ and that
\begin{equation}\label{minPV} \int\gf \,dP_\VR'\xrightarrow{R\to\infty} \min W_U.\end{equation}
But $\int \gf \,dP_\VR' = \dashint_\UcR \chi_\UcR \,dg_\VR$ by definition
of $\gf$ and $P_\VR'$
and using Corollary~\ref{wgcoro}, we have  that $\dashint_\UcR
\chi_\UcR\,d(g_\VR - g_{\mathbf U_{cR}})$ tends to $0$ as
$R\to +\infty$. Therefore $\int \gf\,d(P_\VR' - P_{\mathbf U_{cR}})$ tends to $0$ as well, which together with \eqref{minPV} and \eqref{624} yields
$$W_U(j_V) = \limsup_{R\to +\infty} \int\gf
\,dP_{\mathbf U_{cR}} = \min_\ainfty W_U.$$
It follows that $j_V$ minimizes $W_U$. \end{proof}

\subsection{Proof of Proposition~\ref{pw2bis} and Corollaries \ref{periodisation}, \ref{mesyoung}}
%As we mentioned in the introduction, this limiting problem bears some similarity with the one of diblock copolymers. We follow here the same steps as
%in  \cite{aco}, though their implementation is different.

The first lemma (whose proof is postponed to the end of the subsection)
serves to extract a good boundary. We denote by $W_K$ the renormalized energy relative to the family $\{K_R =[-R,R]^2\}_R$.
\begin{lem} \label{lemmebord}
Let $G$ satisfy the assumptions of Proposition \ref{pw2bis}. Then for any $\gamma \in (0,1)$, any $R$ large enough depending on $\gamma$, and any $p\in[1,2)$ there exists, for any  $j\in G$, some $t \in [R- 2R^\gamma, R- R^\gamma]$ 
such that
\begin{equation}\label{majojp}\int_{\p K_{t} } |j|^p \le C_p  R^{2-\gamma}\end{equation}
 \begin{equation}\label{minowt}\lim_{R\to \infty}
\frac{ W(j, \indic_{K_{t}}
)}{|K_{t}|} = W_K(j),\qquad 
 |\nu (K_{t})-|K_{t}||\le C R^{2-\gamma},
\end{equation}
where $C$, $C_p$ do not depend on $j\in G$, and where the convergence in \eqref{minowt} is  uniform with respect to $j \in G$. \end{lem}

The next steps consist in modifying $j$  in $K_R \backslash K_t$ so that $j\cdot\tau = 0$ on  $\p K_R$, and so that the ensuing modification of $W(j,\indic_{K_R})$ is negligible compared to $R^2$ 
as $R\to +\infty$. This relies on the following two lemmas.

\begin{lem}\label{lemrect1}
Let $\mathcal{R}$ be a rectangle with sidelengths  in $[\frac{L}{2}, \frac{3L}{2}]$. Let $p\in (1,2)$.
Let $g\in L^p(\p \mathcal{R})$ be a function
which is $0$ except on one side of the rectangle  $\mathcal{R}$. Let $m$ be a constant such that  $(m-1)|\mathcal{R}|= - \int_{\p \mathcal{R}}g$. Then
 the mean zero solution  to
\begin{equation}\label{target1}
\left\{\begin{array}{ll}
-\Delta u= m  - 1& \text{in } \ \mathcal{R}\\ [2mm]
\frac{\p u}{\p \nu }= g & \text{on} \ \p \mathcal{R}
\end{array}\right.
\end{equation}
satisfies  for every $q\in [1,2p]$
\begin{equation}\label{tar1e}
\int_{\mathcal{R}} |\nab u|^q \le
C_{p,q}   L^{2-\frac{q}{p}}\|g\|_{L^p(\p \mathcal{R})}^q
. \end{equation}
\end{lem}
\begin{proof}
We write  the solution $u$ of \eqref{target1} as $u=u_1+ u_2$ where
\begin{equation}\label{targetu1}
\left\{\begin{array}{ll}
-\Delta u_1= m  - 1& \text{in } \ \mathcal{R}\\ [2mm]
\frac{\p u_1}{\p \nu }= \bar{g} & \text{on} \ \p \mathcal{R}
\end{array}\right.
\end{equation}
where $\bar{g}$  is equal to  the average of $g$ on the side
where $g$ is supported  and is  $0$ on the other sides; and
$$\left\{\begin{array}{ll}
-\Delta u_2= 0 & \text{in } \ \mathcal{R}\\ [2mm]
\frac{\p u_2}{\p \nu }= g- \bar{g} & \text{on} \ \p \mathcal{R}.
\end{array}\right.
$$
Assume that $\mathcal{R}= [0, \ell_1]\times [0,\ell_2]$, with $\ell_i\in[\frac{L}{2}, \frac{3L}{2}]$, and that
$\bar{g}$ is supported on the side $x_2=0$. Then,  up
 a constant, the  solution of \eqref{targetu1} is
$u_1(x_1,x_2)= \frac{m-1}{2}(x_2-\ell_2)^2$. Therefore
$$\int_{\mathcal{R}} |\nab u_1|^q= (m-1)^q \ell_1 \int_0^{\ell_2}
(x_2-\ell_2)^q \, dx_2
= (m-1)^q \ell_1 \frac{\ell_2^{q+1}}{q+1}\le C(m-1)^qL^{2+q}.$$
Then,  $m-1=-\frac{1}{\ell_1\ell_2}
\int_{\p \mathcal{R}} g$ and using H\"older's inequality
  $|m-1|\le CL^{-2}\|g\|_{L^p(\partial\mathcal R)} L^{1-\frac1p}$.  Inserting  above  we are led to
\begin{equation}
\label{nrju1}
\int_{\mathcal{R}} |\nab u_1|^q\le C_{p,q} L^{2-\frac{q}{p}}
\|g\|_{L^p(\p \mathcal{R})}^{q}.\end{equation}
For $u_2$, note that the conjugate harmonic function $u_2^*$ has trace $\vp$ which satisfies $\partial_\tau\vp = g-\bar g$, hence $\|\nab\vp\|_{L^p(\partial\mathcal R)}\le \|g-\bar g\|_{L^{p}(\partial\mathcal R)}$.
 Then from the Sobolev imbedding $W^{1,p}(\partial \mathcal R)\hookrightarrow W^{1-\frac1{2p}, 2p}(\partial \mathcal R)$, which is the trace space of $W^{1,2p}(\mathcal R)$, and elliptic regularity it follows that
$$\|\nab u_2\|_{L^{2p}(\mathcal R)}= \|\nab u_2^*\|_{L^{2p}(\mathcal R)}\le C_{\mathcal R} \|g-\bar g\|_{L^{p}(\partial\mathcal R)},$$
and it is easy to check that the constant $C_{\mathcal R}$ may be chosen to depend only on $p,L$ and not on $\ell_i$, as long as $\ell_i\in [\frac{L}{2}, \frac{3L}{2}]$. Scaling arguments then show that for some $C_{p,q}$ independent of $L$ we have
\begin{equation}\label{nrju2}
\int_{\mathcal{R}}| \nab u_2|^q\le C_{p,q} L^{2-\frac{q}{p}} \|g\|_{L^p(\p \mathcal{R})}^q.
\end{equation}
Combining \eqref{nrju1} and \eqref{nrju2}, we obtain the result \eqref{tar1e}.
\end{proof}

\begin{lem}\label{lemrect2}
Let $\mathcal{R}$ be a rectangle of barycenter $0$, and sidelengths $\in \sqrt{2\pi}[\hal, \frac{3}{2}]$,
and let $m$ be a constant such that $m|\mathcal{R}|=2\pi$. Then the solution to

$$\left\{\begin{array}{ll}
-\Delta f= 2\pi \delta_0
  - m & \text{in } \ \mathcal{R}\\ [2mm]
\frac{\p f}{\p \nu }= 0  & \text{on} \ \p \mathcal{R}
\end{array}\right.
$$
satisfies
\begin{equation}\label{nrjf2}
\lim_{\eta \to 0}\left|\int_{\mathcal{R}\backslash B(0, \eta)}|\nab f|^2
+2 \pi \log \eta\right|\le C
\end{equation} where $C$ is universal,
and for every $1\le q<2$
\begin{equation}\label{nrjfp}
\int_{\mathcal{R}} |\nab f|^q \le C_q,\end{equation}
where $C_q$ depends only on $q$.
\end{lem}
\begin{proof}
This is a standard computation, of the type of \cite{bbh}, Chap. 1,
observing that $f= -  \log |x|+S(x)$  where $S$ is a $C^1$ function.
\end{proof}

\subsubsection{Proof of Proposition~\ref{pw2bis}}  Let  $j\in G$. Apply
Lemma~\ref{lemmebord} with $p\in(\frac32, 2)$ and $\gamma=\frac34$.
For any $R$ large enough it provides us with a square $K_t$, where $t$ depends on $j\in G$ but satisfies 
\begin{equation}\label{trange}R^{\frac34}\le R-t\le 2R^{\frac34}.\end{equation}

We wish to  extend $j$ in $K_R \backslash K_t$, keeping $j$ and $\Lambda$ unchanged in $K_R \backslash K_t$,  to obtain a $j_R$ satisfying $j_R\cdot \tau =0$ on $\p K_R$ and $\curl j_R = 2\pi \sum_{p\in \Lambda_R} \delta_p-1$ --- while the extension's contribution  to the renormalized energy remains negligible compared to $R^2$, uniformly with respect to $j\in G$.  

Below, the notations $o(\cdot)$, $\sim$ and $O(\cdot)$ are understood with respect to $R\to +\infty$, and {\em uniform with respect to } $j\in G$. % the estimates in the construction below will depend  only on the estimates provided by Lemma \ref{lemmebord}, which are uniform in $j \in G$, and by Lemmas \ref{lemrect1} and \ref{lemrect2}, which are independent of $j$, hence will all be uniform with respect to $j \in G$.

   We divide each side
    of $K_t$ into $[\sqrt R]$ intervals, so that there are a total
     of $4[\sqrt R]$ intervals,  which we label $\{I_i\}_i$, of
     length $\frac {2t}{[\sqrt R]}$, which is equivalent to $2\sqrt R$ as $R\to +\infty$, uniformly with respect to $t$ satisfying \eqref{trange}.
     For each $i$ we consider the square $K_i\subset K_R\sm K_t$
      with one side equal to $I_i$. By perturbing the length of
      the other side of an amount which is $O(1/\sqrt R)$  as $R \to \infty$
      we may  obtain
       a rectangle $\mathcal{R}_i$ whose aspect ratio tends to $1$. 
        and such that
\begin{equation}\label{qtvol}
|\mathcal{R}_i|- \int_{I_i} j\cdot \tau \in 2\pi \mn.\end{equation}
We let $g_i $  denote  the restriction
of $ j \cdot \tau$ to $I_i$ (and extend it by $0$ on the rest of $\p \mathcal{R}_i$),
and let
$m_i= 1- \frac{\int_{\p \mathcal{R}_i} g_i}{|\mathcal{R}_i|}$.
Using H\"older's inequality,  we have
$$|m_i-1|\le \frac{C}{|\mathcal{R}_i|} \(
\int_{I_i } |j\cdot \tau|^p\)^{\frac1p} R^{\hal(1-\frac{1}{p})} \le
C R^{\hal(-1- \frac{1}{p})}\( \int_{\p K_t} |j|^p \)^{\frac{1}{p}}.
$$
In view of \eqref{majojp}, we deduce
$$|m_i - 1| = O(R^{-\hal- \frac1{2p} +\frac2p - \frac\gamma p})= o(1),$$
since $\gamma = \frac34$ and $p>\frac32$.

By \eqref{qtvol} and  by choice of $m_i$ we also  have $m_i |\mathcal{R}_i|\in 2\pi \mn$.
 We may thus tile   $\mathcal{R}_i$   by an  integer number of
 rectangles $\mathcal{R}_{ik}$, whose sidelengths are in $\sqrt{2\pi} [\hal, \frac{3}{2}]$
and such that for each $i,k$, we have
$m_i |\mathcal{R}_{ik}|= 2\pi$. Since $m_i\sim 1 $, the number
of rectangles inside each $\mathcal{R}_i$ is equivalent to $|\mathcal{R}_i|/2\pi\sim R/2\pi$ as $R\to +\infty$.

On each of these rectangles, we may apply Lemma \ref{lemrect2}, which yields a  function $
f_{ik}$ satisfying \eqref{nrjf2}--\eqref{nrjfp}.
We then define the vector field $j_1$ in $\cup_i \mathcal{R}_i$ by
 $j_1= - \nab f_{ik}$ in each $\mathcal{R}_{ik}$. We can check that
 $j_1$ satisfies
 \begin{equation}\label{equaj1glob}
\left\{\begin{array}{ll}
\curl j_1= 2\pi \sum_{p\in \Gamma} \delta_p - \sum_i m_i\indic_{\mathcal{R}_i}
 & \text{in } \ \cup_i \mathcal{R}_i\\ [2mm]
j_1 \cdot \tau= 0  & \text{on} \ \p (\cup_i \mathcal{R}_i)
\end{array}\right.
\end{equation}
where $\Gamma$
is the union over $i,k$, of the centers of the rectangles $\mathcal{R}_{ik}$.
Indeed  since $\frac{\p f_{ij}}{\p \nu}=0$ no curl is created at the interfaces between the $\mathcal{R}_{ik}$'s, and no curl is created either at the interfaces between the $\mathcal{R}_i$'s.

Moreover,
by \eqref{nrjf2}--\eqref{nrjfp},  since the number of $\mathcal{R}_{ik}$
for each $i$ is of order $|\mathcal{R}_i|$, and since the
 number of $\mathcal{R}_i$'s is $O(\sqrt R)$,
$j_1$ satisfies
\begin{equation}\label{nrjj12}
\lim_{\eta\to 0} \left|\int_{\cup_i \mathcal{R}_i \backslash
\cup B(p, \eta)}|j_1|^2  + \pi \# \Gamma   \log \eta \right| \le
C R^{\frac32},\end{equation} and for $q<2$
\begin{equation}\label{nrjj1p}
\int_{\cup_i \mathcal{R}_i} |j_1|^q \le C_q R^{\frac32}.\end{equation}

Since $(m_i-1) |\mathcal{R}_i|=- \int_{\p \mathcal{R}_i} g_i$,
we may also apply Lemma \ref{lemrect1} in each  $\mathcal{R}_i$ with $g=g_i$ for boundary data.
  It yields
a function $u_i$ satisfying \eqref{tar1e}. We then define the vector field
$j_2$ as $- \np u_i$ in each $\mathcal{R}_i$.
It satisfies\begin{equation}\label{equaj2glob}
\left\{\begin{array}{ll}
\curl j_2=  \sum_i m_i\indic_{\mathcal{R}_i}-1
 & \text{in } \  \cup_i \mathcal{R}_i\\ [2mm]
j_2 \cdot \tau= g  & \text{on} \ \p \cup_i \mathcal{R}_i
\end{array}\right.
\end{equation}
where $g= j \cdot \tau $ on $\p K_t$ and $0$ on the rest of $\p (\cup_i \mathcal{R}_i)$.
Indeed,  $g_i$ is only supported on the $I_i$ i.e. on  the sides of the $\mathcal{R}_i$
which are in $\p K_t$, so $j_2 \cdot \tau =0$ on all  the boundaries of the $\mathcal{R}_i$ which intersect, therefore again
no curl is created there.
For every $1\le q\le 2p $ we have, since $|I_i|\sim \sqrt R$,
\begin{equation*}
\int_{\mathcal{R}_i}|j_2|^q\le C_{p,q} R^{1-\frac{q}{2p}} \|g_i\|_{L^p(\p \mathcal{R}_i)}^q.
\end{equation*}
Adding  these relations, we obtain
$$
\int_{\cup_i
\mathcal{R}_i}|j_2|^q\le C_{p,q} R^{1-\frac{q}{2p}}\sum_i \( \int_{\p \mathcal{R}_i}
|g_i|^p\)^{\frac{q}{p}}.$$
But when $q/p>1$ we have $\sum_i x_i^{q/p} \le (\sum_i \max(1,x_i))^{q/p}$ and number of $\mathcal{R}_i$'s is $O(\sqrt R)$ hence,
$$
\int_{\cup_i
\mathcal{R}_i}|j_2|^q\le C_{p,q} R^{1-\frac{q}{2p}} \(
\int_{\p K_t } |g|^p + \sqrt R\)^{\frac{q}{p}}.$$
Using \eqref{majojp}, we deduce, for all $1\le q\le 2p$
\begin{equation}\label{ivs}
\int_{\cup_i
\mathcal{R}_i}|j_2|^q\le C_{p,q} R^{1-\frac{q}{2p}+
\frac{q}{p}(2-\frac34)}
\le C_{p,q}
 R^{1+\frac{3q}{4p}} .\end{equation}
 From now on we choose $p \in (\frac{3}{2}, 2)$ and we have for every $q<\frac{4p}{3}$,
\begin{equation}\label{nrjj2glob2}
\int_{\cup_i
\mathcal{R}_i}|j_2|^q\le C_q
R^\sigma, \quad \text{for some $\sigma <2$}.\end{equation}

We can now define $j_R$ more precisely. In $\cup_i \mathcal{R}_i$ we let
$j_R=j_1+ j_2$ and $\nu_R =2\pi\sum_{p \in \Gamma}  \delta_p$.
 By summing \eqref{equaj1glob} and  \eqref{equaj2glob} we have
$$\left\{\begin{array}{ll}
\curl j_R=  2\pi \sum_{p \in \Gamma} \delta_p -1
 & \text{in } \  \cup_i \mathcal{R}_i\\ [2mm]
j_R \cdot \tau= g  & \text{on} \ \p (\cup_i \mathcal{R}_i),
\end{array}\right.$$
where $g = j\cdot \tau$ on $\p (\cup_i \mathcal{R}_i)\cap \p K_t$ and $g=0$ elsewhere. Also,
\begin{equation}\label{nrjjrl}
 \int_{\cup_i \mathcal{R}_i \backslash \cup  B(p, \eta)} |j_R|^2
= \int_{\cup_i \mathcal{R}_i \backslash \cup B(p, \eta)} |j_1|^2 + |j_2|^2 + 2 j_1 \cdot j_2.
\end{equation}
We have
$$\left|\int_{\cup_i \mathcal{R}_i \backslash \cup B(p, \eta)}
j_1 \cdot j_2\right|\le \( \int_{\cup_i \mathcal{R}_i }
|j_1|^{q'}\)^{1/{q'}} \( \int_{\cup_i \mathcal{R}_i }|j_2|^q\)^{1/q}$$
where  $q>2$ and $\frac{1}{q'}= 1-\frac1q.$
Using \eqref{nrjj1p}  and \eqref{nrjj2glob2} where we can choose $q>2$ since $p>\frac{3}{2}$, we find
$$\left|\int_{\cup_i \mathcal{R}_i \backslash \cup B(p, \eta)}
j_1 \cdot j_2\right|\le C  R^{\frac{3}{2q'}+ \frac{\sigma}{q} }=o(R^2),$$
since $\sigma<2$ and $\frac{1}{q'}+\frac1q =1.$ Inserting into
 \eqref{nrjjrl} and combining with \eqref{nrjj12}
 and \eqref{ivs}  we obtain
$$\lim_{\eta\to 0}
\left|\int_{\cup_i \mathcal{R}_i \backslash \cup  B(p, \eta)} |j_R|^2
 + \pi \# \Gamma
\log\eta\right| = O(R^{\frac32} )+o(R^2)= o(R^2).$$
There remains to define $j_R$ in $A:= K_R \sm( K_t
 \cup_i \mathcal{R}_i) $. First we note that $|A|\in 2\pi\mn$.
  Indeed, from $\curl j = \nu -1$ and \eqref{qtvol}, we have 
$$2\pi\#(\Lambda\cap K_t) - |K_t|=  \int_{\p K_t} j\cdot \tau= \sum_i| \mathcal{R}_i| \ (\mathrm{mod} \,  2\pi) ,$$
thus $|A| = |K_R| - \sum_i| \mathcal{R}_i| - |K_t|$  is in $2\pi\mn$ if $|K_R|\in 2\pi\mn$.

The set $A$ can be described as follows: It is the union of four squares of sidelength $R-t$ positioned at the corners of $K_R$, and a union $\cup_i \mathcal{R}_i'$, where $\mathcal{R}_i'$ is a rectangle having a side of length $|I_i|$ in common with $\mathcal{R}_i$, and such that their union is isometric to $I_i\times[0,R-t]$. Since both dimensions of these rectangles as well as those of the four squares tend to $+\infty$ as $R\to +\infty$, and since $|A|\in 2\pi\mn$, it is possible to tile $A$ by rectangles of area $2\pi$ and aspect ratio close to $1$. Applying Lemma~\ref{lemrect2} in each of them yields a current $j_A$ which satisfies  $\curl j_A = 2\pi\sum_{p\in\Gamma'} \delta_p - 1$ in $A$  and $j_A\cdot\tau = 0$ on $\p A$ --- where $\Gamma'$ is the set of centers of the rectangles tiling $A$.

The cardinal of $\Gamma'$ is $|A|/2\pi$, which is $O(R^{1+\frac34})$, and therefore from \eqref{nrjf2} we deduce
\begin{equation}\label{nrjkt}
\lim_{\eta\to 0} \left|\int_{K_R \backslash
( K_t   \cup_i \mathcal{R}_i) } |j_A|^2
+ \pi \#\Gamma'\log \eta\right|=O(R^{\frac74}),\end{equation}
and from \eqref{nrjfp}, for all $1\le q <2$,
\begin{equation}\label{jaq}
\int_{K_R \backslash
( K_t   \cup_i \mathcal{R}_i) } |j_A|^q\le C_q R^{\frac{7}{4}}.\end{equation}

Letting $j_R = j_A$ in $A$ and
$$\Lambda_R = (\Lambda\cap K_t) \cup\Gamma\cup\Gamma',\quad \nu_R =2\pi  \sum_{p\in\Lambda_R}\delta_p$$
we have $j=j_R$ in $K_t$, $\nu_R= \nu$ in $K_t$, $\curl j_R = \nu_R - 1$ in $K_R$, $j_R\cdot \tau = 0$ on $\p K_R$, and combining \eqref{nrjkt}, \eqref{nrjjrl} and \eqref{minowt} we get
\begin{equation}\label{nrjj}
\frac{W(j_R, \indic_{K_R}) }{|K_R|}= W_K(j) + o(1) \quad \text{as} \ R \to +\infty,\end{equation} where the $o(1)$ is uniform with respect to $j \in G$. This completes the proof of Proposition \ref{pw2bis}.
We also note that from \eqref{nrjj1p}, \eqref{nrjj2glob2} and \eqref{jaq},  for every $1\le q<2$, we have
\begin{equation}\label{jllllp}
\int_{K_R\backslash K_t} |j_R|^q \le R^\sigma\quad \text{for some} \  \sigma <2.\end{equation}

\subsubsection{Proof of Lemma \ref{lemmebord}}
Let $G$ satisfy the assumptions of Proposition~\ref{pw2bis},  and $R>2$ with $|K_R|\in 2\pi \mn$, $0 < \gamma<1$ be given. Assume $j\in G$. In this proofs the constants and limits as $R\to +\infty$ are understood to be uniform with respect to $j\in G$.\medskip

{\it Step 1:}
Denote by $g_R$ the result of applying Proposition~\ref{wspread} in
$K_R$ to $(j,\nu)$. We apply \eqref{wg} to functions of the form
$\chi(x) = \rho(\|x\|_\infty)$, i.e. whose level sets are squares,
with the additional assumption that $\rho'(t) = 0$ outside
$[R-2,R-1]$ and $\rho=0$ on $[R-1,+\infty]$. Since for any Radon measure
$\mu$ on $K_R$ we have
$$ \int \chi\,d\mu =  -\int_0^{R-1} \rho'(t)\mu(K_t)\,dt,$$
we deduce that
\begin{equation}\label{rm1}   
\int_{R-2}^{R-1} \(W(j,\indic_{K_t}) - g_{R}(K_t)\) \rho'(t)\,dt  = -  W(j,\chi) + \int \chi \,dg_{R}\le C n (\log n + 1) \|\rho'\|_\infty,\end{equation}
where the last inequality is \eqref{wg}, and $n=\# \{ p \in \Lambda,
B(p, C) \cap \supp \  \nab \chi\neq \varnothing\}$, so that
$2\pi n \le \nu(K_{R+C}) - \nu(K_{R-C})$.  Here $C$ denotes a universal  constant, hence independant of $j\in G$.  
We deduce by  duality that for some universal $C>0$ we have 
\begin{equation}\label{rm2}\int_{R-2}^{R-1}\left|W(j,\indic_{K_t}) -  g_{R}(K_t)\right| \,dt\le C n (\log n + 1).\end{equation}

On the other hand, Lemma \ref{wlp} yields that for any  $p\in [1,2)$ and $R>0$, 
$$\|j\|_{L^p(K_R)} \le CR^{2/p} \log^{1/2} R,$$ where $C$ depends only on $p$ and on the constants in \eqref{hypunif} and \eqref{cvwunif}. Arguing as in the proof of Lemma~\ref{wnu} this implies that 
\begin{equation}\label{rm3}
n \le\frac{1}{ 2\pi} |\nu(K_{R+C})-\nu(K_R)|\le CR^{\theta(1-\frac1p)} \|j\|_{L^p(K_{R+2C})}\le C R^{\theta(1-\frac1p)+\frac2p}\log^{\frac12} R\le CR^\beta\end{equation}
for some $\beta<2$, where we have chosen $p<2$ close enough to $2$ and used $\theta<2$ in \eqref{hypsetsbis}. It then follows from \eqref{rm2}, \eqref{rm3} that 
\begin{equation}\label{carre1}\frac{1}{R^2}
\int_{R-2}^{R-1}\left|W(j,\indic_{K_t}) -  g_{R}(K_t)\right| \,dt \le C R^{\beta-2}\log R.\end{equation} 

Now denote by $\{\chi_R\}_R$ a family of functions satisfying \eqref{defchi} relative to the family $\{K_R\}_R$. We also assume $\chi_R\le 1$. Since $g_{R}\ge -C $ and since $\chi_{R} = 1$ on $K_{R-1}$   we have for any $t\in[R-2,R-1]$ the inequalities
\begin{equation*}
\int \chi_{R-2} \, dg_{R} - C |K_{R-1}\sm K_{R-3}|
\le
g_{R}(K_t)\le \int \chi_{R} \,dg_R + C|K_R\sm K_{R-2}| \le \int \chi_{R} \,dg_R + CR. \end{equation*}
and thus
\begin{equation}\label{carre2}
\int \chi_{R-2} \, dg_{R} - C R
\le
g_{R}(K_t)\le \int \chi_{R} \,dg_R + CR. \end{equation}

{\it Step 2:} For any integer $k\ge 1$ let $\xi_k = \chi_{k+1} - \chi_k$, and let $\xi_0 = \chi_1$. Then $\xi_k\ge 0$, since $\chi_{k+1} = 1$ on $K_k$ and since $\chi_k\le 1$ and is supported in $K_k$. Moreover $\xi_k$ is supported in  $\mathcal{C}_k := K_{k+1}\sm K_{k-1}$. Since \eqref{hypunif} holds,   the number of integers $k$ in $[ R-2R^\gamma+1, R- R^\gamma]$
such that  $\nu(K_{k+2}\sm K_{k-2})\le C R^{2-\gamma} $ is greater than $R^\gamma/2$ if $C$ is chosen large enough. Similarly,
$$\sum_{k=[R-2R^\gamma]}^{[R- R^\gamma]} \int \xi_k\,d g_R = \int (\chi_{[R-R^\gamma]+1} - \chi_{[R-2R^\gamma]})\,dg_R\le CR^2,$$ where we use Corollary~\ref{wgcoro},  Remark \ref{aprescoro}  and \eqref{cvwunif}. Since $g_R\ge -C$ we have $\int\xi_k\,dg_R\ge - CR$ and therefore the number of integer $k$'s between  $[R-2R^\gamma]+1$ and $[R- R^\gamma]$ such that $\int \xi_k\,dg_R\le CR^{2-\gamma}$ is larger than $R^\gamma/2$ if $C$ and $R$ are chosen large enough, and thus such a  $k$ satisfying
$$\nu(K_{k+2}\sm K_{k-2})\le C R^{2-\gamma},\quad \int\xi_k\,dg_R\le C R^{2-\gamma},$$ for some  $C$ uniform with respect to  $j \in G$.

Applying Proposition~\ref{wspread} in $\mathcal{C}_k$ to $\xi_k$, we have
$$\left|W(j,\xi_k) - \int \xi_k\,dg_R\right|\le C R^{2-\gamma} \log R$$
hence $W(j,\xi_k)\le C R^{2-\gamma} \log R$ and applying Lemma~\ref{wlp} we find that for $p<2$,
$$  \int_{C_k} |\xi_k|^{\frac p2} |j|^p   \le C R^{1-\frac p2} \(R^{2-\gamma} \log R\)^{\frac p2}\le C R^{2-\gamma}.$$
But $\xi_k(x) = 1$ if $\|x\|_\infty =k$ and thus from the gradient bound $\xi_k\ge 1/2$ if $k-1/C\le \|x\|_\infty \le k+1/C$. Therefore
$$\int_{K_{k-\frac1C}\sm K_{k+\frac1C}} |j|^p \le C  R^{2-\gamma}.$$
By a mean value argument on this integral as well as on  \eqref{carre1} (applied with $R= k+1$), we deduce  the existence $t \in [k-1, k]$ --- hence $t\in [R-2R^\gamma,R-R^\gamma]$ ---  such that, one the one hand
$$\int_{\partial K_{t}}|j|^p  \le C  R^{2-\gamma},$$ proving \eqref{majojp} since  $C$ is uniform with respect to $j \in G$ --- and on the other hand  
\begin{equation}\label{parta}  g_R(K_t) - W(j,\indic_{K_{t}})\le CR^{\beta}\log R.\end{equation} 
Now,  from \eqref{carre2} applied to $R = k+1$ and using \eqref{wg} in Proposition~\ref{wspread} together with \eqref{rm3} we obtain  
$$W(j,\chi_{R-2}) - C R^\beta\log R\le g_R(K_t) \le W(j,\chi_R) + CR^\beta\log R, $$
Which together with \eqref{parta} and in view of \eqref{cvwunif} yields \eqref{minowt}. Finally,
$$\left|\nu(K_t) - |K_t|\right| = \left|\int_{\p K_t}j\cdot\tau\right| \le \|j\|_{L^p(\p K_t)} |\p K_t|^{1-\frac1p}\le C R^{2-\gamma}, $$
using \eqref{majojp}. Lemma~\ref{lemmebord} is proved.

%\end{proof}

\subsubsection{Proof of Corollary \ref{periodisation}}
 Let $j\in \ainfty$ be such that $W_K(j) <+\infty$. Let $R$ be such that $R^2 \in 8\pi \mn$, and
 $j_R$ be obtained by applying  Proposition \ref{pw2bis}
  over $K = [0,R]\times [0,R]$. We let $\tilde\jmath_R=j_R-\nab\zeta$, where $\Delta\zeta = \div j_R$ on $K$ and $\zeta = 0$ on $\p K$. Then $\tilde \jmath_R = -\np H_R$  in $K$ since $\div \tilde\jmath_R = 0$ there,
and we have $\p_\nu H_R = 0$ on $\p K$ since $\tilde \jmath_R\cdot \tau = j_R\cdot \tau = 0$ there.
 Thus defining $H_R$ on $K_R = [-R,R]\times[-R,R]$ by letting $H_R(\pm x,\pm y) = H_R(x,y)$ we have $-\Delta H_R = \sum_{p\in \Lambda_R} \delta_p -1$, where $\Lambda_R$ is obtained from the restriction of $\curl j+1$ to $K$ by reflections across the coordinate axis. Moreover $H_R(-R,y) = H_R(R,y)$ and $H_R(x,-R) = H_R(x,R)$ so that we may periodize  $H_R$ to have it defined on $\mr^2$.  Then   $j_R:=-\np H_R$
  belongs to $\ainfty$ and since everything is periodic $W$ can be computed through the results of Section  \ref{rem52}:
$$W_K(j_R)=
 \frac{W(j_R,\indic_{K_R})}{|K_R|} \le  W_K(j)  +o(1) \quad
  \text{as}\
 R \to \infty.$$
 The last assertion of the Corollary follows by taking $j$ to minimize $W_K$ over $\ainfty$ (a minimizer exists by Proposition \ref{pw1}), and remembering that the minimum
does not depend on the choice of shapes used.

\subsubsection{Proof of Corollary \ref{mesyoung}}
The proof of Corollary~\ref{mesyoung} consists in constructing a sequence from a Young measure on micropatterns, to use the terminology of \cite{am}, while retaining an energy control. We thus assume $P$ is a probability measure  on $\Lp$ which is invariant under the action of  translations  and concentrated in  $ \ainfty$. \medskip

First we choose distances which metrize the topologies of $\Lp$ and $\B (\Lploc)$, the set of finite Borel measures on $\Lp$. For $j_1,j_2\in\Lp$  we let
$$d_p(j_1,j_2) = \sum_{k=1}^\infty 2^{-k} \frac{\|j_1 - j_2\|_{L^p(B(0,k)) }}{1+\|j_1 - j_2\|_{L^p(B(0,k))}}. $$
On $\B(\Lploc)$ we define a distance by choosing a sequence of bounded continuous functions $\{\vp_k\}_k$ which is dense in $C_b(\Lploc)$ and we let,  for any $\mu_1,\mu_2\in\B(\Lploc)$, 
$$ d_{\B}(\mu_1,\mu_2) = \sum_{k=1}^\infty 2^{-k} \frac{|\lb \vp_k,\mu_1-\mu_2\rb|}{1+|\lb \vp_k,\mu_1-\mu_2\rb|},$$
where we have used the notation $\lb\vp,\mu\rb = \int\vp\,d\mu$.

We have the following general  facts. 

\begin{lem}\label{etavar} For any $\ep>0$ there exists $\eta_0>0$ such that if $P,Q\in\B(\Lploc)$ and $\|P-Q\|<\eta_0$, then $d(P,Q)<\ep$. Here $\|P-Q\|$ denotes the total variation of the signed measure $P-Q$, i.e. the supremum of $\lb\vp,P-Q\rb$ over measurable functions $\vp$ such that $|\vp|\le 1$. 
\end{lem}
In particular, if $P = \sum_{i=1}^\infty \alpha_i\delta_{x_i}$ and  $Q = \sum_{i=1}^\infty \beta_i\delta_{x_i}$ with $\sum_i|\alpha_i - \beta_i| <\eta_0$, then  $d_{\B}(P,Q)<\ep$.

\begin{lem}\label{etadir} Let $K\subset \Lp$ be compact. For any $\ep>0$ there exists $\eta_1>0$ such that if $x\in K, y\in \Lp$ and $d_p(x,y)<\eta_1$ then $d_{\B}(\delta_x, \delta_y)<\ep$. 
\end{lem}

\begin{lem}\label{etaconv} Let $0<\ep<1$.  If  $\mu$ is a probability measure on a set $A$ and $f,g:A\to \Lp$ are measurable and such that $d_{\B}(\delta_{f(x)}, \delta_{g(x)})<\ep$ for every $x\in A$, then 
$$d_{\B}(f^\#\mu,g^\#\mu) <C \ep(\lep+1).$$
\end{lem} 
\begin{proof} Take any bounded continuous function $\vp_k$ defining the distance on  $\B(\Lploc)$. Then if $d_{\B}(\delta_{f(x)}, \delta_{g(x)})<\ep$ for any $x\in \Lp$ we have in particular  
$$\frac{|\vp_k(f(x)) - \vp_k(g(x))|}{1+ |\vp_k(f(x)) - \vp_k(g(x))|}\le 2^k \ep.$$
It follows that  
$$ d_{\B}(f^\#\mu,g^\#\mu) \le \sum_k 2^{-k} \min(\ep 2^k, 1)\le \ep \({[\log_2\ep]+1}\) + \sum_{k={[\log_2\ep]+1}}^\infty 2^{-k} \le C\ep(\lep+1). $$
\end{proof} \medskip

\noindent{\it Selection of a good subset of  $\Lploc$.} We must restrict to a compact subset of $\Lp$ in a suitable way. This is not surprising when constructing an approximation: note that the set of micropatterns in \cite{am} is assumed to be compact, and we need to reduce to this case. This is the aim of the following Lemma. 

\begin{lem}\label{GH}
Given $P$ as above and $\ep, R>0$ there exist subsets $H_\ep \subset G_\ep $ in  $\Lp$ with $G_\ep$ compact and  such that:
\begin{itemize}
\item[i)] $\eta_0$ being given by Lemma~\ref{etavar} we have 
 \begin{equation}\label{pgep} P({G_\ep}^c) <\min({\eta_0}^2,\eta_0 \ep),\quad  P(H_\ep^c) < \min(\eta_0,\ep).\end{equation}
\item[ii)] For every $j\in H_\ep$ there is a subset  $\Lj\subset K_R$ such that 
\begin{equation}\label{gamj}\text{$|\Lj|< C {R}^2 \eta_0$ and 
 $ \lambda\notin \Lj\implies \theta_\lambda j\in G_\ep.$}\end{equation}
\item[iii)]
\begin{equation}\label{wunif}\text{$W_K(j)$ and $\frac{\nu(K_t)}{t^2}$ are bounded uniformly with respect to $j\in G_\ep$ and  $t>1$,}\end{equation}
 where $\curl j = \nu-1$, and the convergence in the definition of $W_K(j)$  is uniform.
\item[iv)] We have 
 \begin{equation}\label{ppseconde}  d_{\B}(P,P'') < C \ep(\lep+1),\quad\text{where} \quad 
P'' = \int_{H_\ep}\frac1{|K_{R}|}  \int_{K_{R}\sm\Lj} \delta_{\theta_\lambda j}\,d\lambda\,d P(j).\end{equation}
\end{itemize}
Moreover, there exists a partition  $H_\ep=\cup_{i=1}^{N_\ep} H_\ep^i$  such that  $\diam (H_\ep^i) <\eta_3$, where $\eta_3$ is such that 
\begin{equation}\label{eta5} 
j \in H_\ep, \ d_p(j,j') <\eta_3 ,\ \lambda \in K_{R}\backslash \Lj
    \implies
d_{\B} (\delta_{\theta_\lambda j}, \delta_{\theta_\lambda j'}) <\ep;\end{equation}  
 and   for all $1\le i\le N_\ep$ there exists  $J_i \in H_\ep^i$ such that 
 \begin{equation}\label{wjk} W_K(J_i) < \inf_{H_\ep^i} W_K +\ep.\end{equation}
\end{lem}

\begin{proof}\medskip\hbox{}

\noindent{\it Choice of $G_\ep$.}  Since $\Lp$ is Polish we can always find a compact set  $G_\ep$ satisfying \eqref{pgep} and  $P({G_\ep}^c)<\eta_0$. Then from Lemma~\ref{etavar},  $P\llcorner G_\ep$ (the restriction  of $P$ to $G_\ep$)  satisfies $d_{\B}(P,P\llcorner G_\ep)<\ep.$

From the translation-invariance of $P$, we have for any $\lambda$ that  $P(\theta_\lambda G_\ep) >1-\eta_0$ and therefore that $d_{\B}(P, P\llcorner \theta_\lambda G_\ep)< \ep$. It follows that for any $\lambda\in\mr^2$ we have $\|P-P_\lambda\|<\eta_0$ and then $d_{\B}(P,P_\lambda)<\ep$, where  
$$P_\lambda = \int_{\theta_\lambda G_\ep} \delta_{j}\,dP(j) = \int_{G_\ep} \delta_{\theta_\lambda  j}\,dP(j).$$
Then using Lemma~\ref{etaconv} we deduce that if $A\subset \mr^2$ is any measurable set of positive measure, then  
\begin{equation}\label{pprime} d_{\B}(P,P') < C \ep(\lep+1),\quad\text{where}  \quad P' = \int_{G_\ep}\dashint_A \delta_{\theta_\lambda j}\,d\lambda\,dP(j).\end{equation}

Moreover, since $P$ is invariant, choosing $\chi$ to be a smooth positive function with integral $1$ supported in $B(0,1)$, the ergodic theorem  (as in \cite{becker}) ensures that  for $P$-almost every $j$ the limit 
$$\lim_{t\to +\infty} \frac1{|K_t|} \int_{K_t} W(j(\lambda+\cdot),\chi(\lambda+\cdot))\,d\lambda$$
exists. Then   $\indic_{K_t} * \chi$ is  a family of functions  which satisfies \eqref{defchi} with respect to the family of squares $\{K_t\}_t$, and from the definition of the renormalized energy relative to $\{K_t\}_t$  we may rewrite the limit above as 
\begin{equation}\label{cvw} W_K(j) = \lim_{t\to +\infty} \frac1{|K_t|} W(j,\indic_{K_t} * \chi). \end{equation}
By Egoroff's theorem we may choose the compact set $G_\ep$ above to be such that, in addition to \eqref{pprime}, the convergence in \eqref{cvw}  is uniform on $G_\ep$. In fact, since  $W_K(j)<+\infty$ and $\limsup_t\nu(K_t)/t^2<+\infty$ for $P$-a.e. $j$,  where $\curl j = \nu-1$, we may choose  $G_\ep$ such that \eqref{wunif} holds.

The difficulty we have to face next is that $\theta_\lambda j$ need not belong to $G_\ep$ if $j$ does. \medskip

\noindent{\it Choice of $H_\ep$.}  For $j\in G_\ep$,   let $\Lj$ be the set of $\lambda$'s in $ K_{R}$ such that $\theta_\lambda j\not\in G_\ep$. Since, from \eqref{pgep} and the translation-invariance of $P$, for any $\lambda\in\mr^2$ we have $P(\theta_\lambda (G_\ep)^c)<{\eta_0}^2$, it follows from Fubini's theorem that 
$$\int_{G_\ep} |\Lj|\,dP(j) = \int_{K_R} P((\theta_\lambda G_\ep)^c)\,d\lambda < 4 R^2\min({\eta_0}^2,\eta_0 \ep).$$
Therefore, letting 
\begin{equation}\label{defhep} H_\ep = \{j\in G_\ep: |\Lj| <  4R^2  \eta_0\}, \end{equation}  we have that 
\eqref{pgep} holds. Combining \eqref{pgep} and \eqref{defhep} with Lemma~\ref{etavar}, we deduce from \eqref{pprime} that  \eqref{ppseconde} holds. 

Then we use the fact that $G_\ep$ is compact and  Lemma~\ref{etadir} to find that  there exists $\eta_4>0$ such that  
\begin{equation}\label{eta4} \forall j \in G_\ep,\forall j'\in\Lp,\quad d_p(j,j')<\eta_4\  \implies\  d_{\B}\(\delta_{ j},\delta_{  j'}\)<\ep. \end{equation}   Moreover, from the continuity of $(\lambda,j)\mapsto \theta_{\lambda}  j$, there exists $\eta_3>0$ such that 
$$\forall j\in G_\ep, \forall \lambda\in K_R,\quad d_p(j,j')<\eta_3\ \implies\ d_p\(\theta_\lambda j,\theta_\lambda j'\)<\eta_4. $$
Now if $j \in H_\ep$ and  $\lambda \in K_R$ this implies that if $d_p(j,j')<\eta_3$ then $ d_p\(\theta_\lambda j,\theta_\lambda j'\)<\eta_4$. But if $\lambda\notin \backslash \Lj $ we have $\theta_\lambda j \in G_\ep$ hence applying  \eqref{eta4} to $\theta_\lambda j$, $\theta_\lambda j'$,  we get \eqref{eta5}.\medskip

\noindent{\it Choice of $J_1,\dots,J_{N_\ep}$.}  
Now we cover the relatively compact $H_\ep$ by a finite number of balls $B_1,\dots, B_{N_\ep}$ of radius $\eta_3/2$ and  derive from it a partition of $H_\ep$ by sets with diameter less than $\eta_3$  by letting $H^1_\ep = B_1\cap H_\ep$  and 
$$H^{i+1}_\ep = B_{i+1}\cap H_\ep\sm \(B_1\cup\dots\cup B_i\).$$
We then have 
\begin{equation}\label{partition} H_\ep = \bigcup_{i=1}^{N_\ep} H^i_\ep,\quad \diam(H^i_\ep)<\eta_3(R),\end{equation}
where the union is disjoint.  Then we may  choose $J_i\in H_\ep^i$ such that  \eqref{wjk} holds.
\end{proof} \medskip

\noindent{\it Completion of  the construction.}
First we apply Proposition~\ref{pw2bis} with $G = G_\ep$. The proposition yields $R_0>1$ such that for any  $j\in G_\ep$ and any $R>R_0$ such that $|K_R|\in 2\pi\mn$ there exists $j_R$ defined in $K_R$ such that \eqref{eqjr} is satisfied and  such that,  if $ x\in K_{R - 2R^{3/4} } $,  then $j_R(x)=j(x)$. Moreover, 
\begin{equation}\label{restronque} \frac{W(j_R,\indic_{K_R})}{|K_R|} \le  W_K(j)+\ep.\end{equation}

We choose  $R_\ep>R_0$ such that $|K_{R_\ep}|\in 2\pi\mn$ and large enough so that   
\begin{equation}\label{aspratio} K_{R_\ep (1-\eta_0)}\subset \{x: d(x,{K_{R_\ep}}^c)> {R_\ep}^{\tq}\}.\end{equation}
where $\eta_0$ is the constant in Lemma~\ref{etavar}.  

If $\lambda\in K_{R_\ep (1-\eta_0)}$ and  since $j(x)=j_{R_\ep}(x)$ if $d(x,{K_{R_\ep}}^c)> {R_\ep}^{\tq}$, we deduce from \eqref{aspratio} that  $\theta_\lambda j_{R_\ep} = \theta_\lambda j$ in $B(0, {R_\ep}^{\tq})$ as soon as $\eta_0 R_\ep> 2{R_\ep}^\tq$,  so that  from the definition of $d_p$,  taking  $R_\ep$ larger if necessary,  
\begin{equation}\label{proxitronque}  \forall j'\in\Lp,\forall\lambda\in K_{R_\ep(1-\eta_0)},\quad\text{$j' = j_{R_\ep}$ on $K_{R_\ep}$} \implies d_p(\theta_\lambda j, \theta_\lambda j') <\eta_1,\end{equation}
where $\eta_1$ comes from Lemma~\ref{etadir} applied on $G_\ep$, i.e.  is such that 
\begin{equation}\label{eta1}
\text{$j\in G_\ep,j'\in\Lp$ and $d_p(j,j')<\eta_1$} \implies d_{\B}(\delta_{j},\delta_{j'})<\ep.\end{equation}
Having chosen $R_\ep$, we get from Lemma~\ref{GH} a set $H_\ep$ and a partition $H_\ep=\cup_{i=1}^{N_\ep} H_\ep^i$ and in each $H_\ep^i$ a current $J_i$ satisfying \eqref{wjk}. We also choose an arbitrary $J_0\in\ainfty$ such that $W_K(J_0)<+\infty$.

Second we choose an integer $q_\ep$ large enough so that 
\begin{equation}\label{qep} \frac {N_\ep}{{q_\ep}^2} < \eta_0,\qquad \frac{N_\ep}{{q_\ep}^2}\times \max_{0\le i\le N_\ep} W_K(J_i)< \ep.\end{equation}

Now  if  $j\in H_\ep^i$  then $d_p(j,J_i)<\eta_3$ and   we deduce  from \eqref{eta5} that  for every $\lambda\in K_{R_\ep}\sm\Lj$ we have  
\begin{equation}\label{crucial} d_{\B}\(\delta_{\theta_\lambda j}, \delta_{\theta_\lambda J_i}\)<\ep.\end{equation}
Using \eqref{crucial} together with  Lemmas~\ref{etadir}, \ref{etavar},  and the bound  \eqref{gamj}, we deduce from \eqref{ppseconde} that $d_{\B}(P,P''')< C\ep(\lep +1)$, where 
\begin{equation}\label{p3} P''' = \sum_{1\le i\le N_\ep} p_{i}\dashint_{K_{R_\ep}} \delta_{\theta_\lambda J_i}\,d\lambda,\quad\text{where}\quad 
 p_{i} = P\(H_\ep^i\).\end{equation}
We now replace $p_{i}$ in the definition \eqref{p3}  by 
\begin{equation}\label{nkk} \frac{n_i}{{q_\ep}^2} ,\quad\text{where}\quad n_{i} = \left[{q_\ep}^2 p_{i}\right].\end{equation}
Then $\sum_{i=1}^{N_\ep} n_{i}\le {q_\ep}^2$ and 
  \begin{equation}\label{diffpn}
\sum_{1\le i \le N_\ep} 
\left|\frac{n_i}{{q_\ep}^2}  - p_{i}\right|<  \frac{N_\ep}{{q_\ep}^2}<\eta_0.\end{equation}
Then Lemma~\ref{etavar} implies that $d_{\B}(P,P^{(4)})< C\ep(\lep +1)$ where 
\begin{equation}\label{p3bis} P^{(4)} = 
\sum_{1\le i\le N_\ep} \frac{n_{i}}{{q_\ep}^2}\dashint_{K_{R_\ep}} \delta_{\theta_\lambda J_i}\,d\lambda.\end{equation}

Now we let $K_\ep = [-q_\ep R_\ep, q_\ep R_\ep]^2$, and $n_{0} := {q_\ep}^2 - \sum_{i=1}^{N_\ep} n_{i}$, so that $\sum_{i=0}^{N_\ep} n_{i} = {q_\ep}^2$. Then we divide $K_\ep$ in a collection $\LK$ of ${q_\ep}^2$ identical subrectangles which are translates of $K_{R_\ep}$, and for each $0\le i\le N_\ep$  we choose $n_i$ subrectangles in an arbitrary way and  call the collection of these subrectangles $\LKk$, so that $\{\LKk\}_{0\le i\le N_\ep}$ is a partition of $\LK$. 

Let us call  $J_{i,R_\ep}$ the currents obtained from $J_i$ using Proposition~\ref{pw2bis}. They satisfy \eqref{restronque} and \eqref{proxitronque}. We claim that,  as a  consequence of the latter, we have for any $L\in\LKk$ that 
\begin{equation}\label{claimtronque}\text{$j' = J_{i,R_\ep}$ on $K_{R_\ep}$} \implies d_{\B}\(\dashint_{K_{R_\ep}} \delta_{\theta_\lambda J_i}\,d\lambda,\dashint_{K_{R_\ep}} \delta_{\theta_\lambda  j'}\,d\lambda\) < C\ep(\lep +1).\end{equation}
This goes as  follows: (i)  Using \eqref{gamj} and Lemma~\ref{etavar},  integrating on $K_{(1-\eta_0)R_\ep}\sm\Lji$  instead of $K_{R_\ep}$ induces an error of $C\ep$. (ii) From \eqref{proxitronque}, and \eqref{eta1} applied to $\theta_\lambda J_i$ by $\theta_\lambda j'$ we have $d_{\B}(\delta_{\theta_\lambda J_i}, \delta_{\theta_\lambda j'}) < \ep$ and thus in view of Lemma~\ref{etaconv} we may replace $\theta_\lambda J_i$ by $\theta_\lambda j'$ in the integral with  an error of $C\ep\lep$ at most. (iii) Using \eqref{gamj} and Lemma~\ref{etavar} again, we may integrate back on $K_{R_\ep}$ rather than on $K_{(1-\eta_0)R_\ep}\sm\Lji$, with an additional error of $C\ep$. this proves \eqref{claimtronque}. 

Then combining  \eqref{claimtronque} with \eqref{p3bis} and $d_{\B}(P,P^{(4)})< C\ep(\lep +1)$, using Lemma~\ref{etaconv} we find $d_{\B}(P,P^{(5)})< C\ep(\lep +1)$, where 
\begin{equation}\label{PQ} P^{(5)} = \sum_{1\le i\le N_\ep}\frac{n_i}{{q_\ep}^2}\dashint_{K_{R_\ep}} \delta_{\theta_\lambda \tilde J_{i,R_\ep}}\,d\lambda,\end{equation}
and $\tilde J_{i,R_\ep}$ is an arbitrary field in $\Lp$ such that $\tilde J_{i,R_\ep} = J_{i,R_\ep}$ on $K_{R_\ep}$, the constant $C$ being independent of the choice of $\tilde J_{i,R_\ep}$. 

We chose above an arbitrary $J_0$ in $\mathcal A_1$ such that $W_K(J_0)<+\infty$. Let the sum in \eqref{PQ} range over $0\le i\le N_\ep$ instead of $1\le i\le N_\ep$, this defines a measure $P^{(6)}$ such that, by \eqref{qep},
\begin{equation}\label{ppn0}\|P^{(5)} - P^{(6)}\|\le \frac{n_0}{{q_\ep}^2}\le \eta_0,\end{equation}
where we have used \eqref{diffpn} and the fact that $1-\sum_i p_i = P({H_\ep}^c) < \eta_0$, from \eqref{p3}, \eqref{pgep}. Hence using Lemma~\ref{etavar} we have $d_{\B}(P^{(5)},P^{(6)})<\ep$ and then $d_{\B}(P,P^{(6)})< C\ep(\lep +1)$. 
 
We now define the vector field $\jnint:\mr^2\to\mr^2$  by letting  $\jnint(x) = J_{i,R_\ep} (x-x_{L})$ on every $L\in \LKk$, $0\le i\le N_\ep$ --- where $x_L$ is the center of $L$ and thus $L = x_L+ K_{R_\ep}$ ---  and by requiring $\jnint$ to be $K_\ep$-periodic. For every $L\in\LKk$ we have $\jnint(x_L+\cdot) = J_{i,R_\ep}$ on $K_{R_\ep}$, therefore we may choose $\tilde J_{i,R_\ep} = \jnint(x_L+\cdot)$ in \eqref{PQ} and then
$$\int_{K_\ep} \delta_{\theta_\lambda \jnint}\,d\lambda = \sum_{\substack{0\le i\le N_\ep\\ L\in \LKk}} \int_L \delta_{\theta_\lambda \jnint}\,d\lambda = \sum_{\substack{0\le i\le N_\ep\\ L\in \LKk}} \int_{K_{R_\ep}} \delta_{\theta_\lambda \jnint(x_L+\cdot)}\,d\lambda = \sum_{0\le i\le N_\ep} n_i \int_{K_{R_\ep}} \delta_{\theta_\lambda \tilde J_{i,R_\ep}}\,d\lambda.$$
Therefore we may summarize the discussion concerning $P^{(5)}$, $P^{(6)}$ by writing  
\begin{equation}\label{resume} d_{\B}(P,P^{(6)})< C\ep(\lep +1), \quad P^{(6)} = \dashint_{K_\ep} \delta_{\theta_\lambda \jnint}\,d\lambda.\end{equation}

Note that since $J_{i,R_\ep} = 0$ outside $K_{R_\ep}$, and $J_{i,R_\ep}\cdot\tau = 0$ on $\p K_{R_\ep}$ we have, in $K_\ep$,  
\begin{equation}\label{jnint} \text{$\D\jnint=\sum_{\substack{1\le i\le N_\ep\\ L\in\LKk}} J_{i,R_\ep} (\cdot-x_{ L}),\quad \curl \jnint = 2\pi \sum_{p\in\Lambda_\ep}\delta_p - 1,\quad \jnint\cdot\tau = 0$ on $\p K_\ep$.}\end{equation}
where $\Lambda_\ep$ is a finite subset of the interior of $K_\ep$. This completes the construction of $\jnint$.\medskip

\noindent{\it Estimate of the energy.} We have
$$W(\jnint,\indic_{K_\ep})= \sum_{\substack{0\le i\le N_\ep\\ L\in\LKk}} W(J_{i,R_\ep} (\cdot-x_{L}),\indic_L)= \sum_{\substack{0\le i\le N_\ep\\ L\in\LKk}} W(J_{i,R_\ep},\indic_{K_{R_\ep}}).$$
From \eqref{restronque} applied to the $J_i$'s we deduce that 
\begin{equation}\label{wkk} W(\jnint,\indic_{K_\ep}) \le |K_{R_\ep}|  \sum_{i=0}^{N_\ep} n_i\( W_K(J_i) + C \ep\).\end{equation}
Now from \eqref{diffpn} and the fact that $1-\sum_i p_i = P({H_\ep}^c) < \ep$ we deduce that $n_0 \le N_\ep + {q_\ep}^2 \ep$, and then from \eqref{qep} that $n_0 \le C\ep {q_\ep}^2$, where we have included $W_K(J_0)$ in the constant. This and \eqref{wkk}, together with the estimates \eqref{diffpn}, \eqref{qep} and the definition of $p_i$ in \eqref{p3},  implies that 
$$W(\jnint,\indic_{K_\ep}) \le {q_\ep}^2|K_{R_\ep}|\( \sum_{i=1}^{N_\ep} P(H_\ep^i) W_K(J_i) +C \ep\),$$
and then from \eqref{wjk},  and since $|K_\ep| = {q_\ep}^2|K_{R_\ep}|$, that 
\begin{equation}\label{ejnint}W(\jnint,\indic_{K_\ep}) \le |K_{\ep}|\(\int_{H_\ep}W_K(j)\,dP(j) + C \ep\)\le |K_{\ep}|\(\int W_K(j)\,dP(j) + C \ep\),\end{equation}
using \eqref{pgep} and the fact from Proposition~\ref{pw1} that $W_K$ is bounded below. 

Choosing a sequence $\{\ep\}\to 0$ we thus obtain a sequence $\{R\}$ tending to $+\infty$, where $R = q_\ep R_\ep$ and a sequence of currents $\{j_R\}$, where $j_R = j_\ep$ such that \eqref{c4i} is satisfied --- this is \eqref{jnint} ---  and $\D\limsup_{R\to\infty} \frac{W(j_R,\indic_{K_R})}{|K_R|}
\le\int  W_K(j) \, dP(j)$ --- this is \eqref{ejnint}. Moreover from \eqref{resume} we have $P_R\to P$, using the notations of Corollary~\ref{mesyoung}. Thus Corollary~\ref{mesyoung} is proved.

\part{From Ginzburg-Landau to the renormalized energy}

\section{The energy-splitting formula and the blow up}\label{sec-splitting}
In this section, we return to the Ginzburg-Landau energy and  prove an algebraic splitting formula on it  already discussed in Section \ref{sec1.8}, as well as results on the splitting function. We recall that if $\lim_{\ep\to 0} \he/\lep>\lambda_\om$, then $\ho$ may be used as the splitting function. Only when the limit is equal to $\lambda_\om$ does one need to use $\tho$ with $N\neq N_0$ instead.

We recall $\ho$ is the minimizer of \eqref{obs1} and $\tho$ is given by \eqref{deftho}.
 We also introduce the notation of the appendix: for $m \in (- \infty, 1]$,
 $\hm$ denotes the minimizer of
 \begin{equation}\label{mop}
 \min_{h- 1\in H^1_0(\om)} (1-m)\io |- \Delta H + H| + \hal \io |\nab H|^2
 + |H-1|^2.\end{equation} $\hm$ is the solution of an obstacle problem and its properties
 are studied in the appendix. In particular $\hm \ge m$ and $-\Delta \hm + \hm = m \indic_{\omega_m}$, where $\omega_m = \{\hm = m\}$ is the so-called coincidence set. We have
\begin{lem}\label{lemrelated}
For any $ 0 \le N \le\frac{|\om|\he}{2\pi}$ there exists a unique $ m
 \in [1- \frac{1}{2\lambda_\om} , 1] $ (and conversely) such that
$\tho = \he \hm$, and $m $ and $N$ are related by
$2\pi N= \he m |\omega_m|$. Moreover $m$ and $|\omega_m|$ are continuous  increasing functions of $N$.
\end{lem}
\begin{proof} The minimization problem  \eqref{deftho} has a unique minimizer $\tho$ by
convexity.
On the other hand, by the theory of Lagrange multipliers,
 $\tho $ is the minimizer of
$$\min_{ h- \he \in H^1_0(\om) } \hal \io |\nab h|^2 +
|h-\he |^2+\lambda \io |- \Delta h+h|
$$ for some number $\lambda$ characterized by the fact
that the unique minimizer satisfies
  \linebreak$\io |- \Delta h+h|= 2\pi N$. But this minimizer
   is precisely $\he H_m$ with $m$ such that $2\pi N = \he
    m|\omega_m|$ hence $\tho = \he H_m$. From
    Proposition~\ref{obstacleprops}, $m\mapsto|\omega_m|$ is  continuous increasing and one to one  from $I = [1-\frac1{2\lambda_\om}, 1]$ to $[0,|\om|]$, hence if $2\pi N$ is between $0$ and $\he|\om|$, then there exists a unique $m \in I$ such that $\tho = \he H_m$, and $m$ is a continuous increasing function of $2\pi N/\he$ --- hence of $N$ --- characterized by $m|\omega_m| = 2\pi N\he$, and obviously $|\omega_m|$ is too.
\end{proof}
We will denote by $\me$ the $m$ corresponding to $\tho$, and note that
\begin{equation}\label{rangec} 0<1- \frac{1}{2\lambda_\om} \le \me\le 1.\end{equation}
It follows from Lemma~\ref{lemrelated} that   $\tho$ is also  the solution to an obstacle problem, hence $\tho\in C^{1,1}$ (see \cite{frehse}) and satisfies $\he\me\le\tho\le \he$ and
$$ \tmuo:= - \Delta \tho + \tho= \me \he \indic_{\toe}, \quad\text{where}\quad \toe = \{ \tho= \he \me  \}.$$
Note that $\tmuo(\om) = 2\pi N$.

 We will also let
\begin{equation}\label{Cep}
\Ce := \he(\me - m_{0,\ep}) =  \he \me - \he + \hal \plep  .\end{equation}
It is immediate from \eqref{range}, \eqref{rangec} that
$$
\Ce =  O(\he).$$
The minimizer $\ho$ of \eqref{obs1} is equal to
 $h_{\ep,N_0}$,  where $N_0$ is given by \eqref{defn0}.
Moreover, we recall (see  \eqref{mep}) that  $\ho= \he - \hal \plep$ on its coincidence set,
hence $\he m_{\ep,N_0} = \he - \hal \plep$ and $c_{\ep,N_0}=0$.
%\begin{equation}\label{muo}\muo= \he- \hal\L :=\co\he\ \text{on its support $\omega_\ep$}.\end{equation}  so we have in that case $\Ce=0$.
%So when $N=N_0$ then $\ce = \he - \hal \L$.

On the other hand $H_m$ is increasing with respect to $m$, see Proposition~\ref{obstacleprops}, hence
 if $m_1\le m_2$ then $H_{m_1}\le H_{m_2}$ (see Proposition \ref{obstacleprops} in the appendix) so
\begin{equation}\label{signC}
\Ce\ge 0 \quad \text{if } N\ge N_0\qquad \Ce\le 0 \quad \text{if} \ N\le N_0.\end{equation}
We will be most interested in the cases where $N$ is one of the two integers
closest to $N_0$.

\subsection{Energy-splitting}
Let $h_\mu$, a ``splitting function",  be any function such that
$\mu:=  -\Delta h_\mu + h_\mu \in L^2 (\om)$  and $h_\mu= \he$ on $\bo$. Let
\begin{equation}\label{deca} A_1= A- \np h_\mu.\end{equation}
Then,
\begin{lem} \label{lemmasplit}For any $(u,A)$ and $h_\mu$ as above  we have
\begin{multline}\label{dec}
\je(u,A) = \hal\|h_\mu - \he\|^2_{H^1(\om)} + \hal \io \(|u|^2-1\)|\nab h_\mu|^2 +\\
\\ + \hal  \io |\nab_{A_1 } u|^2+(\curl A_1 - \mu)^2 + \frac{(1-|u|^2)^2}{2\ep^2} +
 \io (h_\mu- \he) \(\mu(u,A_1)-\mu\).
\end{multline}
\end{lem}
\begin{proof}
From \eqref{deca}, we have
 $$|\nab_A u|^2 = |\nab_{A_1} u|^2 - 2 \np h_\mu \cdot j(u,A_1) + |u|^2 |\nab h_\mu|^2,$$
 where we have used the notation $j(u,A) = (iu,\nab_A u).$ Also, since $\curl A= \curl A_1 + \Delta h_\mu= \curl A_1 - \mu +h_\mu $, we may write
 $$(\curl A -\he)^2 = (h_\mu -\he)^2 + 2(\curl A_1 - \mu)   (h_\mu-\he) + (\curl A_1 - \mu)^2.$$
 Replacing in \eqref{je}  and integrating by parts the term $\np h_\mu \cdot j(u,A_1)$, using the fact that $h_\mu = \he$ on $\bo$, we find
\begin{multline*}\je(u,A) = \hal\|h_\mu - \he\|^2_{H^1(\om)} + \hal \io \(|u|^2-1\)|\nab h_\mu|^2 +\\
\\ + \hal  \io |\nab_{A_1 } u|^2+(\curl A_1 - \mu)^2 + \frac{(1-|u|^2)^2}{2\ep^2} +
 \io (h_\mu- \he) \(\curl j(u,A_1)+ \curl A_1 - \mu  \).\end{multline*}
 This yields \eqref{dec}, using  the fact that $\mu(u,A_1)= \curl j(u,A_1) +\curl A_1$.
 \end{proof}
Using the particular choice  of splitting function $h_\mu = \tho$ in Lemma \ref{lemmasplit}, we obtain the following result:
\begin{pro}\label{pro1} Let $0 \le N \le \frac{\he|\om|}{2\pi}$ and let  $\tho$ be the corresponding  minimizer of \eqref{deftho}, then
for any $(u,A)$, denoting $\ai = A- \np \tho$, and using the notation \eqref{121b} we have
\begin{equation}
\label{dc} \je(u,A)= \je^N + \jie(u,\ai)
- \hal\io (1-|u|^2)|\nab\tho|^2 ,\end{equation} where
\begin{multline}
\label{g1}\jie(u,A)= \hal \io |\nab_{A} u|^2 + (\curl A-\tmuo)^2 +
\frac{(1-|u|^2)^2}{2\ep^2}\\
 + \io (\tho-\he -\Ce )\mu(u,A)+ \Ce \io (\mu(u,A )- \tmuo).\end{multline}
\end{pro}
\begin{proof}
In \eqref{dec}, we replace $\io (\tho - \he) \( \mu(u,\ai)- \tmuo\)$
by using the fact that by definition of $\tho$, on the support of $\tmuo$, $\tho - \he =
\he \me  - \he= - \hal\plep +\Ce$ and $\io \tmuo= 2\pi N$.  \end{proof}
Note that the functional $F_\ep$ depends on $N$ but for simplicity we will not denote that dependence.

\subsection{Dependence on $N$}
We have the following
\begin{lem} \label{lemvarg}
 For $N \in [0, \frac{|\om|\he} {2\pi}]$ we have
\begin{equation}\label{variationsg}
\frac{d\je^N}{dN}= 2\pi \he ( m-m_{0, \ep} ) \end{equation}
 where $m$ is the one-to-one function of $N$ given by  Lemma \ref{lemrelated}.
Moreover, $\je^N$ is minimized uniquely at $N_0$   and minimized among integers at $\nm $ or $\ns$ (or both) where $\nm$ is the largest integer below $N_0$ and $\ns$ the smallest integer above.
\end{lem}
\begin{proof}
From \eqref{121b} and Lemma \ref{lemrelated} we have \begin{equation}\label{jenform}
\je^N= \pi N \plep+ \hal \he^2 \|H_m - 1\|_{H^1(\om)}^2  \end{equation} where $m$ and $N$ are related by $2\pi N = \he m |\omega_m|$.
On the other hand, since  $H_m$  minimizes \eqref{mop}  we have
\begin{equation}
\label{a1a22}
(1-m) \frac{2\pi N }{\he}  + \hal \| H_m- 1\|_{H^1(\om)}^2 \le (1-m) \frac{2\pi N'}{\he}+ \hal \|H_{m'} - 1
\|_{H^1(\om)}^2\end{equation} where $2\pi N'= \he m' |\omega_{m'}|$. Reversing the roles of $m$ and $m'$ and taking $m'=m+ \delta$ with $\delta\to 0$, we obtain
$$\frac{d}{dm} \|H_m-1\|_{H^1(\om)}^2 = \frac{ 4\pi}{\he} (m-1) \frac{dN}{dm} $$ (in the sense of $BV$ derivatives).
Inserting into \eqref{jenform} we deduce
$$\frac{d\je^N}{dN}= \pi \plep +  2\pi \he (m-1)= 2\pi \he \( m - 1 + \frac{\plep}{2\he}\)$$ hence  the result \eqref{variationsg} in view of \eqref{mep}.  But $m_{0, \ep}$ is the $m$ that corresponds to $N_0$. It immediately follows with the monotonicity of $N \mapsto m $  that $\je^N$ is  decreasing in $[0, N_0]$, increasing in  $[N_0, \frac{|\om|\he}{2\pi}]$, hence minimized at $N_0$, and minimized among integers at $\nm $ or $\ns$.
\end{proof}

{\bf For the rest of Section~\ref{sec-splitting} and Section~\ref{section:lb}, $N \in \{\nm, \ns\}$.}  Since $2\pi N_0= \io \muo \le |\om| \( \he - \hal \plep\)$, see \eqref{omep}, it is clear that
$N_0^-\le N_0^+<    \frac{|\om|\he}{2\pi} $ so Lemma \ref{lemrelated} applies and the corresponding $\tho$ are well-defined.
We also record the following
\begin{lem}If  $N_0\gg 1 $ and  $N \in \{\nm, \ns\}$ then
\begin{equation}\label{limitem}\lim_{\ep\to 0}  \frac{\Ce}{\plep} =0
 \qquad
\lim_{\ep \to 0} \me= m_\lambda, \end{equation}
where $\lambda= \lim_{\ep \to 0} \frac{\he}{\lep}$ and $m_\lambda= 1- \frac{1}{2\lambda} $.
\end{lem}
\begin{proof} 
First we recall that $m_{0,\ep} = 1 - \frac\plep{2\he}$, hence by definition of $\lambda$, we have 
$\lim_{\ep \to 0 } m_{0, \ep}=m_\lambda$.

Assume that $N=\ns$. We use the fact, seen in Lemma \ref{lemrelated}, that $m$ and $|\omega_m|$ are increasing functions of $N$. Thus, since $ N_0\le N\le N_0+1$, we have $|\omega_{\ep, N}|\ge |\omega_{0, \ep}|$ and, using  
$$0\le m_{\ep, N}-m_{0,\ep}=\frac{2\pi \ns}{\he |\omega_{\ep, N}|}-\frac{2\pi N_0}{\he |\omega_{0,\ep}|}\le \frac{2\pi\ns}{\he |\omega_{0,\ep}|}- \frac{2\pi N_0}{\le |\omega_{0,\ep}|}\le \frac{2\pi}{\he |\omega_{0, \ep}|} =\frac{m_{0,\ep}}{N_0}\le \frac{1}{N_0},$$ because we always have $m\le 1$. 
Since we are in the regime $N_0\gg 1$, we find that $m_{\ep, N} $ and $m_{0,\ep}$  have the same limit, that is $m_\lambda$.
The case $N=\nm$ is treated analogously, and we deduce in both cases 
\begin{equation}\label{nmpm}
|m_{\ep, N}-m_{0,\ep}|\le O(\frac{1}{N_0}).\end{equation}

Assume then that  $\he\le O( \lep)$, from 
 \eqref{Cep} and \eqref{nmpm}, we deduce  $\Ce = o(\he)$, which proves \eqref{limitem}.
If $\he\gg \lep$, then  $m_{0, \ep} \sim 1 $  and $ 2\pi N_0\sim \he |\om| $ hence combining with \eqref{Cep} and  \eqref{nmpm}, we find
$\Ce=O(1)$ which also implies the result in this case.
\end{proof}

\subsection{Blow-up procedure}
We  write for simplicity $\omega_\ep$ instead $\toe$ and  $m_\ep$ instead of $\me$, when the precise value of $N \in \{\nm, \ns\}$ does not need to appear explicitly.

Assume \eqref{range} holds.
%Then $\co$, as defined in \eqref{mep}
%satisfies $\co\le 1$ and \eqref{half}, hence is bounded away from
%zero.
Let
$$\el = \frac 1{\sqrt{\he}},$$
and, assuming $0 \in \omega_\ep \subset \om$,  write $x = \el x'$.
Under this change of coordinates the domain $\om$ becomes $\pOe$ and
the subdomain $\omega_\ep$ becomes $\poe$, both becoming infinitely
large as $\ep\to 0$ --- this is obvious for $\pOe$ and for $\poe$, it is proven to be a consequence of \eqref{range} in Proposition~\ref{lemdomaine} below. We call $\pfep$ the expression of $\jie(u,\ai)$
(see \eqref{g1}) in terms of the rescaled unknowns $u'(x')= u(x)$
and $A'(x') = \el \ai(x)$. It is given by
\begin{multline}
\label{Fep2}\pfep(u',A') =\hal \int_{\om'_\ep} |\nab_{A'}
u'|^2 + \frac 1{{\el}^2} |\curl A' -  \co \indic_{\omega'_\ep}|^2
+\frac{(1-|u'|^2)^2}{2(\ep')^2} \\ - \int_{\om'_\ep}\zep'\mu(u',A') +
\Ce \int_{\om'_\ep} (\mu(u',A' )- \co \indic_{\omega_\ep'} ),
\end{multline}
where $\ep'$, $\zep'$ are given by
\begin{equation} \label{scaled}
\ep' = \frac\ep{\el},\quad   \zep'(x') = \he - \tho(x) +\Ce.
\end{equation}
%This way
%\begin{equation}\label{zbord}
%\|\zep'\|_{L^\infty(\p \om_\ep')} = o(\L).\end{equation}
We also define the blown-up current in $\om_\ep'$
\begin{equation}\label{jep}
\pje= \curl (iu', \nab_{A'} u')\end{equation} and the blow-up measure $\mu_\ep' =\mu(u_\ep', A_\ep') $ and extend them by $0$
outside $\om_\ep'$. Note that $\pje(x')= \ell_\ep j_{1,\ep}(\ell_\ep x') $ if $j_{1,\ep}$ denotes the current  $ (iu_\ep, \nab_{\ai} u_\ep)$ in the original variables.

The function $\tho$ is in  $C^{1,1}(\om)$ and  equal to $\he $ on $\bo $. It   attains
its minimum $\he\me $ on $\toe $. From  \eqref{Cep}, \eqref{scaled}
$$\Ce = \min_\om(\tho - \he)+\hal\plep,\quad \zep' (x')= \hal\plep  - \tho(x)+ \min_\om\tho,$$
thus $\zep'\in C^{1,1}(\pOe)$, $\zep' = \Ce$ on $\partial\pOe$ and $\zep'$ attains its maximum on $\poe$. Moreover
\begin{equation}\label{zepprime}\max_{\pOe}  \zep' = \hal\plep,\quad \|\nab\zep'\|_\infty\le C\sqrt{\plep} ,\end{equation}
this last assertion following from \eqref{nabho} in Proposition \ref{lemdomaine} below.

\subsection{Additional results on the splitting function}
In this subsection, we adapt some results from the appendix that we will need below.
In the appendix we introduce an ellipse $E_Q$
of measure $1$ and a function $U_Q\ge 0$ defined in $\mr^2$
such that \begin{equation*} \Delta U_Q = \frac{\Delta
Q}{2}\indic_{\mr^2\sm E_Q},\quad \{U_Q=0\} = E_Q,\end{equation*}  where $Q$ is the quadratic form $ D^2h_0(x_0)$ as introduced in   \eqref{assumption}.

The following proposition can be applied to $N\in\{N_0, \nm, \ns\}.$ In either case we denote by $\omega_\ep$ the associated coincidence set, and by $\poe$ its blow-up.

\begin{pro}\label{lemdomaine}Assume \eqref{range} holds.
Let $\tho$ be as above with $|N- N_0|\le 1$.
\begin{enumerate}
\item
We have $N_0\gg 1$, and if $\he - \hci \ll \lep$
then $N_0 \ll \he$ and
\begin{equation}\label{rangeN}\he - \hci   \sim \lambda_\om N_0
 \log \frac{\he}{N_0}.\end{equation}
\item  For every $x\in\om$
\begin{equation}\label{volome} d(x,\omega_\ep)\le C\sqrt\frac\plep\he,\quad
 |\om \backslash \omega_\ep|\le C \sqrt{\frac{\plep}{\he}}.\end{equation}
Moreover \begin{equation}
 \label{nabho}
 \|\nab \tho\|_{L^\infty(\om)} \le  C \sqrt{\he \plep}. \end{equation}
\item If $K$ is any compact subset of $(\lambda_\om,+ \infty)$,
\begin{equation}\label{item3}\lim_{\delta\to 0} \frac{|\{x\mid d(x,\p\omega_\ep)
 < \delta\}|}{|\omega_\ep|} = 0,\quad \lim_{\delta\to 0}
  \frac{|\{\he\me< \tho   <\he\me+
  \delta\he \}|}{|\omega_\ep|} = 0 \end{equation}
uniformly with respect to  $\ep$ such that $ \frac{\he}{\lep}\in K$.
\item If $\he/\lep\to\lambda_\om$ then there exists  $\{L_\ep\}_\ep$ such that for any $\delta,M>0$ and if $\ep$ is small enough
\begin{equation}
\label{enst2}
\begin{split}
\left\{\frac{\tho - \min_\om\tho}{\he {L_\ep}^2}\ge M\right\}&\subset x_0 + L_\ep\{U_Q\ge M-\delta\},\\
\left\{\frac{\tho - \min_\om\tho}{\he {L_\ep}^2}\le M\right\}&\subset x_0 + L_\ep\{U_Q\le M+\delta\},\\
 \{ d(x, \omQ^c) >\delta L_\ep\} &\subset \omega_\ep\subset \{d(x,\omQ)<\delta L_\ep\},
\end{split}
\end{equation}
where  $\omQ=x_0+ L_\ep E_Q$ and $x_0$ is defined in \eqref{assumption}. Moreover there is a constant $C_\om$ depending only on $\om$ such that
\begin{equation}\label{rellep}
L_\ep^2 |\log L_\ep|\sim  C_\om \frac{\he- \lambda_\om \lep}{\he}.
  \end{equation}
In particular
$\frac{\llep}{\lep}\ll L_\ep^2|\log L_\ep|$ and $\el\ll L_\ep$, and from \eqref{enst2}
$$ L_\ep^2 |\omQ|\sim |\omega_\ep|,\quad \{x\mid d(x,\partial\omega_\ep)\le \delta \el\}\ll |\omega_\ep|.$$

\item For any $R>0$,
\begin{equation}\label{poebr}\left|\{x\mid d(x,(\omega_\ep')^c)>R\}\right|\sim
| \omega_\ep'|\end{equation}
as $\ep\to 0$, i.e. $\omega_\ep'$  satisfies
\eqref{domaine}.

\item
We have
\begin{equation}
\label{enn0} \je^N= \je^{N_0}+
O\(\frac{\he}{N_0}\)
  \le  C \he \plep.\end{equation}
 If $\lep^4\ll \he \ll \frac{1}{\ep^2}$,
 \begin{equation}\label{lq}
\je^N= \hal \he |\om|\plep +o(\he).\end{equation}
\end{enumerate}
\end{pro}

\begin{proof}
{\it Proof of item 1:} We apply Proposition \ref{obstacleprops} to  $m_{0,\ep}= 1- \frac{\plep}{2\he}$.
We denote by $L_\ep$ the $L_m$   corresponding to $m_{0,\ep}$
given by Proposition \ref{obstacleprops}. In particular we have $|\omega_{0,\ep}|\sim L_\ep^2$.
In that proposition $\homb$ denotes  $\min h_0$ and is equal to
$1 - \frac{1}{2\lambda_\om}$.
 Moreover $\he - \lambda_\om \lep \gg
\llep$ is equivalent to $ \frac{1}{2(1-m_{0,\ep})}-\frac{1}{2(1-\homb)}\gg
\frac{\llep}{\lep}$ which in turn is equivalent to  $m_{0,\ep} - \homb \gg
\frac{\llep}{\lep}$. From \eqref{la}
\begin{equation}\label{dfg2}
L_\ep^2 |\log L_\ep|\sim 2\pi
(m_{0,\ep} -\homb)/\homb    = \frac{2\pi}{\homb}
 \( \frac{1}{2\lambda_\om}
- \frac{\plep}{2\he}\). \end{equation}
Then, since $|\omega_{0,\ep}|\sim L_\ep^2$ and $N_0= \he m_{0,\ep} |\omega_{0,\ep}|$
we have $L_\ep^2 |\log L_\ep|\sim   \hal |\omega_{0,\ep}||\log |\omega_{0,\ep}|| \sim  \hal \frac{N_0}{m_{0,\ep} \he}
|\log \frac{N_0}{m_{0,\ep}\he} |$ and inserting into \eqref{dfg2} we find
\begin{equation}\label{dfg3}
\hal \frac{N_0}{m_{0,\ep} \he}
|\log \frac{N_0}{m_{0,\ep}\he} | \sim \frac{2\pi}{\homb}
 \( \frac{1}{2\lambda_\om}
- \frac{\plep}{2\he}\) = \frac\pi{\lambda_\om\homb\he}\(\he - \lambda_\om\lep +\hal\log\he\).\end{equation}
It follows, since $m_{0,\ep} \ge \homb>0$,
 that if \eqref{range} holds, we must have $N_0\gg 1$, for otherwise the left-hand side is $O(\llep)$ while the right-hand side is $\gg \llep$.
 Moreover, if $\he - \hci \ll \lep$ then $m_{0,\ep} \sim  1- \frac{1}{2\lambda_\om}= \homb$ and
 inserting into \eqref{dfg3} and rearranging terms, we obtain $N_0 \ll\he $  and \eqref{rangeN}.

{\it  Proof of  item 2.} From \eqref{mep}, $m_{0,\ep} = 1 - \frac\plep{2\he}$ hence using Proposition~\ref{obstacleprops} we have
$\|\nab\ho\|_\infty = \he \|\nab H_{m_{0,\ep}}\|_\infty \le C\sqrt{\he\plep}$. If $|N-N_0|\le 1$ then from \eqref{limitem} we have $|m_{0,\ep} - \me|\ll \plep/\he$ and therefore $\me\le 1- \frac\plep\he$ for $\ep$ small enough. Proposition~\ref{obstacleprops} applied to $\me$ yields the same bound for $\nab\tho$, and \eqref{volome} follows similarly from Proposition~\ref{obstacleprops} using $\me\le 1 - \frac\plep\he$.

{\it Proof of  item 3.} First note that if $\he/\lep$ is bounded above  as $\ep\to 0$ then $\lep\sim\plep$ hence $\he/\plep$ is bounded as well. Moreover, since $\plep\le \lep$, if $\he/\lep\ge \alpha>\lambda_\om$ then $\he/\plep\ge \alpha$ hence if $\he/\lep$ belongs to a compact subset $K$ of $(\lambda_\om,+\infty)$, then $\he/\plep$ belong to a (different) compact subset $K'$ of $(\lambda_\om,+\infty)$. But, from \eqref{mep}, if this is the case then $m_{0,\ep}$ belongs to a compact subset $\tilde K$ of $\(\homb, 1 \)$. Then from \eqref{limitem} and if $\ep$ is small enough the same is true of $\me$, and the result follows from Proposition~\ref{obstacleprops} applied to $\me$.

{\it Proof of  items 4,5.} This is again Proposition~\ref{obstacleprops}. Indeed $\he/\lep\to \lambda_\om$ implies that $m_{0,\ep}\to\homb = 1 - \frac1{2\lambda_\om}$ hence $\lim_{\ep\to 0} \me = \homb$ and \eqref{la} holds, and for any $\delta, M>0$, \eqref{item4} also if $\ep$ is small enough. It is easy to check using \eqref{limitem}, \eqref{mep} and \eqref{lomhom} that \eqref{la} implies \eqref{rellep}, and \eqref{enst2} is \eqref{item4} since $\tho - \min_\om\tho = \he(H_{\me} - \me)$.

To prove \eqref{poebr} we distinguish the case $\me\to\homb$ from the case where $\me$ is bounded away from $\homb$. In the latter and using Proposition~\ref{obstacleprops}, $|\toe|$ is bounded away from $0$ and
$$\frac{\left|\{x\mid d(x,(\omega_\ep')^c)>R\}\right|}{|\poe|} = \frac{\left|\omega_R\right|}{|\omega_\ep|},\quad\text{where $\omega_R = \{x\mid d(x,{\omega_\ep}^c)>R\el\}$}.$$
Since $R\el\to 0$ we have $\left|\omega_R\right| - |\omega_\ep|\to 0$, proving \eqref{poebr} in this case.

If $\me\to \homb$ then $m_{0,\ep}\to\homb$, i.e. $\he/\plep\to\lambda_\om$, or equivalently $\he/\lep\to\lambda_\om$. Then $L_\ep\gg\el$ and  therefore, using \eqref{enst2},
$$\{d(x,\omQ^c)>\delta L_\ep\} \subset  \omega_R \subset \{d(x,\omQ)<\delta L_\ep\}$$
holds for any $R,\delta>0$ if $\ep$ is small enough. It follows, since $|\omQ| = L_\ep^2$, that $$\left|\omega_R\right|\sim |\omQ|$$ as $\ep\to 0$ for any $R>0$, which implies in particular \eqref{poebr}.

{\it Proof of item 6.} Combining \eqref{variationsg} with \eqref{nmpm} it follows that
$$|\je^N- \je^{N_0}|= \left|\int_{N_0}^N 2\pi \he (m - m_{0, \ep}) \right|\le \frac{2\pi \he }{N_0}$$ where we have used the fact  that $ N \mapsto m$ is increasing and $|N-N_0|\le 1$.

The upper bound part of \eqref{lq} follows from \eqref{enn0} by noting that
$$\plep\int \muo +  \|\ho - \he\|^2_{H^1} \le |\om|\plep,$$
which follows from the minimality of $H_{m_{0,\ep}}$ in \eqref{mop}, using $1$ as a test function. Finally, for the lower bound part of \eqref{lq} we note that from \eqref{volome}
\begin{equation*}\begin{split}
\plep\int \muo +  \|\ho - \he\|^2_{H^1} & \ge \plep |\omega_{0,\ep}| m_{0,\ep}\\
& \ge \plep \(|\om| - C\sqrt\frac\plep\he\)\(1 - \frac\plep{2\he}\)\\
& \ge \plep|\om| + o(\he),
\end{split}\end{equation*}
if $\he\gg\lep^4$.
\end{proof}

\begin{remark}In \eqref{rangeN},
we have recovered  the formula (9.88) from \cite{livre}.
\end{remark}

\subsection{A priori bounds}
\begin{lem}\label{lemapriori} If $(u_\ep, A_\ep)$ minimizes $\je$ then
\begin{equation}
\label{bornG} \je(u_\ep,A_\ep) \le C\he\plep .
\end{equation}
Moreover, for any $(u_\ep, A_\ep)$ satisfying (\ref{bornG}) we have, if $N \in \{\nm, \ns\}$,
\begin{equation}\label{sanserreur}
 \je(u_\ep,A_\ep)= \je^N + \jie(u_\ep,\ai) + o(N),\end{equation}
where $F_\ep$ is defined in  \eqref{g1}.
%and  if $|u|\le 1$, \begin{equation}\label{mua1a}
%\|\mu(u_\ep, A_\ep)- \mu(u_\ep, \ai)\|_{H^{-1}(\om)}\le o(\sqrt{\he |\toe|})=o(N^\hal)
%\end{equation}
%\begin{equation}
%\label{borneenergie}\int e_\ep(u_\ep', \ai')\le C \he \L\end{equation}
\end{lem}
\begin{proof}
The a priori bound \eqref{bornG} is a consequence of the upper bound in Theorem~\ref{th5}, see Corollary
 \ref{coroapb}.
For the other relation  we observe that, since $\nab \tho$ is supported   in $\Omega \setminus \omega_\ep$,
it follows from (\ref{nabho})--(\ref{volome})  that
\begin{equation}
\label{hoest}
 \io |\nab \tho|^4 \le C  \he^{3/2}\plep^{5/2}.\end{equation}
We then  claim that \begin{equation}\label{cl1} \io (1-|u|^2) |\nab
\tho|^2 =  o(\he |\toe|) .\end{equation} Indeed we have $\io (1-|u|^2)^2
\le \ep^2 \je(u,A) \le C \ep^2 \he \plep  .$ Applying  the Cauchy-Schwarz inequality,
it follows  from (\ref{hoest}) that
$$\io (1-|u|^2) |\nab \tho|^2 \le  C \ep \sqrt{\he \plep} \|\nab \tho\|_{L^4}^2
\le C \ep \he^{5/4} \plep^{7/4}.$$  If $\he \le C \lep$ this is
$o(1)$ hence $o(N)$. If $ \lep\ll \he \ll 1/\ep^2 $ then
$\he|\toe|\sim \he |\om|$ and this is also $o(\he|\toe|)$.  (\ref{cl1}) is
proven, since we recall that $2\pi N= \me \he |\toe|$ with  \eqref{rangec}.
Then \eqref{sanserreur}   follows from \eqref{dc}.

%For the last relation, we notice that in view of the definition of $\mu$, we have  for any $(u,A)$,
%$\mu(u, A)- \mu(u, A_1) = \curl \((|u|^2 -1) \np \tho\)$. Taking a test-function $\xi \in H^1_0(\om)$ we thus have
%\begin{multline*}
%\left|\io (\mu(u, A)- \mu(u, A_1)) \xi \right|\le \io ||u|^2-1|\nab \tho||\nab \xi|\\
%\le \( \io (1-|u|^2)^2 |\nab \tho|^2\)^\hal\|\nab \xi\|_{L^2(\om)}
%.\end{multline*}
%If $|u|\le 1$ we have $(1-|u|^2)^2\le (1-|u^2)$ and thus from \eqref{cl1} we deduce
%\begin{equation}\label{cl12}
%\%( \io (1-|u|^2)^2 |\nab \tho|^2\)^\hal \le o(\sqrt{\he|\toe|}),\end{equation} and the result follows.
%(Note that it is not really necessary to assume that $|u|\le 1$, this can be replaced by a control on $\nab \tho$ in $L^\infty$ provided by \eqref{nabho}, at the expense of a weaker control).

%(\ref{bornenergie}) est en fait plus difficile a prouver sans rentrer dans les boules etc.
\end{proof}

We also note some consequences of the second  Ginzburg-Landau equation \eqref{gl2ee}.
\begin{lem} \label{lemdiv} Assume that  $(u_\ep, A_\ep)$ satisfies  \eqref{gl2ee}. Then, letting $\ai = A_\ep - \np\tho$ and $j_{1,\ep} = j(u_\ep,\ai)$,  we have
$$ \|\div j_{1,\ep} \|_{H^{-1}(\om)} =  o(N) $$
\begin{equation}
\label{ar2}  \|  \curl j_{1, \ep} - \mu(u_\ep, A_\ep) + \curl \ai\|_{H^{-1}(\om)} = o(N).\end{equation}

%\\
%\label{ar4} & \|\curl j_{1, \ep} \|_{W^{-1, p}(\om)} \le  \|j_{1, \ep}\|_{L^p(\om)}+ o(\sqrt{\he|\toe|})\ \quad p \le 2\\
%\label{ar5} &
%\|\mu(u_\ep, A_\ep) - \mu(u_\ep, \ai)\|_{H^{-1}(\om)} \le o( \sqrt{\he |\toe|})
%and
%$$\div j_\ep'\to 0 \   \text{in}  \ H^{-1}_{\loc}(\mr^2)$$

\end{lem}
\begin{proof}
By definition of $j_{1,\ep}$ and $\ai$ we have $j_{1,\ep}= (iu_\ep, \nab_{A_\ep} u_\ep) - |u_\ep|^2 \np \tho$. It follows that
$\div j_{1,\ep}= \div (iu_\ep, \nab_{A_\ep}u_\ep)+ \div ((1- |u_\ep|^2) \np \tho)$.  If \eqref{gl2ee} is satisfied then $\div (iu_\ep, \nab_{A_\ep} u_\ep)=0$ and $\div j_{1,\ep}= \div ((1- |u_\ep|^2) \np \tho)$.  On the other hand, combining \eqref{cl1} with $|u_\ep|\le 1$ we find
 \begin{equation}\label{obn}
 \io (1-|u_\ep|^2)^2|\nab \tho|^2 = o(N).\end{equation}
It follows that
$$ \|\div j_{1,\ep}\|_{H^{-1}(\om)} = o(\sqrt{N}).$$
By direct calculation we have $\mu(u_\ep, \ai) - \mu(u_\ep,A_\ep)= \curl ((|u_\ep|^2-1) \np \tho)$ and since $\mu(u, \ai) = \curl j_{1,\ep }+ \curl \ai$,  \eqref{ar2} follows again from \eqref{obn}.

%For \eqref{ar4} we observe that  from the above
% we may write
%$j_{1, \ep}=(1-|u_\ep|^2) \np \tho + \np h_1$ for some $h_1$. Moreover, by the boundary conditions  $(iu_\ep, \nab_{A_\ep} u_\ep) \cdot \nu=0$ on $\bo$ and $\tho = \he $ on $\bo$ we have $\frac{\p h_1}{\p \tau}= 0$ and we may assume $h_1= 0 $ on $\bo$.
%But $\|\nab h_1\|_{L^p(\om)} \le \|j_{1, \ep}\|_{L^p(\om)} + \|(1-|u_\ep|^2)\nab \tho\|_{L^p(\om)}\le   \|j_{1, \ep}\|_{L^p(\om)} + o(\sqrt{\he |\toe|})$,
%so  using the fact that  $h_1=0$ on $\bo$ we deduce
%$\|\Delta h_1\|_{W^{-1, p}(\om)} \le \|j_{1, \ep}\|_{L^p(\om)} + o(\sqrt{\he |\toe|})$. Since $\Delta h_1= \curl j_{1, \ep} - \curl ((1-|u_\ep|^2) \np\tho)$, we deduce \eqref{ar4} from the previous relations.

\end{proof}

\section{Proof of Theorem \ref{th1}, lower
bound}\label{section:lb} In this section we state the key result
from \cite{compagnon} which we need in order  to apply the
framework of Section~\ref{sec:abstract} to the minimization of the Ginzburg-Landau functional as explained in Section \ref{sec1.8}. In
paragraphs 2 (resp. paragraph 3) we use it to derive the lower bound
part of Theorems~\ref{th1},~\ref{th1prime} in the case of moderate
(resp. high) applied fields. In paragraph 4 we prove
Theorems~\ref{th1} and \ref{th1prime} assuming the upper bound of
Theorem~\ref{th5}, proven in Section~\ref{bornesup}.

\subsection{Mass displacement}
 There are two
problems which arise when trying to apply  the abstract scheme
described in Section \ref{sec:abstract} to the Ginzburg-Landau energy. The
first, and less problematic one, is that  $\jie$ is not translation
invariant. Indeed $\indic_{\omega_\ep'}$ and $\zep'$ are constant only
in the subdomain $\omega_\ep'$. Therefore we need to reduce to
integrating on $\omega_\ep'$ rather than $\om_\ep'$.

The second, more delicate problem, is that the integrand in $\jie$
is not positive, and not even expected to be bounded below uniformly
as $\ep\to 0$: Each vortex creates in $\jie$ two terms which get
infinitely large as $\ep\to 0$ and which balance each other. To
capture the difference in the limit we need to absorb the negative
part in the positive part  to obtain an essentially positive
integrand (this will by the way solve in essence the first issue).
Note that the cancellation is not a pointwise cancellation of the
different terms: While the negative contribution is very
concentrated near each vortex, the positive one is more spread out.

The method, introduced in \cite{compagnon} is the same  that  we used  in Proposition \ref{wspread} in a simpler setting, so the reader can refer to that section for an idea of it.

Let us denote  the free energy functional
\begin{equation}\label{je0}
\je^0(u,A) = \int_{\om_\ep} e_\ep,\quad e_\ep = \hal|\nab_A u|^2 + \hal (\curl A)^2 +\frac1{4\ep^2}\(1-|u|^2\)^2,\end{equation}
where $\om_\ep$ is a smooth domain depending on $\ep $ and large as $\ep \to 0$.

We now state  the result of \cite{compagnon} for the sake of completeness. It contains (in a slightly different form) Theorems 1 and 2, as well as Corollary 1.1 of \cite{compagnon}.

First,  $f_+ $ and $f_- $ will denote the positive and negative
parts of a function or measure, both being positive functions or
measures. If $f$ and $g$ are two measures, we will write $f \le g$
in the sense of measures to mean that $g-f$ is a positive measure.

For any set $E$ in the plane, $\widehat{E}$ will denote the
1-tubular neighborhood
 of $E$ in $\om_\ep$ i.e.
$$\widehat{E}= \{x\in \om_\ep, \dist(x, E) \le 1\}.$$
This way
$$\widehat{\p \om_\ep} = \{x \in \om_\ep, \dist(x, \p \om_\ep) \le 1\}.$$
For any function $v$ on $\om_\ep$ we denote (notice the absolute value)
$$\widehat{v}(x)= \sup_{y \in B(x,1)\cap \om_\ep}| v(y)|.$$
 Note that here the choice of the number $1$ is
 arbitrary.

\begin{theo}[\cite{compagnon}]\label{thcompagnon} Let $\{\om_\ep\}_{\ep>0}$ be a family
of bounded open sets in $\mr^2$. Assume that    $\{(u_\ep, A_\ep)\}_\ep$, where $(u_\ep,A_\ep)$ is defined over $\om_\ep$,  satisfies  for some  $0<\beta<1$ small enough
\begin{equation}\label{bornenrj}
\je^0(u_\ep,A_\ep)  \le \ep^{-\beta}.\end{equation}

Then for any  small enough $\ep$, there exists a measure $g_\ep$ defined over
$\om_\ep$ and a measure $\nu_\ep$ depending only on $u_\ep$  of the form $2\pi \sum_i d_i \delta_{a_i}$ for some points $a_i \in \om_\ep$ and some integers
$d_i$ such that, $C$ denoting a generic constant independent of $\ep$:
\begin{enumerate}
\item We have
\begin{equation}\label{estjac}
\|\mu(u_\ep, A_\ep) - \nu_\ep\|_{(C^{0,1}_0(\om_\ep))^*} \le C \sqrt\ep G_\ep(u_\ep, A_\ep),\end{equation}
\item The following inequality holds
$$-C \le   g_\ep \le e_\ep + \dl (\nu_\ep)_-.$$
\item For any measurable set $E\subset \om_\ep$,
\begin{equation}\label{gmoins0}
(g_\ep)_-(E)  \le  C\frac{e_\ep(\HE)}{\lep},\quad (g_\ep)_+(E)\le  C e_\ep(\HE).\end{equation}
\item Letting
$$f_\ep= e_\ep - \dl \nu_\ep,$$
for every Lipschitz function $\xi$ vanishing on
$\p \om_\ep$ and every $0<\eta\le 1$ we have
\begin{equation}\label{erreurrrr}
\int_{\om_\ep} \xi \,d(f_\ep - g_\ep) \le  C\int_{\om_\ep} \widehat{\nab \xi}\, \left[ d|\nu_\ep| +    (\beta+\eta) \,d(g_\ep)_+  +  \frac{|\log\eta|^2}{\eta}\,dx \right] +
C  \int_{\widehat{ \p \om_\ep }} \widehat{\xi}   e_\ep.
\end{equation}

\item  For any measurable set $E\subset \om_\ep$  and
every $0<\eta\le 1$  we have
\begin{equation}\label{controlenu0}
|\nu_\ep|(E)\le C\(\eta (g_\ep)_+(\HE) + \frac1\eta  |\HE| + \frac{e_\ep(\HE\cap\widehat{ \p \om_\ep })}{\lep} \),\quad |\nu_\ep|(E)\le C  \frac{e_\ep(\HE)}\lep .
\end{equation}
%\begin{equation}\label{controlenugood0}
%|\nu_\good|(E)\le \frac{C}{\lep} \int_{\HE} e_\ep
%\end{equation}
\item Assuming  $|u_\ep|\le 1$ in $\om_\ep$,  then  for every ball $B_R$ of radius $R$ such that $B_{R+C}\subset\om_\ep$ and every $p<2$,
$$ \int_{B_R}  |j_\ep|^p\le C_p \(  (g_\ep)_+(B_{R+C}) + R^2  \) .$$
\item Assume  $|u_\ep|\le 1$,   that
 $\dist(0 , \p \om_\ep ) \to + \infty$ as $\ep\to 0$ and that for any $R>0$
\begin{equation}\label{Hypcoer}
\limsup_{\ep\to 0} g_\ep(\UR) <+\infty,\end{equation}
where $\UR$ is any  family satisfying (\ref{hypsets})--\eqref{hypsetsbis}.

Then, up to extraction of  a
subsequence and for any $p<2$,  the  vorticities $\{\mu(u_\ep, A_\ep) \}_\ep$ converge in $W^{-1,p}_{loc}(\mr^2)$  to a
measure $\nu$ of the form $2\pi\sum_{p\in\Lambda}\delta_p$, where
$\Lambda$ is a discrete subset of $\mr^2$,
  the  currents $\{j(u_\ep, A_\ep) \}_\ep$
converge weakly  in $\Lp$  to $j$, and the induced fields $\{\curl A_\ep\}_\ep$ converge weakly in $L^2_{\loc}(\mr^2)$ to $h$ which are such that
$$\curl j = \nu - h,    \quad \text{in} \ \mr^2.$$

\item  If we replace the assumption \eqref{hypcoer} by  the stronger assumption
$$\limsup_{\ep \to 0} g_\ep(\UR) <C R^2,$$
where $C$ is independent of $R$, then the limit $j$ of the currents satisfies, for any $p<2$,
\begin{equation} \label{Lpunif} \limsup_{R\to +\infty} \dashint_{\UR}|j|^p\,dx<+\infty.\end{equation}
Moreover for every family $\chi_{\UR}$ satisfying (\ref{defchi})  we have
\begin{equation}\label{Limlim}
 \liminf_{\ep\to 0} \int_{\mr^2}  \frac{\chi_{\UR} }{|\UR|} \,d\gep  \ge \(
\frac{W(j,\chi_{\UR})}{|\UR|} +  \hal \dashint_{ \UR} h^2  +\frac{ \gamma
}{2\pi} \dashint_{\UR} h \) + o_R(1) ,
\end{equation}
where $\gamma$ is the  constant  in \eqref{defgamma}  and $o_R(1)$ is a function tending to $0$ as $R\to +\infty$.
\end{enumerate}

\end{theo}
The main point is that $g_\ep$ is a modification of $f_\ep = e_\ep - \hal \lep\nu_\ep$ with a small error (measured by \eqref{erreurrrr})  which, contrarily to $f_\ep$, is bounded below  according to item 2.

Now, using the notation of Section~\ref{sec-splitting},
we blow up   at the scale $\el = \sqrt{\he}$, letting $x' =
\el x$, $\pep = \ep/\el$, and
\begin{equation}\label{uprime}
\pue(x') = u_\ep(x),  \quad \pae(x') =
\el(A_\ep(x)-\np \tho(x)).\end{equation}
Then we deduce from Theorem~\ref{thcompagnon} applied in $\{\pOe\}_\ep$ to $\{(\pue,\pae)\}_\ep$

\begin{pro}\label{etalement} Assume that as $\ep \to 0$
$$\frac{\he}{\lep} \to  \lambda\in [\lambda_\om, + \infty], \quad \log\lep\ll \he -\com \lep \quad\text{and}\quad
\he\le\frac{1}{\ep^{\beta}},$$ where $\beta$ is small enough,
 and that $\je(u_\ep,A_\ep) \le C\he\plep$.  We also
assume  that \eqref{gl2ee} holds. 
Then there exists $N\in \{ \nm, \ns\}$ such that there exist measures $\{\pge\}_\ep$ defined on $\om_\ep'$ satisfying the following four properties, for any family
$\{\UR\}_R$ of sets satisfying \eqref{hypsets}, \eqref{hypsetsbis} and any family of
functions $\{\chi_{\UR}\}_R$ satisfying \eqref{defchi}.
\begin{enumerate}
\item $\pge$ is bounded below by a --- not necessarily positive --- constant $C$ independent of $\ep$.
\item Defining $\pfep$ as in \eqref{Fep2}  we have, writing $\poe$ for $\toe'$ and letting $\tom' = \{x\mid d(x,{\poe}^c)\ge2\},$
\begin{equation}\label{intom} \liminf_{\ep \to 0} \frac{\pfep(\pue,\pae) - \pge(\tom')}{ |\poe|} \ge 0    .\end{equation}
and for any  $1\le p<2$,
\begin{equation}
\label{jF}
\int_{\pOe} |j_\ep'|^p \le C_p \( \pfep(\pue, \pae) + |\poe|\),\  \frac1{\el^2} \int_\pOe |\curl\pae - m_\ep\indic_\poe|^2\le C\(\pfep(\pue,\pae)+|\poe|\),\end{equation}
where $j_\ep' :=(i\pue, \nab_{\pae}\pue) $.

\item  If  $\{\pxe\}_\ep$ satisfies $\dist(\pxe,(\poe)^c)\to +\infty$ and
\begin{equation}\label{hypcoer}\text{$\forall R>0$,  $\D\limsup_{\ep\to 0} \pge(\pxe+\UR)<+\infty,$}\end{equation}
then, up to extraction of  a subsequence, the translated
vorticities $\{\pmue(\pxe+\cdot)\}_\ep$ --- where $\pmue := \mu(\pue,\pae)$ --- converge in $W^{-1,p}_{loc}(\mr^2) $ to a measure $\nu$
of the form $2\pi\sum_{p\in\Lambda}\delta_p$, where $\Lambda$ is a
discrete subset of $\mr^2$, and  the translated currents
$\{\pje(\pxe+\cdot)\}_\ep$ converge  in $\Lp$ for any $p<2$ to $j$
such that $\div j = 0$ and $\curl j = \nu - \cl$.

\item If, in addition, we assume that there exists $C>0$ such that for any $R>0$
\begin{equation}\label{hypcoerplus}\limsup_{\ep\to 0} \frac{\pge(\pxe+\UR)}{|\UR|}<C,\end{equation}
then $j \in \admissible$ ($\admissible $ is defined in Definition~\ref{defA}) and
%\begin{equation} \label{jlim}\div j = 0,\quad \curl j= \nu - 1.\end{equation}
\begin{equation} \label{lpunif} \limsup_{R\to +\infty} \dashint_{\UR}|j|^p\,dx<+\infty,\end{equation}
and
\begin{equation}\label{limlim}\liminf_{R\to\infty}\liminf_{\ep\to 0} \frac1{|\UR|}\int \chi_{\UR}(x-\pxe)\,d\pge(x) \ge \limsup_{R\to\infty}\frac{W(j,\chi_{\UR})}{|\UR|} +\frac{ \gamma }{2\pi}\cl.
\end{equation} \end{enumerate}\end{pro}

\begin{proof}[Proof of Proposition~\ref{etalement}]

We apply Theorem~\ref{thcompagnon} to
$\{(\pue,\pae)\}$ on $\om_\ep'$. First, we  check that \eqref{bornenrj} holds. Since $\pae(x') = \el(A_\ep-\np\tho)(x)$,
\begin{multline*}G_{\ep'}^0(\pue,\pae) = \hal \int_{\om} |\nab_{A_\ep - \np \tho}  u_\ep|^2+\ell_\ep^2 (\curl A_\ep - \Delta \tho)^2 + \frac{(1-|u_\ep|^2)^2}{2\ep^2}\\
\le  2\je(u_\ep, A_\ep)+ \io |\nab \tho|^2 + \frac{1}{\he}|\he+\tmuo- \tho|^2 ,
\end{multline*}
using the fact that $-\Delta\tho = \tmuo - \tho$. If  $N\in \{\nm, \ns\}$, then  \eqref{enn0} implies that $\|\he - \tho\|_{H^1}^2\le C\he\plep$, while
$$\int_\om\frac{\tmuo^2}{\he} = \int_\om \he {m_\ep}^2\le C\he.$$
Therefore, if $\he\le \ep^{-\beta}$ and $\je(u_\ep,A_\ep)\le C\he\plep$, then $G_{\ep'}^0(\pue,\pae)\le C\ep^{-2\beta}$ if $\ep$ is small enough. We conclude by noting that if $\he\le \ep^{-\hal}$ then $\ep'\le \ep^{3/4}$ hence if $\beta$ is small enough, Theorem~\ref{thcompagnon} applies.

It gives us a spread out density $(\pge)_0$ of $e_\ep'$
 and a measure $\nu_\ep' = 2\pi \sum_i d_i \delta_{a_i}$
 depending only on $u_\ep'$, hence on $u_\ep$, but {\em not} on our choice of $N$.
Noting that $\frac{1}{2\pi} \nu_\ep'(\om_\ep') $ is an integer, we let
\begin{equation}\label{choice}
\begin{cases} N=\ns & \text{if}\  \frac{1}{2\pi}  \nu_\ep'(\om_\ep')  \ge \ns\\
N=\nm & \text{if} \ \frac{1}{2\pi} \nu_\ep'(\om_\ep')  \le \nm\end{cases}
\end{equation}
and this is the choice of $N$ we make from now on.

%In blown-up coordinates,  \eqref{estjac} becomes $\|\mu(\pue,\pae) - \nu_\ep'\|_{(C^{0,1}(\pOe))^*}$. 
Testing $\mu(\pue,\pae) - \nu_\ep'$ against   $\zep'- \Ce$ which is in $C^{0,1}_0(\pOe)$, we find, in view of \eqref{Fep2}
\begin{multline}\label{g1bis}
\pfep(\pue,\pae) =\hal \int_{\pOe} |\nab_{\pae}
\pue|^2 + \frac 1{{\el}^2} |\curl \pae -  \co \indic_{\omega'_\ep}|^2
+\frac{(1-|\pue|^2)^2}{2(\ep')^2} \\ - \int_{\pOe}\zep'\,d\nu_\ep'  +
\Ce \int_{\pOe} (d\nu_\ep' - \co\indic_{\poe} )+o(1).
\end{multline}
Note that $\co |\poe| = \tmuo(\toe) = 2\pi N$. But, from the choice \eqref{choice}, if $\frac{1}{2\pi}\nu_\ep'(\pOe) \ge \ns$ then we have $N = \ns \ge N_0$ and thus $\Ce\ge 0$ by \eqref{signC}, hence
$$\Ce \int_{\pOe} (d\nu_\ep' - \co \indic_{\omega_\ep'} ) \ge 2\pi \Ce (\ns - N) =0.$$
If on the other hand $\frac{1}{2\pi}\nu_\ep'(\om_\ep')  \le \nm $ then by \eqref{signC} we have  $\Ce\le 0$ hence
$$\Ce \int_{\om'_\ep} (d\nu_\ep' - \co \indic_{\omega_\ep'} )\ge 2\pi \Ce (\nm-N)=0.$$
So in both cases, the last term in \eqref{g1bis} is nonnegative and we are led to
\begin{equation}\label{g1bb}
 \pfep(\pue,\pae) \ge\hal \int_{\om'_\ep} |\nab_{\pae}
\pue|^2 + \frac 1{{\el}^2} |\curl \pae -  \co \indic_{\omega'_\ep}|^2
+\frac{(1-|\pue|^2)^2}{2(\ep')^2} - \int_{\om'_\ep}\zep'\,d\nu_\ep'+o(1) .
\end{equation}
\footnote{At this point we could also choose $N$ to be $\frac{1}{2\pi} \nu_\ep'(\pOe)$ i.e. the total degree of $u_\ep$. Then the term in factor of $  \Ce$ in \eqref{g1bis} is $0$, and we still have \eqref{g1bb}. We may then proceed with an unchanged proof of the lower bound with that $N$. Alternatively, we may analyse further   the positive term $\Ce \int_{\pOe} (d\nu_\ep' - \co \indic_{\omega_\ep'} ) $  that has been discarded. }

To the  spread out density $(\pge)_0$  we add
$$(\pge)_1 = \frac1{\el^2}\(\curl \pae - \co \indic_\poe\)^2 - \hal(\curl\pae)^2$$
and call the result $\pge$.
Using the notation \eqref{je0} we may rewrite \eqref{g1bb} as
\begin{equation}\label{g3}
\pfep(\pue,\pae)\ge \int_{\om_\ep'} e_\ep'
+ (\pge)_1 - \int_{\om'_\ep}\zep'\,d\nu_\ep' +o(1).\end{equation}

We now check that $g_\ep'$ satisfies the required
properties. Since $\co\in (0,1]$ we have that $(\pge)_1$ is bounded below by a universal
constant if,  for instance,  $\el < 1/2$, which is true for small
$\ep$ since $\lim_{\ep\to 0}\el = 0$.

\begin{enumerate}
\item Since $(\pge)_0$ and  $(\pge)_1$ are both  bounded below by a constant independent of $\ep'$, so is $\pge$.

\item Item 2 will be  proven in full below. We prove here the case $\lambda_\om<\lambda<+\infty$, which is technically simpler. Using $\fep'= e_\ep' - \hal \plep \nu_\ep'$, rewrite \eqref{g3} as 
\begin{equation}\label{it1}\pfep(\pue,\pae) \ge \int_{\om_\ep'} \xi\, d f_\ep'  + \int_{\om_\ep'}(1-\xi) e_\ep' + \int_{\om_\ep'}(\pge)_1  +o(1),\end{equation}
where $\xi = 2\zep'/\plep$.  Then, from \eqref{erreurrrr}  applied to $-\xi$ and since 
from \eqref{zepprime}, we have $\xi = 1$ on $\poe$, it follows that 
\begin{equation}\label{it1bis}\int_{\om_\ep'} \xi \,d\fep'\ge \int_{\om_\ep'} \xi \,d(\pge)_0 -  \|\nab\xi\|_\infty\int_{{\widehat{\om_\ep'\sm\poe}}}\,d(|\nu_\ep'|+ (\pge)_0^+).\end{equation}
Since again $\xi = 1$ on $\poe$,  we may write 
\begin{equation}\label{it2} \int \xi\,d(\pge)_0 = (\pge)_0(\tom') + \int_{\om_\ep'\sm\tom'}\xi \,d(\pge)_0^+ - \int_{\om_\ep'\sm\tom}\xi \,d(\pge)_0^-.\end{equation}
Let $A = \widehat{\om_\ep'\sm\poe}\cap \{\xi >1/2\}$. Since, from \eqref{zepprime}, $\|\nab\xi\|_\infty\le C\ell_\ep$, we have $\widehat A\subset \tilde A:=\{\xi >1/4\}\sm\tom'$ if $\ep$ is small enough. Using \eqref{controlenu0} with $\eta=1$ we deduce that 
\begin{equation}\label{it3}|\nu_\ep'|\(\widehat{\om_\ep'\sm\poe}\cap \{\xi >1/2\}\)\le C(\pge)_0^+(\tilde A) + C|\om_\ep'|\le C \int_{\om_\ep'\sm\tom'}\xi\,d(\pge)_0^+ + C\he.\end{equation}
The same bound is trivially true for $(\pge)_0^+\(\widehat{\om_\ep'\sm\poe}\cap \{\xi >1/2\}\)$. 

On the other hand, letting $B = \widehat{\om_\ep'\sm\poe}\cap \{\xi \le 1/2\}$, we have $\xi< 2/3$ on  $\widehat B$ if $\ep$ is small enough, therefore using \eqref{controlenu0} we find that 
\begin{equation}\label{it4}|\nu_\ep'|\(\widehat{\om_\ep'\sm\poe}\cap \{\xi \le 1/2\}\)\le C e_\ep(\{\xi<2/3\})\le C\int (1-\xi) e_\ep', \end{equation}
and the same bound is true for $(\pge)_0^+\(\widehat{\om_\ep'\sm\poe}\cap \{\xi >1/2\}\)$, using \eqref{gmoins0}. 
We have thus obtained that the negative terms in the right-hand side of \eqref{it1bis} can be absorbed (since $\|\nab \xi\|_{\infty}\le C \ell_\ep =o(1)$)  in the positive terms of \eqref{it1} and \eqref{it2}. There remains to absorb the negative term of the right-hand side of  \eqref{it2}.

Moreover for any $n>0$ and letting $A_n = \{\xi >1-1/n\}$ we have if $\ep$ is small enough that $\widehat{A_n^c} \subset A_{2n}^c$ and $C/\lep < 1/{4n}$. Thus, splitting into $A_n$ and $A_n^c$, and   using the fact that $(\pge)_0^-\le C$ and \eqref{gmoins0} we have 
\begin{equation}\label{it5} \int_{\om_\ep'\sm\tom}\xi \,d(\pge)_0^-\le \frac C\lep e_\ep'(A_{2n}^c) + (\pge)_0^-(A_n\sm\tom)\le \hal\int (1-\xi) e_\ep' + C|A_n\sm\tom|.\end{equation}
Using the fact that  $\|\nab\xi\|_\infty = o(1)$, and the fact that $(\pge)_1\ge 0$ outside $\poe$ and is bounded below by $-C$ in $\poe\sm\tom'$ we deduce from \eqref{it1}--\eqref{it5} that for any $n>0$
$$\pfep(\pue,\pae) \ge (\pge)_0(\tom') + (\pge)_1(\tom') - C|A_n\sm\tom|+  o(\lep).$$
Then \eqref{intom} follows by dividing by $|\poe|$ and taking the limit first as $\ep\to 0$ and then as $n\to +\infty$, noting that from \eqref{scaled}, \eqref{Cep}, \eqref{mep}, we have 
$$A_n\sm\tom = \{x\in\poe\mid d(x,{\poe}^c)<2\}\cup\left\{x\in\omega_\ep'\mid\me<\frac{\tho(\ell_\ep x)}\he<\me+\frac\plep{2n}\right\}$$
and thus --- using Lemma~\ref{lemdomaine}, \eqref{item3} --- that $\limsup_{\ep\to 0}\frac{|A_n\sm\tom|}{\lep}$ tends to $0$ as $n\to +\infty$.

\item  For notational simplicity, we assume $\pxe = 0$. Since $(\pge)_1$ is bounded below, $\int_{\UR} \pge$ being
  bounded above independently of $\ep$ for any $R>0$ implies
  that the same is true for  $\int_{\UR} (\pge)_0$ hence
   \eqref{Hypcoer} is satisfied and from item~7
   in Theorem~\ref{thcompagnon}   we deduce the convergence of the
   currents (locally weak $L^p$), vorticities ($W^{-1,p}_{\loc}$)  and fields (weak $L^2_{\loc}$) to $j$, $\nu$ and $h$
   satisfying $\curl j = \nu - h$.
   Since we assume \eqref{gl2ee} we have $\div j(u_\ep, A_\ep)=0$. But a direct computation (see Lemma \ref{lemdiv}) gives $\div j_{1,\ep}= \div j(u_\ep, A_\ep)+ \div ((1-|u_\ep|^2)\np \tho)=0$ in $\toe$ since $\nab \tho=0$ in $\toe$. At the blown-up scale this means  that when $d(0 , (\omega_\ep')^c ) \to +\infty$, we have for any $R>0$ and $\ep$ small enough $\div j_\ep' =  0$ in $B_R$. We  deduce  that   $\div j_\ep' \to 0$ strongly in  $W^{-1,p}_{\loc}(\mr^2)$ and thus $\div j=0$ .
   Since $\curl j_\ep'= \mu_\ep'+ \curl A_\ep'$ we also have that $\curl j_\ep'$ is compact in $W^{-1,p}_{\loc}(\mr^2)$ for $p<2$. It follows that $j_\ep'$ is compact in $L^p_{\loc}$ and the convergence of $j_\ep'$ is strong.

  Moreover we have
\begin{equation}\label{gprime1}\limsup_{\ep\to 0} \dashint_{\UR} \(\frac1{\el^2}\(\curl \pae - \co\indic_\poe\)^2 - \hal(\curl\pae)^2\)<+\infty\end{equation}
for every $R>0$, from which we easily deduce that the limit $h$ of
$\{\curl \pae\}_\ep$ is equal to $\cl$, where
$\lambda\in[\com,+\infty]$ is the limit of $\he/\lep$. Indeed under
the hypothesis $d(0,(\poe)^c)\to +\infty$, we have
$\co\indic_\poe(\pxe+\cdot)\to \cl$ locally from \eqref{limitem}. We thus have $\curl j = \nu - m_\lambda$.

\item Again we assume $\pxe = 0$. The hypothesis that $\limsup_{\ep\to 0}
\dashint_{\UR} \pge$ is bounded above independently of $R$ implies as above that
 the same is true for $\limsup_{\ep\to 0}\dashint_{\UR} (\pge)_0$. Then item~8 of Theorem~\ref{thcompagnon} applies and we obtain \eqref{Lpunif}, hence \eqref{lpunif}, and from \eqref{Limlim} we get
 \begin{equation}\label{lwx} \liminf_{\ep\to 0} \int \frac{\chi_\UR}{|\UR|}\,d(\pge)_0\ge \frac{W(j,\chi_\UR)}{|\UR|} + \hal\dashint_\UR {m_\lambda}^2 + \frac\gamma{2\pi}\dashint_\UR m_\lambda +o_R(1),\end{equation}
 where $\lim_{R\to+\infty} o_R(1) = 0$. On the other hand \eqref{gprime1} implies that $\curl\pae\to m_\lambda$ locally strongly in $L^2$, hence
 $$\liminf_{\ep\to 0} \int \frac{\chi_\UR}{|\UR|}\,d(\pge)_1\ge
  -\hal {m_\lambda}^2.$$
 Together with \eqref{lwx}, this proves \eqref{limlim}.

 To prove that $j\in\admissible$ we integrate $\curl j = \nu - \cl$ over $B_{R+t}$ and $B_{R-t}$ to
obtain
$$\pi\cl(R-t)^2 + \int_{\p B_{R-t}} j\cdot \tau = \nu(B_{R-t})\le \nu(B_{R})\le \nu(B_{R+t}) = \pi\cl(R+t)^2 + \int_{\p B_{R+t}} j\cdot \tau.$$
Then, a mean-value argument and  \eqref{lpunif}, with  $p=1$,
allow  to deduce the existence of $t\in[0,\sqrt R]$ such that
$$\int_{\p B_{R-t}\cup \p B_{R-t}} |j|\le C R^{3/2},$$
and  we deduce that $\nu(B_R)\sim \pi\cl R^2$ as $R\to +\infty$, and so
$j\in\admissible$.
\end{enumerate}

\end{proof}
It remains to prove item 2 in all generality using Theorem \ref{thcompagnon}.

%\begin{pro}\label{pourlautre} Let $N$, $\pzep$ be as above, and  $\poe$ be as above,
%then $\pge' $ (defined as in the proof of Proposition \ref{etalement})  satisfies
%$$ \pfep(u',A')\ge \int_{\tilde{\omega_\ep'} } g_\ep'+o(|\omega_\ep'|). $$
%and
%$$\pfep(u',A') \ge  \frac{1}{C_0} \io  (g_\ep')_+ - C |\omega_\ep'|,$$%
%where  $\tilde\poe$ is the set of
%$x$'s which are at distance at least $2$ from $(\poe)^c$, and $C_0>0$ depends only on the constant in \eqref{controlenu0}.
%\end{pro}

\begin{proof}[Proof of item 2 in Proposition \ref{etalement} in the general case.] Recall that
from  its definition (\ref{scaled}), $\pzep$  achieves its maximum $\hal \plep$ on $\poe$, and is equal to $\Ce$ on $\partial\pOe$. Thus  $$\xi = 2\frac\pzep\plep$$ achieves its maximum $1$ on $\poe$ and
 its minimum  $\Ce/\plep$, which from \eqref{limitem} is $o(1)$,  on $\partial\pOe$.

We let
$$ E_1= \{ x\in \pOe, \xi >1-\delta\}, \quad  E_2= \{x\in \pOe, \xi>1-2 \delta\},$$ and recall that
$$\tom'= \{ x\in \poe, \dist (x, (\poe)^c) \ge 2\}.$$

We will need the following lemma.
\begin{lem}\label{lemco} For any $M>0$ and $\ep$ small enough there exist $\delta>0$ such that
\begin{equation}\label{deltacond1}
\delta>\frac M\plep,\quad \widehat E_1\subset  E_2, \quad |\tom'|>M|\widehat E_2\sm\tom'|.
\end{equation}
\end{lem}

\begin{proof} First we treat the case where $\he/\lep\to \lambda\in (\lambda_\om,+\infty]$.
In this case, using \eqref{item3} in  Proposition~\ref{lemdomaine} we find
$$\lim_{\delta\to 0} \frac{\left|\left\{1 - {2\delta\he}/\plep < \xi<1\right\}\right|}{|\poe|} = 0, $$
uniformly with respect to $\ep\le \ep_0$, if $\ep_0$ is small enough. Since $\he/\plep$ has the same limit as $\he/\lep$, it is bounded away from $0$ as $\ep\to 0$. We deduce easily that $\delta$ may be chosen small enough so that
\begin{equation} \label{fll} \frac{\left|\left\{1 - {3\delta\he}/\plep < \xi<1\right\}\right|}{|\poe|} \le \frac 2M,
\end{equation}
for any $\ep\le \ep_0$.

Then we note, since from \eqref{scaled} we have $|\nab\xi|\le \plep^{-\hal}$, that $\widehat E_1\subset E_2$ holds for $\ep$ small enough, as well as $\widehat E_2\subset \{1-3\delta \he/\plep < \xi\}$. It follows, in view of \eqref{fll} and since $\poe = \{\xi = 1\}$, that for $\ep$ small enough
\begin{equation} \label{fllf} \frac{|\widehat E_2\sm\poe|}{|\poe|} \le \frac 2M.
\end{equation}
To conclude we note that $|\poe\sm\tom'| \le |\{d(x,\partial\poe)\le 2\el\}| = o(|\poe|)$, where we have used \eqref{item3} and scaled. Thus if $\ep$ is small enough and using \eqref{fllf} we find $|\tom'|>M|\widehat E_2\sm\tom'|.$

If $\he/\lep\to +\infty$ then we choose $\delta = 1/2$. As above, $\|\nab\xi\|_\infty\to 0$ implies that $\widehat E_1\subset E_2$ for $\ep$ small enough, and of course $\delta>M/\plep$ is satisfied for $\ep$ small enough depending on $M$. Moreover \eqref{volome} implies that $d(\omega_\ep, \om^c)\to 0$ as $\ep\to 0$, and $d(\tom, (\omega_\ep)^c)=2\el\to 0$ as well. Therefore  $|\om\sm\tom| = o(|\tom|)$ and after  scaling  $|\pOe\sm\tom'| = o(|\tom'|)$, and in particular $|\widehat E_2\sm\tom'| = o(|\tom'|)$.

If $\he/\lep\to \lambda_\om$, we choose $\delta = cL_\ep^2$ with $c>0$ independent of $\ep$ to be chosen small enough depending on $M$. From \eqref{enst2} applied with $M = \frac c{2\lambda_\om}$ and $\delta = \frac\eta{2\lambda_\om}$, and rewritten in terms of
$$ \xi(x) = 1 - 2\frac{\tho(\el x) - \min_\om \tho}\plep$$
we deduce that
$$1 - \xi(x)\le c L_\ep^2 \frac\he{\lambda_\om\plep}\implies U_Q\(\frac{\el x - x_0}{L_\ep}\) \le \frac{c+\eta}{2\lambda_\om}.$$
Therefore, since $\delta = cL_\ep^2$ and $\lambda_\om = \lim_\ep \he/\plep$, for any $\eta>0$ and if $\ep$ is small enough then
$$1 - \xi(x)\le 2\delta\implies  U_Q\(\frac{\el x - x_0}{L_\ep}\) \le \frac{c+\eta}{2\lambda_\om}.$$
Since $\el/L_\ep\to 0$ from Proposition~\ref{lemdomaine} this in turn implies that if $\ep$ is small enough and $|y|<1$ then
$$U_Q\(\frac{\el x- x_0}{L_\ep}    +\frac{\el}{L_\ep} y\)<\frac{c+2\eta}{2\lambda_\om}.$$
Using \eqref{enst2} again, this implies that  $\xi(x+y)>1 -  2 (c+3\eta)L_\ep^2$. Choosing $\eta = c/6$ we have
proven that $\widehat{E_2}\subset \{\xi>1-3\delta\}$ if $\ep$ is small
enough. Similarly, we may prove that $\widehat{E_2^c}\subset
\{\xi<1-\delta\}=E_1^c$. We further deduce that, as $\ep\to 0$,
\begin{equation}\label{et1}
|\widehat{E_2}| < \(\frac{L_\ep}{\el}\)^2 \left|\left\{U_Q < \frac{c}{2\lambda_\om}\right\}\right|
 + o\(\frac\ella\el\)^2.\end{equation}
Moreover from  item 4 in Proposition \ref{lemdomaine} we have
$|\tilde\omega| \sim\( \frac\ella\el\)^2.$ With \eqref{et1}, and since $|\{U_Q <
\frac{2c}{\lambda_\om}\}|\to|E_Q|=1$ as $c \to 0$, and
$\tilde\omega\subset \widehat{E_2}$, we deduce that $|\widehat{E_2}\sm\tilde\omega|/|\tilde\omega|$ can be made arbitrarily small by
choosing $c$, and then $\ep$, small enough. This proves
\eqref{deltacond1} and the lemma.
\end{proof}

Returning to our proof, we start from  \eqref{g3}, and we
bound from below
$$\int_{\om_\ep'} e_\ep'  - \zep'\,d\nu_\ep'.$$
 Let $C_0$ be the constant in \eqref{controlenu0}
i.e. such that for any set $E\subset \om_\ep'$
\begin{equation}\label{CC}
|\nu_\ep|(E) \le \frac{C_0}{ \lep} e_\ep(\widehat{E}) \end{equation}
and assume that $C_0>2$.

For notational simplicity we now write $\om$, $e$, $\nu$, $g$, $\om$, $\tom$, $\zeta$ instead of $\pOe$, $e_\ep'$, $\nu_\ep'$, $(\pge)_0$, $\poe$, $\tom'$, $\pzep$ and we let $f = e - \hal\plep\nu$. Since $\xi  = o(1)$ on $\bo$ and $\|\nab\xi\|_\infty = o(1)$ (recall \eqref{limitem}, \eqref{zepprime}), for any  $\eta>0$ and $\ep$ small enough there exists a smooth positive cut-off
function $\chi\le 1$ such that  $|\nab \chi |\le \eta$, such that $\chi = 0$ on $\{\xi(x) \le \frac{1}{8C_0}\}$  and such that  $\chi = 1$ on  $\{\xi(x)\ge \frac{1}{4C_0}\}$.
Moreover, $\{x\mid d(x,\bo)\le 2\}\subset \{\xi\le \frac{1}{8C_0}\}$.

We then note  that since $f=e- \hal \plep \nu $ and $\zeta=\hal \plep \xi$, we  may write
\begin{equation}\label{eqn1}e - \zeta \nu =
(1-\chi) \(e- \hal \xi \plep  \nu\) + \chi \( \xi(f-g)+
\xi g + (1-\xi)e\). \end{equation}
{\it Step 1:} We first study
\begin{multline}\label{418}
\int_\om  \chi \(\xi \,d(f-g) + \xi\,dg+ (1-\xi)e\)
\\ = \int_{\tilde\omega}\chi\xi  \,dg+
\(\int_{(\tilde\omega)^c}\chi\xi \,dg_+ +\int_\om \chi(1-\xi)e\)
 +\(\int_\om
\chi\xi \,d(f-g)- \int_{(\tilde\omega)^c}\chi\xi \,dg_-\).\end{multline}
 The first
parenthesis on the right-hand side contains positive terms, while
the second one contains negative terms. We use
\eqref{gmoins0}, \eqref{controlenu0} to bound from below the
negative terms, and use the positive terms to balance them. From
(\ref{418}) and (\ref{erreurrrr}) (applied with $\eta=1$), and since $\chi= \xi = 1 $ on $\tilde\omega$ by construction and $\chi=0$ on $\{ \dist(x, \bo) \ge 2\}$,
 we have
\begin{equation}\label{inti1i2}
\int_\om  \chi (\xi \,d(f-g) + \xi \,dg+ (1-\xi)e)
\ge\int_{\tilde\omega}  \,dg + \frac{3}{4} \io \chi (1-\xi) e +
 \hal \int_{(\tilde\omega)^c}\chi\xi \,dg_+  + I_1 + I_2 , \end{equation}
where
\begin{eqnarray*}
 I_1 &=& -C \int_\om
  \widehat{\nab(\chi\xi)}\,dg_+
  +\frac14\int_{(\tilde\omega)^c}\chi \xi  \,dg_+ +\frac18\int_\om\chi (1-\xi)e,\\
I_2 &=& -C \int_\om  \widehat{\nab(\chi\xi)}\,d|\nu|-
\int_{(\tilde\omega)^c } \chi \xi \,dg_- +\frac14\int_{(\tilde\omega)^c}\chi \xi
\,dg_+ +\frac18\int_\om \chi(1-\xi)e.
\end{eqnarray*}
We first study $I_1$. Since $\nab (\chi\xi ) = 0$ in $\omega$,
we have
$\{\widehat{\nab(\chi\xi)}\neq 0\}\subset
(\tilde\omega)^c$. In addition $|\nab (\chi \xi)|\le \eta$ if $\ep$ is small enough since $|\chi|\le 1, |\xi|\le 1 , |\nab \chi |\le \eta $ and $|\nab \xi|=o(1)$.
Thus, noting that $\chi=1 $ on $\{\xi \ge \hal\}$,
\begin{multline*}
2  \int_\om  \widehat{\nab(\chi\xi)}\,dg_+\le
\eta \int_{(\tilde\omega)^c \cap \{ \xi\ge \hal\} }\,dg_++
 \eta \int_{(\tilde\omega)^c \cap \{\xi \le \hal\}} \,dg_+
\\ \le \eta\int_{(\tilde\omega)^c \cap  \{ \xi\ge \hal\} }
 \xi \chi \,dg_+
+C \eta\io (1-\xi)  e\end{multline*} where we have used
\eqref{gmoins0}. We thus obtain, choosing $\eta$ small enough,
$$ I_1\ge - \frac{1}{8}\io (1-\chi)(1-\xi) e.$$

We turn to $I_2$ and split the negative contributions over $E_2$ and $E_2^c$.
Using Lemma  \ref{lemco},  choosing $M$ large enough and using  \eqref{gmoins0},
\eqref{controlenu0} and the above, we have
$$\int_{E_2^c} \(\widehat{\nab(\chi\xi)}\,d|\nu|+\chi \xi \,dg_-  \)\le
\frac{C e(\widehat{ E_2^c})}\plep  \le \frac18 \int_{\widehat{E_2^c}} (1-\xi) e,$$
since  $\widehat{E_2^c}\subset E_1^c$ and $\delta >M/\plep$.

On the other hand, using \eqref{controlenu0} with $\eta = 1$
 and the fact that $\nab (\xi\chi) = 0 $ in $\tilde{\omega}$ and $|\nab (\chi\xi)|\le \eta+o(1)$ we have, since $\widehat{E_2\sm\omega}\subset \widehat E_2\sm\tilde\omega$ and by choosing $\eta$ small enough,
 $$\int_{E_2}
\widehat{\nab(\chi\xi)}\,d|\nu|\le C \eta   \int_{\widehat{E_2}   \backslash
\tilde{\omega}} (dg_++1    )\le \frac14  \int_{\tilde{\omega}^c} \chi    \xi \,dg_+ +C\eta
|\widehat{E_2} \sm  \tilde{\omega}| ,
$$
while using the fact that $g$ is bounded below (hence $g_-\le
C$), we have
$$\int_{\tilde{\omega}^c \cap E_2}
\chi \xi \,dg_-
\le C |E_2\sm \tilde{\omega}|.$$ Combining all the above, we deduce
that choosing $\eta $ small enough we have, if $\ep$ is small, that
$$I_2\ge- \frac18 \io (1-\chi)(1-\xi)  e  - C| \widehat E_2\sm\tilde\omega|.$$  From Lemma
\ref{lemco}, it follows that $\forall M>0$ and $\ep$ small enough, $ I_2\ge - \frac{C}{M}|\tilde\omega|- \frac18 \io (1-\chi)(1-\xi)  e    .$

Inserting the bounds for  $I_1 $ and $I_2$  into \eqref{inti1i2} we find
\begin{multline}\label{i1i2}\int_\om  \chi \(\xi \,d(f-g) + \xi \,dg+ (1-\xi)e\)\ge
   g(\tilde\omega) + \frac34 \io \chi (1-\xi) e +
 \hal \int_{(\tilde\omega)^c}\chi\xi \,dg_+
 \\  - \frac{1}{4}\io (1-\chi)(1-\xi) e - \frac{C}{M}|\tilde\omega|.\end{multline}

{\it Step 2:} We examine
$\io (1-\chi)(e- \hal \xi \plep \,d\nu).$
Since $|\xi |\le \frac{1}{4C_0}$ on the support of $1-\chi$ we have there
$$e-  \hal \xi \plep \nu \ge  e - \frac{1}{8 C_0} \plep |\nu|.$$
On the other hand, from item 2 of Theorem \ref{thcompagnon} we have
$g_+\le e+ \hal \plep |\nu|$ hence on the support of $1-\chi$ we have
$$e-  \hal \xi \plep \nu \ge (1- \frac{1}{4C_0}) e+ \frac{g_+}{4C_0}
 - \frac{2}{8C_0} \plep |\nu |$$
and thus  $$\io (1-\chi) (e-\hal \xi \plep \,d\nu) \ge
(1-\frac{1}{4C_0} ) \io (1-\chi) e + \frac{1}{4C_0}\(\io (1-\chi) \(\,dg_+ -\plep \,d|\nu|\)\).$$

{\it Step 3:} Inserting the above and \eqref{i1i2} in \eqref{eqn1} we obtain
   \begin{multline}\label{e4}\io e - \zeta \,d\nu  \ge
  g(\tilde\omega) + \frac34 \io \chi (1-\xi) e+
 \hal \int_{(\tilde\omega)^c}\chi\xi \,dg_+  \\  - \frac{1}{4}\io (1-\chi)(1-\xi) e - \frac{C}{M}|\tilde\omega| +
 (1-\frac{1}{4C_0} ) \io (1-\chi) e \\ + \frac{1}{4C_0}\io (1-\chi) \,dg_+
 - \frac{\plep}{4C_0}\io
  ( 1-\chi) \,d|\nu|.\end{multline}
 Since $0 \le 1-\xi\le 1+\Ce/\plep=1+o(1)$,  and $C_0>2$   we have
  \begin{multline*}- \frac{1}{4}\io (1-\chi)(1-\xi) e + \frac34 \io \chi (1-\xi) e+
 (1-\frac{1}{4C_0} ) \io (1-\chi) e \\
 \ge \( (1+o(1))(1- \frac{1}{4C_0})- \frac14\)\io   (1-\chi)(1-\xi) e   + \frac34 \io \chi (1-\xi) e \ge\hal  \io (1-\xi) e
.\end{multline*}  On the other hand,
by    \eqref{CC} and choice of $\chi$ we have
 $$\frac{\plep}{4C_0}\io (1-\chi)\,d|\nu|\le \frac14 \int_{\widehat{\{\xi<\frac{1}{4C_0}\}}}e.$$
 Since $|\nab \xi|=o(1)$ we know that $  \widehat{\{\xi
 <\frac{1}{4C_0}\}} \subset \{\xi<\frac14\}$ for $\ep$ small enough and thus
 $$\frac{\plep}{4C_0}\int_\om (1-\chi)\,d|\nu|\le  \frac13 \io (1-\xi) e.$$
 We have obtained
 \begin{multline}\label{eeee}- \frac{1}{4}\io (1-\chi)(1-\xi)  e + \frac34 \io \chi (1-\xi) e+
 (1-\frac{1}{4C_0} ) \io (1-\chi) e - \frac{\plep}{4C_0}\io (1-\chi) \,d|\nu|
 \\
 \ge \frac16 \io (1-\xi)e\ge 0 .\end{multline}
For the terms in  \eqref{e4} involving $g$, $g_+$, we use the fact
  that $\xi> \frac{1}{8C_0} $ on the support of $\chi$:
 $$  g(\tilde\omega)   +\frac{1}{4C_0}\io (1-\chi) \,dg_+ +
 \hal \int_{(\tilde\omega)^c}\chi\xi \,dg_+\ge  \frac{1}{16C_0} g_+(\om)- g_-(\tilde{\omega}) .$$
We may also bound this term below by $g(\tilde{\omega})$.
 Combining with \eqref{eeee} and \eqref{e4} we are led either to
\begin{equation}\label{ed1}\io e-\zeta \,d\nu \ge g(\tilde{\omega}) - \frac{C}{M}|\tilde{\omega}|,\end{equation}
where $M$ can be arbitrarily large,  or, using $g_- \le C $  and keeping $\frac16\int(1-\xi)e$,  to
\begin{equation}\label{ed2}\io e- \zeta \,d\nu \ge  \frac{1}{16C_0}
    g_+(\om)  - C|\tilde{\omega}| + \frac16\int_\om(1-\xi)e .\end{equation}

 {\it Step 4: Conclusion.}  We return to \eqref{g3} and  recall that the letter $g$ stood for  $(\pge)_0$. We deduce from \eqref{ed1}
 $$ \pfep(\pue,\pae)\ge (\pge)_0(\tom')  + (\pge)_1(\pOe) - \frac{C}{M}|\tom'|+o(1).$$
Since $(\pge)_1\ge 0 $ on $\pOe\sm\poe$, since $(\pge)_1\ge -C$
on $\poe$ and since  $|\poe\sm \tom' |=o(|\poe|)$,
we have $(\pge)_1(\pOe)\ge (\pge)_1(\tom') + o(\poe).$ Then we deduce  from \eqref{ed1}, letting $\ep \to 0$ and then $M\to  + \infty$, that
$$\pfep(\pue, \pae)\ge \pge(\tom')+ o(|\poe|), $$
proving  \eqref{intom}.

To prove \eqref{jF} we combine \eqref{ed2} with item~6 in Theorem~\ref{thcompagnon} to obtain
$$\pfep(\pue, \pae)\ge \frac1{C_p} \int_{\{d(x,(\pOe)^c)\ge C\}} |j_\ep'|^p - C|\tom'| + \frac16\int_\pOe (1-\xi) e_\ep'.$$
This proves the first inequality in \eqref{jF}, noting that
$(1-\xi)\ge 1/2$ on $\{d(x,(\pOe)^c)\ge C\}$ if $\ep$ is
small enough and that $e_\ep'\ge \hal|j_\ep'|^2$. Finally, from \eqref{ed1},
\begin{equation}\label{fggo}\pfep(\pue, \pae)\ge (\pge)_0(\tom') + (\pge)_1(\pOe) + o(|\poe|),\end{equation}
and it is straightforward to check that
$$ (\pge)_1\ge\begin{cases} \frac1{2\el^2} (\curl\pae - m_\ep)^2 & \text{on $\poe$}\\
\frac1{2\el^2} (\curl\pae)^2 & \text{on $\pOe\sm\poe$.}\end{cases}$$
Therefore
$$(\pge)_1(\pOe)\ge \int_\pOe \frac{(\curl\pae - m_\ep)^2}{2\el^2} - C|\poe|,$$
which together with \eqref{fggo} finishes the proof of  \eqref{jF}.
\end{proof}

\subsection{The case of small applied field}
Here, by small applied field we mean fields $\he$ that satisfy the assumptions of Proposition \ref{etalement}.
We are  ready to  define the appropriate
 space and $\Gamma$-converging functions to apply
  the scheme described in Section \ref{sec:abstract}.
  We  choose $X = \Lp\times \radon_0$,   where $\radon_0$ denotes the set of measures $\mu$ such that $\mu + C$ is a positive locally bounded measure
  on $\mr^2$, $-C$ being the constant bounding from below $g_\ep$ in  item 1 of Proposition \ref{etalement}.    We consider on $\Lp$ the strong topology and on $\radon_0$ that of weak convergence. The  space of positive measures equipped with the weak convergence is metrizable hence the  space $X$ is a Polish space.
  We will typically denote by $\bix$ an element of $X$.

There is a natural  action $\theta$ of $\mr^2$ on $X$ by translations:
$$\theta_\lambda \(j(\cdot),g(\cdot)\) = \(j(\lambda+\cdot),g(\lambda+\cdot)\)$$ which is clearly continuous with respect to the couple $(j,g)$ as well as with respect to $\lambda$.

Denoting as above $\poe$ the rescaled coincidence set, we know
from Lemma~\ref{lemdomaine} that it satisfies \eqref{domaine}. Now consider  $\{(u_\ep,A_\ep)\}_\ep$ satisfying the hypothesis of Proposition~\ref{etalement},  then the proposition provides us with measures $\pge$ defined on $\pOe$. We   also have the rescaled current  $\pje= \curl (iu_\ep', \nab_{A_\ep'}u_\ep') $  where $(u_\ep',A_\ep')$ are as in Proposition \ref{etalement}. We define functions $\{\gfep\}_\ep$ on $X$ as follows.

Having chosen a smooth  positive function $\chi:\mr^2\to\mr$
with  support in the unit ball  and integral equal to $1$, we let
$$\gfep(\bix) = \begin{cases}\D \int\chi(y-x)\,d\pge(y)
& \text{if $\exists x\in\poe$ s.t. $\bix = \(\pje(x+\cdot),\pge(x+\cdot)\)$}\\ +\infty & \text{otherwise.}\end{cases}
$$
Note that since $\pje$   and $\pge$ vanish outside a compact set, there is at most one $x$ such that $\bix = \(\pje(x+\cdot),\pge(x+\cdot)\)$ unless $\bix = 0$, in which case we let $\gfep(\bix) = +\infty$.

The third statement  of Proposition~\ref{etalement} implies that $\{\gfep\}_\ep$ satisfies the requirement of coercivity \eqref{2.1bis}. Indeed assume that  $\{\bix_\ep\}_\ep$ is a sequence in $X$ such that
\begin{equation}\label{hypcoercive}\limsup_{\ep\to 0}\D \int_{B_R} \gfep(\theta_\lambda \bix_\ep) \, d\lambda<+\infty\end{equation}
for every $R>0$, then in particular this integral is finite if
$\ep$ is small enough, which implies that
$\gfep(\theta_\lambda \bix_\ep)<+\infty$ for almost every
$\lambda\in B_R$. Thus there exists $\{x_\ep\}_\ep$ such that
$\bix_\ep = (\pje(x_\ep+\cdot),\pge(x_\ep+\cdot))$ and $\lambda+x_\ep\in\poe$ for almost every $\lambda\in B_R$, when $\ep$ is small enough. In particular the distance of $x_\ep$ to $\mr^2\setminus\poe$ is larger than $R$ for $\ep $ small enough,  and \eqref{hypcoercive} reads
$$\forall R>0, \quad \limsup_{\ep\to 0}\int_{B_R}\int  \chi(y-\lambda-x_\ep)\,d\pge(y)\, d\lambda<+\infty.$$
Integrating first w.r.t. $\lambda$  we find
$$\limsup_{\ep\to 0}\int \chi_R(y-x_\ep)\,d\pge(y) <+\infty,$$
where  $\chi_R = \chi* \indic_{ B_R}$.

Since $\pge$ is bounded from below independently of $\ep$, and
 since $\chi_R$ is a positive function equal to $1$ on $B_{R-1}$,
 this implies that the second part of \eqref{hypcoer} is satisfied.
   Thus, up to  a subsequence,  the currents $\pje(x_\ep+\cdot)$
   converge in $\Lp$.  Going to a further subsequence,
    $\{\pge(x_\ep+\cdot)\}_\ep$ converges in $\radon$.

Finally, $\chi$ was chosen smooth, it is clear that  the $\Gamma$-liminf requirement \eqref{2.1ter} is satisfied if we define the   function $\gf$  on $X$ by
\begin{equation}\label{defgf}
\gf(j,g) = \int\chi\,dg.\end{equation}
Now Theorem~\ref{gamma} applies and combined with Proposition~\ref{etalement}  gives:

\begin{pro}\label{pro4} Assume that $\{(u_\ep,A_\ep)\}_\ep$ satisfy \eqref{gl2ee}
and $$G_\ep(u_\ep, A_\ep) \le \min_{N \in \{\nm, \ns\}} \je^N + C N_0$$ for some constant $C$,
and assume
$\he$ satisfy the hypothesis of Proposition~\ref{etalement}
and \eqref{bornG}.  Let $(u_\ep', A_\ep')$ and $N$
 be as in Proposition \ref{etalement}. Using the notation above,
 and letting $P_\ep$ be the probability measure on
 $\Lp$ which is the push-forward of the
 normalized uniform measure on $\poe$ by the map
 $x\mapsto \pje(x+\cdot)$, we have the following.
\begin{enumerate}
\item    A subsequence of  $\{P_\ep\}_\ep$ weakly converges to a
translation-invariant probability measure $P$ on $\Lp$ such that $P$-a.e.
$j \in\admissible$.
\item  For any family $\{\UR\}_{R>0}$ satisfying \eqref{hypsets}, \eqref{hypsetsbis}, we have
\begin{equation}\label{gli}\liminf_{\ep\to 0}
\frac{\pfep(\pue,\pae)}{|\poe|}\ge \int W_U(j) \,dP(j) +\cl \frac{ \gamma }{2\pi}\end{equation} and
\begin{equation}\label{gli2}
\je(u_\ep, A_\ep) \ge
\je^N+  N    \(  \frac{2\pi}{m_\lambda}  \int W_U(j) \, dP(j)
 + \gamma +o(1)  \). \end{equation}
  \end{enumerate}
\end{pro}
\begin{remark}\label{autrej} We claim that under the hypothesis of this proposition, the limit of $P_\ep$ which is the image of the normalized Lebesgue measure on $\poe$ by the map
$$\vp: x\mapsto j_\ep' (\cdot-x)$$
is unchanged if,  in the definition of $\vp$,  we
replace $\pae(y) = \el(A_\ep(x+\el y) - \np\tho(x+\el y))$ by
 $\el A_\ep(x+\el y)$, i.e. define $P_\ep$ as in Theorem~\ref{th1prime}.

Indeed, denote by $\psi$ the corresponding modification of
 $\vp$ and by $Q_\ep$ the ensuing modification of $P_\ep$.
  From \eqref{domaine},  there exists $\tom'$ such that
  $\tom'+ B_R\subset \poe$ for any $R>0$ if  $\ep$ is small enough,
   and such that $|\tom'|\sim|\poe|$. Since $\tom'\subset\poe$
   and $|\tom'|\sim|\poe|$, replacing $\poe$ by $\tom'$ in the
    definition of either $P_\ep$ or $Q_\ep$ does not change their
     limit. Then, since $\np\tho(x+\el y) = 0$ for any $x\in \tom'$ and
      $y\in B_R$
      and if $\ep$ is small enough, we find that
       $P_\ep$ and $Q_\ep$,  seen as measures on
       $L^p(B_R)$, coincide if $\ep$ is small enough
       depending on $R$. It follows that their limits in
       $\Lp$ are equal, proving the claim.
\end{remark}
%\begin{coro}\label{borneinffinale}
%Assume $\he$  satisfies $\log \he \ll \lep$, then for any $(u_\ep, A_\ep)$ we have
%begin{multline}
% G_\ep(u_\ep, A_\ep) \ge \hal \|\ho- \he\|_{H^1(\om)}^2 + \hal\L \io \muo \\
 %\ge
% \hal \|\ho- \he\|_{H^1(\om)}^2 + \hal\L \io \muo \\
% + (1- 1/2\lambda) \he |\omega_\ep|\( \min_{\admissible} W+ \frac{\gamma}{2\pi} - \frac{1}{4} \log (1- 1/2\lambda) \) + o(\he).\end{multline}
% \end{coro}

\begin{proof}[Proof of Proposition~\ref{pro4}] It is immediate that
 $(u_\ep, A_\ep)$ satisfies the hypotheses of Proposition \ref{etalement}.
Let $N$ be given by Proposition \ref{etalement}.  By Lemma \ref{lemapriori}, we have, for that $N$,
$G_\ep(u_\ep, A_\ep)=\je^N +F_\ep(u_\ep, \ai) + o(N).$ But from the upper bound assumption, $G_\ep (u_\ep, A_\ep) \le \je^N  +C N_0 $ so we deduce
 $$F_\ep (u_\ep, \ai) \le C N_0 \simeq CN =O( \he |\toe|).$$  In blown-up coordinates, this means $\pfep(\pue, \pae) \le C |\tom'|.$

 On the other hand, from  \eqref{intom} in Proposition~\ref{etalement}, we have
$$\pfep(\pue,\pae)\ge \pge(\tom')  - o(|\tom'|),$$
where we recall $\tom' = \{d(x,(\poe)^c)\ge 2\}$.
Also,   since $\int \chi=1$, since $\chi$ has support in $B_1$
and since $\pge$ is bounded below,
$$ \pge(\poe) \ge \int_{\boe}\gfep(\theta_\lambda \pje,\theta_\lambda \pge)\,d\lambda - o(|\poe|),$$
where $\boe= \{d(x,(\poe)^c)\ge 3\}$. Combining these facts, we have
$$F_\ep(\pje,\pge) := \dashint_{\boe}\gfep(\theta_\lambda \pje,\theta_\lambda \pge)\,d\lambda \le C $$ and we now   apply
Theorem~\ref{gamma} to these functionals.  We deduce  that the measures $\{Q_\ep\}_\ep$, where $Q_\ep$ is the image under
$$x\mapsto(\pje(x+\cdot),\pge(x+\cdot))$$
of the uniform normalized  probability measure on $\boe$,
 converge to a translation invariant probability measure $Q$
  on $X$ and
$$\liminf_{\ep\to 0} \frac{\pfep(\pue,\pae)}{|\boe|}\ge \int_X
\( \lim_{R\to +\infty} \dashint_{\UR} \gf(\theta_\lambda \bix)\)
\,dQ(\bix).$$
Moreover,  since from \eqref{domaine} we have $|\poe|\sim|\boe|$, we may replace
$\boe$ with $\poe$ in the definition of $Q_\ep$ and obtain the same limit.
But, writing as above $\chi_{\UR} = \chi *\indic_{\UR}$, we have as
 above from \eqref{defgf}
$$\lim_{R\to +\infty} \dashint_{\UR} \gf(\theta_\lambda\bix)\,
d\lambda= \lim_{R\to +\infty} \frac 1{|\UR|}\int\chi_{\UR}(x)\,dg(x),$$
where $\bix = (j,g)$. Now, if $\gf (\bix)$ is finite and
 $\bix$ is in the support of $Q$, then there exists (see Remark
 \ref{remF})  a sequence $\{x_\ep\}_\ep$ such that
 $(\pje(x_\ep+\cdot),\pge(x_\ep+\cdot))$ converges to $\bix$
 in $X$, with \eqref{hypcoerplus}  satisfied with $\UR$.
  It follows from Proposition \ref{etalement} that $j
  \in\admissible$  and that \eqref{limlim} is  satisfied, thus
\begin{multline*}
\lim_{R\to +\infty} \frac 1{|\UR|}\int\chi_\UR(x)\,dg(x) = \lim_{R\to +\infty} \lim_{\ep\to 0} \frac 1{|\UR|}\int\chi_\UR(x-x_\ep) \,d\pge(x)\\
\ge \limsup_{R \to \infty} \frac{W(j, \chi_{\UR})}{|\UR|}+ \cl\frac{\gamma}{2\pi} = W_U(j) + \cl\frac{\gamma}{2\pi}.\end{multline*}
Letting $P_\ep(j)$ and $P(j)$ denote the marginals of
the measures $Q_\ep$ and $Q$ with respect to the  first  variable,
 we immediately deduce that $P_\ep\to P$ and that \eqref{gli} is satisfied.
Replacing \eqref{gli} in \eqref{sanserreur} we find (\ref{gli2}), since $|\poe| = \he|\toe|= \frac{2\pi N}{ \me} = \frac{2\pi N}{m_\lambda} +o(N)$.
\end{proof}
\begin{remark}
 As in Remark \ref{rem2.1} we can also obtain a result of equipartition of energy of $F_\ep'$ on $\omega_\ep'$.\end{remark}

\subsection{The case of larger applied field}

We prove that the conclusions of Proposition~\ref{pro4} hold for larger fields as well, that is fields which do not satisfy the assumptions of Proposition \ref{etalement}. The technical difficulty is that the applied field is too large to have the energy upper bound that is needed to apply the result of Theorem \ref{thcompagnon}, so we use the strategy of \cite{livre} Chapter 8: average over smaller balls where most of the time the local energy is small enough to apply Theorem \ref{thcompagnon} and the result for small applied fields. Note that in this regime, by item 6 in Proposition \ref{lemdomaine} we have for $N \in \{ \nm, \ns\}$, $\je^N= \je^{N_0} +o(\he) = \hal |\om | \he \plep +o(\he)$.

\begin{pro} \label{pro5}Assume that
$$ \lep^4 \ll \he  \ll 1/\ep^2$$
and that $\je(u_\ep,A_\ep) \le \frac{\he}{2} |\om| \plep+ C \he$.  We also assume that $|u_\ep |\le 1$ and that
$A_\ep$ is a critical point of $\je(u_\ep,\cdot)$.
Then, using the same notation as in Propositions~\ref{etalement} and \ref{pro4}, we have, for any $U$
\begin{equation}\label{gli2big}
\je(u_\ep,A_\ep)\ge \hal  |\om| \he \plep +
 |\om| \he\(\int W_U(j) \, dP(j) + \frac{\gamma}{2\pi}+o(1)\),
 \end{equation}
 and $P$-a.e.  $j \in \ainfty$.
 \end{pro}

\begin{proof}[Step 1: Blow-up]
The proof follows the ideas in Chapter 8 of \cite{livre,erratum}.
We recall  the rescaling formula from there. Define $\tix =  x/\facteur$ and
$$ \lue(\tix) = u_\ep(x),
 \quad  \lAe(\tix) =\facteur A_\ep(x),\quad \lom =
\frac\om\facteur,\quad \leep  = \frac\ep\facteur, \quad\lhe = \facteur^2\he.$$
Then, denoting $B_x^\facteur =
B(x,\facteur)$, we have up to translation
$\je(u_\ep,A_\ep,B_x^\facteur) = \lG(\lue,\lAe, B_\tix^1)$, where $G_\ep (u_\ep, A_\ep, B_x^\facteur)$ denotes the Ginzburg-Landau energy restricted to $B_x^\facteur$, and
$$\lG(\lue,\lAe,B_1)= \frac{1}{2}\int_{B_1} |\nab_{\lAe} \lue|^2 + \frac1{\facteur^2}
\(\curl \lAe -\lhe\)^2+ \frac{\(1-|\lue|^2\)^2}{2{\leep}^2}.$$

As in \cite{livre}, Chap~8,   if $\lep^4\ll\he\ll 1/\ep^2$  we may  choose $\facteur\ll 1$
such that
\begin{equation}\label{deflambda} \lhe = (\log\leep)^4 \qquad |\log \leep|\sim \plep
,\end{equation} so that $\sigma^2 \log^4\frac{1}{\sigma}= \ep^2 \he$.
\end{proof}

\begin{proof}[Step 2: Fubini] We give a formulation of  the energy which follows from  Fubini's theorem:
\begin{equation}\label{fub}\je(u_\ep,A_\ep,\om) =  \int_{x\in\mr^2}\frac{\je(u_\ep,A_\ep,B_x^\facteur\cap\om)}{|B_x^\facteur|}.\end{equation}
Note that if we restrict the  integration to the set $\oms$ of  those $x$'s such that $B_x^\facteur\subset\om$, then we have an inequality.

For $x\in\oms$ we define  $P_\ep^x$ to be the push-forward of
the normalized Lebesgue measure on $B_x^\facteur$ by the map
$x\mapsto \pje(\frac x\el+\cdot)$, where again $j_\ep' =
(iu_\ep', \nab_{A_\ep'} u_\ep') $,
and $u_\ep', A_\ep'$ are
 as in \eqref{uprime}. It  is an element of $\probas$, the set of
  probability measures on $X = \Lp$. On $\probas$ we define the function
$$\gfep(P) =  \begin{cases}\D \frac{\je(u_\ep,A_\ep,B_x^\facteur)}{|B_x^\facteur|} - \frac\he 2 \plep& \text{if $\exists x\in\oms$ s.t. $P = P_\ep^x$}\\ +\infty & \text{otherwise.}\end{cases}$$
We also define $Q_\ep$ to be the push-forward under $x\mapsto P_\ep^x$ of the normalized Lebesgue measure on $\oms$.
It is a probability measure on $\probas$. Now \eqref{fub} becomes, after subtracting $\frac\he 2 |\oms|\plep$,
$$\je(u_\ep,A_\ep,\om) - \frac\he 2|\oms|\plep =|\oms| \int_\probas \gfep(P) \,dQ_\ep(P).$$

Note that since $|\om\sm\oms| = O(\facteur)$, and since $\facteur\plep = o(1)$ --- this easily follows from \eqref{deflambda} --- we deduce from the above that
\begin{equation}\label{bizarre2}
\je(u_\ep,A_\ep,\om) - \frac\he 2|\om|\plep  = o(\he)+|\oms| \int_\probas \gfep(P) \,dQ_\ep(P).\end{equation}
\end{proof}

\begin{proof}[Step 3: $\Gamma$-convergence of $\frac1\he\gfep$]
Assume that  $P_\ep$ is  a probability measure such that
$\gfep(P_\ep)\le C\he$ and that $P_\ep\in\probas$ converges to $P$.
Then since  $\gfep (P_\ep)<\infty$,    there exists for each $\ep$ some $x_\ep\in \oms$ such that $P_\ep = P_\ep^{x_\ep}$ and
\begin{equation}\begin{split}\label{gfepgl} \gfep(P_\ep)
&= \frac{\je(u_\ep,A_\ep,B_{x_\ep}^\facteur)}{|B^\facteur|} - \frac\he 2 \plep\\
&= \frac1{\facteur^2}\(\frac{\lG(\lue,\lAe,B^1_{\tix_\ep})}{|B_1|} - \frac{\facteur^2\he}2 \plep\)\\
&=  \frac1{\facteur^2}\(\frac{\lG(\lue,\lAe,B^1_{\tix_\ep})}{|B_1|} - \frac{\lhe}2 \plep\).
\end{split}\end{equation}
Since $\gfep(P_\ep)\le C\he$ we get $\lG(\lue,\lAe,B^1_{\tix_\ep})\le C (\he\facteur^2 + \he\facteur^2 \plep) = C\lhe(1+\plep)$ and in view of  \eqref{deflambda}, note that $\ep' = \ep\sqrt\he = \tilde\ep\sqrt\lhe$,  we may apply Proposition~\ref{pro4} with  $\ep $ replaced by $\tilde{\ep}$, $\he$ by $\lhe$, etc.
This way we find that, $\tilde{N}$ being given by Proposition
\ref{pro4}, $\tilde P$ is concentrated  on $\ainfty$ and that
\begin{multline}\label{star}
\lG(\lue,\lAe,B_{\tix_\ep}^1) \ge
 \hal \|h_{\tilde{\ep}, \tilde{N}}   - \lhe\|_{H^1(B_{\tix_\ep}^1)}^2 + \pi \tilde{N} \plep \\
 + \lhe |\omega_\leep|
\( \int W_U(j) \, d\tilde P(j)
 + \frac{\gamma}{2\pi}+o(1),  \) \end{multline}
 where $h_{ \tilde{\ep}, \tilde{N}}  $ is the solution of
 the obstacle problem \eqref{obs1}, replacing $\om$ with
 $B_{\tix_\ep}^1$, $\he$ with $\lhe$ and $\ep$ with $\leep$~;
 and $\lmuo = - \Delta \lho + \lho $, $\omega_\leep = \supp(\lmuo)$.
   We can further check, since $\lhe = |\log\leep|^4$,
    that
    from \eqref{volome}, \eqref{lq} in Proposition~\ref{lemdomaine},  we have, as $\leep\to 0$,
 \begin{equation}\label{broutilles}| \omega_{\tilde{\ep}} |=|B_1|+o(1) \qquad     \|  h_{\tilde{\ep}, \tilde{N}}      - \lhe\|_{H^1}^2+
 \pi \tilde{N} \plep = \hal \lhe |B_1| \plep +o(\lhe).\end{equation}

 In \eqref{star},  $\tilde P$ is such that $\tilde P$-a.e.
  $j
  \in\ainfty$, and is
  the limit of the probability measures $\{P_\leep\}_\ep$.
   Using Remark~\ref{autrej}, the measure $\tilde P$ is the limit
    of the image of the normalized Lebesgue measure on
     $\omega_{\tilde{\ep}}$ by $\vp:x\mapsto j(u_x,A_x)$, where
 $$u_x(y) = \lue(x+\el y) = u_\ep(\facteur x+\el y),\quad A_x(y) = \el \lAe(x+\el y) = \el u_\ep(\facteur x+\el y),$$
 while $P_\ep$ is the image of the normalized Lebesgue measure on
 $B_{\tix_\ep}^\facteur$ by the usual blow-up map
 $x\mapsto j_\ep'(x+\cdot)$  which, changing the variable to $\tilde x$, is equal to the image of the normalized Lebesgue measure on $B_{\tix_\ep}^1$ by $\vp$. From  $|\omega_{\tilde{\ep}}|\sim |B_1|$ and $\omega_{\tilde{\ep}}\subset B_{\tix_\ep}^1$, we then deduce that  $\{P_\ep\}_\ep$ and $\{P_\leep\}_\ep$ have the same limit, i.e. that $\tilde P = P$. Then \eqref{star}, \eqref{broutilles} yield
 $$\lG(\lue,\lAe,B_1^{\tix_\ep}) \ge
 \hal\lhe |B_1| \lL +
 \lhe |B_1|
\( \int W_U(j) \, dP(j)
 + \frac{\gamma}{2\pi} \) +o(\lhe), $$
 where $P = \lim_\ep P_\ep$ is concentrated on $\ainfty$. Then, from \eqref{gfepgl},
\begin{equation}\label{gamlim}
\liminf_{\ep\to 0} \frac1\he\gfep(P_\ep)\ge \int W_U(j) \, dP(j)
 + \frac{\gamma}{2\pi}.
 \end{equation}
 \end{proof}

\begin{proof}[Step 4: $\{Q_\ep\}_\ep$ is tight] A consequence
of \eqref{gamlim} and Proposition~\ref{pw1} below is that the
function $\gfep$ is  bounded below independently of $\ep$, thus the results of Section~\ref{sec:abstract} apply.  The measures $Q_\ep$ are regular from their definition hence given $\delta>0$ there exists compact sets $K_\ep$ in $\probas$ such that $Q_\ep(K_\ep)>1-\delta$, and since $\he^{-1} \int \gfep(P)\,dQ_\ep(P) < C$ from \eqref{bizarre2} and the energy upper bound, we can also require that $\he^{-1}\gfep < 1/\delta$ on $K_\ep$.

Now assume $P_\ep\in K_\ep$ for each $\ep>0$. Then $\gfep(P_\ep) /\he$ is bounded independently of $\ep$ and therefore from the previous step and after taking a subsequence, $P_\ep\to P$. This shows that the hypotheses of Lemma~\ref{lesigne} are satisfied and thus that any subsequence of $\{Q_\ep\}_\ep$ has a convergent subsequence. In addition, we  deduce that $Q$-almost every $P$ satisfies that $P$-almost every
$j \in\ainfty$. \end{proof}

\begin{proof}[Step 5: Conclusion] Combining the convergence of $\{Q_\ep\}_\ep$ and \eqref{gamlim}, we deduce with the help  of Lemma~\ref{abstrait} that
\begin{equation}\label{451bis}
\liminf_{\ep\to 0} \frac1\he\int \gfep(P)\,dQ_\ep(P)\ge \int\(
\int W_U(j) \, dP(j)\)\,dQ(P) + \frac{\gamma}{2\pi}.\end{equation}
It remains to show that
\begin{equation}\label{probprob}  \int\( \int W_U(j) \, dP(j)\)\,
dQ(P)=  \int W_U(j)\,d\overline P(j),\end{equation}
where $\overline P = \lim_\ep P_\ep$ and $P_\ep$ is the push-forward of the normalized uniform measure on $\poe$ by the map $x\mapsto \pje(x+\cdot)$ where $\pje $
 is the current  defined from \eqref{uprime}.

But, if $\vp$ is a continuous and bounded function on $X$, by definition of $Q_\ep$ we have
$$\int \vp\,dP_\ep = \int \(\int \vp\,dP\) \,dQ_\ep(P)+o(1),$$ since $|\omega_\ep'|\sim \oms$.
Hence, passing to the limit, we find
$$\int \vp\,d\overline P = \int \(\int \vp\,dP\) \,dQ(P).$$
It is straightforward to check that this equality extends to positive measurable functions, and in particular  to $W$, which was proven to be measurable in  Proposition~\ref{pw1}. This proves \eqref{probprob}.
In addition since  $Q$ almost every $P$  satisfies $j \in \ainfty$,
 and since  $\overline{P} =
\int P\, dQ(P)$, we also have  that $\overline{P}$ almost every $j\in
\ainfty$.
Renaming $\overline{P}$ by  $P$,  and since $|\oms|\sim |\om|$ as $\ep\to 0$, combining \eqref{bizarre2}, \eqref{451bis}   and \eqref{probprob} proves \eqref{gli2big}.\end{proof}

\subsection{Proof of Theorem~\ref{th1prime}}
We  may now  combine
the results of
 Propositions \ref{pro4}, \ref{pro5}, Lemma \ref{lemapriori} and the upper bound of Theorem~\ref{th5} below to prove Theorem~\ref{th1prime}.

Assertion 2 of the Theorem is part of Theorem~\ref{th5} below.

For Assertion 1, in view of Remark~\ref{autrej}, propositions~\ref{pro4} and~\ref{pro5} imply, for a suitable choice of $N\in \{\nm,\ns\}$, the convergence of $P_\ep$ to a translation invariant measure $P$ on $L^p_\loc$ such that \eqref{borneinffinale2} holds.

We next  prove \eqref{mumuo0}.  We have seen in the subsections just above that the upper bound condition implies
 $\fe(u_\ep,\ai)\le C\he|\toe|$ or, blowing-up, $\fe'(\pue,\pae)\le C|\poe|$. We may write $\mu(u_\ep, A_\ep)  - \tmuo = \curl j_{1,\ep} +\alpha+\beta$, where
$$\alpha = \curl\ai - \tmuo,\quad \beta = \mu(u_\ep , A_\ep) - \curl\ai - \curl j_{1,\ep}.$$
From \eqref{jF} we have $\|\alpha\|_{L^2}^2\le C(\fe'+|\poe|)\le C\he|\toe|$, hence $\|\alpha\|_{W^{-1,p}} \le C\sqrt{N}$. The same bound holds for $\beta$ from \eqref{ar2}. Finally, from \eqref{jF}, we have $\int |j'|^p \le C (F_\ep ' + |\omega_\ep'|) \le C N$. Rescaling this relation, we get 
$$\|j\|_{L^p(\om)}\le\he^{\hal-\frac{1}{p}} \|j'\|_{L^p(\om_\ep')} \le C \he^{\hal-\frac{1}{p}} N^{\frac{1}{p}}   \le C N^\hal$$ where we have used the fact that $ N \le \frac{\he|\om|}{2\pi}\le C \he$.   Therefore $\|\curl j_{1,\ep}\|_{W^{-1,p}} \le C\sqrt{N}$, which concludes the proof of \eqref{mumuo0}.

It remains to prove the statement concerning   minimizers $\{(u_\ep,A_\ep)\}_\ep$ of $G_\ep$. In  the case of small applied fields $\he \le \ep^{-\beta}$,  Corollary \ref{coroapb} gives $G_\ep(u_\ep, A_\ep) \le \min_{N\in \{ \nm, \ns\}}   \je^N + C N_0 $
hence Proposition \ref{pro4} applies.  Comparing the lower bound \eqref{gli2}   to the upper bound \eqref{resth5} in Theorem \ref{th5},  we deduce
that $N$ minimizes the right-hand side and that
$$
 \int W_U(j) \, dP(j)  \le \min_{\admissible}  W .$$
Since $P$ is supported on $\admissible$, we obtain that $P$-a.e.
 $j$ minimizes $W_U$ over $\admissible$. Since minimizers of $W_U$ are independent of $U$, we have the result.

In the case of large applied fields  $ \lep^4\ll \frac{1}{\ep^2}$,  Corollary \ref{corogc}
yields $$\min G_\ep \le \frac{\he}{2}|\om|\plep + |\om|\he \( \min_{\ainfty} W + \frac{\gamma}{2\pi} + o(1)\), $$
thus Proposition \ref{pro5} applies and comparing the above to \eqref{gli2big}, we deduce that there is equality and that
$$  \int W_U(j)\, dP(j) \le \min_{\ainfty} W.$$We again deduce
 that $P$-almost every $j$ minimizes
  $W$ over $\ainfty$.
From \eqref{lq} in Proposition \ref{lemdomaine}, we get that
\eqref{borneinffinale2}
holds.

Note that Theorem \ref{th1prime} implies Theorem~\ref{th1} since,
from \eqref{enn0} in Proposition~\ref{lemdomaine}, if $\he = \lambda\lep$ with $\lambda>\lambda_\om$ then $G_\ep^N - G_\ep^{N_0} = O(1)$ as $\ep\to 0$.

\begin{remark}\label{remsec6} If we had chosen from the beginning $N=\bar{N}= \frac{1}{2\pi} \nu_\ep(\Omega_\ep')$  as indicated in footnote in the proof of Proposition \ref{etalement}, the we would obtain  the lower bound \eqref{borneinffinale2} with that $\bar{N}$. Combining with the upper bound of Theorem \ref{th5} we may deduce that $\je^{\bar{N}}= \min_{N\in \mn} \je^N+o(\bar{N})+o(N_0)$. A careful examination  of the variation of $\je^N$ with $N$, based on Lemma \ref{lemvarg} should then allow to obtain that $\bar{N}= \nm $ or $\ns$  up to an error which is quantified by the examination of the growth of $\je^N= \je^{N_0}$ (this is quite delicate, though). We expect this error to be $0$ for small enough applied fields, in particular for $\he < \hci + O(\sqrt{\lep})$.
\end{remark}

\section{Upper bound}\label{bornesup}
In this section, we use the  notation of Section \ref{sec-splitting}, in particular the definitions of $\tho, \toe...$ can be found there.
We make here, and here only, the assumption that $\om$ is convex. This guarantees the smoothness of the solutions to the obstacle problem \eqref{obs32}, see below. We do not believe this is a serious restriction, but the possible presence of cusps in the coincidence set would certainly add technical difficulties to our construction.

We prove the upper bound matching the lower bound of Theorem \ref{th1prime} (recall the definition of $\je^N$ in \eqref{121b}):
\begin{theo}\label{th5} Assume that  $\om$ is convex and that (\ref{range}) holds. Then for any family of integers $\{N\}$ depending on $\ep$ and  satisfying
\begin{equation}\label{rangen}1\ll N\le \frac{|\om|\he}{2\pi}, \quad \text{as $\ep\to 0$}\end{equation} the following holds:
\begin{enumerate}
\item
There exists $(u_\ep,A_\ep)$ such that, as $\ep \to 0$,
\begin{equation}\label{resth5} \je(u_\ep, A_\ep) \le  \je^N +
 N\(\frac{2\pi}{m_\lambda} \min_{\admissible}W +
 \gamma  +o(1)\),\end{equation}
where $\lambda\in[\lambda_\om,+\infty]$ is the limit of $\he/\lep$ as $\ep\to 0$, where $\admissible$ is as  in Definition \ref{defA}, and where $\ep' = \ep\sqrt\he$.
\item Let $1<p<2$ be given.  For any probability $P $ on $\Lp$
which is invariant under the action of translations and concentrated on $\admissible$, there exists $(u_\ep, A_\ep)$  such that, letting $P_\ep$ be the push-forward of the normalized Lebesgue measure on $\toe$ by the map $x \mapsto
\frac{1}{\sqrt{\he}}  j(u_\ep, A_\ep) \( x + \frac{\cdot}{\sqrt{\he}} \) $, we have  as $\ep \to 0$, $P_\ep \to P$ weakly and
\begin{equation}\label{resth5kk} \je(u_\ep, A_\ep) \le \je^N + N \(\frac{2\pi}{m_\lambda} \int W_K(j) \, dP(j)  +\gamma  +o(1)\).\end{equation}
\end{enumerate}
\end{theo}

\begin{coro}\label{coroapb} Under the same assumptions, we have
$$\min G_\ep \le \min_{N \in \{\nm, \ns\}} \je^N+ CN_0\le C \he \plep.$$
\end{coro}
\begin{proof} To obtain an upper bound, we apply the result above with $N= \nm$ and $N=\ns$ (recall that $N_0$ is not necessarily an integer) and use \eqref{enn0} in Proposition~\ref{lemdomaine}. \end{proof}

\begin{coro}\label{corogc}Under the same assumptions,
if $\lep^4\ll \he \ll \frac{1}{\ep^2} $,  we have
$$\min G_\ep \le  \hal|\om |\he\plep + \he |\om| \(\min_{\ainfty}W+\frac{\gamma}{2\pi} +o(1)\).$$
\end{coro}
\begin{proof} This follows from  \eqref{resth5} applied with $N = \nm$, \eqref{lq} in Proposition \ref{lemdomaine},
$|\toe|\le |\om|$ and $\lambda= +\infty$.\end{proof}

We now prove Theorem \ref{th5}.

\subsection{Properties of $\toe$} The convexity of $\om$ guarantees
 that the coincidence sets $\{\omega_m\}_m$  for the minimizers
  of \eqref{obs32} are convex (Friedman-Phillips, \cite{fp},
  see also \cite{ka}, \cite{dm}). Then the density criterion of
  Caffarelli \cite{caf1} and regularity improvement of
  Kinderlehrer-Nirenberg \cite{kn} and Isakov \cite{isa} imply
  that it is in fact analytic for any $m$. We state a density estimate which is uniform with respect to $m\in(\homb,1]$. This is the only place where we use the assumption that $\om$ is convex.
\begin{lem} Assume $\om$ is convex (so that \eqref{assumption} is satisfied), and
let $L_m$ be as in Proposition~\ref{obstacleprops}. Then there exists $\a>0$ and $r_0>0$ such that for any $m\in(\homb,1]$,   any $r<r_0$, and any $x\in\omega_m$
$$\frac{|\omega_m\cap B(x,L_m r)|}{|B(x,L_m r)|}\ge \alpha.$$
\end{lem}
\begin{proof} We call $d(r)$ the density ratio above. Since
 $\omega_m$ is convex, $r\mapsto d(r)$ is decreasing. But
  from Proposition~\ref{obstacleprops}, the diameter
   of $\omega_m$ is bounded by $CL_m$ and $|\omega_m|\ge cL_m^2$,
    where $c,C>0$ are independent of $m\in(\homb,1]$. Therefore
    $d(C)\ge c/C^2$. Letting $\a  = c/C^2$ and $r_0 = C$ proves the lemma.
\end{proof}
Now assume the hypothesis of Theorem~\ref{th5} are satisfied.
Then Lemma~\ref{lemrelated} applies and $\tmuo:= - \Delta \tho+ \tho= \me \he   \indic_{\toe}$ where $\me $ satisfies \eqref{rangec} and $\me\he  |\toe|= 2\pi N$. Let
$$\ell_\ep'= \frac{1}{\sqrt{\me \he}}.$$ We have $(\ell_\ep')^{-2}|\toe|\in 2\pi \mn$.

Rescaling the previous lemma we find
\begin{coro}\label{convexreg} There exists $\a>0$ such that for any $R>0$, any $\ep$ small enough depending on $R$, and $x\in\toe$ we have $|\toe\cap B(x,R\ell_\ep')|\ge\a |B(x,R\ell_\ep')|$.
\end{coro}
\begin{proof} From Proposition~\ref{lemdomaine} we have
$\ell_\ep'\ll L_\ep$, where $L_\ep$ is the value of $L_m$
corresponding to $m = \me$. Thus, if $\ep$ is small enough,
 we have $R\ell_\ep'\le r_0L_\ep$, and the previous lemma applies.\end{proof}

\subsection{Definition of the test current}

The construction
follows similar lines as \cite{livre},   Chapters 7 and 10,
but the estimates must be more precise as only an error of $o(1)$
per vortex is allowed. From now on we assume \eqref{assumption},
\eqref{range} and \eqref{rangen}.  We write for simplicity
$\omega_\ep$ instead of $\toe$ and $m_\ep$ instead of $\me$.

Let $R\in 4\pi \mn$ be given.  We tile $\mr^2$ in
 the obvious way by a collection $\{\mathcal{K}_i\}_i$ of squares
  of sidelength $2R\elp$. We let
$$I= \{ i, \mathcal{K}_i \subset \omega_\ep, \dist \(\mathcal{K}_i, \p \omega_\ep\) \ge \ell_\ep' \},\quad \tom = \cup_{i\in I} \mathcal{K}_i\quad
  \oh=\omega_\ep \sm\tom.$$
To prove the first item of the theorem, we
 apply Corollary~\ref{periodisation} to a minimizer of $W$  to find
$j_R$  in $K_R$ such that
$j_R\cdot\tau$ on $\p K_R$ where $K_R = [-R,R]^2$.
To prove the second item, let $P$ concentrated on $\admissible$ be given, and let  us define $\bar{P}$ to be the push-forward of $P$ under the rescaling $j \mapsto \frac{1}{\sqrt{m_\lambda}} j\(\frac{\cdot}{\sqrt{m_\lambda}} \).$ Then $\bar{P}$ is concentrated on $\ainfty$ and from \eqref{minalai1}, we have
\begin{equation}\label{waiai}
\int W_K(j) \, d\bar{P} (j) = \frac{1}{m_\lambda} \int W_K(j) \, dP(j) +\frac14 \log m_\lambda.\end{equation}
We then apply Corollary~\ref{mesyoung} to the probability $\bar{P}$,
it gives again a  $j_R$  in $K_R$ such that
$j_R\cdot\tau=0$ on $\p K_R$. We continue the construction in the same way
in either of these two cases.

 We next extend $j_R$ by periodicity to $\mr^2$ and, denoting
 by $c_0$ the center of $\mathcal{K}_{i_0}$, where $i_0\in I$ is
  arbitrary, we  let
$$\tej(x) = \begin{cases}\D\frac1\elp
j_R\(\frac {x-c_0}\elp\)& \text{ in $\tom$}
\\ 0 & \text{ in $\mr^2\sm\tom$.}\end{cases}$$
In particular, letting $\widetilde\Lambda =
 \(c_0+ \elp \Lambda_R\)\cap\tom$, where $\Lambda_R$ denotes the support of $\nu_R$, and since $\tej \cdot \tau = 0$
on $\p\tom$, we have
$$\curl \tej = 2\pi\sum_{p\in\widetilde\Lambda}\delta_p - m_\ep\he\indic_\tom,\quad\text{in $\mr^2$}.$$

We define a current $\hej$ as follows. First we note that since
 $|\mathcal K_i|, |\omega_\ep|\in 2\pi\elp^2\mn$, we have
  $|\oh|\in 2\pi\elp^2\mn$. Then, using Corollary~\ref{convexreg}
   we may --- we omit the cumbersome details --- find disjoint
   measurable
   sets $\mathcal{C}_1,\dots,\mathcal{C}_n$ and $y_i\in \mathcal{C}_i$
    such that, for some $c,C>0$ independent of $R,\ep$,
\begin{equation}\label{propci} \indic_\oh = \sum_i\indic_{\mathcal{C}_i},\quad
|\mathcal{C}_i| = 2\pi\elp^2,\quad B(y_i,c\elp)\subset \mathcal{C}_i
\subset B(y_i,C\elp).\end{equation}
We let $j_i = -\np f_i$, where
\begin{equation}\label{defji} \begin{cases} - \Delta f_i =
2\pi\delta_{y_i} - m_\ep\he \indic_{\mathcal{C}_i} &
\text{in $B(y_i,C\elp)$}\\ \p_\nu f_i = 0 &
\text{ on $\p B(y_i,C\elp)$},\end{cases}\end{equation}
and then, letting $j_i = 0$ on $\mr^2\sm B(y_i,C\elp)$, we let $\hej = \sum_{i=1}^n j_i$. We have, letting $\widehat\Lambda = \{y_1,\dots,y_n\}$,
$$ \curl\hej = 2\pi\sum_{p\in\widehat\Lambda} \delta_p - m_\ep\he \indic_\oh.$$
Finally we let $j_\ep = \tej +\hej$, $\Lambda = \widetilde\Lambda\cup\widehat\Lambda$. We have
\begin{pro}\label{finition} The current $j_\ep$ satisfies
\begin{equation}\label{curljconstr}  \left\{\begin{array}{ll}
\curl j_\ep= 2\pi \sum_{p \in \Lambda} \delta_p -
\he m_\ep \mathbf{1}_{\omega_\ep}  & \text{in $\mr^2$}\\
j_\ep  =0 & \text{on $\mr^2\sm\om$.} \end{array}\right.\end{equation}
Moreover
\begin{equation} \label{wje} \limsup_{\eta\to 0}\frac1{m_\ep\he
|\omega_\ep|}\(\hal\int_{\om\sm\cup_{p\in\Lambda} B(p,\eta\elp)}
|j_\ep|^2 + \pi\#\Lambda\log\eta\)\le \frac{W(j_R,\indic_{K_R})}{R^2} + o_\ep(1),\end{equation}
where $\lim_{\ep\to 0} o_\ep(1) = 0$.

Finally, there exists $\eta_0>0$ such that for any $\ep$
small enough, any $p\in \Lambda$ and any $q\in [1,+\infty)$
we have
\begin{equation}\label{jelq} \|j_\ep - \np\log|\cdot-p|\|_{L^q(B(p,\eta_0\elp))}\le C_q\elp^{\frac2q-1}.\end{equation}
\end{pro}
\begin{proof} The fact that $\curl j_\ep= 2\pi \sum_{p \in \Lambda}
 \delta_p - \he m_\ep$ is obvious, and $j_\ep = 0$ on $\om^c$
  follows from the definition of $j_\ep$ and the fact that
   $d(\omega_\ep,\om^c)\ge C\elp$
    if $\ep$ is small enough. This is a consequence of the fact that on $\omega_\ep$ we have $h_{\ep,N} = \me \he$ while on $\bo$ we have $h_{\ep,N} = \he$. The difference is $$\Delta = \he(\moep - \me+ 1 - \moep) = \Ce +\hal \plep\approx \hal\plep,$$
    using \eqref{limitem}. It follows using \eqref{nabho} that $d(\omega_\ep,\om^c)\ge \sqrt{\plep/\he} \gg \elp$.

We estimate $\hej$. From \eqref{defji}, $f_i(y) = - \log|y-y_i| + g_i((y-y_i)/\elp)$, where $g_i$ solves
$\Delta g_i(x) = \indic_{\mathcal{C}_i}(y_i+\elp x)$ in $B(0,C)$
 and $\p_\nu g_i = 0$ on $\p B(0,C)$. Since $\indic_{\mathcal{C}_i}
 \in L^\infty$, elliptic regularity implies that
 $\|\nab g_i\|_{L^q}\le C_q$ for every $q\in[1,+\infty)$.
 We easily deduce that
\begin{equation}\label{jilqlog} \|j_i - \np\log|\cdot-y_i|\|_{L^q(B(y_i,C\elp))}\le C_q\elp^{\frac2q-1}.\end{equation}
Since $j_i = 0$ outside $B(y_i,C\elp)$ we deduce
\begin{equation}\label{jilq} \|j_i\|_{L^q(\mr^2\sm B(y_i,c\elp))}\le C_q\elp^{\frac2q-1}\(1+c^{2-q}\),\end{equation}
\begin{equation}\label{jil2} \|j_i\|_{L^2(\mr^2\sm B(y_i,c\elp))}
\le C \(1+\log\frac1c\).\end{equation}
Then we compute estimates for $\tej$.
Since $j_R$ is defined independently of $\ep$, there exists $\eta_0>0$ (depending on $R$) which bounds from below the distances between the points in $\Lambda_R$ and between $\Lambda_R $ and $\p K_R$.   Since $\div j_R=0$  in $K_R$ we have
that $\|j_R - \np\log|\cdot-y|\|_{L^q(B(y,\eta_0))}\le C_q$ for
any $y\in\Lambda_R$ and moreover $j_R\in C^\infty_\loc(K_R\sm\Lambda_R)$. It follows that $\forall y\in\widetilde \Lambda$ and $\forall q\in [1,+\infty)$ we have
\begin{equation}\label{jtelq} \|\tej- \np\log|\cdot-y\|_{L^q(B(y,\eta_0\elp))}\le C_q\elp^{\frac2q-1}.\end{equation}
Moreover, since $j_R$ is uniformly locally bounded in $L^{q'}$, for every $q'\in[1,2)$, we have for every $x\in\mr^2$ and every $M>0$
\begin{equation}\label{jtelp}
 \|\tej\|_{L^{q'}(B(x,M\elp))}\le C_{q'} M^{\frac2{q'}}\elp^{\frac2{q'} - 1}.
 \end{equation}

We are ready to derive estimates for
$j_\ep = \tej +\sum_{i=1}^n j_i$. First we note that from
 \eqref{propci} we have that $|y_i-y_j|$ and $d(y_i,\tom)$
 are bounded below by $c\elp$, if $i\neq j$, therefore since
  $\supp j_i\subset B(y_i,C\elp)$ and $\supp \tej\subset\tom$
  the overlap number of the supports of $\tej$ and the $j_i$'s
  is bounded by a constant $C$ independent of $R$, $\ep$.
   Moreover $\tom$ is included in the complement of
   $\cup_i B(y_i,c\elp)$.

We have $\int_{\mr^2} \hej\cdot\tej = \sum_{i}
\int_{\mr^2} j_i
\cdot  \tej$, but the number of $i$'s for which the integrals above are nonzero
is bounded by $C\frac{|\supp\tej\cap\supp\hej|}{\elp^2}$
since $\supp\tej\cap\supp\hej\subset
\{d(x,\omega_\ep^c)\le C\elp\}$
and using Proposition~\ref{lemdomaine}, item 4.
 Applying H\"older's inequality to each of these nonzero integrals,
and using
 \eqref{jilq}, \eqref{jtelp}, we deduce
\begin{equation}\label{rect} \int_{\mr^2} \hej\cdot\tej \le
C\frac{|\supp\tej\cap\supp\hej|}{\elp^2}\times
C_{q}\elp^{\frac2q - 1} \times C_{q'} \elp^{\frac2{q'} - 1}
\le C m_\ep\he o_\ep(1),\end{equation}
 Note that  still from Proposition~\ref{lemdomaine}, item 4,
  $|\oh| = o_\ep(1)|\omega_\ep|$
and, since $\#\widehat\Lambda  = \elp^{-2} (2\pi)^{-1} |\oh|$
and from \eqref{jilqlog} applied with $q=2$ and \eqref{jil2},
 we deduce
\begin{equation}\label{what} \left|\hal\int_{\mr^2\sm\cup_{p\in
\widehat\Lambda} B(p,\eta\elp)} |\hej|^2 + \pi\#\widehat\Lambda
\log\eta\right| \le C\#\widehat\Lambda\le Cm_\ep\he|\oh| =  o_\ep(1)\he|\omega_\ep|.\end{equation}
Also, since by definition
$$W(j_R,\indic_{K_R}) = \lim_{\eta\to 0} \(\hal
\int_{K_R\sm\cup_{p\in\Lambda_R}B(p,\eta)}|j_R|^2 + \pi\#(\Lambda_R\cap K_R)\log\eta\),$$
we have, multiplying by the number of squares in $\tom$,
 which is $|\tom|\elp^{-2}/R^2 = \he m_\ep|\tom|/R^2$, that
\begin{equation}\label{wtilde} \lim_{\eta\to 0}\frac1{m_\ep\he |\tom|}\(\hal\int_{\mr^2\sm\cup_{p\in\widetilde\Lambda} B(p,\eta\elp)} |\tej|^2 + \pi\#\widetilde\Lambda\log\eta\)  = \frac{W(j_R,\indic_{K_R})}{R^2}, \end{equation}
and the limit is uniform in $\ep$ since, despite the notation, the left-hand side only depends on $R$.
From \eqref{rect}, \eqref{what} and \eqref{wtilde} we deduce
$$ \limsup_{\eta\to 0}\frac1{m_\ep\he|\omega_\ep|}
\(\hal\int_{\mr^2\sm\cup_{p\in\Lambda} B(p,\eta\elp)} |j_\ep|^2 +
\pi\#\Lambda\log\eta\)\le \(\frac{W(j_R,\indic_{K_R})}{R^2}\frac{|\tom|}{|\omega_\ep|}  + o_\ep(1)\),$$
and this holds uniformly in $\ep$. Since $|\oh| = o_\ep(1)|\omega_\ep|$, which implies that $|\tom| = (1 - o_\ep(1))|\omega_\ep|$, we obtain \eqref{wje}. Then \eqref{jelq} follows from \eqref{jilqlog}, \eqref{jtelq}.
\end{proof}

\subsection{Definition of the test-configuration}

We next find $\ai$  such that $\curl \ai=m_\ep\he
\mathbf{1}_{\omega_\ep}= \tmuo$, and set
$$A_\ep= \ai +  \np \tho .$$
 %By definition of $\ho$ (recall $\muo= - \Delta \ho+ \ho$) we have
%\begin{equation}
%\label{hconstr} \left\{ \begin{array}{ll} -\Delta h_\ep+ h_\ep= 2\pi
%\sum_{p \in \Lambda} \delta_{p} & \text{in} \ \om
%\\ h=\he & \text{on} \ \bo.\end{array}\right.\end{equation}
To define  $u_\ep$, we start by
defining its phase $\vp_\ep$ by requiring
\begin{equation} \label{pah}\nabla \vp_\ep=\ai +j_\ep  .\end{equation}
Indeed, denoting by $\Theta $ the phase of
 $\prod_{p \in \Lambda} \frac{z-p}{|z-p|}$,
  we have by \eqref{curljconstr}
    $$\curl (\ai +j_\ep - \nab \Theta)=\tmuo
     +   2\pi \sum_{p \in \Lambda}\delta_p  -\co \he \indic_{\omega_\ep}  - 2\pi \sum_{p \in
\Lambda} \delta_{p} =0,$$
 therefore $\ai+j_\ep - \nab \Theta$
  is the gradient of a function $\psi$, we may then
   let $\vp_\ep= \Theta + \psi$,
    this function is well-defined modulo $2\pi$ in
    $\om \backslash \Lambda  $  and satisfies (\ref{pah}).
  Hence $e^{i \vp_\ep}$ is well-defined in
  $\om \backslash \Lambda$ and $\nab \vp_\ep= \ai + j_\ep$.

Fixing $M>1$, we then define
\begin{eqnarray*}
 &u_\ep(x)= e^{i \vp_\ep(x)}\quad \text{in} \  \om \backslash
\cup_{p \in \Lambda}
  B(p , M \ep) \\
 & u_\ep(x)= \frac{1}{f(M)} f\( \frac{|x-p|}{\ep}\) e^{i
\vp(x)}\
 \text{ in} \ B(p , M\ep),
\end{eqnarray*}
where $f$ is the modulus of the unique radial degree-one vortex
$u_0(r, \theta) =f(r) e^{i \theta}$ (see \cite{bbh,miro,hh}). $f$ is increasing from $0$ to $M$, so that  in particular we deduce $|u_\ep|\le 1$ everywhere.

This construction is possible since, for fixed $R$, the distances between the  points in $\Lambda$ are bounded
below by $\eta_0 \ell_\ep'$  (for some $\eta_0$ possibly smaller than the one used before)  and  $\ell_\ep'= \frac{1}{\sqrt{m_\ep
\he}} \gg \ep$ since we assume $\he \ll \frac{1}{\ep^2}$. The test-configuration $(u_\ep,A_\ep)$ is
now defined and there remains to evaluate its energy and show that it satisfies (\ref{resth5}), respectively \eqref{resth5kk}.

\subsection{Splitting of the energy of $(u_\ep, A_\ep)$}
According to Proposition \ref{pro1}, we have the relation
\begin{equation*}
 \je(u_\ep,A_\ep)=\je^N  +  \jie(u_\ep, \ai) - \io
(1-|u_\ep|^2) |\nab \tho|^2,\end{equation*} where
\begin{multline*}
\jie(u_\ep, \ai)=\hal \io |\nab_{\ai}
u|^2 + (\curl \ai-\tmuo)^2 + \frac{(1-|u_\ep|^2)^2}{2\ep^2}\\
 + \io
(\tho-\he - \Ce)\mu(u_\ep,\ai)+ \Ce \io (\mu(u_\ep, \ai)- \tmuo )  .\end{multline*}

First we observe that
 $\io (1-|u_\ep|^2) |\nab \tho|^2=0$   since $\nab \tho
=0$ in $\omega_\ep$ and $|u_\ep|=1$ outside $\cup_{p\in\Lambda} B(p,M\ep)$, which is included in $\omega_\ep$ if $\ep$ is small enough. Thus
$ \je(u_\ep,A_\ep)=\je^N+  \jie(u_\ep, \ai).$
There  remains to evaluate $\jie(u_\ep, \ai)$. By definition
\begin{multline*}
\mu(u_\ep, \ai) =  \curl (iu_\ep, \nab_{\ai} u_\ep) + \curl \ai=
\curl ( |u_\ep|^2 (\nab \vp_\ep- \ai )) + \tmuo \\= \curl (|u_\ep|^2 j_\ep
)+ \tmuo .\end{multline*}
Since $j_\ep = 0$ on $\om^c$, we
 have $\int_\om \mu(u_\ep,\ai) = \int_\om \tmuo = 2\pi N$.
 Moreover, a direct computation shows that  $\mu(u, A)=0$ where $|u|\equiv 1$ so  $\mu(u_\ep,\ai)$ is supported in $\omega_\ep$,
 where $\tho - \he - \Ce = -\hal\plep$ (see \eqref{Cep}), so
 we deduce
$$ \io (\tho-\he- \Ce) \mu (u_\ep, \ai) +   \Ce  \io (\mu(u_\ep, \ai)- \tmuo )   =
  -  \pi N \plep .$$
On the other hand, by choice of $\vp_\ep$,
$$\io |\nab_{\ai} u_\ep|^2 = \io |\nab |u_\ep||^2 + \io |u_\ep|^2
 |\nab \vp_\ep -
\ai|^2= \io |\nab |u_\ep||^2+ \io |u_\ep|^2 |j_\ep|^2
.$$
Recalling that $\curl \ai - \tmuo=0$, we are thus led to
\begin{equation}
\label{vii}
 \je(u_\ep,A_\ep)
=\je^N - \pi N \plep  + \hal \io |u_\ep|^2 |j_\ep|^2+
 |\nab |u_\ep||^2+
\frac{(1-|u_\ep|^2)^2}{2\ep^2}.\end{equation}
It remains to estimate  the terms on the right-hand side.

\begin{lem} [Energy in  $B(p, M\ep)$]
For every $p \in \Lambda$, we have
 \begin{equation}
\label{nrjint}
\hal\int_{  B(p, M\ep) } |u_\ep|^2 |j_\ep|^2 + |\nab |u_\ep||^2+ \frac{(1-|u_\ep|^2)^2}{2\ep^2}
=  \pi \log M + \gamma +o_{M} (1)+o_\ep(1)
\end{equation}
where $o_M(1)\to 0$ as $M\to \infty$ and $o_\ep(1)\to 0$ as $\ep \to 0$, and $\gamma$ is the constant of \eqref{defgamma}.
\end{lem}
\begin{proof} From \eqref{defgamma}, and since $u_0(r,\theta) = f(r) e^{i\theta}$, we have
$$\gamma = \lim_{M\to +\infty}\(\hal\int_0^M \({f'}^2 +\frac{f^2}{r^2} + \frac{(1-f^2)^2}{2}\)2\pi r\,dr - \pi\log M\).$$
On the other hand, from \eqref{jelq} in Proposition~\ref{finition} and since $|u_\ep|\le 1$, we have for any $p\in\Lambda$, by H\"older's inequality
\begin{multline*}\int_{B(p,M\ep)}|u_\ep|^2\left|j_\ep-
\np\log|\cdot-p|\right|^2\le \\|B(p,M\ep)|^{1-\frac2q}\|j_\ep-\np\log|\cdot-p|\|_{L^q(B(p,M\ep))}^2\le \(\frac{M\ep}\elp\)^{2-\frac4q} = o_\ep(1),\end{multline*}
choosing $q>2$ and since $\elp\gg\ep$.
Moreover, choosing $q>2$ and $\frac{1}{q}+\frac{1}{q'}=1$, we have by H\"older again
\begin{multline*}
\left |\int_{B(p, M\ep)}
|u_\ep|^2\np\log|\cdot-p |\cdot(j_\ep-\np\log|\cdot-p|) \right|
\le \\\|j_\ep-\np\log|\cdot-p|\|_{L^q (B(p, M\ep) ) }\|\frac{1}{|x|}
 \|_{L
^{q'}(B(0, \ep))}
 \le C  (\ell_\ep')^{\frac{2}{q}-1} \ep^{\frac{2}{q'}-1}=o_\ep(1)
\end{multline*}
where we used \eqref{jelq} and the fact that $\ell_\ep' \gg \ep$ by \eqref{range}.
Finally, since
$$\int_{B(p,M\ep)} |u_\ep|^2\left|\np\log|\cdot-p|\right|^2
 = \int_{B(p,M\ep)} \frac{f^2(|x|/\ep)}{|x|^2}\,dx,$$
we deduce that for each $p \in \Lambda$,
 $$\int_{B(p, M\ep)} |u_\ep|^2 |j_\ep|^2 = \int_{B(0, M\ep)}
  \frac{f^2(|x|/\ep) }{|x|^2} + o_\ep(1). $$
Following \cite{livre} p. 210,  we deduce that (\ref{nrjint})
holds.\end{proof}

Next, we consider the energy in the annuli $B(p,\ell_\ep' \eta )
\sm B(p, M\ep)$, which  are disjoint when $\eta<\eta_0 $.
\begin{lem}[Energy in the annuli]  For every
$p \in \Lambda$, we have
\begin{equation}  \label{nrjinter}  \hal \int_{B(p, \ell_\ep'
 \eta )\sm B(p, M\ep)}
|u_\ep|^2 |j_\ep |^2 + |\nab |u_\ep||^2+ \frac{(1-|u_\ep|^2)^2}
{2\ep^2}
\le  \pi \log \frac{ \eta\ell_\ep'}{M\ep} +C\eta.
\end{equation} \end{lem}
\begin{proof}
Since  $|u_\ep|=1$ on the annulus $A = B(p, \ell_\ep' \eta )\sm B(p, M\ep)$ only the first term in \eqref{nrjinter} needs to be bounded. Using \eqref{jelq} in Proposition~\ref{finition} we have for $q>2$
$$\int_{A } |j_\ep-\np\log|\cdot-p||^2 \le |A|^{1-\frac2q} \|j_\ep- \np\log|\cdot-p|\|_{L^q(A)}^2 \le C_q\eta^{2-\frac4q}.$$
A similar argument yields, for any $q'\in[1,2)$,
$$\left|
 \int_{A } (j_\ep-\np\log|\cdot-p|)\cdot \np\log|\cdot-p| \right|
 \le   C_{q'}\eta^{\frac2{q'}-1}.$$
We deduce, choosing for instance $q=4$, $q'=1$ above,
$$ \hal \int_{A} |j_\ep |^2 \le \hal \int_{A}
|\np\log|\cdot-p| |^2+ C\eta = \pi \log \frac{ \eta\ell_\ep'}{M\ep} +C\eta, $$
proving \eqref{nrjinter}.
\end{proof}
From \eqref{vii},
\eqref{wje}, \eqref{nrjint}, \eqref{nrjinter}, and letting $M\to + \infty$, $\eta\to 0$,
  we deduce, since $m_\ep\to m_\lambda$ as $\ep\to 0$ and since $2\pi N = m_\ep\he|\omega_\ep|$ and $\frac{\ell_\ep'}{\ep}=
  \frac{1}{\sqrt{m_\ep} \ep' }$ that for any $R>0$
\begin{equation}\label{720}
 G_\ep(u_\ep, A_\ep) \le \je^N +   N\( \pi \log \frac{1}{\sqrt{m_\lambda}} + \gamma \)+
2\pi N \frac{W (j_R, \mathbf{1}_{K_R})}{|K_R|} +o_\ep(N).
\end{equation}
We recall our choice of $j_R$. To prove item 1 of the theorem  $j_R$ was the result of applying Corollary~\ref{periodisation} to a minimizer of $W$, hence from \eqref{bsupper} was such that $$\limsup_{R \to \infty} \frac{W (j_R, \mathbf{1}_{K_R})}{|K_R|} \le \min_{\ainfty} W,$$ so that  letting $R\to \infty$ in \eqref{720} we find
$$
 \je(u_\ep,A_\ep)\le \je^N  +  N  \( 2\pi (   \min_{\ainfty} W - \frac14 \log m_\lambda ) +   \gamma    +o(1)\).
$$  In view of  \eqref{minalai}, this proves \eqref{resth5}.

 To prove item 2 of the theorem, we chose $j_R$ given by Corollary \ref{mesyoung} applied to  $\bar{P}$, so that, using \eqref{waiai},
 \begin{equation*}
 \limsup_{R \to \infty} \frac{W(j_R,\mathbf{1}_{K_R})}{|K_R|} \le \int W_K(j) \, d\bar{P}(j) = \frac{1}{m_\lambda} \int W_K(j) \, dP(j) +\frac14 \log m_\lambda.
\end{equation*}
By letting $R \to \infty$ in \eqref{720}, the result \eqref{resth5kk} follows.

To conclude the proof of the theorem, it remains to show that  $P_\ep \to P$, where $P_\ep$ is the push-forward of the normalized Lebesgue measure on $\toe$
by  $x \mapsto
\frac{1}{\sqrt{\he}}  j(u_\ep, A_\ep) \( x + \frac{\cdot}{\sqrt{\he}} \) $.
Equivalently it suffices to show that $\bar{P}_\ep\to \bar{P}$ where $\bar{P}_\ep$ is the push-forward of the
normalized Lebesgue measure on $\toe$
by  $x \mapsto
\frac{1}{\sqrt{m_\lambda \he}}  j(u_\ep, A_\ep) ( x + \frac{\cdot}{\sqrt{m_\lambda \he}} ) $.   Let $\Phi$ be a continuous function on $\Lp$.
By definition
$$\int \Phi(j)\, d\bar{P}_\ep (j) = \dashint_{\toe} \Phi\( \frac{1}{\sqrt{m_\lambda \he}}  j(u_\ep, A_\ep) ( x + \frac{\cdot}{\sqrt{m_\lambda \he}} )\) \, dx.$$
For any $\eta>0$, we also have
$$
\int \Phi(j)\, d\bar{P}_\ep (j) = \dashint_{\{ \dist(x, \p \tilde{\omega_\ep} \ge \eta) \} }
 \Phi\( \frac{1}{\sqrt{m_\lambda \he}}  j(u_\ep, A_\ep) ( x + \frac{\cdot}{\sqrt{m_\lambda \he}} )\) \, dx +o_\eta(1)$$
where $o_\eta(1)\to 0$ as $\eta\to 0$.  On the other hand by definition of $(u_\ep, A_\ep)$,
$j(u_\ep, A_\ep)= j_\ep+ (|u_\ep|^2 -1) j_\ep - |u_\ep|^2 \np \tho.$
But $\tho $ is constant in $\toe$, and $(|u_\ep|^2 - 1) j_\ep \to 0$ in $L^p$ so we deduce that in $\{\dist(x, \p \tilde{\omega_\ep}) \ge \eta\}$ we have $\frac{1}{\sqrt{m_\lambda \he} } j(u_\ep, A_\ep)   ( x + \frac{\cdot}{\sqrt{m_\lambda \he}}) \to j_R $ in $\Lp$, as $\ep \to 0$.
It follows that
$$\int \Phi(j)\, d \bar{P}_\ep (j) = \dashint_{K_R} \Phi(j_R(x _ \cdot) ) \, dx
+o_\ep(1)+ o_\eta(1) = \int\Phi(j) \, dP_R(j) + o_\ep(1)+o_\eta(1),$$
 where we have used the periodicity of $j_R$, and where $P_R$ is as in Corollary \ref{mesyoung}.
But  $P_R\to \bar{P}$ as $R \to \infty$ so letting $\ep\to 0$, $\eta\to 0$,  and $R \to +\infty$, we obtain the desired result.

\appendix
\section{Additional results on the obstacle problem}
 Here we gather a few  results on the obstacle problem
that we need at various places in the paper. Although these results
 may be known to experts, we have not been able to find them in the
 literature, so they may be of independent interest.
 We focus on the particular type of obstacle problem we are
 concerned with, which is an obstacle problem with constant
 obstacle.
More precisely, letting $\om$ be any smooth bounded domain in
$\mr^2$,  for any $m\in
(-\infty,1]$ we denote by $\hm$ the minimizer of
\begin{equation}\label{obs32}
\min_{H-1\in H^1_0(\om)} (1-m)\io |- \Delta H+H|+ \hal \io |\nab H|^2 + |H-1|^2.\end{equation}
By convex duality and the maximum principle (cf. \cite{bre,brs})
it is equivalent to
 the obstacle problem
\begin{equation}
\label{obs33} \min_{\substack{H-1\in H^1_0(\om)\\
 H\ge m  }
}\frac{1 }{2} \io |\nab H|^2+H^2.\end{equation} We  also define the
coincidence set
$$\oma = \{x \in \om  | \hm(x) = m\}.$$ For general references on obstacle problems we
refer for example to \cite{ks}. (\ref{obs33})   is a standard
obstacle problem where the obstacle is constant and equal to $m$,
and we are interested in the properties of the coincidence set
$\oma$ as $m$ varies.
It is known that
$$-\Delta \hm +\hm= m \indic_{\omega_m}.$$

 Note that the regularity of $\oma$ for fixed
$m$ is well-known (see  \cite{caf1,bk} or the survey \cite{monneau}), however this
is not sufficient for our purposes, since we need estimates which
are uniform in $m$. More precisely let us  define $\hom$ to be the
minimizer of the unconstrained problem, i.e. the solution of
$$\left\{\begin{array}{ll}
 -\Delta \hom+\hom=0 & \text{in} \ \om\\

 \hom = 1 & \text{ on}  \  \p\om, \end{array}\right.$$
 and
set $$ \homb = \min_\om\hom.$$
Then the situation is as follows:
\begin{enumerate}
\item If $m<\homb $ then $\oma= \varnothing $ and $\hm= \hom$.
Thus
\begin{equation}\label{lomhom}\homb= 1-\frac{1}{2\lambda_\om}\end{equation}
where $\lambda_\om$ is as in Section \ref{sec12}.
\item If $\homb<m\le 1$ then $\oma\neq \varnothing$. Moreover, as
$m \nearrow 1$, $\oma\to \om$, and as $m \searrow\homb$, $\omega_m$
reduces to the set of points where $\hom$ achieves its minimum
$\homb$. If we assume in addition that $\homb$ is achieved at a
unique point $x_0$ then as $m\searrow\homb$,
 $\omega_m$ is expected to shrink down to $x_0$ in an ellipse shape.
\end{enumerate}
Our task here is to establish more precisely this behaviour, in
particular obtain some uniform estimates of convergence of $\oma$
(after blow-up at a suitable scale) to an ellipse.
We recall that  we make the assumption \eqref{assumption}.
 It is
standard (see \cite{safftotik})
 that there exists an ellipse $E_Q$
of measure $1$ and a nonnegative function $U_Q$ defined in $\mr^2$
such that \begin{equation}\label{EQ} \Delta U_Q = \frac{\Delta
Q}{2}\indic_{\mr^2\sm E_Q},\quad \{U_Q=0\} = E_Q.\end{equation}
One may check that  this $U_Q$ is unique and that 
$$U_Q = \text{Cst}+ \frac{Q}{2} -\frac{\Delta Q }{4\pi} \log * \indic_{E_Q}.$$
This ellipse will be shown to be the limit of the coincidence sets $\oma$
as $m \searrow \homb$, rescaled at a scale $\ella$ which is not
completely obvious to guess since it contains a logarithmic factor,
more precisely
\begin{equation}
\label{valella}
 \ella \sim \sqrt{\frac{\pi (m-\homb)}{\homb|\log
(m - \homb)|}}\quad \text{as $m \searrow \homb$.}\end{equation}
 Our main result is the
following.
\begin{pro}\label{obstacleprops} Let $\hm$ be as above the minimizer of \eqref{obs33}.  The following holds.
\begin{enumerate}

\item  The coincidence set $\oma$ is empty if $m<\homb$ and has positive measure if $m>\homb$.
$\hm$ is increasing with $m$,
the coincidence set $\oma$ as well, and  $m \mapsto m |\omega_m|$
is a continuous, strictly increasing bijection from $[\homb, 1]$ to  $[0, |\om|]$.

\item If $K$ is any compact subset of $(\homb,1)$ then, uniformly with respect to  $m\in K$,
$$\lim_{\delta\to 0} \frac{|\{x\mid d(x,\p\oma) < \delta\}|}{|\oma|} = 0,\quad \lim_{\delta\to 0} \frac{|\{m<\hm<m+ \delta\}|}{|\oma|} = 0 .$$

\item For any $\delta>0$, if $1-m$ is small enough then $x\in\om\sm\oma$ implies that $d(x,\p\om)^2 < (1-m)(2+\delta)$. In particular there exists a  constant $C$ depending only on $\om$  such that
\begin{equation}
\label{estsolob} |\om \backslash \oma|\le C\sqrt{1-m} \qquad {\|\nab
\hm\|}_{L^\infty(\om)} \le C \sqrt{1-m} . \end{equation}

\item Assuming (\ref{assumption}),  there is a length $\ella$ such that\begin{equation}\label{la}
{\ella}^2|\log\ella| \sim 2\pi (m - \homb)/\homb\quad  \text{as}\ m\to\homb\end{equation}
and such that  for any $M,\delta>0$, if $m$ is sufficiently close
to $\homb$ then
\begin{equation}\label{item4}\begin{split}
\{d(x,\omQ^c)>\delta\ella\} \subset\oma\subset \{d(x,\omQ)<\delta\ella\},\\
\left\{\frac{\hm - m}{\ella^2} \le M\right\}\subset \xom+\ella\{U_Q\le M+\delta\},\\
\left\{\frac{\hm - m}{\ella^2} \ge M\right\}\subset \xom+\ella\{U_Q\ge M-\delta\}.\end{split}
\end{equation}
where $\omQ = \xom +\ella E_Q$. In particular $|\oma|\sim \ella^2 $ as $m \to \homb$.
\end{enumerate}
\end{pro}
The main technique is the construction of  barriers:
\begin{lem} \label{lemp3}
Let $\hm$ be as above. Let $h$ be a continuous function.
\begin{description}
\item[\rm Interior barrier.] If $h\ge \hm$ on $\p\om$, if $h\ge m$ in $\om$ and if $-\Delta h+h\ge 0$ in $\om$, then
$$\text{$h\ge \hm$ in $\om$, in particular $\{h = m\}\subset\oma$.}$$
\item[\rm Exterior barrier.] If $h\le \hm$ on $\p\om$, if $h\ge m$ in $\om$ and if $-\Delta h+h\le m\indic_{\{h=m\}}$ in $\om$, then
$$\text{$h\le \hm$ in $\om$, in particular $\oma \subset \{h = m\}$.}$$
\end{description}
\end{lem}
Both assertions are easy and standard applications of the maximum
principle.

We now turn to the proof of the proposition.  First  recall that
from a well known result of J. Frehse (\cite{frehse}), $\hm$ is
$C^{1,1}$ and, as a consequence, $-\Delta \hm + \hm =
m\indic_\oma.$

\begin{proof}[Proof of 1)]

If $m<\beta$,
then $\hm + (\beta - m)$ is an interior barrier for $h_\beta$,
thus $\oma\subset\omega_\beta$ (hence $\oma$ is increasing in $m$). If $m< \homb$ then $h_0$ is an
exterior barrier for $\hm$ hence $\oma = \varnothing$. If $m>
\homb$ and $\oma$ is negligible, then $-\Delta \hm + \hm =
m\indic_\oma = 0$ in $\om$ and $\hm = 1$ on $\p\om$  hence $\hm =
\hom$, which is impossible. By Lemma \ref{lemp3} too,
for $m\le m'$ we have $\hm\le H_{m'}\le \hm+(m'-m)$, so $\hm$ is increasing in $m$, and also
 $H_{m'}\to \hm$ uniformly in $\om$ as $m'$ tends to $m$ from above, and then in the distributions sense too. Hence,  using $-\Delta (H_{m'} - \hm) + (H_{m'} - \hm) \ge  m\indic_{\omega_{m'}\sm\oma}$, we find that $|\omega_{m'}\sm\oma|$ tends to zero as $m'\to m$. It follows that  $m \mapsto m |\oma|$ is a continuous strictly increasing function of $m$. For $m=\homb$ it is equal to $0$, while for $m=1$, it is immediate that $\hm= 1$ and $\oma=\om$, so $m \mapsto m |\oma|$ maps $[\homb,1]$ to $[0, |\om|]$ bijectively.
\end{proof}
\begin{proof}[Proof of 2)] Let $\omega_{m,\delta} = \{x\mid d(x,\p\oma) < \delta\}$. Then it is true that $\lim_{\delta\to 0} |\omega_{m,\delta}| = 0$, indeed $\cap_{\delta>0}\omega_{m,\delta} = \p\oma$ and  $|\p\oma| = 0$, see  \cite{bk}. From the previous step  $\lim_{\delta\to 0} |\omega_{m+\delta}\sm\omega_{m-\delta}| = 0$.
Thus for any $\ep>0$ and any $m\in K$ there exists $\beta>m$
such that $|\omega_{\beta}\sm\omega_{m}| <\ep$ and then there
exists $\delta>0$ such that $|\omega_{\beta,\delta}|<\ep$ and
$|\omega_{m, \delta}| <\ep $. Then for any $m'\in[m,\beta]$, it
holds that
$$ \omega_{m',\delta} \subset \omega_{m,\delta}\cup(\omega_{\beta}\sm\omega_{m})\cup \omega_{\beta,\delta}$$
hence $|\omega_{m',\delta}|<3\ep$. Then by the compactness of $K$,
$\lim_{\delta\to 0}|\omega_{m,\delta}|=0$ uniformly in $m\in K$,
which is what we want since from 1) we have that  $\inf_{m\in
K}|\oma|>0$.

Similarly, we have $\cap_{\delta>0}\{m<\hm<m+ \delta\}=\varnothing$  and, since $m'\ge m$ implies $\hm\le H_{m'}\le \hm+(m'-m)$, we have
$$\{m'<H_{m'}<m'+ \delta\}\subset \{m<\hm<m'+ \delta\}.$$
 Thus for any $m\in K$ and $\ep>0$, taking $\delta$ such that $|\{m<\hm<m+ 2\delta\}|<\ep$ we find  $|\{m'<H_{m'}<m'+ \delta\}|<\ep$ for any $m'\in [m,m+\delta]$ and it follows as above that  $\lim_{\delta\to 0}|\{m<\hm<m+ \delta\}|=0$ uniformly in $m\in K$.
\end{proof}
\begin{proof}[Proof of 3)] We let $d(x) = d(x, \bo)$ and
$$ h(x)= \begin{cases} (1-m) \frac{(d-\eta)^2}{\eta^2} + m  & \text{if $d(x)\le \eta$}\\
m & \text{if $d(x)> \eta$.}\end{cases}$$  We claim that if $\eta$
is well-chosen and $ 1-m$ is small enough, the function $h$ is an
interior barrier for $\hm$. Indeed $h(x)=f\circ d$ hence $\Delta h=
f''(d)|\nab d|^2 + f'(d)\Delta d$  with
$$f'(d)= \frac{2(1-m)}{\eta^2}(d-\eta)\indic_{ d \le \eta} ,
\quad  f''(d)= \frac{2(1-m)}{\eta^2}\indic_{ d\le \eta}.$$ In
addition we have $|\nab d|=1$ and $|\Delta d|\le \kappa_\om$, where
$\kappa_\om$ is the maximum of the curvature on $\bo$. Thus $-\Delta
h+h$ is trivially positive on $\{d\ge\eta\}$ while on the set
$\{d<\eta\}$ we have
$$-\Delta h+h\ge -2\frac{1-m}{\eta^2}  -2\frac{1-m}{\eta}\kappa_\om + m.$$
Letting  $\eta^2= (2+\delta)(1-m)$, the right-hand side is positive
for $m$ close enough to $1$ (depending on $\delta$).  Then $h$ is
an interior barrier for $\hm$ and we deduce that $ \{d^2 >
(2+\delta)(1-m)\} \subset \oma$. The first result in
(\ref{estsolob}), i.e.  $|\om \backslash \omega_m|\le C
\sqrt{1-m},$ follows immediately, for some appropriate $C$
depending on $\om$. The second assertion is a consequence of the
first one, together with the estimate
$$\|\hm\|_{C^{1,1}(\om) }\le C,$$
with $C$ independent of $m$ (see \cite{bk}).
Indeed, either  $x\in \omega_m$ and then $\nab \hm(x)=0$, or $x
\notin \omega_m$ and since there exists $y\in \omega_m$ such that
$|x-y|\le C \sqrt{1-m}$, the $C^{1,1}$ estimate implies that
$$|\nab \hm(x)|\le |\nab \hm(y)|+ C\sqrt{1-m} = C\sqrt{1-m}.$$
\end{proof}

\begin{proof}[Proof of 4)] Let $\omQ = \xom + \ella E_Q$, for some $\ella$ to be specified below, and define $h$ to be the solution of $-\Delta h+h = m\indic_{\omQ}$ in $\om$ and $h=1$ on $\p\om$.
Then we may express $h$ as
$$h(x)= m \int G_\om(x,y) \indic_\omQ(y) \, dy + h_0(x),$$
where  $\gom(\cdot,y)$ is the solution of $-\Delta \gom +\gom =
\delta_y$ in $\om$ and $\gom = 0$ on $\p\om$. We further split
$\gom$ as
$$\gom(x,y) = -\frac1{2\pi}\log|x-y| +\som(x,y),$$
where $\som(\cdot,y)$ solves $-\Delta \som +\som =
(2\pi)^{-1}\log|\cdot-y|$ in $\om$, thus is $C^1$ locally in $\om$.

Replacing accordingly in the expression of $h$ and  writing $x
 =
\xom +\ella x'$, we obtain (using the fact that the volume of $E_Q$
is $1$)
$$ h(x) = \homb -\frac {m}{2\pi} \ella^2\log\ella +\ella^2
m\som(\xom,\xom) + \ella^2\(\hal Q(x') -
\frac{m}{2\pi}\log\ast\indic_{E_Q}(x')\) + R(x),$$ where
$$ R(x) =
\(\hom(x) - \homb - \hal\ella^2Q(x')\)+ m \int\( \som(x,y)  -
\som(\xom,\xom)\)\indic_\omQ(y) \, dy.$$
The first term  in $R(x)$ is the remainder of the Taylor expansion of order $2$ of  $\hom$ at $\xom$. It is therefore $O(|x-\xom|^3)$ and its derivatives are $O(|x-\xom|^2)$ since $\hom$ is analytic. Since $|\omQ| = \ella^2$ and since $y\in\omQ\implies |y-\xom|\le C\ella$, the second term in $R(x)$ is clearly $O(\ella^2|x-\xom|)$ and, since   $\som$ is $C^1$, differentiating under the integral sign shows its derivatives are $O(\ella^2)$. Differentiating twice allows to bound its second derivatives by  $\|D^2\som\|_{L^q}\|\indic_\omQ\|_{L^{q'}} $ and using the equation satisfied by $\som$ we have that $\|D^2\som\|_{L^q}$ for every $q<+\infty$, while $\|\indic_\omQ\|_{L^{q'}}=\ella^{\frac2{q'}}$. Thus the second derivatives of the second term are bounded by $C_p \ella^p$, for every $p<2$. To sum up, for any $p<2$ we have
\begin{equation}\label{restebarriere}R(x)=O\(L_m^3+ |x-\xom|^3\),\  \nabla R(x)=O\(L_m^2+ |x-\xom|^2\),\  \nabla^2 R(x) = O\(L_m^p+ |x-\xom|\).\end{equation}
Then we note that $\hal Q - \frac\homb{2\pi}\log\ast\indic_{E_Q}$ is
equal to $U_Q+ C_{\om,Q}$ where $C_{\om, Q}$ is a constant. Indeed,
note that since $Q= D^2 h_0(x_0)$ and  $-\Delta h_0+h_0=0$ we have $\hal\Delta Q= \Delta
h_0(x_0) =h_0(x_0)=\homb$, and the claim follows from (\ref{EQ}).
Thus,
 if we choose for $\ella$ a solution of the following  equation
\begin{equation}\label{edla}
 m-\homb  = -\frac\homb{2\pi} \ella^2\log\ella + \ella^2\(
\homb\som(\xom,\xom)+ C_{\om,Q}\),\end{equation} we have
\begin{equation}\label{avantbarriere}
h(x) = m + \ella^2 U_Q(x') + R'(x),\end{equation} where
$$R'(x)= R(x) + \ella^2(\homb-m) \(
\frac{1}{2\pi}\log \ast \indic_{E_Q}(x')+ \frac{1}{2\pi}\log \ella -
S_\om(x_0,x_0)\) .$$ Since $|\homb - m |\le C \ella^2
\log \ella$, we easily deduce  that $R'$ satisfies the same
properties as $R$, i.e. (\ref{restebarriere}). Returning to
(\ref{edla}),  since the term $\ella^2\log\ella$ dominates, we
deduce that ${\ella}^2|\log\ella| \sim 2\pi (m - \homb)/\homb$ as
$m \to \homb$ and that \eqref{la} and  (\ref{valella}) holds. From
(\ref{avantbarriere})  and (\ref{restebarriere}) for $R'$, we deduce
that for any $M,\delta>0$ and if $m-\homb$ is small enough, then
\begin{equation}\label{hsets} \{U_Q < M-\delta\}\subset
 \left\{x'\mid \frac{h(\xom+\ella x') - m}{\ella^2} <  M\right\}
 \subset \{U_Q < M+\delta\}.\end{equation}
Indeed \eqref{restebarriere}, \eqref{avantbarriere} imply local
uniform convergence of $\frac{h(\xom+\ella x') - m}{\ella^2}$ to
$U_Q$ as $m-\homb$ decreases to $0$, which implies that $\ella$ decreases to $0$ by \eqref{valella}. Thus it suffices to check that $h(\xom+\ella x')\le
m+M\ella^2$ implies a uniform bound for $x'$. This is the case
because the equation satisfied by $h$ implies that
$\|h-\hom\|_\infty\to 0$ as $m\searrow\homb$. Thus $h(x)\le
m+M\ella^2$ implies that $x-\xom$ is small as $m\searrow\homb$.
Then,  $|x-\xom|^3 = o(|x-\xom|^2) = o(\ella^2 + \ella^2 U_Q(x'))$
since $U_Q(x')\sim\hal Q(x')$ as $|x'|\to +\infty$. Therefore $R'(x)
= o(\ella^2 +  \ella^2 U_Q(x'))$ and we deduce from
\eqref{avantbarriere} a bound for $U_Q(x')$, hence for $x'$. Note in particular that the minimum of $h(\xom + \ella x')$ is achieved in a fixed compact set of $\mr^2$, and then \eqref{restebarriere}, \eqref{avantbarriere} yield
\begin{equation}\label{minh} \bh:=\min_\om h = O(\ella^3  +m ).\end{equation}

Next, we  use $h$ to construct an exterior and an interior barrier for
$\hm$. The exterior barrier is defined by $\extb =
\max(h-\delta\ella^2,m)$. Where $\extb = m$ it is obvious that
$-\Delta \extb +\extb = m$. Then
by definition
$\extb(x)\neq m$ implies that
$h(x)-\delta\ella^2>m$ and then
from
\eqref{avantbarriere}--\eqref{restebarriere}  if $m-\homb$ is small
enough, we must have  $x\notin \omQ$ (recall
that $U_Q=0$ in $E_Q$). Therefore where $\extb >m$ we have
$$-\Delta \extb +\extb = -\Delta h + h-\delta\ella^2 = -\delta\ella^2.$$
Therefore $-\Delta \extb +\extb\le m\indic_{\{\extb=m\}}$. The
other properties of exterior barriers are trivially verified by
$\extb$ thus $\extb\le \hm$ and  then, using \eqref{hsets},
\begin{align}\label{exterior}
\oma\subset\{\extb = m\}
 &= \{h \le m+\delta\ella^2\}\subset \{\xom+\ella y\mid U_Q(y)\le 2\delta\}
 \\
\nonumber\{\hm\le m+ M\ella^2\}&\subset\{\extb\le m+
M\ella^2\}\subset \{h\le m+ (M+\delta)\ella^2\} \\
 \nonumber &\qquad\subset
\{\xom+\ella y\mid U_Q(y)\le M+2\delta\}.
\end{align}

For the interior barrier, let $\chi:\mr_+\to\mr_+$ be smooth and
such that $\chi = 1$ in a  neighborhood of $0$ and  $\chi_{|[1,+\infty)}=0$. Then let  $\bh
= \min(m,\min_\om h)$ and
$$\hint(x) = m + (h(x) - \bh)\vp(x),\quad \vp(x) = \chi\(\frac{d(x,\omQ^c)}{\delta\ella}\).$$
We have $-\Delta \hint + \hint= - \Delta h + m+h - \bh  = m- \bh\ge 0$ on $\omQ^c$ and
$\hint = 1+ (m-\bh)\ge 1$ on $\p\om$. Moreover, $\hint\ge m$. It
remains to check that $-\Delta \hint + \hint\ge 0$ in $\omQ$.

Since $U_Q = 0$ on $E_Q$ and using
\eqref{avantbarriere},\eqref{restebarriere}, we have on $\omQ$
$$\Delta \hint  = \vp\Delta h + 2\nabla\vp\cdot\nabla h + (h-\bh)\Delta\vp=  O\(\ella+ \frac{\ella^2}\delta+\frac{\ella^3}{\delta^2}\).
$$
Indeed, since $U_Q = 0$ on $\omQ$, we have
$h-\bh = O(\ella^3)$   and
$\nab(h-\bh) = O(\ella^2)$,   $\Delta(h-\bh) = O(\ella)$ in $\omQ$.
 It
follows that if $m-\homb$ is small enough depending on $\delta$,
then $-\Delta\hint+\hint \ge m/2$ in $\omQ$, finishing the proof
that $\hint$ is an interior barrier for $\hm$. Then $\hint\ge \hm$
in $\om$ and it follows easily using \eqref{hsets} that
\begin{align}\label{interieur}
\{d(x,\omQ^c)>\delta\ella\} &\subset \{\hint = m\}\subset\oma\\
\nonumber\{\hm\ge m+ M\ella^2\}&\subset \{\hint\ge m+
M\ella^2\}=\{h\ge m+ M\ella^2\}\\
\nonumber & \qquad \subset \{\xom+\ella y\mid U_Q(y)\ge M-\delta\}.
\end{align}
The relation $\{\hint\ge m+M\ella^2\} = \{h\ge m+M\ella^2\}$ follows from the fact that  $\hint = h$ outside $\omQ$ and that if $m-\homb$ is small, then both sets are disjoint from $\omQ$: Indeed if $x\in\omQ$ then from \eqref{restebarriere}, \eqref{avantbarriere} we have $h(x) - m = O(\ella^3)$ and $\hint(x) - m = O(\ella^3)$ from \eqref{minh}. Then, \eqref{item4} follows from
\eqref{exterior}, \eqref{interieur}. The fact that $|\oma|\sim \ella^2$ as $m \to \homb$
is an easy consequence.

\end{proof}

\noindent
{\sc Etienne Sandier}\\
Universit\'e Paris-Est,\\
LAMA -- CNRS UMR 8050,\\
61, Avenue du Général de Gaulle, 94010 Créteil. France\\
\& Institut Universitaire de France\\
{\tt sandier@u-pec.fr}\\ \\
{\sc Sylvia Serfaty}\\
UPMC Univ  Paris 06, UMR 7598 Laboratoire Jacques-Louis Lions,\\
 Paris, F-75005 France ;\\
 CNRS, UMR 7598 LJLL, Paris, F-75005 France \\
 \&  Courant Institute, New York University\\
251 Mercer st, NY NY 10012, USA\\
{\tt serfaty@ann.jussieu.fr}

\end{document}